\documentclass[11pt]{article}

\usepackage{mathrsfs}
\usepackage{changes}

\usepackage[T1]{fontenc}
\usepackage[utf8]{inputenc}
\usepackage{lmodern}
\usepackage{amssymb,amsmath,amsthm}
\usepackage{bbm}
\usepackage{graphicx}
\usepackage{float}
\usepackage{xcolor}
\usepackage[colorlinks,linkcolor=black!70!black,citecolor=black]{hyperref} % Or: hidelinks
\usepackage[a4paper,margin=1.5cm]{geometry}
\usepackage{tikz}
\usepackage{tikz-cd}
\usepackage{enumerate}

\usetikzlibrary{decorations.pathreplacing, patterns,shapes,snakes}

\usepackage{pgfplots}

\numberwithin{equation}{section}

\newcommand{\ra}{\rightarrow}

\newcommand{\be}{\begin{equation}}
\newcommand{\ee}{\end{equation}}
\newcommand{\bi}{\begin{itemize}}
\newcommand{\ei}{\end{itemize}}

\newcommand{\commentout}[1]{}

\newcommand{\Hm}{\mathbb{H}}

\newcommand{\Leb}{{\text{Leb}}}

\newcommand{\calO}{{\mathcal{O}}}

\newcommand{\Nm}{{\mathbb{N}}}

\newcommand{\Rm}{{\mathbb R}}
\newcommand{\R}{{\mathbb R}}

\newcommand{\Pm}{{\mathbb P}}

\newcommand{\expE}{{\mathbb E}}
\newcommand{\Ind}{{\mathbbm{1}}}

\newcommand{\Es}[1]{{\mathbb E}\left[#1\right]}

\newcommand{\p}[1]{{\mathbb P}\left(#1\right)}
\newcommand{\psub}[2]{{\mathbb P}_{#1}\left(#2\right)}
\newcommand{\Esub}[2]{{\mathbb E}_{#1}\left[#2\right]}
\def\1{\mathbbm{1}}

\newtheorem{theo}{Theorem}[section]
\newtheorem{lem}[theo]{Lemma}
\newtheorem{defin}[theo]{Definition}

\newtheorem{prop}[theo]{Proposition}
\newtheorem{cor}[theo]{Corollary}
\newtheorem{rmk}[theo]{Remark}

\newtheorem{assum}[theo]{Assumption}

\newtheorem{counterexample}[theo]{Counterexample}

\commentout{

\newcommand{\calO}{{\mathcal{O}}}

\newcommand{\Nm}{{\mathbb{N}}}

\newtheorem*{rep@theorem}{\rep@title}
\newcommand{\newreptheorem}[2]{%
\newenvironment{rep#1}[1]{%
 \def\rep@title{#2 \ref{##1}}%
 \begin{rep@theorem}}%
 {\end{rep@theorem}}}
\makeatother
\newreptheorem{theorem}{Theorem}
\newreptheorem{lemma}{Lemma}
\newreptheorem{defin}{Definition}
\newreptheorem{cond}{Condition}
\newreptheorem{prop}{Proposition}
\newreptheorem{lem}{Lemma}
\newreptheorem{cor}{Corollary}
\newreptheorem{rmk}{Remark}
\newreptheorem{exer}{Exercise}
\newreptheorem{conj}{Conjecture}
\newreptheorem{assum}{Assumption}
\newreptheorem{equation}{Equation}
\newreptheorem{problem}{Problem}

}

\newtheorem*{rep@theorem}{\rep@title}
\newcommand{\newreptheorem}[2]{%
\newenvironment{rep#1}[1]{%
 \def\rep@title{#2 \ref{##1}}%
 \begin{rep@theorem}}%
 {\end{rep@theorem}}}
\makeatother
\newreptheorem{theorem}{Theorem}
\newreptheorem{lemma}{Lemma}
\newreptheorem{defin}{Definition}
\newreptheorem{cond}{Condition}
\newreptheorem{prop}{Proposition}
\newreptheorem{lem}{Lemma}
\newreptheorem{cor}{Corollary}
\newreptheorem{rmk}{Remark}
\newreptheorem{exer}{Exercise}
\newreptheorem{conj}{Conjecture}
\newreptheorem{assum}{Assumption}
\newreptheorem{equation}{Equation}

\begin{document}
\setcounter{page}{1}

\title{Convergence and front position for an FKPP-type free boundary problem}
\author{Julien Berestycki\footnote{Department of Statistics, University of Oxford, UK},\, Sarah Penington\footnote{Department of Mathematical Sciences, University of Bath, UK}\, and Oliver Tough\footnote{Department of Mathematical Sciences, Durham University, UK}}
\date{\today}

\maketitle

\begin{abstract}
The free boundary problem\[ \begin{cases}
\partial_tu=\frac{1}{2}\Delta u+u,\quad &t>0, \, x>L_t,\\ u(t,x)=0,\quad &t>0,\, x\le L_t,\\ \int_{L_t}^{\infty}u(t,y)dy=1,\quad &t> 0,\\ u(t,x)dx \to u_0(dx)&\text{weakly as }t\to 0, \end{cases}\] has long been conjectured to be in the universality class of the so-called FKPP reaction-diffusion equation. It appears naturally as the hydrodynamic limit of a branching-selection particle system, the $N$-BBM. In the present work, we show that for any initial condition $u_0(dx)$ that decays fast enough as $x\to\infty$, the solution of the free boundary problem converges to the minimal travelling wave solution. We further show how the decay of the initial condition precisely determines the position of the free boundary $L_t$ at large times $t$, mirroring the celebrated results of Bramson \cite{Bramson1983} in the context of the FKPP equation. Our conditions for convergence to the minimal travelling wave, and for $L_t$ to have the Bramson asymptotics
\[ L_t=\sqrt{2}t-\frac{3}{2\sqrt{2}}\log t+c+o(1)\quad\text{as }t\to\infty,\]
are necessary and sufficient. We also apply our results to a more general free boundary problem that depends on a parameter $\beta$, where we see a transition from \emph{pulled} to \emph{pushed} behaviour (with \emph{pushmi-pullyu} behaviour at the critical value of $\beta$). We obtain analogous sharp conditions for convergence to the minimal travelling wave, along with precise asymptotics for the front position, in each of these regimes. To our knowledge, such necessary and sufficient conditions had not previously been established in the pushmi-pullyu or pushed regimes, even for classical monostable reaction-diffusion equations. Our results prove and extend non-rigorous predictions in the physics literature of the first author, Brunet and Derrida.
\end{abstract}

%\medskip
\tableofcontents

\section{Introduction and main results}\label{section:main}

We consider solutions of the following free boundary partial differential equation:
\begin{equation}\label{eq:FBP}
		\begin{cases}
		\partial_tu=\frac{1}{2}\Delta u+u,\quad &t>0, \, x>L_t,\\
		u(t,x)=0, \quad &t>0, \, x\le L_t,\\
		\int_{L_t}^{\infty}u(t,y)dy=1, \quad &t> 0,\\
		u(t,x )dx \to u_0(dx) &\text{weakly as } t\to 0.
	\end{cases}
\end{equation}
Existence and uniqueness of classical solutions of~\eqref{eq:FBP} was established in~\cite{Berestycki2018}; we will be more precise about the definition of a classical solution, and discuss existing (rigorous and non-rigorous) results about solutions of the free boundary problem, and the motivation for studying it, in Section~\ref{subsec:background} below.
In particular, $u(t,x)$ is the hydrodynamic limit of a branching-selection particle system known as the $N$-BBM~\cite{DeMasi2019} (see Section \ref{subsec:N-BBM}).

The minimal travelling wave solution of~\eqref{eq:FBP}, i.e.~the unique travelling wave solution with minimal wave speed $\sqrt 2$, is given by
\begin{equation}\label{eq:minimal travelling wave}
	u(t,x)=\pi_{\min}(x-\sqrt{2}t) \; \text{ for }t>0,\, x\in \R, \quad \text{where }\; \pi_{\min}(y)=2ye^{-\sqrt{2}y}\1_{\{y>0\}} \;\text{ for }y\in \R.
\end{equation}
For each $c> \sqrt{2}$, a unique travelling wave solution with speed $c$ exists, given by $u(t,x)=\pi_{c}(x-c t)$ for $t>0$, $x\in \R,$ where 
\begin{equation}\label{eq:speed c travelling wave}
	\pi_{c}(y)=(c^2-2)^{-1/2}[e^{(-c+\sqrt{c^2-2})y}-e^{(-c-\sqrt{c^2-2})y}]\1_{\{y>0\}} \;\text{ for }y\in \R.
\end{equation}
We let $\pi_{\sqrt 2}=\pi_{\min}$.

Writing $U(t,x):=\int_x^{\infty}u(t,y)dy$ for $t>0$ and $x\in \R$, where $(u(t,x),L_t)$ is the solution of~\eqref{eq:FBP}, we have that $(U(t,x),L_t)$ solves the free boundary problem
\begin{equation}\label{eq:FBP_CDF}
	\begin{cases}
		\partial_tU=\frac{1}{2}\Delta U+U,\quad &t>0, \; x>L_t,\\
		U(t,x)= 1,\quad &t>0, \; x\leq L_t,\\
		\partial_xU(t, L_t)= 0,\quad &t>0,\\
		U(t,\cdot ) \to U_0(\cdot) \quad &\text{in }L^1_{\mathrm{loc}}\text{ as } t\to 0,
	\end{cases}
\end{equation}
where $U_0(x):=u_0((x,\infty))$ for $x\in \R$.
Moreover, if $(U(t,x),L_t)$ solves~\eqref{eq:FBP_CDF}, then letting $u(t,x)=-\partial_x U(t,x)$ for $t>0$ and $x\in \R$, we have that $(u(t,x),L_t)$ solves~\eqref{eq:FBP}, i.e.~the solutions of \eqref{eq:FBP} and \eqref{eq:FBP_CDF} are in one-to-one correspondence.
Although~\eqref{eq:FBP} is the formulation that makes the most intuitive sense when thinking about the hydrodynamic limit of the $N$-BBM, in this article it will almost always be convenient to work with the integrated version~\eqref{eq:FBP_CDF} of the free boundary problem.

\subsection{Main results} \label{subsec:mainresults}

We say that $(U(t,x),L_t)$ is a classical solution to~\eqref{eq:FBP_CDF} if:
\begin{enumerate}[(i)]
\item  $L_t\in \R$ $\forall t>0$ and $t\mapsto L_t$ is continuous on $(0,\infty)$,
\item  $U :  (0,\infty) \times \R  \to [0, 1]$ with $U \in C^{1,2}(\{(t,x):t>0,\, x>L_t\})\cap C( (0,\infty) \times \R)$,
\item  $(U,L)$ satisfies~\eqref{eq:FBP_CDF}.
\end{enumerate}
In the present article, we will always be dealing with classical solutions of~\eqref{eq:FBP_CDF}. 
We will always make the following (equivalent) assumptions on the initial conditions $U_0$ and $u_0$ of~\eqref{eq:FBP_CDF} and~\eqref{eq:FBP} respectively.
\begin{assum}[Initial condition] \label{assum:standing assumption ic}
Suppose $U_0:\Rm\to [0,1]$ is a c\`adl\`ag non-increasing function with $U_0(x)\to 1$ as $x\to -\infty$ and $U_0(x)\to 0$ as $x\to \infty$.
Equivalently, suppose $u_0$ is a Borel probability measure on $\R$.
\end{assum}
An existence and uniqueness result for classical solutions of~\eqref{eq:FBP_CDF}, for initial conditions satisfying Assumption \ref{assum:standing assumption ic}, will be given in Proposition~\ref{prop:fbpsoln} below (which follows from results in~\cite{Berestycki2018}).

The minimal travelling wave solution of~\eqref{eq:FBP_CDF} has shape given by
\begin{equation} \label{eq:Pimindefn}
\Pi_{\min}(x):=\int_x^{\infty}\pi_{\min}(y)dy\;\text{ for } x\in \R,
\end{equation}
and wave speed given by $\sqrt 2$.
For $c\ge\sqrt 2$, we write
\begin{equation} \label{eq:Picdefn}
\Pi_{c}(x):=\int_x^{\infty}\pi_{c}(y)dy\;\text{ for } x\in \R
\end{equation}
for the shape of the travelling wave solution with speed $c$ (recalling that we let $\pi_{\sqrt 2}=\pi_{\min}$).
Our first main result characterises the domain of attraction of the minimal travelling wave solution of~\eqref{eq:FBP_CDF}. 
\begin{theo} \label{theo:convtoPimin}
Let $U_0$ satisfy Assumption~\ref{assum:standing assumption ic}, and let $(U(t,x),L_t)$ denote the solution of the free boundary problem~\eqref{eq:FBP_CDF}.	Then the following are equivalent:
	\begin{enumerate}
		\item $\limsup_{x\ra\infty}\frac{1}{x}\log U_0(x)\leq -\sqrt{2}$;
		\item $\limsup_{t\ra\infty}\frac{L_t}{t}\leq \sqrt{2}$;
		\item $\lim_{t\ra\infty}\frac{L_t}{t}= \sqrt{2}$;\label{enum:convergence of velocity}
		\item $U(t,x+L_t)\ra \Pi_{\min}(x)$ uniformly in $x$ as $t\ra\infty$. \label{enum:convergence of profile}
	\end{enumerate}
\end{theo}
Informally, Theorem~\ref{theo:convtoPimin} tells us that if the initial condition $U_0(x)$ decays sufficiently quickly as $x\to \infty$, then the asymptotic speed of the front is $\sqrt{2}$ and the front shape converges to the minimal travelling wave shape $\Pi_{\min}$.
This is, of course, precisely the behaviour expected here by analogy with the classical FKPP equation~\cite{Bramson1983,Kolmogorov1937} (see Section~\ref{subsubsec:FKPPlinearised}). 
The next natural question is the position of the front at large times; in a manner akin to the classical FKPP equation~\cite{Bramson1983}, the answer depends in a subtle way on the decay of $U_0(x)$ as $x\to \infty$.
\begin{theo} \label{theo:Ltposition}
	Suppose that $U_0$ satisfies Assumption~\ref{assum:standing assumption ic}, and define
	\[
	 b(t)=2^{-1 / 2} \log \left(\int_0^{\infty} y e^{\sqrt 2 y} U_0(y) e^{-y^2 / (2 t)} d y+1\right) \quad \text{for }t>0.
	\]
	Suppose that for some $\gamma<1/2$,
\begin{equation}\label{eq:stretched exponential U0}
U_0(x)\leq e^{x^{\gamma}-\sqrt{2}x}
\end{equation}
for all $x$ sufficiently large.
	\commentout{
	
	for some $\delta<1 / 3$, for $t$ sufficiently large,
	\begin{equation}\label{eq:assumption on b Lt asymptotics}
	b(t)\le t^\delta.
	%\int_0^{\infty} y e^{\sqrt 2 y} U_0(y) e^{-y^2 / 2 t} d y \leq e^{t^\delta}.
	\end{equation}
	}Let $(U(t,x),L_t)$ denote the solution of the free boundary problem~\eqref{eq:FBP_CDF}, and let
	\[
	m(t)=\sqrt 2 t-\frac 3{2\sqrt 2} \log t+b(t)  \quad \text{for }t>0.
	\]
	Then there exists $a=a(U_0)\in \R$ such that
	\[
	L_t-m(t) \to a \quad \text{as }t\to \infty.
	\]	
		In particular, if $\int_0^{\infty} y e^{\sqrt 2 y} U_0(y)  d y<\infty$, then there exists $c=c(U_0)\in \R$ such that 
	\[
	L_t = \sqrt 2 t -\frac3{2\sqrt 2 } \log t +c +o(1) \quad \text{as }t\to \infty.
	\]
If instead $\int_0^{\infty} y e^{\sqrt 2 y} U_0(y)  d y=\infty$, then $b(t)\to \infty$ as $t\to \infty$, and 
\[
	a(U_0)= -\frac 1 {\sqrt{2}} \log \sqrt{\pi}.
	\]
Moreover, if $\int_0^{\infty} y e^{\sqrt 2 y} U_0(y)  d y=\infty$, then the conclusion 
\[
L_t-\left[\sqrt{2}t-\frac{3}{2\sqrt{2}}\log t\right]\ra \infty \quad \text{as }t\to \infty
\]
remains valid without assuming \eqref{eq:stretched exponential U0}.
\end{theo}
\begin{defin}[Finite and infinite initial mass]\label{defin: fin init mass}
In the remainder of the article, the cases 
\begin{align}
\int_0^{\infty} y e^{\sqrt 2 y} U_0(y)d y<\infty  \label{eq:finite initial mass U0}\\
\text{and }\qquad \int_0^{\infty} y e^{\sqrt 2 y} U_0(y)  d y=\infty \label{eq:infinite initial mass U0}
\end{align}
will be referred to as the \emph{finite initial mass} and \emph{infinite initial mass} cases, respectively, as in~\cite{Bramson1983}. 
\end{defin}
Observe that for $U_0$ satisfying Assumption~\ref{assum:standing assumption ic}, by integration by parts we have that $U_0$ satisfies~\eqref{eq:finite initial mass U0} if and only if the probability measure $u_0$, defined by $U_0(x):=u_0((x,\infty))$ for $x\in\Rm$, satisfies
\begin{equation}\label{eq:finite initial mass little u0}
\int_{[0,\infty)} y e^{\sqrt 2 y} u_0(dy)<\infty.
\end{equation}
We will therefore equivalently refer to \eqref{eq:finite initial mass little u0} as the finite initial mass case, and
\begin{equation}\label{eq:infinite initial mass little u0}
\int_{[0,\infty)} y e^{\sqrt 2 y} u_0(dy)=\infty
\end{equation}
as the infinite initial mass case.

We treat these cases separately in the proof; we first prove Theorems~\ref{theo:convtoPimin} and~\ref{theo:Ltposition} under the assumption of finite initial mass,
and then complete the proofs of Theorems~\ref{theo:convtoPimin} and~\ref{theo:Ltposition} by covering the infinite initial mass case (see Section~\ref{subsec:outline} for an outline of the proof).
\begin{rmk}\label{rmk:FKPP asymptotics heavy-tailed}
As an illustrative example of the results in Theorem~\ref{theo:Ltposition}, suppose $U_0(x) \sim Ax^\nu e^{-\sqrt 2 x}$ as $x\to \infty$ for some fixed $\nu\in \R$ and $A\in (0,\infty)$, and let $(U(t,x),L_t)$ solve~\eqref{eq:FBP_CDF}. Then by Theorem~\ref{theo:Ltposition},
\begin{enumerate}
\item if $\nu<-2$, then there exists $c\in \R$ such that $L_t =\sqrt 2 t -\frac{3}{2\sqrt 2} \log t + c +o(1)$ as $t\to \infty$.
\item if $\nu=-2$, then $L_t=\sqrt{2}t-\frac{3}{2\sqrt{2}}\log t+\frac{1}{\sqrt{2}}\log \log t+\frac{1}{\sqrt{2}}\log \left(\frac{A}{2\sqrt \pi}\right)+o(1)$ as $t\to \infty$.
	\item if $\nu>-2$, then $L_t=\sqrt{2}t+\frac{\nu-1}{2\sqrt{2}}\log t+\frac{1}{\sqrt{2}}\log\left(\frac A {\sqrt \pi}\int_{0}^{\infty} y^{1+\nu} e^{-y^2 / 2 } d y\right)+o(1)$ as $t\to \infty$.
\end{enumerate}
Analogous behaviour is seen for the FKPP equation~\cite{Bramson1983} and for an FKPP boundary value problem with a given boundary~\cite{BBHR17} (see Section~\ref{subsubsec:FKPPlinearised}).
\end{rmk}

\begin{rmk}
Theorem~\ref{theo:Ltposition} tells us that the asymptotics $L_t =\sqrt 2 t -\frac{3}{2\sqrt 2} \log t + c +o(1)$ as $t\to \infty$ for some $c\in \mathbb{R}$ hold if and only if $U_0$ satisfies the finite initial mass condition~\eqref{eq:finite initial mass U0}, and otherwise $L_t-[\sqrt 2 t -\frac{3}{2\sqrt 2} \log t]\ra \infty$ as $t\ra\infty$. We will establish a similar characterisation in the pushmi-pullyu and pushed settings for solutions of a generalised free boundary problem in Section \ref{section:proof on generalised FBP} (see Remark \ref{rmk:necessary and sufficient for asymptotics generalised}).\end{rmk}

We can also characterise the domains of attraction of the faster travelling wave solutions $\Pi_c$ for $c>\sqrt 2$, and if $U_0(x)$ decays exponentially as $x\to \infty$ at a slower rate than $\sqrt{2}$, we can determine the long-term asymptotics of $(U(t,x),L_t)$.
\begin{theo} \label{theo:slowerdecay}
Let $U_0$ satisfy Assumption~\ref{assum:standing assumption ic}, and let $(U(t,x),L_t)$ denote the solution of the free boundary problem~\eqref{eq:FBP_CDF}.	
For $t> 0$, let
\begin{equation} \label{eq:mtslowdecay}
m(t):=\sup\left\{x\in \R:e^t \int_{-\infty}^\infty U_0(y)\frac{e^{-(x-y)^2/(2t)}}{\sqrt{2\pi t}}dy\ge 1\right\}.
\end{equation}
Then for any $c>\sqrt{2}$, the following are equivalent:
	\begin{enumerate}
		\item $\lim_{x\ra\infty}\frac{1}{x}\log U_0(x)=-c+\sqrt{c^2-2}$;\label{enum:asymptotics U0 heavy tailed for speed c}
		\item $U(t,x+L_t)\ra \Pi_c(x)$ uniformly in $x$ as $t\ra\infty$.\label{enum:convergence to c travelling wave}
	\end{enumerate}
		Moreover, if \ref{enum:asymptotics U0 heavy tailed for speed c} or equivalently \ref{enum:convergence to c travelling wave} holds, then
	\[
	\lim_{t\to\infty}\frac{L_t}t =c \quad \text{and}\quad \lim_{t\to\infty} \left(L_t-m(t)-\frac{1}{c-\sqrt{c^{2}-2}} \left(\log \sqrt{c^{2}-2}+\log\Big(c-\sqrt{c^{2}-2}\Big)\right)\right)=0.
	\]
\end{theo}

\begin{rmk}
Theorems \ref{theo:convtoPimin} and \ref{theo:slowerdecay} imply that if
\[
\liminf_{x\rightarrow \infty}\tfrac{1}{x}\log U_0(x)< \limsup_{x\rightarrow \infty}\tfrac{1}{x}\log U_0(x),
\]
and if the latter is strictly greater than $-\sqrt{2}$, then $U(t,L_t+\cdot)$ cannot converge to any travelling wave as $t\rightarrow \infty$.
\end{rmk}
\medskip

The following result (sometimes referred to as {\it the magical relation}~\cite{Brunet2023}) is an important step in the proof of Theorems~\ref{theo:convtoPimin} and~\ref{theo:Ltposition}, and is also of interest in its own right. It relates the Laplace transform of the initial condition $U_0$ to another integral transform of the free boundary $(L_t)_{t>0}$. It was first developed by Brunet and Derrida in a discrete setting~\cite{BD2015} and then employed non-rigorously in the context of the present free boundary problem and of the FKPP equation ~\cite{Berestycki2018a,Berestycki2018b}; see also the recent work~\cite{Brunet2023} where an analogue of this relation is used to obtain fine asymptotics of the solution to the FKPP equation ahead of the front. 

\begin{theo}[Brunet-Derrida relation] \label{theo:magic formula}
Suppose $U_0$ satisfies Assumption~\ref{assum:standing assumption ic}, and let $(U(t,x),L_t)$ denote the solution of~\eqref{eq:FBP_CDF}.
Let $L_0=\inf\{x\in \R:U_0(x)<1\}\in \{-\infty\}\cup \R.$
Then for any $r\in (-\infty,\sqrt{2})\setminus \{0\}$,
	\begin{equation}\label{eq:magic formula}
		\int_{L_0}^{\infty}U_0(x)e^{rx}dx=-\frac{1}{r}e^{rL_0}+\frac{1}{r}\int_0^{\infty}e^{rL_t-(1+\frac{1}{2}r^2)t}dt.
	\end{equation}
\end{theo} 
Note that the statement of Theorem~\ref{theo:magic formula} in the cases $r\in (0,\sqrt 2)$ or $L_0=-\infty$ includes the possibility that both sides of~\eqref{eq:magic formula} are $+\infty$ (both sides are necessarily finite if $r<0$ and $L_0\in \R$; see Lemma~\ref{lem:boundary locally Lipschitz from the left} below).

The (non-rigorous) derivation of this relation  in~\cite{Berestycki2018a} included the condition that $U(t,L_t+x) \to \Pi_c(x)$ uniformly in $x$ as $t\to \infty$, for some $c\ge \sqrt 2$. No such condition is required here. Nonetheless, many of the key steps of the proof of Theorem~\ref{theo:magic formula} can be found  in~\cite{Berestycki2018a}. Our statement also clarifies when the relation is valid. In particular, we show in Section~\ref{section:magic formula} that for $r\in (\sqrt{2},\infty)$, both sides of~\eqref{eq:magic formula} may be finite whilst~\eqref{eq:magic formula} fails to hold (see Counterexample~\ref{counterexample:magic formula}).

We obtain the following relationship between $U_0$ and $\limsup_{t\ra\infty}\frac{1}{t}L_t$ as a corollary of Theorem~\ref{theo:magic formula}. For $U_0$ satisfying Assumption~\ref{assum:standing assumption ic},  let 
\begin{equation} \label{eq:r0U0defn}
r_0(U_0):=\sup\left( \{0\}\cup \{r\in (0,\sqrt{2}):\int_{0}^{\infty}e^{rx}U_0(x)dx<\infty\}\right).
\end{equation}
 We can define the above integral from $0$ to $+\infty$ (instead of from $L_0$ to $+\infty$) since $\int_{-\infty}^{y}e^{rx}U_0(x)dx$ is always finite, for any $r\in (0,\sqrt{2})$ and any $y\in \R$.
\begin{theo}\label{theo:initial condition limsup bdy relation}
Suppose $U_0$ satisfies Assumption~\ref{assum:standing assumption ic}, and let $(U(t,x),L_t)$ denote the solution of~\eqref{eq:FBP_CDF}.
Define $r_0=r_0(U_0)$ as in~\eqref{eq:r0U0defn}.
Then
\begin{equation}
\limsup_{t\ra\infty}\frac{1}{t}L_t=
\begin{cases}
\frac{1}{r_0}+\frac{r_0}{2} \quad &\text{if }r_0>0,\\
\infty \quad &\text{if }r_0=0.
\end{cases}
\end{equation}
\end{theo} 
Another simple consequence of Theorem~\ref{theo:magic formula} is that two solutions of~\eqref{eq:FBP_CDF} with different initial conditions cannot have the same free boundary.
\begin{theo} \label{theo:samebdysameU0}
	Suppose that $U_0^{(1)}$ and $U_0^{(2)}$ satisfy Assumption~\ref{assum:standing assumption ic}, and let $(U^{(1)}(t,x),L^{(1)}_t)$ and
	$(U^{(2)}(t,x),L^{(2)}_t)$ denote the solutions of~\eqref{eq:FBP_CDF} with initial conditions $U_0^{(1)}$ and $U_0^{(2)}$ respectively.
If $L^{(1)}_t=L^{(2)}_t$ for all $t>0$, then $U_0^{(1)}=U_0^{(2)}$.
\end{theo}

\medskip

In the remainder of this introduction, we present applications of Theorems~\ref{theo:convtoPimin},~\ref{theo:Ltposition},~\ref{theo:slowerdecay} and~\ref{theo:initial condition limsup bdy relation} to a generalised free boundary problem in Section~\ref{subsec:pushedFBP}, and then in Section~\ref{subsec:background}
we discuss background and related work.
In Section~\ref{subsec:notation}, we introduce some general notation, and in Section~\ref{subsec:outline} we give an outline of the proofs of our main results, along with a heuristic for the difference between the finite and infinite initial mass cases.
Sections~\ref{section:properties of free-boundary}-\ref{section:proof on generalised FBP}
contain the proofs of our results.
In Section \ref{section:properties of free-boundary}, we collect basic properties of solutions of the free boundary problem~\eqref{eq:FBP_CDF}. This includes Lemma~\ref{lem:FKforinfinitemass}, a novel Feynman-Kac representation for~\eqref{eq:FBP_CDF}, which will be fundamental to our proof in the infinite initial mass setting. Then, in Section \ref{section:magic formula}, we prove the Brunet-Derrida relation, Theorem~\ref{theo:magic formula}, along with Theorems~\ref{theo:initial condition limsup bdy relation} and~\ref{theo:samebdysameU0}.
In Section~\ref{sec:finitemass}, we prove the results stated in Theorems~\ref{theo:convtoPimin} and~\ref{theo:Ltposition} under the assumption of finite initial mass, and then in Section~\ref{sec:infiniteinitialmass}, we prove the results stated in Theorems~\ref{theo:convtoPimin} and~\ref{theo:Ltposition} under the assumption of infinite initial mass.
In Section~\ref{sec:mainthmpfs}, we combine the results proved in previous sections to complete the proofs of  Theorems~\ref{theo:convtoPimin},~\ref{theo:Ltposition} and~\ref{theo:slowerdecay}.
Finally, in Section~\ref{section:proof on generalised FBP} we state full versions of the results outlined in Section~\ref{subsec:pushedFBP}, and use Theorems~\ref{theo:convtoPimin},~\ref{theo:Ltposition},~\ref{theo:slowerdecay} and~\ref{theo:initial condition limsup bdy relation} to prove these results.

\subsection{Pushed and pulled fronts: a generalised version of \eqref{eq:FBP_CDF}}\label{subsec:pushedFBP}

In this subsection, we show some applications of our main results to the following generalisation of the free boundary problem~\eqref{eq:FBP_CDF},
which depends on a parameter $\beta\in \mathbb{R}$:
\begin{equation}\label{eq:generalised FBP_CDF}
	\begin{cases}
		\partial_tV=\frac{1}{2}\Delta V+V,\quad &t>0, \; x>L_t,\\
		V(t,x)= 1,\quad &t>0, \; x\leq L_t,\\
		\partial_xV(t, L_{t}+)= -\beta,\quad &t>0,\\
		V(t,\cdot ) \to V_0(\cdot) \quad &\text{in }L^1_{\mathrm{loc}}\text{ as } t\to 0.
	\end{cases}
\end{equation}
Note that setting $\beta=0$ in~\eqref{eq:generalised FBP_CDF} gives us the free boundary problem~\eqref{eq:FBP_CDF}.
This more general free boundary problem~\eqref{eq:generalised FBP_CDF} was studied in~\cite{Berestycki2018a}, where the first author, Brunet and Derrida showed (non-rigorously) that there is a transition from pulled to pushed behaviour at $\beta=\sqrt{2}$. 
(Note that compared to the notation in~\cite{Berestycki2018a}, we have a factor $\frac 12$ in front of the Laplacian and a minus sign in front of $\beta$.)
The same family of free boundary problems is analysed in forthcoming work of Demircigil and Henderson~\cite{Demircigil2025}, using rigorous PDE methods (at the end of this subsection, we will compare our results to the results in~\cite{Demircigil2025}).

We say that $(V(t,x),L_t)$ is a classical solution of~\eqref{eq:generalised FBP_CDF} if:
\begin{enumerate}[(i)]
\item  $L_t\in \R$ $\forall t>0$ and $t\mapsto L_t$ is continuous on $(0,\infty)$,
\item  $V :  (0,\infty) \times \R  \to [0, 1]$ with $V \in C^{1,2}(\{(t,x):t>0,\, x>L_t\})\cap C( (0,\infty) \times \R)$,
\item $\partial_x V(t_n,x_n)\to -\beta$ as $n\to \infty$ for any $t>0$ and $(t_n,x_n)\to (t,L_t)$ with $x_n>L_{t_n}$ $\forall n\in \mathbb N$,\label{enum:continuous along boundary -beta}
\item  $(V,L)$ satisfies~\eqref{eq:generalised FBP_CDF}.
\end{enumerate}
We will always be dealing with classical solutions of~\eqref{eq:generalised FBP_CDF} in this article. 

The notion of classical solution here differs from the notion of classical solution of~\eqref{eq:FBP_CDF} stated at the start of Section~\ref{subsec:mainresults}, in that here we assume that $\partial_xV(t,x)$ converges to $-\beta$ continuously along the boundary (condition~\eqref{enum:continuous along boundary -beta} above), whereas classical solutions of \eqref{eq:FBP_CDF} are only assumed to satisfy $\partial_xU(t,L_t)=0$ pointwise in $t>0$. This is because it is proved in~\cite{Berestycki2018a} that if $\partial_xU$ vanishes along the boundary pointwise in $t>0$, and the other conditions of a classical solution are satisfied, then $\partial_xU$ vanishes continuously along the boundary. Nevertheless, it is unclear (to us) whether this implication holds for~\eqref{eq:generalised FBP_CDF}, and the prescription of the derivative along the boundary in a continuous manner (as in condition~\eqref{enum:continuous along boundary -beta} above) is classical for free boundary problems (see e.g. \cite[p.216, Section 8]{Friedman1964} and \cite[p.4698, end of Section 1.2]{Chen2022}).

We will consider the case $\beta>0$, and impose the following standing assumption on the initial condition $V_0$ (which is the same as our Assumption~\ref{assum:standing assumption ic} for $U_0$):
 \begin{assum}\label{assum:standing assumption initial condition general beta}
Suppose $V_0:\Rm\to [0,1]$ is a c\`adl\`ag non-increasing function with $V_0(x)\to 1$ as $x\to -\infty$ and $V_0(x)\to 0$ as $x\to \infty$.
\end{assum}
We firstly provide an existence and uniqueness result for classical solutions of~\eqref{eq:generalised FBP_CDF}. 
\begin{theo}\label{theo:existence uniqueness classical solution V}
Take $\beta>0$, and
suppose that $V_0$ satisfies Assumption~\ref{assum:standing assumption initial condition general beta}. 
Then there exists a unique classical solution of~\eqref{eq:generalised FBP_CDF} with initial condition $V_0$.
\end{theo}

The following result shows that a solution of~\eqref{eq:generalised FBP_CDF} with initial condition $V_0$ can be constructed using the solution of~\eqref{eq:FBP_CDF} with an initial condition $U_0$ given by a function of $\beta$ and $V_0$; this result will allow us to apply our main results in Section~\ref{subsec:mainresults} to~\eqref{eq:generalised FBP_CDF}. This mapping from solutions of~\eqref{eq:FBP_CDF} to solutions of~\eqref{eq:generalised FBP_CDF} was first introduced in~\cite[Section 4.2]{Berestycki2018a}; here we make the mapping
fully rigorous under the aforementioned assumptions on $\beta$ and $V_0$. We will also construct, in Proposition~\ref{prop:mapping inversion} below, a mapping which is inverse to this (an argument that did not appear in~\cite{Berestycki2018a}, even non-rigorously). Taken together, these results will allow us to establish Theorem~\ref{theo:existence uniqueness classical solution V} from the existence and uniqueness of classical solutions to~\eqref{eq:FBP_CDF}.
\begin{theo}\label{theo:mapping between FBP and general problem}
Take $\beta>0$, and
suppose that $V_0$ satisfies Assumption~\ref{assum:standing assumption initial condition general beta}. 
Let
\begin{equation}\label{eq:U0 formula from V0}
U_0(x):=
\frac{2}{\beta}e^{-\frac{2}{\beta}x}\int_{-\infty}^xe^{\frac{2}{\beta}z}V_0(z)dz
\quad \forall x \in \R.
\end{equation}
Then $U_0$ satisfies Assumption~\ref{assum:standing assumption ic} (and in particular is non-increasing). 

Now let $(U(t,x),L_t)$ denote the unique classical solution of~\eqref{eq:FBP_CDF} with initial condition $U_0$, and define
\begin{equation}\label{eq:V defined in terms of U}
V(t,x):=\begin{cases}
U(t,x)+\frac{\beta}{2}\partial_xU(t,x),\quad & x>L_t,\; t>0,\\
1,\quad & x\leq L_t,\; t>0.
\end{cases}
\end{equation}
Then $(V(t,x),L_t)$ is a classical solution of~\eqref{eq:generalised FBP_CDF} with initial condition $V_0$. Moreover, $V(t,\cdot)$ is non-increasing for all $t> 0$. 
\end{theo}

We now outline our results about the long-term behaviour of classical solutions of the generalised free boundary problem~\eqref{eq:generalised FBP_CDF}. These results will be a straightforward consequence of applying our main results on the free boundary problem~\eqref{eq:FBP_CDF}, Theorems~\ref{theo:convtoPimin},~\ref{theo:Ltposition},~\ref{theo:slowerdecay} and~\ref{theo:initial condition limsup bdy relation}, with initial condition $U_0$ given by~\eqref{eq:U0 formula from V0}, and then using Theorems~\ref{theo:existence uniqueness classical solution V} and~\ref{theo:mapping between FBP and general problem}.
The full versions of these results are rather lengthy, and so we postpone them to Section~\ref{section:proof on generalised FBP}, where we also prove the results (and prove Theorems~\ref{theo:existence uniqueness classical solution V} and~\ref{theo:mapping between FBP and general problem}).

Firstly, for each $\beta>0$ we let $\Pi^{(\beta)}_{\min}$ denote the non-negative travelling wave solution of~\eqref{eq:generalised FBP_CDF} with minimal wave speed $c^{(\beta)}_{\min}$ (see Proposition~\ref{prop:travelling waves general beta} in Section~\ref{section:proof on generalised FBP}).
In Theorem~\ref{theo:convtoPimin general beta} in Section~\ref{section:proof on generalised FBP}, we show that if $V_0$ satisfies Assumption~\ref{assum:standing assumption initial condition general beta}, then letting $(V(t,x),L_t)$ denote the solution of~\eqref{eq:generalised FBP_CDF},
\[
V(t,x+L_t)\ra \Pi^{(\beta)}_{\min}(x) \quad \text{ uniformly in }x \text{ as }t\ra\infty
\]
if and only if $\limsup_{x\ra\infty}\frac{1}{x}\log V_0(x)\leq -\min(\sqrt{2},\frac{2}{\beta})$.

Secondly, in Theorem~\ref{theo:Ltposition general beta} in Section~\ref{section:proof on generalised FBP}, we determine the long-term asymptotics of the front position for solutions of~\eqref{eq:generalised FBP_CDF} with initial condition $V_0$ satisfying Assumption~\ref{assum:standing assumption initial condition general beta}, under the condition that 
\[
\limsup_{x\ra\infty}\tfrac{1}{x}\log V_0(x)\leq -\min(\sqrt{2},\tfrac{2}{\beta}).
\]
We let 
\begin{equation} \label{eq:Ibetaintro}
I_\beta:=
\begin{cases}
\int_0^{\infty}xe^{\sqrt{2}x}V_0(x)dx \quad &\text{if }0<\beta< \sqrt{2}, \\
\int_{-\infty}^{\infty}e^{\frac{2}{\beta}x}V_0(x)dx \quad &\text{if }\beta\geq  \sqrt{2}.
\end{cases}
\end{equation}
For simplicity, here we also assume that $I_\beta<\infty$.
This turns out to be the equivalent of the finite initial mass case for solutions of~\eqref{eq:generalised FBP_CDF};
in the full statement of Theorem~\ref{theo:Ltposition general beta} we also have results in the infinite initial mass case $I_\beta=\infty$.
Let $(V(t,x),L_t)$ denote the solution of~\eqref{eq:generalised FBP_CDF}; we show that
\begin{enumerate}
\item if $0<\beta <\sqrt{2}$ then $L_t=\sqrt{2}t-\frac{3}{2\sqrt{2}}\log t+c+o(1)$ as $t\ra\infty$, for some $c=c(V_0)\in \R$;
\item if $\beta=\sqrt{2}$ then $L_t=\sqrt{2}t-\frac{1}{2\sqrt{2}}\log t+\frac{1}{\sqrt{2}}\Big(\log(\sqrt 2 I_{\sqrt 2})-\log\sqrt{\pi}\Big)+o(1)$ as $t\ra\infty$;
\item if $\beta>\sqrt{2}$ then $L_t=c^{(\beta)}_{\min}t+\frac{\beta}{2} \left(\log (\frac{2}{\beta}I_{\beta})+\log(\beta^2-2)-2\log \beta\right)+o(1)$ as $t\ra\infty.$
\end{enumerate}
Note that in the cases $\beta=\sqrt{2}$ and $\beta>\sqrt{2}$, the constant term in the asymptotics is given by an explicit function of the initial condition $V_0$.

In the three regimes 1-3 above, we see a change in behaviour from \emph{pulled} behaviour when $\beta<\sqrt{2}$, to \emph{pushmi-pullyu} behaviour when $\beta=\sqrt{2}$, to \emph{pushed} behaviour when $\beta>\sqrt{2}$.
Roughly speaking, a front expansion is pulled when it is driven by the behaviour far ahead of the front, and pushed when it is driven by the behaviour close to the front; pushmi-pullyu behaviour is seen at the transition between these two regimes.
For initial conditions decaying sufficiently fast,
in pulled fronts (see 1~above) there is a logarithmic correction in the position of the front, whereas in pushmi-pullyu fronts (see 2~above) there is a different (smaller) logarithmic correction (first observed in~\cite{An2023b,Giletti2022,An2023a}),
and finally in pushed fronts (see 3~above) there is no logarithmic correction. 
The reader is referred to~\cite{Garnier2012,An2024} for more background on pulled, pushmi-pullyu and pushed fronts.

The pushmi-pullyu and pushed behaviour of the front position described above comes, via the mapping in Theorem~\ref{theo:mapping between FBP and general problem}, from the behaviour of solutions of~\eqref{eq:FBP_CDF} with heavy-tailed initial conditions (see the proof of Theorem~\ref{theo:Ltposition general beta} in Section~\ref{section:proof on generalised FBP}). This is perhaps surprising, as these are very different settings with very different phenomenologies.

We end this subsection by briefly comparing our work to the forthcoming work of Demircigil and Henderson~\cite{Demircigil2025}, who analyse the same family of free boundary problems~\eqref{eq:generalised FBP_CDF} for $\beta\geq 0$. In contrast to the largely (but not entirely) probabilistic techniques that we employ, their methods are entirely based on (very different) PDE techniques. 
Indeed, their formulation of the PDE is different, viewing it instead as a PDE on the real line with two disjoint regimes, rather than as a free boundary problem per se. Nevertheless, the PDE they consider is equivalent to~\eqref{eq:generalised FBP_CDF}.

As in the present article, Demircigil and Henderson establish existence and uniqueness of solutions of~\eqref{eq:generalised FBP_CDF} for $\beta\ge 0$.
Moreover, they establish the asymptotics of $L_t$ up to $\mathcal{O}(1)$, under various conditions on the initial condition $V_0$; these 
conditions fall within, and are much more restrictive than, the (sharp) \textit{finite initial mass} condition under which we establish similar asymptotics for $L_t$ up to $o(1)$. In contrast to the present work, none of their results cover the \textit{infinite initial mass} case.

\subsection{Background and related work} \label{subsec:background}

The free boundary problem~\eqref{eq:FBP} studied here and its variants are now the subject of a rich and growing literature.  In this subsection, we review some of the relevant works on this topic and  give some mathematical context to our results.

\subsubsection{Existence, uniqueness and equivalent formulations} \label{subsubsec:existencefbp}

The existence and uniqueness of a classical global solution for the free boundary problem~\eqref{eq:FBP} for an arbitrary initial Borel probability measure $u_0$ was first established in~\cite{Berestycki2018}, which improved on earlier work of Lee~\cite{Lee}, where a local existence result was proved under certain regularity conditions on $u_0$. By a classical solution, we mean a pair $(u(t,x),L_t)$ where 
\begin{itemize}
	\item $L_t\in \R$ $\forall t>0$ and $t\mapsto L_t$ is continuous on $(0,\infty)$,
	\item $u: (0,\infty)\times \R \to [0,\infty)$ with $u\in C^{1,2} (\{ (t,x): t>0, \, x>L_t \}) \cap C( (0,\infty)\times \R) $,
	\item $(u,L)$ satisfies~\eqref{eq:FBP}.
\end{itemize}
See the start of Section~\ref{section:properties of free-boundary} for a precise statement of the existence and uniqueness result from~\cite{Berestycki2018}.
We have already explained (immediately after~\eqref{eq:FBP_CDF}) how the solutions of~\eqref{eq:FBP} are in  one-to-one correspondence with those of~\eqref{eq:FBP_CDF}, but the same problem can also be further reformulated in several different ways.

\medskip

One approach which is particularly relevant for us here is the {\it inverse first passage time problem} representation.   Given a random variable $\zeta$ taking values in $(0,\infty)$, the inverse first passage time problem consists of finding an upper semi-continuous function $b :[0,\infty) \mapsto  \Rm$ such that the first passage time, 
\[
\tau_b := \inf \{t>0 : B_t \le b(t) \},
\]
of $b$ by a standard Brownian motion $B$ has the same distribution as $\zeta$. The initial distribution of $B_0$ can be specified as some probability measure $u_0$. 
When  $\zeta$ is an exponential random variable with mean one, it can be shown (see e.g.~\cite{Chen2022,Chen2011}) that the inverse first passage time problem has a unique solution $L$ which coincides with the solution of \eqref{eq:FBP} in the sense that $L_t$ is given by the free boundary in \eqref{eq:FBP} and $u(t,x)dx = \mathbb{P}(B_t \in dx  | \tau_L>t)$
(see Lemma \ref{lem:FKforwardstime} below, which is~\cite[Theorem~B.1]{Berestycki2024}).  

Inverse first passage time problems are a well-studied topic in probability theory, which can be traced back at least to a question asked by Shiryaev in 1976 at a Banach centre meeting (see~\cite[p.1320]{Zucca2009}), and to the work of Dudley and Gutmann in 1977~\cite{Dudley1977} and Anulova in 1981~\cite{Anulova1981}. Existence was first proved by Anulova for a two-sided version of the problem~\cite{Anulova1981}. Chen et al.~\cite{Cheng2006,Chen2011} then proved existence and uniqueness by making the connection with a free boundary problem similar to~\eqref{eq:FBP} and a corresponding variational inequality (when $\zeta$ is non-atomic). A second approach by Ekstr\"{o}m and Janson~\cite{Ekstrm2016} established a connection to an optimal stopping problem and proved uniqueness in a more general framework (see also the comprehensive book by Peskir and Shiryaev~\cite{Peskir2006}). The recent works~\cite{Berard2023, Klump2022, Klump2023} also make the connection with certain particle systems similar to the $N$-BBM.

It is interesting to note that it is shown in~\cite{Chen2022} that for any initial distribution $u_0$, the boundary is almost as smooth as the distribution of $\zeta$. It follows that $L\in C^\infty (0,\infty)$ in our case  (since $\zeta$ has an exponential distribution). The regularity of $L$ at 0 is also studied in~\cite{Chen2022} but depends on the behaviour of $u_0$ near $L_0$ in a non-trivial way.
We do not use the results in~\cite{Chen2022} to study \eqref{eq:FBP_CDF} in this article, 
except in the proof of Theorem~\ref{theo:mapping between FBP and general problem} which establishes a mapping to solutions of the generalised problem \eqref{eq:generalised FBP_CDF}.
In contrast, in \cite{Berestycki2018} the boundary $L$ is only proven to be continuous and, in general, the regularity of the boundary in free boundary problems is a delicate matter.  

\medskip

As observed in \cite{Lee}, the free boundary problem can also be recast as a so-called {\it Stefan problem}, by the following formal argument. Let $(u,L)$ be the solution of~\eqref{eq:FBP} and assume that $L$ is differentiable. Then by differentiating the relation $\int_{L_t}^\infty u(t,x)dx =1$ and using that $u(t,L_t)=0$, and then using that $\lim_{x\to \infty} \partial_x u(t,x) =0$, we get 
\[
0 = 0+\int_{L_t}^\infty \left( \tfrac12 \partial_{xx} u(t,x) + u(t,x)\right) dx =  1- \tfrac 12 \partial_x u(t,L_t),
\]
which gives us that $\partial_x u(t,L_t) =2$ for all $t>0$. 
Define $h(t,x) :=  \partial_x u(t,x)$.  Then $(h,L)$ solves
\begin{equation}\label{eq:Stefan}
	\left\{ 
	\begin{array}{ll}
		\partial_t h =  \frac1 2 \Delta h  +h, & t>0,\; x>L_t,\\
		h(t,L_t) =2, \, & t>0,  \\ 
		\dot{ L}_t = -\frac14  \partial_x h (t,L_t), \, &  t>0,
	\end{array}
	\right.
\end{equation}
where the last equality is obtained by differentiating $u(t,L_t)=0$, assuming that $\partial_{xx} u(t,L_t)$ exists (in the sense of a right limit). 
Conversely, it can be checked (following the same argument as in~\cite{Lee}) that if $(h,L)$ is a solution of~\eqref{eq:generalised FBP_CDF}, then if we define $u(t,x) = \int_{L_t}^{x\wedge L_t} h(t,y) dy$, we have that $(u,L)$ is a solution of \eqref{eq:FBP}.

Stefan problems originated from 19th century studies on phase changes, particularly ice formation and melting, where the interface between the solid and liquid phases corresponds to the moving boundary. As noted in \cite{Berestycki2018}, there is a vast literature devoted to the study of Stefan problems, see e.g.~\cite{Gupta2017}. The problem~\eqref{eq:Stefan} above is non-standard since it involves a (linear) reaction term in addition to the diffusion, and since $h$ is not non-negative.

\subsubsection{A linearised version of the FKPP equation}\label{subsubsec:FKPPlinearised}

One of the key motivations for studying~\eqref{eq:FBP_CDF} is that it is a linearised version of the well-known FKPP equation,
\begin{equation}\label{eq:true_FKPP}
\partial_t v  = \tfrac 12 \Delta v +v(1-v), \quad  t>0, \, x\in \R, 
\end{equation}
where $v : (0,\infty)  \times \R \ni (t,x)\mapsto v(t,x) \in [0,1]$. Introduced in 1937 independently by
Fisher~\cite{Fisher1937} and by Kolmogorov, Petrovskii and Piskunov~\cite{Kolmogorov1937}, it is the prototypical {\it reaction-diffusion} equation. Very heuristically, the behaviour of the solutions of \eqref{eq:true_FKPP} and of \eqref{eq:FBP_CDF} are driven by the same three key ingredients: the Laplacian which corresponds to diffusion, a growth term (the ``$+v$''), and a saturation mechanism (``$-v^2$'' in the FKPP equation, corresponding to the mass constraint in the free boundary problem), and should therefore be very similar.

Indeed, it is well known that for suitable initial conditions $v_0$ where $0\le v_0\le 1$, $v_0(x)$ is bounded away from 0 as
$x\to -\infty$ and $v_0(x)\to 0$ fast enough as $x\to \infty$, the solution  $v(t,x)$ of~\eqref{eq:true_FKPP} converges to a travelling wave: there exists a centring term $m(t)$ and an asymptotic shape $\omega_c(x)$
such that
\[
\lim_{t\to \infty} v(t,m(t)+x) = \omega_c(x),
 \]
where $\lim_{t\to \infty}m(t)/t = c$ and $\omega_c$ is a travelling wave solution of~\eqref{eq:true_FKPP} with speed $c$. In his seminal work~\cite{Bramson1983}, Bramson precisely quantified what ``fast enough'' decay of $v_0(x)$ as $x\to \infty$ means and how the centring term $m(t)$ and the asymptotic speed $c$ depend on the initial condition. Our Theorems~\ref{theo:convtoPimin},~\ref{theo:Ltposition}, and~\ref{theo:slowerdecay} are direct analogues of Bramson's results (see in particular Theorems~A (on p.5) and~B (p.6), and Theorems~3 (p.141), 4 (p.166) and~5 (p.177), all in~\cite{Bramson1983}). In particular, in the case of initial conditions $v_0$ with a {\it compact interface} (i.e.~such that for some $K<\infty$ we have $v_0(x) =0$ for $x\ge K$ and $v_0(x)=1$ for $x\le -K$), Bramson's centring term  can be chosen as 
\[
m(t)  = \sqrt 2 t -\frac 3{2\sqrt 2 } \log t. 
\]

In \cite{Ebert98}, Ebert and van Saarloos predicted  non-rigorously that if one defines $\mu^{(1/2)}_t := \sup\{x: v(t,x)\ge 1/2\}$, then for initial conditions $v_0$ with a compact interface, there exists $a=a(v_0)\in \R$ such that
\[
\mu^{(1/2)}_t   = \sqrt 2 t -\frac 3{2\sqrt 2 } \log t +a - \frac{3\sqrt \pi}{\sqrt{2t}}+o\left(\frac{1}{\sqrt{t}}\right) \quad\text{as $t\rightarrow \infty$}.
\]
This was (non-rigorously) shown to hold for solutions of~\eqref{eq:FBP_CDF} (with $L_t$ in lieu of $\mu^{(1/2)}_t $) in~\cite{Berestycki2018a}, using the Brunet-Derrida relation (Theorem~\ref{theo:magic formula}) and non-rigorous singularity analysis. More precisely, 
it is predicted in~\cite{Berestycki2018a} that if $\int^\infty_0 x^3e^{\sqrt 2 x} U_0(x)dx<\infty $  and $(U,L)$ is the solution of~\eqref{eq:FBP_CDF} with initial condition $U_0$, we have the full expansion 
\begin{equation} \label{eq:Ltfullexpansion}
L_t= \sqrt 2 t -\frac 3{2\sqrt 2 } \log t +a - \frac{3\sqrt \pi}{\sqrt{2t}}+\frac 9 {8\sqrt 2}  (5-6\log 2) \frac{\log t }{t} + o \left(\frac{\log t }{t}\right) \quad\text{as $t\rightarrow \infty$}.
\end{equation}
The same non-rigorous approach was then used for the true FKPP equation~\eqref{eq:true_FKPP}  in~\cite{Berestycki2018b} to obtain a similar expansion for $\mu^{(1/2)}_t $. This was subsequently proved fully rigorously by Graham~\cite{Graham} (in the  compact interface initial condition case) using purely analytical methods. 
Proving rigorously that for solutions of~\eqref{eq:FBP_CDF}, $L_t$ satisfies the fine asymptotics~\eqref{eq:Ltfullexpansion} in e.g.~the compact interface initial condition case remains open, since our results only determine $L_t$ up to $o(1)$ as $t\to \infty$.

Studying the fine asymptotics of FKPP front localisation was also the motivation behind \cite{ BBHR17,Henderson17}. In those two works, the free boundary problem \eqref{eq:FBP} was replaced by the following {\it boundary value problem}:
\[
\begin{cases}
	\partial_t u = \frac 1 2 \Delta u + u , \quad &t>0,\, x>L_t, \\  u(t,L_t)=0,\quad &t>0,
\end{cases}
\]
 where $t \mapsto L_t$ is a given function. 
For such a boundary value problem, the question becomes ``for which choices of $t \mapsto L_t $ does $u(t,x+L_t)$ converge in the large-time limit $t\to \infty$?''  If one picks a function $L_t$ which grows too slowly, then the total mass will grow unboundedly; conversely, if $L_t$  increases too fast, then it ``eats'' the mass and $u$ converges to $0$. Now take an initial condition such that $u(0,x) \sim A x^\nu e^{-\sqrt{2}x}$ as $x\to \infty$ and a twice differentiable boundary $t\mapsto L_t$ with the property that $\frac{d^2}{dt^2} L_t = \mathcal{O}(t^{-2})$ as $t \to  \infty$. It is shown in~\cite{BBHR17} that $u(t,L_t+\cdot)$ will converge if and only if $L_t$ has the same asymptotics (with a possibly different constant term) as those found for the position of the free boundary in our Theorem~\ref{theo:Ltposition} (see Remark~\ref{rmk:FKPP asymptotics heavy-tailed}). The Ebert-van Saarloos term is shown to arise if one wants the convergence to happen at the optimal rate.

We recall from Section~\ref{subsubsec:existencefbp} that it has been shown in~\cite{Chen2022} that solutions $(u,L)$ of the free boundary problem~\eqref{eq:FBP} have a $C^{\infty}(0,\infty)$ boundary $L$. It is tempting to take such a solution $(u,L)$ to~\eqref{eq:FBP} for which we know that $u(t,L_t+\cdot)$ converges, consider $L$ as a fixed boundary, and apply the results of~\cite{BBHR17}, as an alternative approach for obtaining the position of the front $L_t$ up to $o(1)$ (as in our Theorem~\ref{theo:Ltposition}). The problem is that this would require knowing that $\frac{d^2}{dt^2}L_t=\mathcal{O}(t^{-2})$, which we do not. This is therefore not a viable approach for now.

We also believe that for solutions of~\eqref{eq:FBP_CDF} with e.g.~compact interface initial conditions, $U(t,L_t+\cdot)$ should converge at a rate of $\frac{1}{t}$ to the minimal travelling wave $\Pi_{\min}$, as has been established for the classical FKPP equation~\cite{An2024}. The authors of~\cite{An2024} established a framework for obtaining the rate of convergence from knowing the asymptotics of the front up to $\mathcal{O}(1)$. However, this concerned PDEs on the whole real line, and did not cover the case of free boundary PDEs. We believe it should be possible to extend the results of~\cite{An2024} to free boundary PDEs, and thereby obtain the rate of convergence of $U(t,L_t+\cdot)$ to $\Pi_{\min}$ using the asymptotics for $L_t$ obtained in the present article, but this would be by no means straightforward.

\subsubsection{The hydrodynamic limit of the $N$-BBM}\label{subsec:N-BBM}

A second major motivation for studying the free boundary problem~\eqref{eq:FBP} comes from the fact that it arises as the hydrodynamic limit of the so-called $N$-BBM, a branching particle system with selection in which particles diffuse as independent Brownian motions on the real line, branch at rate one into two particles at the position of their parent, and --crucially-- the total number $N$ of particles is kept constant by removing the leftmost particle in the system after each branching event.  The overall picture is that we have a cloud of diffusing branching particles which is drifting to the right because of the selection rule. A discrete-time analogue of this model was introduced in the early 2000s by Brunet and Derrida \cite{Brunet1997,Brunet1999,BD2015,Brunet2006,Brunet2007} for two reasons: first it is a natural toy model for the evolution of a population under selection in which one may study the genealogy of the system, and second it can be thought of as a noisy version of the FKPP equation. 
The $N$-BBM itself was introduced in~\cite{Maillard2016} as a natural continuous-time Brunet-Derrida branching-selection system.

It was shown in~\cite{DeMasi2019} (and for a closely related particle system in the earlier work~\cite{Durrett2011}) that if the $N$ initial particle positions are i.i.d.~with distribution given by $u_0$, then the hydrodynamic limit of this system exists and coincides with the solution of~\eqref{eq:FBP}. More precisely, if we write $X_1(t), \ldots , X_N(t)$ for the positions of the $N$ particles at time $t$ and $\mu^N_t(\cdot) = \frac 1N \sum_{i=1}^N \delta_{X_i(t)}(\cdot)$ for the corresponding normalised empirical measure, then for any $a \in \R$ and $t>0$,
\[
\lim_{N\to \infty}\mu^N_t ( [a,\infty)) = U(t,a) \qquad \text{ a.s. and in } L^1,
\]
where $(U,L)$ is the solution of the free boundary problem~\eqref{eq:FBP_CDF} with initial condition $U_0(x)= u_0((x,\infty))$. The book \cite{Carinci2016} provides an excellent account of the free boundary problem~\eqref{eq:FBP_CDF} and its relationship to the $N$-BBM. More recently, Atar~\cite{Atar2025} and B{\'e}rard and Fr{\'e}nais~\cite{Berard2023} have introduced general formulations of~\eqref{eq:FBP_CDF}, which are used to generalise the above hydrodynamic limit theorem.

Recently, the first and third authors of the present article solved the so-called {\it selection principle} for the $N$-BBM \cite{Berestycki2024}. The $N$-BBM process centred by the position of its leftmost particle is a recurrent process, which therefore admits a unique invariant distribution $\psi^N $. It is natural to ask what $\psi^N$ should look like for large $N$. Note that since the centred $N$-BBM converges to its invariant distribution as $t\to \infty$, we are trying to describe what happens when we first let $t\to \infty $ and then $N\to \infty$. The {\it selection principle} asserts that the particle density under $\psi^N$ converges as $N\to \infty$ to $\pi_{\min}$, the minimal travelling wave solution of \eqref{eq:FBP} (i.e.~the particle system {\it selects} the minimal travelling wave among all possible travelling waves). 
In other words, the complementary cumulative distribution function of the empirical measure under $\psi^N$ converges to $\Pi_{\min}$ as $N\to \infty$.
This can be interpreted as saying that the limits $t\to \infty$ and $N\to \infty$ commute, since when we first let $N \to  \infty$ we get the solution of~\eqref{eq:FBP_CDF}, and in the present work we show that if we then let $t \to \infty$, the solution of~\eqref{eq:FBP_CDF} converges to $\Pi_{\min}$, under a suitable integrability condition on $U_0$.  
The situation is summarised in the following diagram:
 
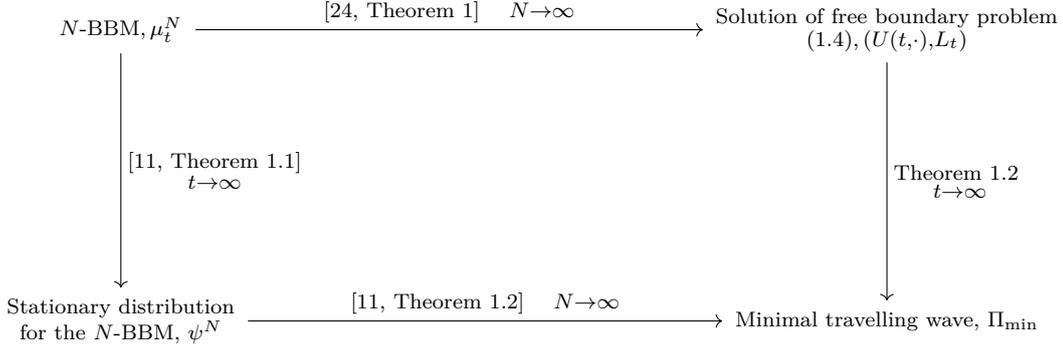
\begin{figure}[H]
 	\begin{center}
 		\begin{tikzcd}[column sep=6cm, row sep=3cm]
 			\substack{\text{$N$-BBM},\,  \mu^N_t} \arrow[r, "\text{\cite[Theorem 1]{DeMasi2019}}\quad N\ra\infty"] \arrow[d, "\substack{\text{\cite[Theorem 1.1]{Berestycki2024}}\\t\ra\infty}"]
 			& \substack{\text{Solution of free boundary problem}\\
 			\text{\eqref{eq:FBP_CDF}}, \, (U(t,\cdot),L_t)} \arrow[d, "\substack{\text{Theorem \ref{theo:convtoPimin}  }  \\t \ra \infty}"] \\
 			\substack{\text{Stationary distribution}\\ \text{for the $N$-BBM, $\psi^N$}} \arrow[r, "\text{\cite[Theorem 1.2]{Berestycki2024}}\quad  N\ra\infty"]
 			& \substack{\text{Minimal travelling wave, $\Pi_{\min}$}}
 		\end{tikzcd}
 	\end{center}
 \end{figure}

Note, however, that the aforementioned integrability conditions for convergence of solutions of~\eqref{eq:FBP_CDF} to the minimal travelling wave do not appear for the convergence of $\psi^N$ as $N\to \infty$. This represents the possibility of non-commuting limits. Indeed, take an $N$-BBM in which initially the particles are i.i.d.~with distribution given by $\pi_c$, for some $c>\sqrt{2}$. If we firstly take $t\to \infty$ and then $N\to\infty$, the centred $N$-BBM will converge to $\psi^N$ as $t\to \infty$ for fixed $N$, which then converges to $\pi_{\min}$ as $N\to\infty$. In contrast, if we first take $N\to\infty$, then by~\cite{DeMasi2019} the (normalised empirical measure of the) $N$-BBM will converge to the travelling wave with speed $c$. The shape of this travelling wave then trivially converges to $\pi_c\neq \pi_{\min}$ as $t\to\infty$. Heuristically, this demonstrates that if we start the $N$-BBM very close to one of the non-minimal travelling waves, then it will gradually move away from this travelling wave and stabilise on the minimal one.

\subsubsection{Connection with Brownian motion on $(0,\infty)$ with drift $-\sqrt 2$} \label{subsubsec:BMwithdrift}

Recall from Section~\ref{subsubsec:existencefbp} (see also Lemma~\ref{lem:FKforwardstime} below) that the inverse first passage time problem representation for the free boundary problem~\eqref{eq:FBP} with initial condition $u_0$ can be formulated as follows: We initialise a Brownian motion $(B_t)_{t\ge 0}$ according to $B_0\sim u_0$, and kill it instantaneously upon hitting the boundary $L_t$, i.e.~at the stopping time $\tau=\inf\{t>0:B_t\le L_t\}$. The boundary $L$ is chosen so that $\tau$ is an exponential random variable with mean 1, i.e.~$\p{\tau>t}=e^{-t}$. Then $(u(t,x),L_t)$ is the solution of~\eqref{eq:FBP}, where $u(t,x)$ satisfies
\begin{equation} \label{eq:stochrepintro}
	e^t\Pm_{u_0}(B_t\in dx,\tau>t)=u(t,x)dx.
\end{equation}

Equivalently, we can consider
\[
d X_t:=-dL_t+dB_t,
\]
and kill $X_t$ instantaneously upon hitting $0$. Note that $L_t$ is a finite variation process (one can prove this using results in Section~\ref{subsec:stretching} below, but we do not need this for any of our proofs), justifying the above semimartingale decomposition. 

Instead of choosing a finite variation process $L_t$ to fix the probability of survival, we could instead fix a negative drift, and allow the probability of survival to vary. If $u_0$ has finite initial mass, we will show that $(L_{t+s}-L_t)_{s\geq 0}$ converges uniformly on compacts as $t\to \infty$ to $(\sqrt{2}s)_{s\geq 0}$ (see the proof of Theorem~\ref{theo:convergence to the minimal travelling wave for finite initial mass} in Section~\ref{subsec:finiteinitmassconv}), so it is natural to fix the drift as $-\sqrt{2}$. Therefore we consider the killed process
\begin{equation}\label{eq:SDE for BM cst drift}
dX_t=dB_t-\sqrt{2}dt,\quad 0\leq t<\tau_0:=\inf\{t>0:X_{t-}=0\}.
\end{equation}

The asymptotic behaviour of this process is well understood, and displays surprising similarities with that of the FKPP equation and our free boundary problem. The analogue of travelling waves are known as quasi-stationary distributions (QSDs), and similarly there is a one-parameter family of them, $(\alpha_{\lambda})_{0<\lambda\leq 1}$ (see~\cite{Meleard2011} for a survey on absorbed Markov processes and their QSDs). Indeed, these correspond exactly to the travelling waves of the free boundary problem considered in this article, as observed in~\cite{GroismanJonckheer}. The analogue of the minimal travelling wave is known as the minimal QSD, and denoted by $\alpha_{\min}:=\alpha_1$.

Given a QSD $\alpha$, its domain of attraction consists of those probability measures $\mu$ such that
\begin{equation}\label{eq:convergence to QSD}
\mathcal{L}_{\mu}(X_t\lvert \tau_0>t)\ra \alpha \quad \text{as }t\rightarrow \infty.
\end{equation}
If~\eqref{eq:convergence to QSD} holds then we say that $\alpha$ is the quasi-limiting distribution for $\mu$; the investigation of this limit represents the principal question investigated in the theory of absorbed Markov processes.

The domain of attraction of the minimal QSD, $\alpha_{\text{min}}$, was obtained in \cite[Theorem 1.3]{Martinez1998} and \cite[Theorem 1.1]{Tough23}. Strikingly, it is exactly the same as the domain of attraction we obtain in Theorem \ref{theo:convtoPimin}, and the domain of attraction obtained by Bramson for the FKPP equation \cite[Theorem~A]{Bramson1983}. 

When~\eqref{eq:convergence to QSD} holds, one can investigate the rate of convergence. It was established by Oçafrain in \cite[Theorem 3.1]{Ocafrain2020} that one has a $\frac{1}{t}$ rate of convergence to the minimal QSD in Wasserstein distance, for all initial conditions $\mu$ satisfying an exponential moment condition (so in particular including all compactly supported initial conditions). This is precisely the rate of convergence that is expected to hold in the free boundary setting considered in this article, and is known more generally for pulled FKPP equations~\cite{An2024}.

Finally, we mention that just as the free boundary problem \eqref{eq:FBP} can be understood as the large $N$ limit of the $N$-BBM, as discussed in Section \ref{subsec:N-BBM}, the law of Brownian motion with drift $-\sqrt{2}$ can similarly be understood as the large $N$ limit of a particle system called the Fleming-Viot particle system \cite{Burdzy2000,Villemonais2011}. Just as there was a long-standing ``selection problem'' for the $N$-BBM (discussed in Section \ref{subsec:N-BBM}), there was a similar conjectured selection principle for this Fleming-Viot particle system. This selection principle was proven by the third author in \cite{Tough23} and, in fact, ideas from this proof played an important role in the proof of the selection principle for the $N$-BBM by the first and third authors \cite{Berestycki2024}. See \cite{GroismanJonckheer} for an excellent survey on the relationship between these two particle systems and their selection problems.

\subsection{Notation} \label{subsec:notation}
Throughout this article, for $x\in \Rm$, let $\mathbb P_x$ denote the probability measure under which $(B_t)_{t\ge 0}$ is a Brownian motion started at $x$, and let $\mathbb E_x$ denote the corresponding expectation.
Moreover, for a Borel probability measure $\mu$ on $\R$, let $\mathbb P_\mu$ denote the probability measure under which $(B_t)_{t\ge 0}$ is a Brownian motion with $B_0\sim \mu$, and let $\mathbb E_{\mu}$ denote the corresponding expectation.

For $t\ge 0$ and  $x,y\in \Rm$, let $(\xi^t_{x,y}(s),0\le s \le t)$ denote a Brownian bridge from $x$ at time 0 to $y$ at time $t$.

Let $(U^H(t,x),L_t^H)$ denote the solution of the free boundary problem~\eqref{eq:FBP_CDF} with Heaviside initial condition $H(x)=\1_{\{x<0\}}$. Let $L^H_0=0$.

\subsection{Outline of the proof of Theorems~\ref{theo:convtoPimin} and~\ref{theo:Ltposition}} \label{subsec:outline}

In this subsection, we outline some of the key ideas in the proofs of our main results, Theorems~\ref{theo:convtoPimin} and~\ref{theo:Ltposition}.
As mentioned after Definition~\ref{defin: fin init mass}, the proof is divided into two cases: the \textit{finite initial mass} case -- where the initial condition $U_0$ satisfies~\eqref{eq:finite initial mass U0}, and the \textit{infinite initial mass} case -- where $U_0$ satisfies~\eqref{eq:infinite initial mass U0}.

Here, we focus on the proof strategy in the finite initial mass case, which is very different from the strategy employed by Bramson in~\cite{Bramson1983} for the FKPP equation in the finite initial mass setting. The proof in the infinite initial mass  case is found in Section~\ref{sec:infiniteinitialmass}; it uses a novel Feynman-Kac representation for solutions of~\eqref{eq:FBP_CDF}, Lemma~\ref{lem:FKforinfinitemass}, and follows broadly a similar proof strategy to that employed by Bramson for the FKPP equation in the same infinite initial mass setting~\cite{Bramson1983}. We give a more detailed outline of the strategy at the beginning of Section~\ref{sec:infiniteinitialmass}.

\medskip

\noindent \textbf{Proof strategy for finite initial mass case.}
Recall from the stochastic representation of solutions of~\eqref{eq:FBP} given in~\eqref{eq:stochrepintro} that for $t>0$, 
\begin{equation} \label{eq:stochrepintro2}
\psub{u_0}{B_s>L_s \; \forall s\in (0,t)}=e^{-t}.
\end{equation}
The authors of~\cite{BBHR17} established that if we replace the free boundary $L_t$ with a fixed boundary 
\[
m(t):=\sqrt{2}t-\frac{3}{2\sqrt{2}}\log t+a+o(1),
\]
with some additional conditions on the $o(1)$ term, then for a broad subset of initial conditions $u_0$ with finite initial mass in the sense of~\eqref{eq:finite initial mass little u0},
\begin{equation}\label{eq:convtopimin}
e^t \psub{u_0}{B_t-m(t)\in dx,\,  B_s>m(s) \; \forall s\in (0,t)} \to K\pi_{\min}(x) \quad \text{as }t\to \infty,
\end{equation}
for some positive constant $K$.
(This can also be phrased as convergence of the solution of a PDE with Dirichlet boundary conditions at the boundary $m$.)
However, the results in~\cite{BBHR17} assumed that $u_0$ has a density (which we write as $u_0(x)$ in an abuse of notation)
that satisfies
\begin{equation} \label{eq:BBHR_condition}
u_0(x)=\mathcal{O}(x^{-\nu}e^{-\sqrt{2}x}) \quad \text{for some }\nu>\sqrt{2}.
\end{equation}
This is a slightly stronger assumption than the finite initial mass assumption~\eqref{eq:finite initial mass little u0}; for instance, an initial density of $u_0(x)=e^{-x}x^{-2}\log^{-2} x$ has finite initial mass, but does not satisfy~\eqref{eq:BBHR_condition}. We refine the estimates from~\cite{BBHR17} so as to extend the result~\eqref{eq:convtopimin} to all initial conditions $u_0$ with finite initial mass, when the boundary is given by 
\begin{equation} \label{eq:mtformulaoutline}
m(t):=\sqrt{2}t-\frac{3}{2\sqrt{2}}\log (t+1)+a;
\end{equation}
see Theorem~\ref{theo:extension of BBHR}.

In Section~\ref{subsec:finiteinitmassconv}, we use Theorem~\ref{theo:extension of BBHR} to establish convergence to the minimal travelling wave under the finite initial mass assumption.
We first consider the solution $(U^H(t,x),L^H_t)$ of \eqref{eq:FBP_CDF} with Heaviside initial condition.
By comparing $U^H$ to solutions of the classical FKPP equation with Heaviside initial condition, and using Bramson's front position asymptotics for the FKPP equation~\cite{Bramson1983}, we see that
\begin{equation} \label{eq:LHtloweroutline}
L^H_t\geq \sqrt{2}t-\frac{3}{2\sqrt{2}}\log t+\mathcal{O}(1).
\end{equation}
Moreover, using a {\it stretching lemma} established in~\cite{Berestycki2024}, we show in Section~\ref{subsec:stretching} that for $\epsilon>0$, for sufficiently large times $t$ we have $L^H_{t+s}-L^H_t\ge (\sqrt 2-\epsilon)s$ for $s\ge 0$,
and by comparison with the travelling wave $\Pi_{\min}$ we also have $L^H_{t+s}-L^H_t\le \sqrt 2 s$ for $s\ge 0$.
Therefore $(L^H_{t+s}-L^H_t)_{s\geq 0}$ converges locally uniformly as $t\to \infty$ to $(\sqrt{2}s)_{s\geq 0}$. 

By~\eqref{eq:LHtloweroutline}, we can choose $a$ in~\eqref{eq:mtformulaoutline} such that $m(t)\leq L^H_t$ for all $t$. 
Take an initial condition $U_0$ with finite initial mass and with $U_0(x)=1$ for $x<0$ (this assumption can be removed later on), and let $(U(t,x),L_t)$ denote the solution of~\eqref{eq:FBP_CDF}.
Another consequence of the stretching lemma (see Lemma \ref{lem:less stretching means slower boundary})
tells us that $L_t-L^H_t$ is non-decreasing in $t$, hence in particular $L_t\ge L^H_t$ $\forall t\ge 0$ and
\[
m(t)\leq L^H_t\leq L_t \quad \forall t\ge 0.
\]
Let $u_0$ denote the probability measure given by $u_0((x,\infty))=U_0(x)$ for $x\in \R$.
Our extension of~\eqref{eq:convtopimin} tells us that for large $t$, the probability that a Brownian motion survives until time $t$, after being initialised from $u_0$ and killed upon hitting $m(s)$, is asymptotically like $Ke^{-t}$ for some $K\in (0,\infty)$. Using the fact that $L_s\geq m(s)$ $\forall s\ge 0$ (and so the Brownian motion cannot hit $m(s)$ before hitting $L_s$), we argue that if $L_t-m(t)$ were to be unbounded, then there must be some $t<\infty$ such that the probability of a Brownian motion $B_s$ initiated from $u_0$ not hitting $L_s$ by time $t$ must be less than $e^{-t}$. By~\eqref{eq:stochrepintro2}, this is a contradiction, and so we conclude that $L_t-m(t)$ is bounded. This implies that
\[
L_t=\sqrt{2}t-\frac{3}{2\sqrt{2}}\log t+\mathcal{O}(1).
\]
It also implies that $L_t-L^H_t$ is bounded, and so since it is non-decreasing, $L_t-L^H_t$ must converge as $t\to \infty$. This implies that $(L_{t+s}-L_t)_{s\geq 0}$ converges locally uniformly to $(\sqrt{2}s)_{s\geq 0}$, using the corresponding fact for $L^H_t$. 
Using Theorem~\ref{theo:magic formula}, the Brunet-Derrida relation, with initial condition $U(t,L_t+\cdot)$, we conclude that $U(t,L_t+\cdot)$ converges uniformly to the unique initial condition for which the solution of~\eqref{eq:FBP_CDF} has free boundary $(\sqrt{2}s)_{s\geq 0}$, namely $\Pi_{\min}$.

We have now established convergence to the travelling wave and asymptotics of the free boundary up to $\mathcal{O}(1)$ as $t\to \infty$. Our next step is to improve our boundary asymptotics to $o(1)$. This is accomplished in Section \ref{subsec:bootstrap}. Since we know that $L_t-L^H_t$ converges as $t\ra\infty$, it suffices to prove the asymptotics for Heaviside initial condition, i.e.~to establish that for some $c\in \R$,
\[
L^H_t=\sqrt{2}t-\frac{3}{2\sqrt{2}}\log t+c+o(1) \quad \text{as }t\to \infty.
\]

We begin by taking $a\in \mathbb{R}$ sufficiently large that for $m(t)$ defined as in~\eqref{eq:mtformulaoutline}, we have  $m(t)\geq L^H_t$ for all $t\ge 0$. We consider Brownian motion started from $0$, killed upon hitting the barrier
\[
L'_t:=\begin{cases}
L^H_t\quad &\text{for }t\leq T_0,\\
m(t)\quad &\text{for }t>T_0,
\end{cases}
\]
for some $T_0<\infty$ which we will choose later. For some $R<\infty$ which we will also choose later on, we consider the events
\[
E'_t:=\{B_s>L'_s\; \forall s\in (0,t),\; B_t>L^H_t+R\}\quad\text{and}\quad E^H_t:=\{B_s>L^H_s\; \forall s\in (0,t),\; B_t>L^H_t+R\}.
\]
Note that $E'_t\subseteq E^H_t$, due to the ordering of the barriers. 

On the one hand, it follows from our extension of ~\eqref{eq:convtopimin}, by conditioning on $B_{T_0}$, that as $t\to \infty$,
\[
\Pm_0(B_s>L'_s \; \forall s\in (0,t),\; B_t>L^H_t+R)\sim K_{T_0}e^{-t}\Pi_{\min}(L^H_t-L'_t+R),
\]
for some constant $K_{T_0}\in (0,\infty)$, which depends on $T_0$ but not on $R$.
On the other hand, using the stochastic representation~\eqref{eq:stochrepintro}, and then convergence to the travelling wave for Heaviside initial condition, we see that as $t\to \infty$,
\[
\Pm_0(B_s>L^H_s\; \forall s\in (0,t),\; B_t>L^H_t+R)=e^{-t}U^H(t,L^H_t+R) \sim \Pi_{\min}(R)e^{-t}.
\]
Therefore, as $t\to \infty$,
\[
\psub{0}{B_s>L'_s\; \forall s\in (0,t) \; \big\lvert \; B_s>L^H_s \; \forall s\in (0,t),\; B_t>L^H_t+R}\sim \frac{K_{T_0}\Pi_{\min}(L^H_t-L'_t+R)}{\Pi_{\min}(R)}.
\]

We further prove, using that $L'_t-L^H_t$ is bounded, that $R$, $T_0<\infty$ can be chosen to make 
\[
\psub{0}{B_s>L'_s\; \forall s\in (0,t) \; \big\lvert \; B_s>L^H_s\; \forall s\in (0,t),\; B_t>L^H_t+R}
\]
as close to $1$ as we like for large times $t$. 
This ensures that large $R$, $T_0<\infty$ can be chosen to make 
\[
\frac{K_{T_0}\Pi_{\min}(L^H_t-L'_t+R)}{\Pi_{\min}(R)} 
\]
as ``close to convergent'' as $t\to \infty$ as we like (in a certain appropriate sense). However, for large $R<\infty$ we have
\[
\frac{K_{T_0}\Pi_{\min}(L^H_t-L'_t+R)}{\Pi_{\min}(R)}\approx K_{T_0}e^{-\sqrt{2}(L^H_t-L'_t)}.
\]  
This implies that $L^H_t-L'_t$ is as ``close to convergent'' as we like, and so must converge as $t\to \infty$, and we are done.

\medskip

\noindent \textbf{Heuristic for finite initial mass and infinite initial mass dichotomy.} As mentioned at the start of this subsection, the proof strategy for the finite initial mass case is very different to the strategy for the infinite initial mass case.
The fact that the two strategies only apply to one of the two cases is explained, at least heuristically, by the following dichotomy.

Recall again the stochastic representation of solutions of the free boundary problem~\eqref{eq:FBP} stated at the start of Section~\ref{subsubsec:BMwithdrift}.
We now consider some probability measure $u_0$, let $(u(t,x),L_t)$ denote the solution of the free boundary problem~\eqref{eq:FBP} with initial condition $u_0$, and take a Brownian motion $(B_t)_{t\ge 0}$ with $B_0\sim u_0$, killed at time $\tau=\inf\{t>0:B_t\le L_t\}$. One can then ask, if our Brownian motion survives until time $t$, where was it started? Of course, if a Brownian motion is started far to the right, then it will be easier for it to survive until time $t$. However, this must be balanced with the probability of the Brownian motion being started far to the right, which depends on the tail behaviour of $u_0$.

It turns out that there is a transition in behaviour. If the initial condition $u_0$ is sufficiently light tailed, then it is easier to be born close to $0$ and then survive for a long time, than to be born far from $0$ and survive for a long time. This corresponds to the initial condition having finite initial mass. Therefore, in the finite initial mass case, a Brownian motion that survives for a long time is likely to have been born close to $0$.

Conversely, in the infinite initial mass case, the initial condition is sufficiently heavy tailed that it is easier to be born far away from $0$ and survive for a long time, than to be born close to $0$ and survive for a long time. Therefore, a Brownian motion that survives for a long time is likely to have been born far to the right of $0$.

To demonstrate heuristically that these change in behaviour should depend on whether or not $u_0$ has finite initial mass, we consider the analogous killed process $(X_t)_{0\leq t<\tau_0}$ defined in~\eqref{eq:SDE for BM cst drift} in Section~\ref{subsubsec:BMwithdrift}, given by Brownian motion with constant drift $-\sqrt{2}$, killed when it hits $0$. The survival probabilities for $X_t$ over a given time horizon are known exactly, and are given by the formula~\cite[(2.4)]{Martinez1998}. We fix a probability measure $\mu$ and an arbitrary $a<\infty$, and define $\mu_a:=\frac{\mu_{\lvert_{(0,a]}}}{\mu((0,a])}$. Using \cite[(2.4)]{Martinez1998}, we can calculate
\[
\begin{split}
\psub{\mu}{X_0\leq a\lvert \tau_0>t}&=\frac{\Pm_{\mu}(X_0\leq a,\;\tau_0>t)}{\Pm_{\mu}(\tau_0>t)}\\
&=\frac{\mu((0,a])\Pm_{\mu_a}(\tau_0>t)}{\Pm_{\mu}(\tau_0>t)}\\&=\frac{e^{t+\frac{3}{2}\log t}\mu((0,a])\Pm_{\mu_a}(\tau_0>t)}{e^{t+\frac{3}{2}\log t}\Pm_{\mu}(\tau_0>t)}
\to \begin{cases}
\frac{\int_0^{a}xe^{\sqrt{2}x}\mu(dx)}{\int_0^{\infty}xe^{\sqrt{2}x}\mu(dx)}\quad &\text{if }\int_0^{\infty}xe^{\sqrt{2}x}\mu(dx)<\infty,\\
0\quad &\text{if }\int_0^{\infty}xe^{\sqrt{2}x}\mu(dx)=\infty
\end{cases}
\end{split}
\]
as $t\to \infty$.
This is exactly the claimed dichotomy.

\medskip

\noindent \textbf{Acknowledgements:}
The authors are grateful to Mete Demircigil and Christopher Henderson for fruitful discussions about their work and its relationship to our own. The authors are also grateful to Centre International de Rencontres Mathématiques for hosting the workshops `Random Walks: Interacting, Branching, and more' and `Non-local branching processes', and to Banff International Research Station for hosting the workshop `Emerging Connections between Reaction-Diffusion, Branching Processes, and Biology', all of which were greatly beneficial to the development of this article.  

SP is supported by a Royal Society University Research Fellowship.
The work of OT was partially supported by the EPSRC MathRad programme grant EP/W026899/2.

\section{Basic properties}\label{section:properties of free-boundary}

This section contains basic properties of solutions of the free boundary problem~\eqref{eq:FBP_CDF} that will be used in the remainder of the article.

Recall the definition of a classical solution of~\eqref{eq:FBP_CDF} at the start of Section~\ref{subsec:mainresults}.
The following existence and uniqueness result follows directly from results in~\cite{Berestycki2018}.
\begin{prop}[Existence and uniqueness] \label{prop:fbpsoln}
Suppose $U_0$ satisfies Assumption~\ref{assum:standing assumption ic}; then
a classical solution $(U(t,x),L_t)$ of the free boundary problem~\eqref{eq:FBP_CDF} exists and is unique, and satisfies
\begin{enumerate}[(i)]
\item $U(t,x)=1 \; \Leftrightarrow \; x\le L_t$ for $t>0$.
\item $\lim_{t\to 0}L_t = \inf\{x\in \R:U_0(x)<1\}\in \{-\infty\}\cup \R$. \label{enum:Lt_time0_limit}
\item $U(t,\cdot)$ is non-increasing with $U(t,x)>0$ $\forall x\in \R$ and $U(t,y)\to 0$ as $y\to \infty$ for $t>0$.
\item $U(t,\cdot)$ is strictly decreasing on $[L_t,\infty)$ for $t>0$, with $\partial_x U(t,x)<0$ for $t>0$, $x>L_t$.
\item $\partial_x U\in C^{1,2}(\{(t,x):t>0,\, x>L_t\})\cap C( (0,\infty) \times \R)$, and $u=-\partial_x U$ is the unique classical solution of~\eqref{eq:FBP}, where $u_0$ is the Borel measure satisfying $u_0((x,\infty))=U_0(x)$ $\forall x\in \R$.
Moreover, $|\partial_x U(t,x)|\le 3+t^{-1/2}$ for $t>0$ and $x\in \R$.
\item  $U(t,x)\to U_0(x)$ as $t\to 0$ for any $x\in \R$ such that $U_0$ is continuous at $x$ (i.e.~Lebesgue-a.e.~$x$).
%$U(t,\cdot)\to U_0$ almost everywhere as $t\to 0$, and
\end{enumerate}
\end{prop}
\begin{proof}
The existence and uniqueness property and statements (i), (ii) and (vi) follow directly from~\cite[Theorems~1.1 and~1.2 and Proposition~1.3]{Berestycki2018}.
Statement (iii) follows from the comparison principle in~\cite[Theorem~1.1]{Berestycki2018} and~\cite[Corollary~3.3]{Berestycki2018}.
For statement~(v),~\cite[Theorem~1.1]{Berestycki2018} states that $\partial_x U\in C( (0,\infty) \times \R)$, 
~\cite[Corollary~2.1]{Berestycki2018} states that $u$ is the unique classical solution of~\eqref{eq:FBP},
~\cite[Lemma~4.3]{Berestycki2018} implies that $|\partial_x U(t,x)|\le 3+t^{-1/2}$,
and~\cite[Proposition~3.1 and Lemma~3.4]{Berestycki2018} combined with statement~(iii)
imply the rest of statement~(v).

Finally, for statement~(iv), for $t>0$ and $x\in \R$, let
$u(t,x):=-\partial_x U(t,x)\ge 0$.
Now fix $t>0$ and $x>L_t$, and take $\varepsilon\in (0,x-L_t)$.
Let $\tau=\frac t 2 \wedge \inf\{s\ge 0:B_s\le L_{t-s}+\varepsilon\}$.
Then by statement~(iv) and~\cite[Proposition~3.1 and Lemma~4.3]{Berestycki2018},
\[
u(t,x)=\Esub{x}{e^{\tau}u(t-\tau,B_\tau)}\ge e^{t/2}\Esub{x}{\1_{\{B_s>L_{t-s}+\varepsilon\;\forall s\le t/2\}}u(t/2,B_{t/2})}>0,
\] 
where the last inequality follows since $\inf_{y\ge L_{t/2}+2\epsilon}\psub{x}{B_s>L_{t-s}+\varepsilon\; \forall s\le t/2 \Big| B_{t/2}=y}>0$ and, by statement~(iii), $\int_{L_{t/2}+2\varepsilon}^\infty u(t/2,y)dy=U(t/2,L_{t/2}+2\varepsilon)>0$.
\end{proof}
In the remainder of the article, for $(U(t,x),L_t)$ a solution of~\eqref{eq:FBP_CDF}, we write
$U(0,\cdot):=U_0$ and 
\begin{equation} \label{eq:L0defn}
L_0:=
\inf\{x\in \R:U_0(x)<1\}
=\lim_{t\to 0}L_t \in \{-\infty\}\cup \R.
\end{equation}
The next result is a comparison principle proved in~\cite[Theorem~1.1]{Berestycki2018}.
\begin{prop}[Comparison principle] \label{prop:fbpcomparison}
Suppose $U_0^{(1)}$ and $U_0^{(2)}$ satisfy Assumption~\ref{assum:standing assumption ic} with
$U_0^{(1)}\le U_0^{(2)}$. Let $(U^{(1)}(t,x),L^{(1)}_t)$ and $(U^{(2)}(t,x),L^{(2)}_t)$ denote the solutions of~\eqref{eq:FBP_CDF} with initial conditions $U_0^{(1)}$ and $U_0^{(2)}$ respectively; then
\[
U^{(1)}(t,x)\le U^{(2)}(t,x) \quad \forall t>0, \, x\in \Rm.
\]
\end{prop}
The next result shows that the solution $U$ of~\eqref{eq:FBP_CDF} can be written as the limit of a sequence of solutions to FKPP equations, and is proved in~\cite[Theorem~1.2 and Proposition~1.3]{Berestycki2018}. 
\begin{prop} \label{prop:Un}
Suppose $U_0$ satisfies Assumption~\ref{assum:standing assumption ic}.
For each $n\in \Nm$,
let $U_n(t,x)$ solve
\begin{equation}\label{eq:CDF_n}
	\begin{cases}
		\partial_tU_n=\frac{1}{2}\Delta U_n+U_n-U_n^n, \quad t>0, \; x\in \Rm\\
		U_n(0,\cdot)=U_0.
	\end{cases}
\end{equation}
Then for any $t>0$ and $x\in \Rm$, the sequence $n\mapsto U_n(t,x)$ is non-decreasing with the following pointwise limit:
\begin{equation} \label{eq:UntoU}
	U(t,x)=\lim_{n\to \infty}U_n(t,x).
\end{equation}
\end{prop}
The following elementary lemma bounds the difference between solutions of~\eqref{eq:FBP_CDF} with different initial conditions.
\begin{lem} \label{lem:Bramsongronwall}
	Suppose $U_0^1$ and $U_0^2$ satisfy Assumption~\ref{assum:standing assumption ic}, and let $(U^1(t,x),L^1_t)$ and $(U^2(t,x),L^2_t)$ solve~\eqref{eq:FBP_CDF} with initial conditions $U_0^1$ and $U_0^2$ respectively. Then for $t>0$,
	\[
	U^2(t,x)-U^1(t,x)\le e^t \sup_{y\in \Rm}\left(U^2_0(y)-U^1_0(y)\right) \quad \forall x\in \Rm.
	\]
\end{lem}
\begin{proof}
Let $U^{n,1}$ and $U^{n,2}$ denote solutions of~\eqref{eq:CDF_n} with initial conditions $U_0^1$ and $U_0^2$ respectively.
By the comparison principle applied to the PDE satisfied by $U^{n,1}-U^{n,2}$ (see~\cite[Lemma 3.2]{Bramson1983}),
 and then letting $n\to \infty$ and using~\eqref{eq:UntoU} in Proposition~\ref{prop:Un}, the result follows.
\end{proof}

\subsection{Feynman-Kac formulae} \label{subsec:FK}

We will use the following Feynman-Kac formula later in this section, and also in Section~\ref{section:magic formula} and extensively in Section~\ref{sec:infiniteinitialmass}.
\begin{lem} \label{lem:FKforinfinitemass}
Suppose $U_0$ satisfies Assumption~\ref{assum:standing assumption ic}, and let $(U(t,x),L_t)$ denote the solution of~\eqref{eq:FBP_CDF}.
	For $t>0$ and $x\in \Rm$,
	\[
	U(t,x)=\Esub{x}{U_0(B_t)e^{\Leb(\{s\in [0,t] :B_s\ge L_{t-s}\})}}.
	\]
\end{lem}
\begin{proof}
Recall the definition of $U_n$ in~\eqref{eq:CDF_n}.
	For $n\in \Nm$, by the Feynman-Kac formula (see e.g.~\cite[Corollary~3.2]{Berestycki2018}), for $t>0$ and $x\in \Rm$,
	\begin{equation} \label{eq:FKun}
		U_n(t,x)=\Esub{x}{U_0(B_t)e^{\int_0^t (1-U_n(t-s,B_s)^{n-1})ds}}.
	\end{equation}
	By~\cite[Propositions~4.7 and~5.4]{Berestycki2018},
	for $s> 0$ and $y<L_s$, we have
	\[
	U_n(s,y)^{n-1} \to 1 \quad \text{as }n\to \infty.
	\]
	Moreover, by Proposition~\ref{prop:Un} and then by Proposition~\ref{prop:fbpsoln}(i), for $s> 0$ and $y>L_s$, 
	\[
	U_n(s,y)^{n-1} \le U(s,y)^{n-1} \to 0 \quad \text{as }n\to \infty.
	\]
	Therefore, by dominated convergence, $\mathbb P_x$-almost surely
	\[
	\int_0^t (1-U_n(t-s,B_s)^{n-1})ds
	=\int_0^t (1-U_n(t-s,B_s)^{n-1})\1_{B_s\neq L_{t-s}} ds
	\to \Leb(\{s\in [0,t]:B_s\ge L_{t-s}\})
	\]
	as $n\to \infty$.
	The result follows by taking the limit as $n\to \infty$ on both sides of~\eqref{eq:FKun} and using~\eqref{eq:UntoU} in Proposition~\ref{prop:Un} on the left-hand side and another application of dominated convergence on the right-hand side.
\end{proof}
We say that $U_0:\R\to [0,1]$ has a \textit{compact interface} if there exists $K_0<\infty$ such that $U_0(x)=1$ for $x\le -K_0$ and $U_0(x)=0$ for $x\ge K_0$.
We can now prove Gaussian tail bounds on $U(t,L_t+\cdot)$ and $|\partial_x U(t,L_t+\cdot)|$ when $U_0$ has a compact interface.
\begin{lem} \label{lem:GaussiantailU}
Suppose $U_0$ satisfies Assumption~\ref{assum:standing assumption ic} and has a compact interface; let $(U(t,x),L_t)$ solve~\eqref{eq:FBP_CDF}.
Then for $0<\delta<T$, there exist $c_{\delta,T}>0$ and $C_{\delta,T}<\infty$ such that
\[
U(t,L_t+x)\le C_{\delta,T} e^{-c_{\delta,T}x^2}
\quad \text{and}\quad
|\partial_x U(t,L_t+x)|\le C_{\delta,T} e^{-c_{\delta,T}x^2}
\quad \forall t\in [\delta, T], \, x\ge 0.
\]
\end{lem}
\begin{proof}
Take $K_0<\infty$ such that $U_0(x)=0$ for $x\ge K_0$.
By Lemma~\ref{lem:FKforinfinitemass}, for $s\in [0,T]$ and $x\in \R$,
\begin{equation} \label{eq:GaussiantailU}
U(s,x)\le e^s \Esub{x}{U_0(B_s)}\le e^T \psub{x}{B_s\le K_0}
\le e^T (e^{-\frac 1 {2T}(x-K_0)^2} +\1_{\{x<K_0\}}).
\end{equation}
Take $K_1<\infty$ such that $|L_t|\le K_1$ $\forall t\le T$.
For $t>0$ and $x\in \R$, let $u(t,x)=-\partial_x U(t,x)\ge 0$.
Then by Proposition~\ref{prop:fbpsoln}(v) and the Feynman-Kac formula (see e.g.~\cite[Proposition~3.1]{Berestycki2018}),
for $t>0$ and $x>L_t+1$, letting $\tau=\frac t 2 \wedge \inf\{s\ge 0:B_s\le L_{t-s}+1\}$,
\begin{align} \label{eq:Gaussiantailu}
u(t,x)&=\Esub{x}{e^\tau u(t-\tau,B_{\tau})} \notag \\
&\le e^{t/2}(3+(t/2)^{-1/2})\psub{x}{\inf_{s\le t/2}B_s\le K_1+1}
+e^{t/2}\int_{L_{t/2}+1}^\infty u(t/2,y)\frac 1 {\sqrt{\pi t}}e^{-(x-y)^2/t}dy,
\end{align}
where in the second inequality we used Proposition~\ref{prop:fbpsoln}(v) and that $L_{t-s}+1\le K_1+1$ $\forall s\le t/2$.
Using the reflection principle and a Gaussian tail bound for the first term on the right-hand side of~\eqref{eq:Gaussiantailu}, and using integration by parts and~\eqref{eq:GaussiantailU} (with $s=t/2$) for the second term on the right-hand side of~\eqref{eq:Gaussiantailu}, it follows that there exist $c>0$ and $C<\infty$ such that
\begin{equation} \label{eq:Gaussiantailu2}
u(t,x)\le Ce^{-cx^2} \quad \forall t\in [\delta, T], \, x>L_t+1.
\end{equation}
Using the bound on $\partial_x U$ in Proposition~\ref{prop:fbpsoln}(v) again, the result follows from~\eqref{eq:GaussiantailU} and~\eqref{eq:Gaussiantailu2}.
\end{proof}
Recall from Section~\ref{subsec:notation} that
for a Borel probability measure $\mu$ on $\R$, we let $\mathbb P_\mu$ denote the probability measure under which $(B_t)_{t\ge 0}$ is a Brownian motion with $B_0\sim \mu$.
The following Feynman-Kac representation is proved in~\cite{Berestycki2024}, and will be used in Section~\ref{sec:finitemass}.
\begin{lem}[Theorem~B.1 in~\cite{Berestycki2024}] \label{lem:FKforwardstime}
Suppose $U_0$ satisfies Assumption~\ref{assum:standing assumption ic}, and let $u_0$ denote the probability measure such that $u_0((x,\infty))=U_0(x)$ $\forall x\in \R$.
Let $(U(t,x),L_t)$ denote the solution of~\eqref{eq:FBP_CDF} with initial condition $U_0$, and let
$\tau=\inf\{t>0:B_t\le L_t\}$. Then for any $t>0$,
\begin{align*}
\psub{u_0}{\tau>t}&=e^{-t}\\
\text{and }\quad \psub{u_0}{B_t>x \lvert \tau>t}&=U(t,x) \; \forall x\in \R.
\end{align*}
\end{lem}

\subsection{Stretching lemma and consequences} \label{subsec:stretching}

Suppose $U,V:\R\to [0,1]$.
We say that $U$ is \textit{more stretched} than $V$, denoted by $U\geq_s V$, if for any $c\in \R$ and $x_1,x_2\in \R$ with $x_1\leq x_2$ we have
\begin{equation}\label{eq:defin 1 of being more stretched}
	U(x_1)> V(x_1+c)\Rightarrow U(x_2)\geq V(x_2+c).
\end{equation}
We write $U\geq_s V$ and $V\leq_s U$ interchangeably.
The following {\it stretching lemma} was proved in \cite{Berestycki2024}.
\begin{lem}[Stretching lemma, Lemma 2.11 in \cite{Berestycki2024}]\label{lem:extended maximum principle}
Suppose $U_0$ and $V_0$ satisfy Assumption~\ref{assum:standing assumption ic}, and let $U(t,x)$, $V(t,x)$ denote the solutions of~\eqref{eq:FBP_CDF} with initial conditions $U_0$, $V_0$ respectively. If $U_0\geq_s V_0$, then $U(t,\cdot) \geq_s V(t,\cdot)$ for all $t\geq 0$. 
\end{lem}

Note that \eqref{eq:defin 1 of being more stretched} differs slightly from the notion of being `more stretched' given for the FKPP equation by Bramson in~\cite[p.33]{Bramson1983}, where~\eqref{eq:defin 1 of being more stretched} is assumed to hold with both inequalities given by $>$, and also with both inequalities given by $\ge$. Since solutions $(U(t,x),L_t)$ of \eqref{eq:FBP_CDF} have $U(t,x)=1$ for $t>0$ and $x\le L_t$, the stretching lemma would not be true were we to define `more stretched' as in \cite[p.33]{Bramson1983}.

We note the following easy consequence of the definition of $U\ge_s V$.
\begin{lem} \label{lem:stretchdecr}
Suppose $U,V:\R\to [0,1]$ and $x_1,x_2\in \R$ with $U\ge_s V$, $U(x_1)=V(x_2)$ and $V$ continuous and strictly decreasing on $[x_2,\infty)$.
Then $U(x_1+y)\ge V(x_2+y)$ $\forall y\ge 0$.
\end{lem}
\begin{proof}
Since $V$ is strictly decreasing on $[x_2,\infty)$, for any $\varepsilon>0$ we have $U(x_1)=V(x_2)>V(x_2+\varepsilon)$.
Hence by the definition of $U\ge_s V$ in~\eqref{eq:defin 1 of being more stretched}, for $y\ge 0$ we have $U(x_1+y)\ge V(x_2+\varepsilon+y)$. By letting $\varepsilon\to 0$, the result follows.
\end{proof}

Recall from Section~\ref{subsec:notation} that $(U^H(t,x),L_t^H)$ is the solution of~\eqref{eq:FBP_CDF} with Heaviside initial condition
$H(x)=\1_{\{x<0\}}$. We can use the stretching lemma and a classical result of Bramson for the FKPP equation to prove the following lower bound on $L^H_t$ for large $t$.
\begin{lem} \label{lem:LHtlower}
There exists $C_0<\infty$ such that 
\[
L^H_t \ge \sqrt{2} t -\frac{3}{2\sqrt 2}\log(t+1)-C_0 \quad \forall t\ge 0.
\]
\end{lem}
\begin{proof}
For $n\in \Nm$, let $U^H_n(t,x)$ solve~\eqref{eq:CDF_n} with $U_0(x)=\1_{\{x<0\}}$, and for $t\ge 0$, let
\[
m_{1/2}^{(n)}(t)=\sup\{x\in \R:U^H_n(t,x)\ge 1/2\}\quad \text{and}\quad m_{1/2}(t)=\sup\{x\in \R:U^H(t,x)\ge 1/2\}.
\]
Then by Proposition~\ref{prop:Un} we have $m_{1/2}(t)\ge m_{1/2}^{(n)}(t)$ $\forall n\in \Nm$, $t\ge 0$.
By~\cite[Theorem~3, p.~141]{Bramson1983},
there exists $K<\infty$ such that
\[
m_{1/2}^{(2)}(t) \ge \sqrt{2} t -\frac{3}{2\sqrt 2}\log(t+1)-K \quad \forall t\ge 0.
\]
By Lemma~\ref{lem:extended maximum principle}, since $H\le_s \Pi_{\min}$ we have $U^H(t,\cdot)\le_s \Pi_{\min}$ for $t\ge 0$.
Then by Lemma~\ref{lem:stretchdecr} and Proposition~\ref{prop:fbpsoln}(iv), and since $\Pi_{\min}(0)=1$, for $t>0$ it follows that
$U^H(t,L^H_t+\Pi_{\min}^{-1}(1/2))\le 1/2$.
Therefore
\[
m_{1/2}(t)\le L^H_t+\Pi_{\min}^{-1}(1/2) \quad \forall t>0,
\]
and the result follows.
\end{proof}

Recall~\eqref{eq:L0defn};
the following lemma is proved in~\cite{Berestycki2024} as a consequence of the comparison principle.
\begin{lem}[Lemma 2.12 in \cite{Berestycki2024}] \label{lem:less stretching means slower boundary}
Suppose $U_0$ and $V_0$ satisfy Assumption~\ref{assum:standing assumption ic}, and let $(U(t,x),L^U_t)$ and $(V(t,x),L^V_t)$ denote the solutions of~\eqref{eq:FBP_CDF} with initial conditions $U_0$, $V_0$ respectively.
	If $U_0\geq_s V_0$ and $L^U_0,L^V_0>-\infty$, then $L^U_t-L^U_0\geq L^V_t-L^V_0$ for all $t\geq 0$.
\end{lem}
As a useful consequence of Lemma~\ref{lem:less stretching means slower boundary}, we can establish the following lower bound on the speed of the free boundary.
\begin{lem}\label{lem:boundary locally Lipschitz from the left}
	There exist $(t_c:-\infty<c<\sqrt{2})$ such that:
	\begin{enumerate}
		\item $t_c\ra 0$ as $c\ra -\infty$;
		\item for any $U_0$ satisfying Assumption~\ref{assum:standing assumption ic}, letting $(U(t,x),L_t)$ denote the solution of \eqref{eq:FBP_CDF}, 
		\[
		L_t-L_s\geq c(t-s) \; \forall t_{c}\leq s\leq t.
		\]
	\end{enumerate}
\end{lem}

\begin{proof}
Recall from Section~\ref{subsec:notation} that $(U^H(t,x),L_t^H)$ is the solution of~\eqref{eq:FBP_CDF} with Heaviside initial condition
$H(x)=\1_{\{x<0\}}$. It follows from the stretching lemma (Lemma~\ref{lem:extended maximum principle}) that 
	\begin{equation}\label{eq:Heaviside becomes more stretched}
		H \leq_s U^H(s,\cdot) \leq_s U^H(t,\cdot) \quad\text{for all}\quad 0\leq s\leq t<\infty.
	\end{equation}
For any $n>0$, by \eqref{eq:Heaviside becomes more stretched} and Lemma \ref{lem:less stretching means slower boundary}, $(L^H_{\frac{k+1}{n}}-L^H_{\frac{k}{n}})_{k=0}^\infty$ is non-decreasing in $k$.
Therefore, by first considering the case $s,t\in \mathbb Q$ and then using that $u\mapsto L^H_u$ is continuous,
\begin{equation} \label{eq:LHtgetsfaster}
\frac{L_s^H-L_0^H}{s}\leq \frac{L_{t}^H-L_0^H}{t} \quad \forall 0<s\le t.
\end{equation}
	
	We now essentially follow the proof of the mean value theorem. For any $c<\sqrt{2}$ we can choose $ T_c>0$ such that $L^H_{ T_c}-L^H_0\geq c T_c$ by Lemma~\ref{lem:LHtlower}. 
	We now let
	\[
	f(s)= L^H_s-L^H_0-\frac{L^H_{ T_c}-L^H_0}{ T_c}s\quad \text{for }s\in [0,T_c],
	\]
	and choose $t_c\in (0, T_c)$ such that $f(t_c)=\inf_{s\in [0,T_c]}f(s)$, which is possible since $f(0)=0=f(T_c)$ and $f(s)\le 0$ $\forall s\in (0,T_c)$ by~\eqref{eq:LHtgetsfaster}.
	Then for any $s\in [t_c, T_c]$,
	\[
	L^H_s-L^H_{t_c}=(L^H_s-L^H_0)-(L^H_{t_c}-L^H_0)\geq \frac{L^H_{ T_c}-L^H_0}{ T_c}(s-t_c)\geq c(s-t_c),
	\]
	where the last inequality follows from our choice of $T_c$.
	By \eqref{eq:Heaviside becomes more stretched} and Lemma \ref{lem:less stretching means slower boundary}, it follows that
	$L^H_s-L^H_{t}\geq c(s-t)$ for all $s\geq t\ge t_c$.
	
	Now to see that $t_c$ can be chosen so that $t_c\ra 0$ as $c\ra -\infty$, note that $\frac{L^H_{\delta}-L^H_0}{\delta}$ is finite for any $\delta>0$, and therefore, for any $\delta>0$, for $K>0$ sufficiently large, for $c\le -K$ we can choose $T_c\le \delta$ and hence $t_c<\delta$ in the above argument.
	
	Finally, take $U_0$ satisfying Assumption~\ref{assum:standing assumption ic}, and let $(U(t,x),L_t)$ denote the solution of \eqref{eq:FBP_CDF}.
By Lemma~\ref{lem:extended maximum principle}, we have $U(s,\cdot)\geq_s U^H(s,\cdot)$ for any $s\ge 0$, and so the result follows from Lemma~\ref{lem:less stretching means slower boundary}.
\end{proof}
Recall the definition of $\Pi_c$ for $c\ge \sqrt{2}$ in~\eqref{eq:Picdefn}. We now give an upper bound on the speed of the free boundary, if the initial condition is less stretched than $\Pi_c$ for some $c\ge \sqrt 2$.
\begin{lem}\label{lem:boundary Lipschitz from the right when less stretched than a travelling wave}
Suppose $U_0$ satisfies Assumption~\ref{assum:standing assumption ic}, and that
$U_0\leq_s \Pi_c$ for some $c\ge \sqrt{2}$. 
Let $(U(t,x),L_t)$ denote the solution of~\eqref{eq:FBP_CDF}.
Then $L_t-L_s\leq c(t-s)$ for all $0\leq s\leq t$.
\end{lem}
\begin{proof}
Observe that by~\eqref{eq:defin 1 of being more stretched}, $U_0\leq_s \Pi_c$ for some $c\ge \sqrt 2$ and $U_0\not\equiv 0$ implies $L_0>-\infty$.
Also, by Lemma~\ref{lem:extended maximum principle} we have $U(s,\cdot)\leq_s \Pi_c$ for all $s\ge 0.$
The result is then a simple consequence of Lemma~\ref{lem:less stretching means slower boundary} applied with initial condition $U(s,\cdot)$.
\end{proof}

We obtain the following immediate corollary of Lemma~\ref{lem:boundary locally Lipschitz from the left} and Lemma~\ref{lem:boundary Lipschitz from the right when less stretched than a travelling wave}.
\begin{cor}\label{cor:locally Lipschitz if less stretched than a travelling wave}
Suppose $U_0$ satisfies Assumption~\ref{assum:standing assumption ic}, and that
$U_0\leq_s \Pi_c$ for some $c\ge \sqrt{2}$. 
Let $(U(t,x),L_t)$ denote the solution of~\eqref{eq:FBP_CDF}.
Then $t\mapsto L_t$ is Lipschitz continuous on $[\delta, \infty)$
 for any $\delta>0$.
\end{cor}
We will also use the following result about the long-term asymptotics of the position of the free boundary.
\begin{lem} \label{lem:Lttsqrt2}
Suppose $U_0$ satisfies Assumption~\ref{assum:standing assumption ic} with $\limsup_{x\to \infty} \frac 1x \log U_0(x)\le -\sqrt{2}$.
Let $(U(t,x),L_t)$ denote the solution of~\eqref{eq:FBP_CDF};
then
\[
\lim_{t\to \infty}\frac{L_t}t=\sqrt 2.
\]
\end{lem}
\begin{proof}
For $t\ge 0$ and $x\in \R$, let
\[
U^+(t,x):=e^t \Esub{x}{U_0(B_t)},
\]
and for $t\ge 0$,
let $L^+_t:=\sup\{x\in \Rm:U^+(t,x)\ge 1\}$.
Then since $\limsup_{x\to \infty} \frac 1x \log U_0(x)\le -\sqrt{2}$ and $\lim_{x\to -\infty}U_0(x)=1$, by an easy calculation
(see e.g.~\cite[Corollary~1 of Lemma~4.2]{Bramson1983}),
\begin{equation} \label{eq:mphilimit}
	\lim_{t\to \infty} \frac{L^+_t}{t}= \sqrt{2}.
\end{equation}
By Lemma~\ref{lem:FKforinfinitemass}, 
we have $U(t,x)\le U^+(t,x)$ $\forall t>0,$ $x\in \Rm$, and so
$L_t\le L^+_t$ $\forall t> 0$.
Since we also have $L_{t}-L_1\ge L^H_{t-1}$ for $t\ge 1$ by Lemma~\ref{lem:less stretching means slower boundary}, the result follows from Lemma~\ref{lem:LHtlower}.
\end{proof}

We end this section with the following convergence result, which will be used in Section~\ref{section:magic formula}.
\begin{lem} \label{lem:convofLn}
Suppose, for each $n\in \Nm$, $U_0^{(n)}$ satisfies Assumption~\ref{assum:standing assumption ic}, and let $(U^{(n)}(t,x),L^{(n)}_t)$ solve~\eqref{eq:FBP_CDF} with initial condition $U_0^{(n)}$.
Suppose $U_0$ satisfies Assumption~\ref{assum:standing assumption ic}, and $(U^{(n)}_0)_{n\in \Nm}$ is a non-decreasing sequence converging uniformly on $\R$ to $U_0$ as $n\to \infty$.
Let $(U(t,x),L_t)$ solve~\eqref{eq:FBP_CDF} with initial condition $U_0$.
Then for each $t>0$, $(L^{(n)}_t)_{n\in \Nm}$ is a non-decreasing sequence with
\[
L^{(n)}_t \to L_t \quad \text{as }n\to \infty.
\]
\end{lem}
\begin{proof}
Note first that by the comparison principle in Proposition~\ref{prop:fbpcomparison}, $L^{(n)}_t\le L_t$ $\forall t>0,$ $n\in \Nm$.
By Lemma~\ref{lem:Bramsongronwall}, for any $s>0$ and $x\in \R$ we have that $U^{(n)}(s,x)\to U(s,x)$ as $n\to \infty$.

Now fix $t>0$, and take $\varepsilon\in (0,t/2)$ sufficiently small that
\begin{equation} \label{eq:epssmallassump}
e^{2\varepsilon}(1-\varepsilon)(1-4e^{-\varepsilon^{-1/3}/4})>1.
\end{equation}
Then take $N_\varepsilon<\infty$ such that $U^{(n)}(t-2\varepsilon,L_{t-2\varepsilon})\ge 1-\varepsilon$ for $n\ge N_\varepsilon$.
Take $n\ge N_\varepsilon$ and suppose (aiming for a contradiction) that $L^{(n)}_s<L_{t-2\varepsilon}-2\varepsilon^{1/3}$ $\forall s\in [t-2\varepsilon,t]$.
Then by Lemma~\ref{lem:FKforinfinitemass},
\begin{align*}
U^{(n)}(t,L_{t-2\varepsilon}-\varepsilon^{1/3})&\ge e^{2\varepsilon}(1-\varepsilon)\psub{L_{t-2\varepsilon}-\varepsilon^{1/3}}{B_s\ge L_{t-2\varepsilon}-2\varepsilon^{1/3}\, \forall s\in [t-2\varepsilon,t], B_t\le L_{t-2\varepsilon}}\\
&\ge e^{2\varepsilon}(1-\varepsilon)(1-4e^{-\varepsilon^{2/3}/(4\varepsilon)}),
\end{align*}
where the second inequality follows from the reflection principle and a Gaussian tail bound.
By~\eqref{eq:epssmallassump}, this gives us a contradiction.
Therefore, using Lemma~\ref{lem:less stretching means slower boundary} and recalling the definition of $L^H_t$ in Section~\ref{subsec:notation}, for $n\ge N_\varepsilon$,
\[
L^{(n)}_t\ge L_{t-2\varepsilon}-2\varepsilon^{1/3}+\inf_{u\in [0,2\varepsilon]}L^H_u.
\]
Since $\varepsilon>0$ could be chosen arbitrarily small, and $(L_s)_{s>0}$ and $(L^H_s)_{s\ge 0}$ are continuous with $L^H_0=0$, the result follows.
\end{proof}

\section{The Brunet-Derrida relation}\label{section:magic formula}

In this section, we prove Theorems~\ref{theo:magic formula},~\ref{theo:initial condition limsup bdy relation} and~\ref{theo:samebdysameU0}.
Before proving these results, we start by briefly explaining the following 
illustrative example mentioned in Section~\ref{subsec:mainresults} after Theorem~\ref{theo:magic formula}, which shows that
for $r\in (\sqrt{2},\infty)$, both sides of~\eqref{eq:magic formula} may be finite whilst~\eqref{eq:magic formula} fails to hold.
Recall from Section~\ref{subsec:notation} that we let $(U^H(t,x),L_t^H)$ denote the solution of~\eqref{eq:FBP_CDF} with Heaviside initial condition.
\begin{counterexample}\label{counterexample:magic formula}
	Take an arbitrary $r\in (\sqrt{2},\infty)$. For $T>0$, let $(V^T(t,x),L^T_t)$ denote the solution of~\eqref{eq:FBP_CDF} with initial condition $V^T_0(x):=U^H(T,x+L^H_T)$ $\forall x\in \R$.
	Then for any $T>0$, both $\int_{L_0^T}^{\infty}V^T_0(x)e^{rx}dx$ and $\int_{0}^{\infty}e^{rL^T_t-(1+\frac{1}{2}r^2)t}dt$ are finite, but there exists $T>0$ such that
	\begin{equation} \label{eq:magiccounter}
	\int_{L_0^T}^{\infty}V^T_0(x)e^{rx}dx\neq -\frac{1}{r}+\frac{1}{r}	\int_0^{\infty}e^{rL^T_t-(1+\frac{1}{2}r^2)t}dt.
	\end{equation}
\end{counterexample}
\begin{proof}	
By the definition of $V_0^T$, for $t\ge 0$ and $x\in \R$ we have
	\[
	V^T(t,x)=U^H(T+t, x+L^H_{T}) \quad \text{and}\quad  L^T_t:=L^H_{T+t}-L^H_T.
	\]
By Lemma~\ref{lem:GaussiantailU},	
	for any fixed $T>0$, $V^T_0(x)$ satisfies a Gaussian tail bound for $x\ge 0$, and so the left-hand side of \eqref{eq:magiccounter} is finite. 	
	However, by Theorem~\ref{theo:convtoPimin} we have that $V^T_0$ converges uniformly to $\Pi_{\min}$ as $T\to \infty$. Since $\int_{0}^{\infty}e^{rx}\Pi_{\min}(x)dx=+\infty$ (by~\eqref{eq:Pimindefn} and~\eqref{eq:minimal travelling wave}, and since $r>\sqrt{2}$), 
and since $L^T_0=0$,	
	it follows that 
	\begin{equation} \label{eq:VT0intinf}
	\int_{L^T_0}^{\infty}V^T_0(x)e^{rx}dx\ra \infty \quad \text{as }T\to \infty.
	\end{equation}
	
	On the other hand, $L^T_{t}= L^H_{T+t}-L^H_T\leq \sqrt{2}t$ for any $t\geq 0$, by Lemma \ref{lem:boundary Lipschitz from the right when less stretched than a travelling wave}. Therefore, for any $T>0$,
	\[
	\int_0^{\infty}e^{rL^T_t-(1+\frac{1}{2}r^2)t}dt\leq \int_0^{\infty}e^{(\sqrt{2}r-(1+\frac{1}{2}r^2))t}dt<\infty,
	\]
	where the second inequality holds because $r>\sqrt{2}$.
	The result now follows by~\eqref{eq:VT0intinf}.	
\end{proof}

\subsection{Proof of Theorem \ref{theo:magic formula}} \label{subsec:magicproof}
We will prove Theorem \ref{theo:magic formula} by establishing~\eqref{eq:magic formula} for $r\in (-\infty,\sqrt{2})\setminus \{0\}$ for an increasingly large family of initial conditions $U_0$, in the following three steps. 

In Step 1, we will prove~\eqref{eq:magic formula} for a narrow class of initial conditions $U_0$, precisely those satisfying the following assumption.
Recall from Section~\ref{section:properties of free-boundary} that we say that $U_0:\R\to [0,1]$ has a \textit{compact interface} if there exists $K_0<\infty$ such that $U_0(x)=1$ for $x\le -K_0$ and $U_0(x)=0$ for $x\ge K_0$.
\begin{assum}\label{assum:initial assumption on initial condition}
Suppose $U_0$ satisfies Assumption \ref{assum:standing assumption ic}, $U_0$ has a compact interface, $L_0:=\inf\{x\in \R:U_0(x)<1\}=0$,
and 
 $U_0\leq_s \Pi_c$ for some $c\in [\sqrt{2},\infty)$. 
\end{assum}
We will prove~\eqref{eq:magic formula} under this assumption by following the strategy in the non-rigorous argument in~\cite{Berestycki2018a}.

In Step 2, we will then extend \eqref{eq:magic formula} to those initial conditions satisfying the following assumption.
\begin{assum}\label{assum:smooth assumption on initial condition}
Suppose $U_0$ satisfies Assumption \ref{assum:standing assumption ic}, $L_0:=\inf\{x\in \R:U_0(x)<1\}>-\infty$, $U_0\in C^2((L_0,\infty))\cap C(\Rm)$, and $U_0'(x)<0$ $\forall x>L_0$. 
\end{assum}
Finally, in Step 3, we will extend \eqref{eq:magic formula} to all initial conditions $U_0$ satisfying Assumption \ref{assum:standing assumption ic}.

We will use the following existence and uniqueness result for solutions of integral equations.
\begin{lem} \label{lem:picardintegralform}
Suppose $F:[0,\infty)\times \R \to \R$ is measurable, and there exists $C<\infty$ such that
\[
|F(t,x)-F(t,y)|\le C|x-y| \quad \text{and} \quad |F(t,x)|\le C(1+|x|) \quad \forall t\ge 0, \, x,y\in \R.
\]
Then for any $y_0\in \R$, a continuous solution $(y(t),t\ge 0)$ of the integral equation
\begin{equation} \label{eq:picardintegraleqn}
y(t)-y(0)=\int_0^tF(s,y(s))ds \quad \text{for }t\geq 0,\quad y(0)=y_0
\end{equation}
exists and is unique.
\end{lem}
This result follows from the same proof as for the Picard-Lindel\"of theorem, which gives an existence and uniqueness result for the solution of the ODE $\dot{y}(t)=F(t,y(t))$. In the Picard-Lindel\"of theorem, $F(t,y)$ is also assumed to be continuous in $t$, in order to show that the solution of the integral equation is a classical solution of the ODE, which we do not claim here.
\begin{proof}
The proof is standard; we recall it here briefly for completeness.
Fix $\delta \in (0,1/C)$, and let $C([0,\delta];y_0)$ denote the space of continuous functions $(x(s),s\in [0,\delta])$ with $x(0)=y_0$, equipped with the subspace metric coming from the uniform norm.
Define a map
\begin{align*}
\Gamma:C([0,\delta];y_0) &\to C([0,\delta];y_0)\\
x=(x(t),t\in [0,\delta]) &\mapsto (\Gamma_t(x),t\in [0,\delta]), \quad  \text{where}\quad 
\Gamma_t(x)=y_0+\int_0^t F(s,x(s))ds.
\end{align*}
Then for $x_1,x_2\in C([0,\delta];y_0)$, we have $\|\Gamma(x_1)-\Gamma(x_2)\|_\infty \le C \delta \|x_1-x_2\|_{\infty}$.
Since $C\delta<1$, it follows from the Banach fixed-point theorem that $\Gamma$ has a unique fixed point in $C([0,\delta];y_0)$, and therefore~\eqref{eq:picardintegraleqn} has a unique continuous solution on $[0,\delta]$.
By an induction argument over successive time intervals $[k\delta,(k+1)\delta]$ for $k\in \Nm$, the result follows.
\end{proof}

\begin{proof}[Proof of Theorem~\ref{theo:magic formula}]
\textbf{Step 1:} Proof of~\eqref{eq:magic formula} for $U_0$ satisfying Assumption~\ref{assum:initial assumption on initial condition}.

We will first prove \eqref{eq:magic formula} for initial conditions $U_0$ satisfying the following assumption.
\begin{assum}\label{assum:initial assumption 1 on initial condition}
Take $U'_0$ satisfying Assumption \ref{assum:initial assumption on initial condition} and take $\delta>0$, 
let $(U'(t,x),L^{U'}_t)$ solve~\eqref{eq:FBP_CDF} with initial condition $U'_0$,
and let
\begin{equation}\label{eq:U0' evolved time delta}
U_0(x)= U'(\delta, L^{U'}_{\delta}+x) \quad \forall x\in \R.
\end{equation}
\end{assum}
In other words, we obtain $U_0$ by taking some initial condition satisfying Assumption \ref{assum:initial assumption on initial condition} and evolving it for some small time $\delta>0$. 
Once we have established~\eqref{eq:magic formula} for $U_0$ satisfying Assumption~\ref{assum:initial assumption 1 on initial condition},
we will then conclude Step 1 by taking the limit as $\delta\ra 0$.

Now fix $\delta>0$ and $U_0'$ satisfying Assumption \ref{assum:initial assumption on initial condition}, and let $U_0$ be given by~\eqref{eq:U0' evolved time delta}. 
Let $(U(t,x),L_t)$ solve~\eqref{eq:FBP_CDF} with initial condition $U_0$, and let $L_0:=0=\lim_{t\to 0}L_t$ and $U(0,\cdot)=U_0$.

It follows from Corollary \ref{cor:locally Lipschitz if less stretched than a travelling wave} that $t\mapsto L_t$ is Lipschitz continuous on $[0,\infty)$, and hence is weakly differentiable with essentially bounded derivative. In particular, $t\mapsto L_t$ is in the Sobolev space $W^{1,\infty}$. We write $\dot{L}_t$ for the weak derivative of $L_t$. In the remainder of Step~1, time derivatives should be understood as weak derivatives.

By Lemma~\ref{lem:GaussiantailU}, using the definition of $U_0$ in~\eqref{eq:U0' evolved time delta} and that $U'_0$ has a compact interface,
for any $T>0$ there exist $C_T<\infty$ and $c_T>0$ such that 
\begin{equation}\label{eq:convergence of first derivative to 0}
	U(t,L_t+x)\leq C_Te^{-c_T x^2} \quad\text{and}\quad \lvert \partial_xU(t,L_t+x)\rvert\leq C_Te^{-c_Tx ^2}
	\quad \forall t\in [0, T],\, x\geq 0.
\end{equation}
By Proposition~\ref{prop:fbpsoln}, $U'\in C^{1,2}(\{(t,x):t>0,\, x>L^{U'}_t\})$, and therefore, by~\eqref{eq:U0' evolved time delta}, for any $\epsilon>0$ and $T>0$, $\Delta U$ is bounded on $\{(t,x):0\leq t\leq T, \, L_t+\epsilon\leq x\leq L_t+\epsilon^{-1}\}$.

We now follow the proof strategy from the non-rigorous argument in~\cite{Berestycki2018a}. We define
\begin{equation} \label{eq:grtdefn}
	g(r,t):=\int_0^{\infty}U(t,L_t+x)e^{rx}dx \quad \text{for }r< \sqrt{2}, \, t\ge 0.
\end{equation}
We will show that $(g(r,t),t\ge 0)$ is the unique solution of an integral equation.
Take $T\in (0,\infty)$, $0\leq t_1<t_2\le T$ and $r<\sqrt{2}$. Then we have that
\[
g(r,t_2)-g(r,t_1)=\lim_{\epsilon\ra 0}\Big[\int_{\epsilon}^{\frac{1}{\epsilon}}U(t_2,L_{t_2}+x)e^{rx}dx-\int_{\epsilon}^{\frac{1}{\epsilon}}U(t_1,L_{t_1}+x)e^{rx}dx\Big].
\]
Then for $\epsilon>0$, 
since $\dot{L}_t$ and $U$ are bounded, and $\partial_xU$ and $\Delta U$ are bounded on $\{(t,x):0\leq t\leq T, \, L_t+\epsilon\leq x\leq L_t+\epsilon^{-1}\}$,
and using Fubini's theorem for the last equality, 
 we have that
\begin{align} \label{eq:gt1t2Fubini}
	&\int_{\epsilon}^{\frac{1}{\epsilon}}U(t_2,L_{t_2}+x)e^{rx}dx-\int_{\epsilon}^{\frac{1}{\epsilon}}U(t_1,L_{t_1}+x)e^{rx}dx \notag \\
	&\quad =\int_{\epsilon}^{\frac{1}{\epsilon}}\int_{t_1}^{t_2}\partial_t[U(t,L_t+x)e^{rx}]dt dx \notag \\
	&\quad =\int_{\epsilon}^{\frac{1}{\epsilon}}\int_{t_1}^{t_2}\Big[\partial_xU(t,L_t+x)\dot{L}_t+\partial_tU(t,L_t+x)\Big]e^{rx}dt dx \notag \\
	&\quad =\int_{\epsilon}^{\frac{1}{\epsilon}}\int_{t_1}^{t_2}\Big[\partial_xU(t,L_t+x)\dot{L}_t+\tfrac{1}{2}\Delta U(t,L_t+x)+U(t,L_t+x)\Big]e^{rx}dtdx \notag \\
	&\quad = \int_{t_1}^{t_2}\int_{\epsilon}^{\frac{1}{\epsilon}}\Big[\partial_xU(t,L_t+x)\dot{L}_t+U(t,L_t+x)\Big]e^{rx}dx dt 
	+\int_{t_1}^{t_2}\int_{\epsilon}^{\frac{1}{\epsilon}}\tfrac{1}{2}\Delta U(t,L_t+x)e^{rx}dx dt.
\end{align}
Then by integration by parts, we have that for $t\in [t_1,t_2]$,
\begin{align} \label{eq:DeltaUintbyparts}
&\int_{\epsilon}^{\frac{1}{\epsilon}}\Delta U(t,L_t+x)e^{rx}dx \notag \\
	&\quad = \Big[\partial_x U (t,L_t+x)e^{rx}\Big]_{x=\epsilon}^{x=\frac{1}{\epsilon}}-r\Big[ U (t,L_t+x)e^{rx}\Big]_{x=\epsilon}^{x=\frac{1}{\epsilon}}+r^2\int_{\epsilon}^{\frac{1}{\epsilon}}U(t,L_t+x)e^{rx} dx.
\end{align}
Putting~\eqref{eq:gt1t2Fubini} and~\eqref{eq:DeltaUintbyparts} together, we obtain that
\begin{align} \label{eq:gt1gt2witheps}
&g(r,t_2)-g(r,t_1) \notag \\
&=\lim_{\epsilon\ra 0}\int_{t_1}^{t_2}\Big\{\int_{\epsilon}^{\frac{1}{\epsilon}}\Big[\partial_xU(t,L_t+x)\dot{L}_t+U(t,L_t+x)\Big]e^{rx}dx \notag \\ 
&\qquad +\frac{1}{2}\Big[\partial_x U (t,L_t+x)e^{rx}\Big]_{x=\epsilon}^{x=\frac{1}{\epsilon}}-\frac{r}{2}\Big[ U (t,L_t+x)e^{rx}\Big]_{x=\epsilon}^{x=\frac{1}{\epsilon}}+\frac{r^2}{2}\int_{\epsilon}^{\frac{1}{\epsilon}}U(t,L_t+x)e^{rx} dx\Big\}dt.
\end{align}
By Proposition~\ref{prop:fbpsoln} and the definition of $U_0$ in~\eqref{eq:U0' evolved time delta}, for any $t\ge 0$ we have
\[
U(t,L_t+x)\to 1 \quad \text{and} \quad \partial_xU(t,L_t+x)\to 0 \quad \text{as } x\downarrow 0.
\]
Therefore, using the dominated convergence theorem, the boundedness of $\dot{L}_t$, and the Gaussian tail bounds from~\eqref{eq:convergence of first derivative to 0}, we can take the $\epsilon\ra 0$ limit of each of the terms in~\eqref{eq:gt1gt2witheps} separately, and obtain 
\begin{equation} \label{eq:gt1t2limit}
g(r,t_2)-g(r,t_1)=\int_{t_1}^{t_2}\Big\{\dot{L}_t\int_{0}^{\infty}\partial_xU(t,L_t+x)e^{rx}dx+\frac{r}{2}+\left(1+\frac{r^2}{2}\right)g(r,t)\Big\}dt.
\end{equation}
We now use integration by parts and~\eqref{eq:convergence of first derivative to 0} to see that
\[
\int_{0}^{\infty}\partial_xU(t,L_t+x)e^{rx}dx=[U(t,L_t+x)e^{rx}]_{x=0}^{x=\infty}-r\int_{0}^{\infty}U(t,L_t+x)e^{rx}dx=-1-rg(r,t).
\]
Substituting into~\eqref{eq:gt1t2limit}, we see that for any $T\in (0,\infty)$ and $0\leq t_1<t_2\le T$,
\begin{equation}
	g(r,t_2)-g(r,t_1)=\int_{t_1}^{t_2}\left( -\dot{L}_t-r\dot{L}_tg(r,t)+\frac{r}{2}+\left(1+\frac{r^2}{2}\right)g(r,t)\right) dt.
\end{equation}
It follows that for any $r<\sqrt{2}$ fixed, $(g(r,t),t\ge 0)$ is a solution to the integral equation
\begin{equation}\label{eq:integral equation g satisfies}
	y(t)-y(0)=\int_0^t\left( \Big[\frac{r}{2}-\dot{L}_s\Big]+\Big[\frac{r^2}{2}+1-r\dot{L}_s\Big]y(s)\right) ds \quad \text{for }t\ge 0,\quad y(0)=g(r,0),
\end{equation}
and by~\eqref{eq:convergence of first derivative to 0} we see that $(g(r,t),t\ge 0)$ is continuous in $t$.

For $t\ge 0$ and $y\in \R$, let 
\[
F(t,y):=\left(\frac{r}{2}-\dot{L}_t\right)+\left(\frac{r^2}{2}+1-r\dot{L}_t\right)y.
\]
Then by applying Lemma~\ref{lem:picardintegralform}, using that
$\dot{L}_t$ is bounded (recall that this followed from Corollary~\ref{cor:locally Lipschitz if less stretched than a travelling wave} and our assumption on $U_0$), we see that $(g(r,t),t\ge 0)$ is the unique continuous solution to the integral equation~\eqref{eq:integral equation g satisfies}.

We now take $r\in (-\infty,\sqrt{2})\setminus \{0\}$ and define, for $t\ge 0$,
\begin{equation} \label{eq:Itdefn}
I_t:=\Big(\frac{r^2}{2}+1\Big)t-rL_t
\end{equation}
and
\[
Y(t):=-\frac{1}{r}+\Big[g(r,0)+\frac{1}{r}-\frac{1}{r}\int_0^te^{-I_s}ds\Big]e^{I_t}.
\]
Since $L_t$ is weakly differentiable in $t$, so is $Y(t)$. Since $L_0=0$, we have
$
Y(0)=g(r,0).
$
Moreover, for $t>0$, writing $\dot{I}_t$ for the weak derivative of $I_t$,
\begin{align*}
	\partial_tY(t)&=-\frac{1}{r}e^{-I_t}e^{I_t}+\Big[g(r,0)+\frac{1}{r}-\frac{1}{r}\int_0^te^{-I_s}ds\Big]\dot{I}_te^{I_t}\\
	&=-\frac{1}{r}+\dot{I}_t\Big(Y(t)+\frac{1}{r}\Big)\\
	&=-\frac{1}{r}+\Big(\frac{r^2}{2}+1-r\dot{L}_t\Big)\Big(Y(t)+\frac{1}{r}\Big)\\
	&=\Big(\frac{r^2}{2}+1-r\dot{L}_t\Big)Y(t)+\Big(\frac{r}{2}-\dot{L}_t\Big).
\end{align*}
It follows that $(Y(t),t\ge 0)$ is also a solution of the integral equation \eqref{eq:integral equation g satisfies}, and so for $t\ge 0$,
\begin{equation}\label{eq:formula for g as ODE solution}
	g(r,t)=Y(t)=-\frac{1}{r}+\Big[g(r,0)+\frac{1}{r}-\frac{1}{r}\int_0^te^{-I_s}ds\Big]e^{I_t}.
\end{equation}

We now show that $g(r,t)$ is bounded as $t\to \infty$, and $I_t\to \infty$ as $t\to \infty$, which we will then combine with~\eqref{eq:formula for g as ODE solution} to show that~\eqref{eq:magic formula} holds.
To bound $g(r,t)$, note first that
since $U_0'$ has a compact interface (by Assumption~\ref{assum:initial assumption 1 on initial condition}), 
there exists $C<\infty$ such that
\begin{equation} \label{eq:U0sandwich}
\1_{\{x+C<0\}}\leq U'_0(x)\leq \1_{\{x-C<0\}} \quad \forall x\in \R.
\end{equation}
Recall from Section~\ref{subsec:notation} that we write $(U^H(t,x),L_t^H)$ for the solution of~\eqref{eq:FBP_CDF} with Heaviside initial condition
$H(x)=\1_{\{x<0\}}$. Then by the comparison principle in Proposition \ref{prop:fbpcomparison},
\[
U^H(t,x+C)\leq U'(t,x)\leq U^H(t,x-C) \quad \forall t>0,\, x\in \R.
\]
In particular, since $U'(t,L^H_t-C)\ge U^H(t,L^H_t)=1$ for $t>0$, we have
\begin{equation} \label{eq:LU'lower}
L^{U'}_t\geq L^H_t-C \quad \forall t>0.
\end{equation}
On the other hand, by Lemma~\ref{lem:extended maximum principle} we have $U^H(t,\cdot)\le_s \Pi_{\min}$ $\forall t>0$, 
and so by Lemma~\ref{lem:stretchdecr} and Proposition~\ref{prop:fbpsoln}(iv), and since $\Pi_{\min}(0)=1$, for $t>0$ we have
$U^H(t,L^H_t+x)\leq \Pi_{\min}(x)$ for all $x\geq 0$. 
Therefore, for $t>0$ and $x\ge 0$,
\begin{equation} \label{eq:U'upper}
U'(t,L^H_t+C+x)\leq U^H(t,L^H_t+x)\leq \Pi_{\min}(x).
\end{equation}
Combining~\eqref{eq:LU'lower} and~\eqref{eq:U'upper}, we obtain
\[
U'(t,L^{U'}_t+2C+x)\leq \Pi_{\min}(x) \quad \forall t>0,\, x\ge 0.
\]
By the definition of $U_0$ in~\eqref{eq:U0' evolved time delta} and the definition of $g(r,t)$ in~\eqref{eq:grtdefn}, and since $\Pi_{\min}(y)=1$ for $y\le 0$,
it follows that 
\begin{equation} \label{eq:gsupbound}
\sup_{t\ge 0}g(r,t)\leq\int_{0}^{\infty}\Pi_{\min}(x-2C)e^{rx}dx<\infty,
\end{equation}
where the second inequality follows by~\eqref{eq:Pimindefn} and~\eqref{eq:minimal travelling wave} since $r<\sqrt{2}$. 

We now establish a lower bound on $I_t$ for large $t$.
By~\eqref{eq:U0sandwich}, we have $U'_0(x)\leq \Pi_{\min}(x-C)$ $\forall x\in \R$, and so by Proposition~\ref{prop:fbpcomparison} and since $\Pi_{\min}$ is the shape of the travelling wave with speed $\sqrt{2}$, 
for $t>0$ we have
$U'(t,\sqrt{2}t+x)\le \Pi_{\min}(x-C)$ $\forall x\in \R$, which implies
 that
\[
L_t^{U'}\leq \sqrt{2}t+C \quad \forall t>0.
\]
Moreover, for $t\ge 0$ we have $L_t=L^{U'}_{t+\delta}-L^{U'}_\delta\le L^{U'}_{t+\delta}+C-L^H_{\delta}$ by~\eqref{eq:LU'lower}.
Therefore, by~\eqref{eq:Itdefn}, in the case $r\in (0,\sqrt{2})$, for $t\ge 0$,
\[
I_t=\Big(\frac{r^2}{2}+1\Big)t-rL_t\geq \Big(\frac{r^2}{2}-\sqrt{2}r+1\Big)t-\sqrt{2}r\delta -2rC+rL^H_\delta.
%=\Big(\frac{r}{\sqrt{2}}-1\Big)^2-C\ra \infty
\]
Since $\frac{r^2}{2}-\sqrt{2}r+1=(\frac{r}{\sqrt{2}}-1)^2>0$, it follows that there exists $c>0$ such that $I_t\ge ct$ for $t$ sufficiently large.
In the case $r<0$, by Lemma~\ref{lem:less stretching means slower boundary}
and then by Lemma~\ref{lem:LHtlower} we have
$L_t=L^{U'}_{t+\delta}-L^{U'}_\delta\ge L^H_t-L^H_0\ge 0$ for $t$ sufficiently large,
and so again there exists $c>0$ such that $I_t\ge ct$ for $t$ sufficiently large.

We now have that for any $r\in (-\infty,\sqrt{2})\setminus \{0\}$, by~\eqref{eq:gsupbound} and since $g(r,t)\ge 0$ by~\eqref{eq:grtdefn},
$\sup_{t\ge 0}|g(r,t)|<\infty$, and $\int_0^{\infty}e^{-I_t}dt<\infty$ with $I_t\to \infty$ as $t\to \infty$.
It therefore follows from~\eqref{eq:formula for g as ODE solution} that
\[
g(r,0)+\frac{1}{r}-\frac{1}{r}\int_0^{\infty}e^{-I_t}dt=0.
\]
Note that the above step is where the assumption of convergence to a travelling wave solution was used in the non-rigorous argument in~\cite{Berestycki2018a}, but as we have seen, this is not needed.

By the definition of $g(r,0)$ in~\eqref{eq:grtdefn} and since $U(0,\cdot)=U_0$ and $L_0=0$, and by the definition of $I_t$ in~\eqref{eq:Itdefn}, we conclude that~\eqref{eq:magic formula} holds (with both sides finite) for any $r\in (-\infty,\sqrt{2})\setminus \{0\}$ whenever the initial condition $U_0$ satisfies Assumption~\ref{assum:initial assumption 1 on initial condition}.
	
It follows that for any $U_0$ satisfying Assumption~\ref{assum:initial assumption on initial condition}, letting $(U(t,x),L_t)$ denote the solution of~\eqref{eq:FBP_CDF}, for any $r\in (-\infty,\sqrt{2})\setminus \{0\}$ and $\delta>0$,
\begin{equation}\label{eq:magic formula delta}
		\int_{0}^{\infty}U(\delta,L_\delta+x)e^{rx}dx=-\frac{1}{r}+\frac{1}{r}\int_0^{\infty}e^{r(L_{t+\delta}-L_{\delta})-(1+\frac{1}{2}r^2)t}dt.
	\end{equation}
We now send $\delta\ra 0$ to complete Step~1.
For the left-hand side of~\eqref{eq:magic formula delta}, we can write
\[
\int_{0}^{\infty}U(\delta,L_\delta+x)e^{rx}dx
=e^{-rL_{\delta}}\int_{L_\delta}^{\infty}U(\delta,x)e^{rx}dx
\to \int_{0}^{\infty}U_0(x)e^{rx}dx \quad \text{as }\delta \to 0,
\]
by dominated convergence, since $U(\delta,\cdot)\to U_0$ as $\delta \to 0$ almost everywhere by Proposition~\ref{prop:fbpsoln}(vi), and $L_\delta\to L_0=0$ as $\delta\to 0$, and using the Gaussian tail bound on $U$ from Lemma~\ref{lem:GaussiantailU} combined with the fact that $\sup_{t\le 1}|L_t|<\infty$ to give a dominating function.
 	For the right-hand side of~\eqref{eq:magic formula delta}, we have
 	\[
 	\int_0^{\infty}e^{r(L_{t+\delta}-L_{\delta})-(1+\frac{1}{2}r^2)t}dt
 	=e^{-rL_\delta +(1+\frac{1}{2}r^2)\delta}\int_{\delta}^{\infty}e^{rL_s-(1+\frac{1}{2}r^2)s}ds
 	\to \int_{0}^{\infty}e^{rL_s-(1+\frac{1}{2}r^2)s}ds
 	\]
 	as $\delta\to 0$, since $L_0=0$ and so $\sup_{s\le \delta}|L_s|\to 0$ as $\delta\to 0$.
Therefore, by~\eqref{eq:magic formula delta}, for any $U_0$ satisfying Assumption~\ref{assum:initial assumption on initial condition}, \eqref{eq:magic formula} holds for $r\in (-\infty,\sqrt{2})\setminus\{0\}$.

\medskip

\noindent \textbf{Step 2:} Proof of~\eqref{eq:magic formula} for $U_0$ satisfying Assumption~\ref{assum:smooth assumption on initial condition}.

We now take an arbitrary $U_0$ satisfying Assumption \ref{assum:smooth assumption on initial condition} such that $L_0:=\inf\{x\in \R:U_0(x)<1\}=0$. We recall from the start of Section~\ref{subsec:magicproof} that this means that $U_0$ satisfies Assumption~\ref{assum:standing assumption ic}, $U_0\in C^2((0,\infty))\cap C(\Rm)$, and $U_0'(x)<0$ $\forall x>0$. 

We claim that there exists a non-decreasing sequence of initial conditions $(U_0^{(n)})_{n=1}^\infty$ converging pointwise to $U_0$ such that each $U_0^{(n)}$ satisfies Assumption~\ref{assum:initial assumption on initial condition}.
Indeed, we can construct such a sequence as follows.
Since $U_0$ is strictly decreasing on $(0,\infty)$ with $U_0(0)=1$, we can take $f:(0,1)\ra (0,\infty)$ such that $U_0'(x)=-f(U_0(x))$ $\forall x>0$. 
Note that $f$ is locally Lipschitz continuous by the following argument:
For $y\in (0,1)$, take $K<\infty$ such that $U_0(1/K)>y>U_0(K)$, and let $c:=\inf_{x\in [1/K,K]}(-U'_0(x))>0$ and $C=\sup_{x\in [1/K,K]}|U_0''(x)|<\infty$.
Take $\delta>0$ such that $(y-\delta,y+\delta)\subseteq [U_0(K),U_0(1/K)]$. Then for $y_1,y_2\in (y-\delta,y+\delta)$, there exist $x_1,x_2\in [1/K,K]$ such that $U_0(x_1)=y_1$ and $U_0(x_2)=y_2$.
We now have $|y_1-y_2|=|U_0(x_1)-U_0(x_2)|\ge c|x_1-x_2|$ and 
\[|f(y_1)-f(y_2)|=|U'_0(x_1)-U'_0(x_2)|\le C|x_1-x_2|\le C c^{-1}|y_1-y_2|,\]
which implies that $f$ is locally Lipschitz continuous, as claimed.

For $n\in \Nm$, we then define
\[
f_n(y):=f(y)\vee \frac 1 n \quad \text{for }y\in (0,1).
\]
Since $f$ is locally Lipschitz continuous, so too is $f_n$.
For $u\in [0,1]$, let
\[
F_n(u):=\int_u^1 \frac{ds}{f_n(s)}<\infty.
\]
Then $F_n$ is $C^1$ on $(0,1)$, with $F_n'(u)=-1/f_n(u)<0$ for $u\in (0,1)$.
Therefore, by the inverse function theorem, $F_n^{-1}:(0,F_n(0))\to (0,1)$ is $C^1$ with $(F_n^{-1})'(y)=-f_n(F_n^{-1}(y))$.
Define $U_0^{(n)}:\R\to [0,1]$ by letting
\[
U^{(n)}_0(x)=
\begin{cases}
1 \quad &\text{for }x\le 0\\
F_n^{-1}(x) \quad &\text{for }x\in (0,F_n(0))\\
0 \quad &\text{for }x>F_n(0).
\end{cases}
\]
Then $U_0^{(n)}$ is continuous and non-increasing, and has a compact interface with $L_0^{(n)}:=\inf\{x\in \R:U_0^{(n)}(x)<1\}=0$.
Moreover, by~\eqref{eq:Picdefn} and~\eqref{eq:speed c travelling wave}, for $c\ge \sqrt 2$ sufficiently large,
$\partial_x \Pi_c(x)\ge -1/n$ $\forall x\in \R$, and so since we also have
$(U_0^{(n)})'(y)\le -1/n$ $\forall y\in (0,F_n(0))$ and $U_0^{(n)}(y)=0$ $\forall y\ge F_n(0)$, it follows from the definition of $\le_s$ in~\eqref{eq:defin 1 of being more stretched} that $U_0^{(n)}\le_s \Pi_c$.
Therefore $U_0^{(n)}$ satisfies Assumption~\ref{assum:initial assumption on initial condition}, and so by Step~1,
letting $(U^{(n)}(t,x),L^{(n)}_t)$ denote the solution of~\eqref{eq:FBP_CDF} with initial condition $U_0^{(n)}$, for $r\in (-\infty,\sqrt{2})\setminus \{0\}$,
	\begin{equation}\label{eq:magic formulaU0(n)}
		\int_{0}^{\infty}U_0^{(n)}(x)e^{rx}dx=-\frac{1}{r}+\frac{1}{r}\int_0^{\infty}e^{rL^{(n)}_t-(1+\frac{1}{2}r^2)t}dt.
	\end{equation}
	
To see that $U_0^{(n)}\uparrow U_0$ pointwise as $n\to \infty$,
note that for $0<x<y$,
\[
\int_{U_0(y)}^{U_0(x)}\frac{ds}{f(s)}=-\int_x^y \frac{U'_0(z)}{f(U_0(z))}dz=y-x,
\]
and so letting $x\to 0$, for $y>0$,
\[
\int_{U_0(y)}^{1}\frac{ds}{f(s)}=y.
\]
Now fix $x>0$ and take $\epsilon \in (0,U_0(x))$.
Then by monotone convergence, for $n$ sufficiently large,
\[
F_n(U_0(x)-\epsilon)\ge \int_{U_0(x)}^1 \frac{ds}{f(s)}=x,
\]
and $F_n(U_0(x))\le x$, and so $U_0(x)-\epsilon\le U^{(n)}_0(x)\le U_0(x)$.
Also, for $n\in \Nm$ we have $F_n(0)\le F_{n+1}(0)$, and for $x\in (0,F_n(0))$ we have
$x=F_n(U^{(n)}_0(x))\le F_{n+1}(U_0^{(n)}(x))$,
and so $U^{(n+1)}_0(x)\ge U^{(n)}_0(x)$.
Hence $U_0^{(n)}\uparrow U_0$ pointwise as $n\to \infty$, as claimed.

By Dini's theorem and since $U_0(x)\to 0$ as $x\to \infty$, it follows that $U_0^{(n)}\to U_0$ uniformly on $\R$ as $n\to \infty$.
Let $(U(t,x),L_t)$ solve~\eqref{eq:FBP_CDF} with initial condition $U_0$.
Then by Lemma~\ref{lem:convofLn}, for each $t>0$, $L^{(n)}_t \uparrow L_t$ as $n\to \infty$.
Moreover, recalling the definition of $L^H_t$ in Section~\ref{subsec:notation},
we have
 $L^{(n)}_t\ge L^H_t$ $\forall t\ge 0,$ $n\in \Nm$ by Lemma~\ref{lem:less stretching means slower boundary} and $L^H_t\ge L^H_{t_0}$ $\forall t\ge t_0$ by Lemma~\ref{lem:boundary locally Lipschitz from the left}.
 Therefore,
using monotone convergence on the left-hand side of~\eqref{eq:magic formulaU0(n)}, and on the right-hand side using monotone convergence in the case $r>0$ and dominated convergence in the case $r<0$, we have that~\eqref{eq:magic formula} holds with initial condition $U_0$ for any $r\in (-\infty, \sqrt 2) \setminus \{0\}$.

Then for general $U_0$ satisfying Assumption~\ref{assum:smooth assumption on initial condition} (without assuming $L_0=0$), since we have now shown that~\eqref{eq:magic formula} holds with initial condition $U_0(\cdot+L_0)$, it follows immediately that~\eqref{eq:magic formula} holds with initial condition $U_0$, which completes Step 2.

\medskip

\noindent \textbf{Step 3:} Proof of~\eqref{eq:magic formula} for $U_0$ satisfying Assumption~\ref{assum:standing assumption ic}.

We fix $r\in (-\infty,\sqrt{2})\setminus \{0\}$. We now take $U_0$ to be an arbitrary initial condition satisfying Assumption \ref{assum:standing assumption ic}, let $(U(t,x),L_t)$ denote the corresponding solution to \eqref{eq:FBP_CDF}, and let $L_0:=\inf\{x\in \R:U_0(x)<1\}$. Then $U(t,\cdot)$ satisfies Assumption \ref{assum:smooth assumption on initial condition} for any $t>0$ by Proposition \ref{prop:fbpsoln}. 
Therefore, by Step 2, for any $t>0$ we have
\begin{equation} \label{eq:magicformula Ut}
-\frac{1}{r}e^{rL_t}+\frac{1}{r}\int_0^{\infty}e^{rL_{t+s}-(1+\frac{1}{2}r^2)s}ds=\int_{L_t}^{\infty}U(t,x)e^{rx}dx.
\end{equation}
We will complete the proof by taking the limit of both sides as $t\ra 0$.

It is easy to see that as $t\to 0$,
\[
\int_0^{\infty}e^{rL_{t+s}-(1+\frac{1}{2}r^2)s}ds
=e^{(1+\frac 12 r^2)t}\int_t^{\infty}e^{rL_{s}-(1+\frac{1}{2}r^2)s}ds
\ra \int_0^{\infty}e^{rL_s-(1+\frac{1}{2}r^2)s}ds,
\]
where possibly the right-hand side or both sides are equal to $+\infty$. Moreover, by Proposition~\ref{prop:fbpsoln}, as $t\to 0$ we have $e^{rL_t}\to e^{rL_0}\in \Rm\cup\{+\infty\}$.

Therefore by Fatou's lemma and Proposition~\ref{prop:fbpsoln}, and then by~\eqref{eq:magicformula Ut}, we have
\[
\int_{L_0}^{\infty}U_0(x)e^{rx}dx\leq \liminf_{t\ra 0}\int_{L_t}^{\infty}U(t,x)e^{rx}dx=-\frac{1}{r}e^{rL_0}+\frac{1}{r}\int_0^{\infty}e^{rL_s-(1+\frac{1}{2}r^2)s}ds.
\]
Thus if $\int_{L_0}^{\infty}U_0(x)e^{rx}dx=+\infty$, then \eqref{eq:magic formula} holds, with both sides being $+\infty$. 

We now assume the converse, i.e.~that
\begin{equation}\label{eq:ic bded exp moment}
\int_{L_0}^{\infty}U_0(x)e^{rx}dx<\infty,
\end{equation}
which implies in particular that $rL_0<\infty$.
It remains only to establish that under assumption~\eqref{eq:ic bded exp moment}, as $t\to 0$,
\begin{equation}\label{eq:convergence exp func of ics}
\int_{L_t}^{\infty}U(t,x)e^{rx}dx\ra \int_{L_0}^{\infty}U_0(x)e^{rx}dx.
\end{equation}
Our strategy is to apply the dominated convergence theorem, so we seek to construct a dominating function.

By the Feynman-Kac formula in Lemma~\ref{lem:FKforinfinitemass}, we see that
\[
U(t,x)\leq e^t\Esub{x}{U_0(B_t)} \quad \forall t>0,\, x\in \R.
\]
Therefore, since $U_0$ is non-increasing, for  $t\in [0,1]$ and $x\in \R$,
\[
U(t,x)\leq e\left( \Esub{x}{U_0(B_t)\1_{\{B_t<x\}}}+U_0(x)\right).
\]
Then since $\psub{x}{B_1<u}\geq \psub{x}{B_1<u}$ for all $t\in [0,1]$, $x\in \R$ and $u<x$, it follows that
\[
\Esub{x}{U_0(B_t)\1_{\{B_t<x\}}}\le \Esub{x}{U_0(B_1)\1_{\{B_1<x\}}}\le \Esub{x}{U_0(B_1)}.
\]
We conclude that for  $t\in [0,1]$ and $x\in \R$,
\[
U(t,x)\leq e\left(  \Esub{0}{U_0(x+B_1)}+U_0(x)\right).
\]
We now observe that 
by Fubini's theorem,
\begin{align*}
\int_{L_0-1}^{\infty}e^{rx}\Esub{0}{U_0(x+B_1)}dx
%=\Esub{0}{\int_{L_0-1}^\infty e^{rx}U_0(x+B_1)dx}
&=\Esub{0}{e^{-rB_1}\int_{L_0-1+B_1}^\infty e^{ry}U_0(y)dy}\\
&\le \Esub{0}{e^{-rB_1}\Big((|B_1|+1)(e^{rL_0}+e^{r(L_0-1+B_1)})+\int_{L_0}^\infty e^{ry}U_0(y)dy\Big)}.
\end{align*}
Hence if~\eqref{eq:ic bded exp moment} holds, then
\[
\int_{L_0-1}^{\infty}e^{rx}\Big(e\left(  \Esub{0}{U_0(x+B_1)}+U_0(x)\right)\Big)dx<\infty.
\]
We can therefore apply the dominated convergence theorem to establish~\eqref{eq:convergence exp func of ics}. 

This concludes the proof of Theorem \ref{theo:magic formula}.
\end{proof}

\subsection{Proof of Theorem \ref{theo:initial condition limsup bdy relation}}

Recall the definition of $r_0(U_0)$ in~\eqref{eq:r0U0defn}. We now use Theorem~\ref{theo:magic formula} to prove Theorem~\ref{theo:initial condition limsup bdy relation}.
\begin{proof}[Proof of Theorem~\ref{theo:initial condition limsup bdy relation}]
Note that by Lemma~\ref{lem:less stretching means slower boundary} we have $L_t-L_1\ge L^H_{t-1}$ for $t\ge 1$, and so by Lemma~\ref{lem:LHtlower} we have $\limsup_{t\ra\infty}\frac{1}{t}L_t\ge \sqrt 2$.
We define
\begin{equation} \label{eq:Lbardefn}
\bar L:=\limsup_{t\ra\infty}\frac{1}{t}L_t\in [\sqrt{2},\infty ]
\end{equation}
and let $r_0=r_0(U_0)$.

Now suppose that $r_0>0$, and take $r\in (0,r_0)$.
Then by~\eqref{eq:r0U0defn} we have $\int_0^{\infty}e^{rx}U_0(x)dx<\infty$ and $r<\sqrt 2$, and so by Theorem~\ref{theo:magic formula} and since $\int_{L_0\wedge 0}^0e^{rx}U_0(x)dx<\infty$ and $e^{rL_0}<\infty$, we have
\begin{equation} \label{eq:finiteLtintegral}
\int_{0}^\infty e^{rL_t-(1+\frac{1}{2}r^2)t}dt<\infty.
\end{equation}
Recall from Lemma~\ref{lem:boundary locally Lipschitz from the left} that there exists $t_0>0$ such that, in particular,
$\inf_{s\in [0,1]}L_{t+s}\ge L_t$ for any $t\ge t_0$.
Therefore we must have
\[
r\bar L-(1+\tfrac{1}{2}r^2)\leq 0,
\]
because otherwise the integral on the left-hand side of~\eqref{eq:finiteLtintegral} would be infinite. 
It follows that
$
\bar L\leq \frac{1}{r}+\frac{r}{2}.
$
Since $r\in (0,r_0)$ was arbitrary, it follows that if $r_0>0$ then
\begin{equation} \label{eq:Lbarupperbd}
\bar L\leq \frac{1}{r_0}+\frac{r_0}{2}.
\end{equation}

Conversely, suppose that $r_0\in [0,\sqrt{2})$ and take $r\in (r_0,\sqrt 2)$. Then by~\eqref{eq:r0U0defn},
\[
\int_0^{\infty}e^{rx}U_0(x)dx=\infty.
\]
By Theorem~\ref{theo:magic formula} and since $L_0<\infty$,
it follows that
\begin{equation} \label{eq:infiniteLtintegral}
\int_{0}^{\infty}e^{rL_t-(1+\frac{1}{2}r^2)t}dt=\infty.
\end{equation}
Therefore
\[
r\bar L-(1+\tfrac{1}{2}r^2)\geq 0,
\]
since otherwise the integral on the left-hand side of~\eqref{eq:infiniteLtintegral} would be finite. Therefore
$
\bar L\geq \frac{1}{r}+\frac{r}{2}.
$
Now since $r\in (r_0,\sqrt 2)$ was arbitrary, we see that if $r_0=0$ then $\bar L=\infty$ and if $r_0\in (0,\sqrt{2})$ then
\begin{equation} \label{eq:Lbarlowerbd}
\bar L\geq \frac{1}{r_0}+\frac{r_0}{2}.
\end{equation}
By combining~\eqref{eq:Lbarupperbd} and~\eqref{eq:Lbarlowerbd}, and using~\eqref{eq:Lbardefn} in the case $r_0=\sqrt 2$, the result follows.
\end{proof}

\subsection{Proof of Theorem~\ref{theo:samebdysameU0}}
The following result will be applied together with Theorem~\ref{theo:magic formula} to prove Theorem~\ref{theo:samebdysameU0}.
\begin{lem}\label{lem:uniquness given same transform on open set}
	Suppose $U_0^{(1)}$ and $U_0^{(2)}$ both satisfy Assumption~\ref{assum:standing assumption ic} and are continuous, 
and let $L_0^{(i)}:=\inf\{x\in \R:U_0^{(i)}(x)<1\}$
for $i\in \{1,2\}$. Suppose $L_0^{(1)}=L_0^{(2)}>-\infty$,	
	and suppose that there is some open $W\subseteq (-\infty,0)$ such that 
	\[
	\int_{L_0^{(1)}}^\infty e^{rx}U_0^{(1)}(x)dx=\int_{L_0^{(2)}}^\infty e^{rx}U_0^{(2)}(x)dx \quad \forall r\in W.
	\]
Then $U_0^{(1)}=U_0^{(2)}$.
\end{lem}
\begin{proof}
	Let $\mu_0^{(1)}$ and $\mu_0^{(2)}$ denote the Borel probability measures such that $U_0^{(1)}(x)=\mu_0^{(1)}((x,\infty))$ and $U_0^{(2)}(x)=\mu_0^{(2)}((x,\infty))$ for all $x\in \Rm$. Then by integration by parts, using that $U_0^{(i)}(L_0^{(i)})=1$ and $L_0^{(i)}>-\infty$ for $i\in \{1,2\}$ and that $W\subseteq (-\infty,0)$, we see that
	\begin{equation} \label{eq:intbypartsW}
		-\frac 1 r e^{rL_0^{(1)}}+\frac 1r \int_{L_0^{(1)}}^\infty e^{rx}\mu_0^{(1)}(dx)=-\frac 1 r e^{rL_0^{(2)}}+\frac 1r \int_{L_0^{(2)}}^\infty e^{rx}\mu_0^{(2)}(dx) \quad \forall r\in W.
	\end{equation}
	We write $\Hm:=\{z\in \mathbb{C}:\text{Re}(z)<0\}$. For $z\in \Hm$, we let
	\[
	f_1(z)= \int_{L_0^{(1)}}^\infty e^{z x}\mu_0^{(1)}(dx)\quad\text{and}\quad f_2(z)= \int_{L_0^{(2)}}^\infty e^{z x}\mu_0^{(2)}(dx).
	\]
Then since $\mu_0^{(1)}$ and $\mu_0^{(2)}$ are probability measures and $L_0^{(i)}>-\infty$ for $i\in \{1,2\}$, $f_1$ and $f_2$ are both well defined and complex differentiable on $\Hm$, with continuous extensions to $\partial \Hm$. 
Then by~\eqref{eq:intbypartsW} and since $L_0^{(1)}=L_0^{(2)}$,
we have that $f_1$ and $f_2$ are holomorphic functions with $f_1=f_2$ on $W$, and therefore $f_1=f_2$ on $\bar \Hm$.
Therefore $\mu_0^{(1)}$ and $\mu_0^{(2)}$ have the same Fourier transform, and so $U_0^{(1)}=U_0^{(2)}$.
\end{proof}
We can now prove Theorem~\ref{theo:samebdysameU0}.
\begin{proof}[Proof of Theorem~\ref{theo:samebdysameU0}]
Suppose $L^{(1)}_t=L^{(2)}_t$ $\forall t>0$.
Take $\delta>0$. By Proposition~\ref{prop:fbpsoln}, for each $i\in \{1,2\}$, we have that $U^{(i)}(\delta,\cdot)$ is continuous and satisfies Assumption~\ref{assum:standing assumption ic}, and $L^{(i)}_\delta>-\infty$.
By Theorem~\ref{theo:magic formula}, for any $r\in (-\infty,0)$,
\[
\int_{L^{(1)}_\delta}^{\infty}U^{(1)}(\delta, x)e^{rx}dx=-\frac{1}{r}e^{rL^{(1)}_\delta}+\frac{1}{r}\int_0^{\infty}e^{rL^{(1)}_{t+\delta}-(1+\frac{1}{2}r^2)t}dt
=\int_{L^{(2)}_\delta}^{\infty}U^{(2)}(\delta, x)e^{rx}dx,
\]
where the second equality follows since $L^{(1)}_s=L^{(2)}_s$ $\forall s>0$, and by applying Theorem~\ref{theo:magic formula} again.
Therefore, since $L^{(1)}_\delta=L^{(2)}_\delta$, by Lemma~\ref{lem:uniquness given same transform on open set} we have
$U^{(1)}(\delta,\cdot)=U^{(2)}(\delta,\cdot)$.
For each $i\in \{1,2\}$, we have $U^{(i)}(\delta,\cdot)\to U^{(i)}_0$ in $L^1_{\mathrm{loc}}$ as $\delta\to 0$, and $U_0^{(i)}$ is c\`adl\`ag non-increasing, and so $U_0^{(1)}=U_0^{(2)}$.
\end{proof}

\section{Finite initial mass case} \label{sec:finitemass}

In this section, we will often assume that $U_0:\R\to [0,1]$ satisfies
\begin{equation} \label{eq:finiteinitialmassinsection}
\int_0^{\infty} y e^{\sqrt 2 y} U_0(y)  d y<\infty,
\end{equation}
which we refer to as the finite initial mass condition (as mentioned after Theorem~\ref{theo:Ltposition}).
We will show that for $U_0$ satisfying Assumption~\ref{assum:standing assumption ic} and~\eqref{eq:finiteinitialmassinsection}, letting $(U(t,x),L_t)$ denote the solution of the free boundary problem~\eqref{eq:FBP_CDF} with initial condition $U_0$,
\[
U(t,x+L_t)\ra \Pi_{\min}(x)\text{ uniformly in }x \text{ as }t\ra\infty,
\]
and there exists $c=c(U_0)\in \R$ such that 
	\[
	L_t = \sqrt 2 t -\frac3{2\sqrt 2 } \log t +c +o(1) \quad \text{as }t\to \infty.
	\]

\subsection{Convergence to $K\pi_{\min}$ for solutions of a PDE with given boundary} \label{subsec:linearproblem}
 
We begin by proving the following result, which is an extension of \cite[Theorem 1, Case (d)]{BBHR17}. See Section \ref{subsec:outline} for a discussion of the difference between the two results. Recall from before Lemma~\ref{lem:FKforwardstime} in Section~\ref{subsec:FK} that for a Borel probability measure $\mu$ on $\R$, we let $\mathbb P_\mu$ denote the probability measure under which $(B_t)_{t\ge 0}$ is a Brownian motion with $B_0\sim \mu$.
This result will be combined with the Feynman-Kac representation in Lemma~\ref{lem:FKforwardstime} for solutions of~\eqref{eq:FBP_CDF} to determine the asymptotics for $L_t$ in Sections~\ref{subsec:finiteinitmassconv} and~\ref{subsec:bootstrap} below.
\begin{theo}\label{theo:extension of BBHR}
	Take $a\in \R$ and $t_0\in [0,\infty)$, and for $t\ge 0$, let
	\begin{equation} \label{eq:mtdefn}
	m(t):=\sqrt 2(t+t_0)-\frac{3}{2\sqrt 2}\log(t+t_0+1)+a.
	\end{equation}
	Suppose $\mu$ is a Borel probability measure on $\R$, supported on $[m(0),\infty)$, and that
	\begin{equation} \label{eq:finiteinitmassforlinear}
		\int_{[0,\infty)}xe^{\sqrt 2 x}\mu(dx)<\infty.
	\end{equation}
Then
there exists $K\in (0,\infty)$ such that
	\begin{equation} \label{eq:pmuBt}
	e^t \psub{\mu}{B_t>m(t)+y, \, B_s>m(s)\; \forall s\in (0,t)}\to K \Pi_{\min}(y) 
	\end{equation}
	as $t\to \infty$, uniformly in $y\ge 0$.
\end{theo}

\begin{rmk}
	For $(m(t),t\ge 0)$ and $\mu$ as in the statement of Theorem~\ref{theo:extension of BBHR}, let $u(t,x)$ solve
	\begin{equation} \label{eq:linearPDEwbdy}
		\begin{cases}
		\partial_tu=\frac{1}{2}\Delta u+u,\quad &t>0, \, x>m(t),\\
		u(t,x)=0, \quad &t>0, \, x\le m(t),\\
		u(t,x )dx \to \mu(dx) &\text{weakly as } t\to 0.
	\end{cases}
\end{equation}
Then for $t>0$ and $y\in \R$,
\[
u(t,y)dy = e^t \psub{\mu}{B_t \in dy,\, B_s>m(s)\; \forall s\in (0,t)}
\]
(see Proposition~\ref{prop:BBHRresult} below), and so Theorem~\ref{theo:extension of BBHR} is equivalent to the statement that uniformly in $x\ge 0$,
\[
\int_{x}^\infty u(t,m(t)+y)dy \to K\Pi_{\min}(x) \quad \text{as }t\to \infty.
\]
However, the interpretation in~\eqref{eq:pmuBt} is the one that will we use later in this section.
\end{rmk}

We will use the following result from~\cite{BBHR17} about the fundamental solution of~\eqref{eq:linearPDEwbdy} in the proof of Theorem~\ref{theo:extension of BBHR}.
\begin{prop} \label{prop:BBHRresult}
Suppose $(m(s),s\ge 0)$ is twice continuously differentiable, and satisfies $m(0)=0$ and $m''(s)=\mathcal O(s^{-2})$ as $s\to \infty$.
For $x\in \R$ fixed, let $q(t,x,y)$ solve
\begin{equation*}
		\begin{cases}
		\partial_t q=\frac{1}{2}\partial_y^2 q+q,\quad &t>0, \, y>m(t),\\
		q(t,x,y)=0, \quad &t>0, \, y\le m(t),\\
		q(0,x,y)=\delta(y-x),
	\end{cases}
\end{equation*}
where $\delta$ is the Dirac distribution, i.e. for $t>0$ and $x,y\in \R$,
\[
q(t,x,y)dy=e^t \psub{x}{B_t\in dy, B_s>m(s) \; \forall s\in (0,t)}.
\]
Then there exist $0<K_1<K_2<\infty$ and $T<\infty$ such that for $t>0$ and $x,y\ge 0$,
\[
		q(t,x,m(t)+y)=\frac{\sqrt{2}\sinh(\frac{xy}{t})}{\sqrt{\pi t}}e^{\frac{m(t)}{t}(x-y)-\frac{x^2+y^2}{2t}+t-\frac{1}{2}\int_0^t(m'(s))^2ds}\psi_t(x,y),
	\]
	where $\psi_t(x,y)$ satisfies
	\begin{enumerate}[(i)]
	\item $K_1\le \psi_t(x,y)\le K_2$ $\forall t>0$, $x,y\ge 0$;
	\item $|\psi_t(x,y)-\psi_t(x,0)|\le K_2 y\frac{\log t}{t}\psi_t(x,0)$ $\forall t\ge T$, $x,y\ge 0$;
	\item $\Psi(x):=\lim_{t\to \infty}\psi_t(x,0)$ exists for $x\ge 0$, and satisfies $0<\Psi(x)<\infty$.
	\end{enumerate}
\end{prop}
\begin{proof}
The result follows directly from~\cite[Eq.~(28)]{BBHR17} and~\cite[Proposition 8]{BBHR17}, noting that we exchanged the roles of $x$ and $y$ in our notation compared to the notation in~\cite{BBHR17}, and in the definition of $q$ in~\cite[Eq.~(19)]{BBHR17} there is no $\frac 12$ factor in the $\partial_x^2 q$ term, so $q(t,x,y)$ in our notation corresponds to $\sqrt{2}q(t,\sqrt{2}y,\sqrt{2}y)$ in the notation of~\cite{BBHR17},
and $m(s)$ in our notation corresponds to $\sqrt{2}m(s)$ in the notation of~\cite{BBHR17}.
\end{proof}

\begin{proof}[Proof of Theorem~\ref{theo:extension of BBHR}]
Since $\mu$ is a finite measure,~\eqref{eq:finiteinitmassforlinear} implies that ~$\int_0^{\infty}xe^{\sqrt 2 x}\mu_{y}(dx)<\infty$, where $\mu_{y}$ is given by $\mu_{y}(E)=\mu(E+y)$ for $E\in \mathscr B(\R)$, for any $y\in \R$. Therefore
we may assume without loss of generality that $a=\frac{3}{2\sqrt 2}\log(t_0+1)-\sqrt 2 t_0$ and so $m(0)=0$.

Then by~\eqref{eq:mtdefn}, for $t>0$,
	\begin{align*}
	m'(t)&=\sqrt 2-\frac{3}{2\sqrt{2}(t+t_0+1)},\quad m''(t)=\frac{3}{2\sqrt{2}(t+t_0+1)^2}, \\
	\text{and }\quad \frac{m(t)}{t}&=\sqrt{2}-\frac{3}{2\sqrt{2}t}\log(t+t_0+1)+\frac{\sqrt 2 t_0+a}{t}.
	\end{align*}
	Moreover, for $t>0$,
	\begin{align*}
	\int_0^t(m'(s))^2ds &=\int_0^t \left(2-\frac{3}{1+s+t_0}+\frac{9}{8(1+s+t_0)^2}\right)ds\\
	&=2t-3[\log(1+t+t_0)-\log (1+t_0)]+\frac{9}{8}[(1+t_0)^{-1}-(1+t)^{-1}].
	\end{align*}
	
	In the remainder of the proof, we let $A=A(t)$ denote a real-valued function of $t>0$ that converges as $t\ra\infty$ to some positive real number, and which may change from line to line. We write $\bar A=\lim_{t\to \infty}A(t)\in (0,\infty)$ for its limit, which may also change from line to line.
	Moreover, $C=C(t,x,y)$ will denote a real-valued function of $t>0$, $x\ge 0$ and $y\geq 0$ that is bounded uniformly over all $t\geq T_0$, $x\ge 0$ and $y\ge 0$ for some $T_0\in (0,\infty)$, and $T_0$ and $C$ may increase from line to line. 
	
	Then by Proposition~\ref{prop:BBHRresult}, for $t>0$ and $x,y\ge 0$,
	\begin{align} \label{eq:qtformulaforthm}
	q(t,x,m(t)+y)&= A\sinh\left(\frac{xy}{t}\right)t^{-\frac{1}{2}}
	e^{(\sqrt 2-\frac{3}{2\sqrt 2 t}\log(t+t_0+1)+\frac{\sqrt 2 t_0+a}{t})(x-y)-\frac{x^2+y^2}{2t}+t-t+\frac{3}{2}\log(1+t+t_0)}\psi_t(x,y) \notag \\
		&=A t\sinh\left(\frac{xy}{t}\right)
	e^{(\sqrt 2-\frac{3}{2\sqrt 2 t}\log(t+t_0+1)+\frac{\sqrt 2 t_0+a}{t})(x-y)-\frac{x^2+y^2}{2t}}\psi_t(x,y).
	\end{align}
	We will now decompose $q(t,x,m(t)+y)$ by writing
	\begin{equation} \label{eq:qtdecompose}
	q(t,x,m(t)+y)=q(t,x,m(t)+y)\1_{\{xy\leq 2t\}}+q(t,x,m(t)+y)\1_{\{xy> 2t\}},
	\end{equation}
	and we will control the two terms on the right-hand side separately.
	
First, note that by property~(i) in Proposition~\ref{prop:BBHRresult} and since $\sup_{z\in (0,2]}(z^{-1}\sinh(z))<\infty$, we have
that for $t>0$ sufficiently large that $\sqrt 2-\frac{3}{2\sqrt 2 t}\log(t+t_0+1)+\frac{\sqrt 2 t_0+a}{t}\in [1,\sqrt 2]$,
for $x,y\ge 0$,
\begin{equation} \label{eq:qdomconvbd}
q(t,x,m(t)+y)\1_{\{xy\leq 2t\}}
%\le Cxy e^{\sqrt{2}x}e^{-(\sqrt 2-\frac{3}{2\sqrt 2 t}\log(t+t_0+1)+\frac{\sqrt 2 t_0+a}{t})y}
	\le Cxy
	e^{\sqrt{2}x}e^{-y}.
\end{equation}
Moreover, since $\sinh(z)\sim z$ as $z\to 0$, and by properties~(ii) and~(iii) in Proposition~\ref{prop:BBHRresult},
for $x,y\ge 0$ fixed, as $t\to \infty$,
\begin{equation} \label{eq:qpointwiselimit}
q(t,x,m(t)+y)\1_{\{xy\leq 2t\}}
\to \bar A xy e^{\sqrt{2}(x-y)}\Psi(x),
\end{equation}
where we recall from property~(iii) in Proposition~\ref{prop:BBHRresult} that $\Psi(x)\in (0,\infty)$.

Recall our assumption~\eqref{eq:finiteinitmassforlinear} on $\mu$.
Then by~\eqref{eq:qdomconvbd} and~\eqref{eq:qpointwiselimit} and the dominated convergence theorem,
using that
$ye^{-y}$ is integrable with respect to $dy$ over $\{y\ge 0\}$ and $xe^{\sqrt 2 x}$ is integrable with respect to $\mu(dx)$ over $\{x\ge 0\}$,
	as $t\to \infty$, 
	\begin{equation}\label{eq:L1 limit of first piece Roberts extension}
		\int_0^\infty \left|\int_{[0,\infty)} q(t,x,m(t)+y)\1_{\{xy\leq 2t\}} \mu(dx) -Kye^{-\sqrt 2y}\right| dy \to 0,
	\end{equation}
	where
	\[
	K=\bar A\int_{[0,\infty)} xe^{\sqrt 2x}\Psi(x)\mu(dx)\in (0,\infty).
	\]
	
	We now decompose the second term on the right-hand side of~\eqref{eq:qtdecompose} by writing
	\begin{equation} \label{eq:qtdecomposeagain}
	q(t,x,m(t)+y)\1_{\{xy> 2t\}}=q(t,x,m(t)+y)\1_{\{xy> 2t, \, y\leq t^{1/3}\}}
		+q(t,x,m(t)+y)\1_{\{xy> 2t, \, y> t^{1/3}\}}.
	\end{equation}
	Note that for $t>0$ and $x,y\ge 0$, 
	\begin{equation} \label{eq:sinhbd}
	\sinh\left(\frac{xy}{t}\right)
	e^{-\frac{x^2+y^2}{2t}}\le \tfrac 12 e^{\frac{xy}{t}}e^{-\frac{x^2+y^2}{2t}}=\tfrac 12 e^{-\frac{(x-y)^2}{2t}}.
	\end{equation}
	Therefore by~\eqref{eq:qtformulaforthm} and property~(i) in Proposition~\ref{prop:BBHRresult}, for $t>0$ sufficiently large and $x,y\ge 0$, 
	\[
	q(t,x,m(t)+y)\1_{\{xy> 2t, \, y\leq t^{1/3}\}}\leq Cte^{-\frac{(x-y)^2}{2t}} e^{\sqrt 2x}e^{-y}\1_{\{xy> 2t, \, y\leq t^{1/3}\}}.
	\]
	For $t\ge 1$ and $x,y\ge 0$ with $xy>2t$ and $y\leq t^{1/3}$, we have $x>2t^{2/3}$ and so $(x-y)^2>t^{4/3}$. Thus, for $t>0$ sufficiently large and $x,y\ge 0$,
	\[
	q(t,x,m(t)+y)\1_{\{xy> 2t, \, y\leq t^{1/3}\}}\leq Cte^{-\frac{1}{2}t^{1/3}}e^{\sqrt 2x}e^{-y}.
	\]
	By~\eqref{eq:finiteinitmassforlinear} and since $\mu$ is a finite measure, the right-hand side is integrable with respect to $\mu(dx)dy$ over $\{x\ge 0,\, y \ge 0\}$ with integral converging to $0$ as $t\ra\infty$, and so as $t\ra\infty$,
	\begin{equation}\label{eq:convergence of second piece Roberts extension}
		\int_0^\infty \int_{[0,\infty)} q(t,x,m(t)+y)\1_{\{xy> 2t, \, y\leq t^{1/3}\}}\mu(dx) dy\ra 0.
	\end{equation}

	Finally, for the second term on the right-hand side of~\eqref{eq:qtdecomposeagain}, 
	by~\eqref{eq:qtformulaforthm},~\eqref{eq:sinhbd} and property~(i) in Proposition~\ref{prop:BBHRresult},
	for $t>0$ sufficiently large and $x,y\ge 0$,
	\[
		q(t,x,m(t)+y)\1_{\{xy> 2t, \, y> t^{1/3}\}}\leq Cte^{\sqrt{2}x}e^{-y}\1_{\{xy> 2t, \, y> t^{1/3}\}}
		\leq Cte^{\sqrt{2}x}e^{-\frac{1}{2}t^{1/3}}e^{-\frac{y}{2}}.
	\]
Thus, using~\eqref{eq:finiteinitmassforlinear} and that $\mu$ is a finite measure, as $t\ra\infty$,
	\begin{equation}\label{eq:convergence of third piece Roberts extension}
		\int_0^\infty \int_{[0,\infty)} q(t,x,m(t)+y)\1_{\{xy> 2t, \, y> t^{1/3}\}} \mu(dx) dy\ra 0.
	\end{equation}
	
	By the definition of $q(t,x,y)$ in Proposition~\ref{prop:BBHRresult}, we have that for $t>0$ and $z\ge 0$,
	\[
	e^t \psub{\mu}{B_t>m(t)+z, \, B_s>m(s)\; \forall s\in (0,t)}
=\int_{z}^\infty \int_{[0,\infty)} q(t,x,m(t)+y)\mu(dx) dy.
	\]
	Therefore, recalling the definition of $\Pi_{\min}$ in~\eqref{eq:Pimindefn} and~\eqref{eq:minimal travelling wave}, the result follows directly from \eqref{eq:L1 limit of first piece Roberts extension}, \eqref{eq:convergence of second piece Roberts extension} and \eqref{eq:convergence of third piece Roberts extension}.
\end{proof}

\subsection{Convergence to $\Pi_{\min}$ for solutions of free boundary problem}\label{subsec:finiteinitmassconv}
In this subsection, under the finite initial mass assumption~\eqref{eq:finiteinitialmassinsection} on the initial condition $U_0$, we will prove that solutions of the free boundary problem~\eqref{eq:FBP_CDF} converge to the travelling wave $\Pi_{\min}$, and establish preliminary results on the asymptotics of the free boundary position, which we will then strengthen in Section~\ref{subsec:bootstrap} below.
Recall that $(U^H(t,x),L^H_t)$, defined in Section~\ref{subsec:notation}, is the solution of ~\eqref{eq:FBP_CDF} with Heaviside initial condition.
\begin{theo}\label{theo:convergence to the minimal travelling wave for finite initial mass}
Suppose $U_0$ satisfies Assumption~\ref{assum:standing assumption ic} and~\eqref{eq:finiteinitialmassinsection}.
	Let $(U(t,x),L_t)$ denote the solution of~\eqref{eq:FBP_CDF} with initial condition $U_0$.
	Then
	\begin{equation} \label{eq:theoconvtoPimin}
	U(t,L_t+x)\ra \Pi_{\text{min}}(x)
	\end{equation}
	uniformly in $x\in \R$ as $t\ra \infty$. Moreover, as $t\to \infty$,
	\begin{equation}\label{eq:O(1) asymptotics of free-bdy}
		L_t=\sqrt{2}t-\frac{3}{2\sqrt{2}}\log t+\calO(1).
	\end{equation}
	Furthermore, there exists $a\in\Rm$ such that
	\begin{equation} \label{eq:theoLHtLtconv}
		L_t - L^H_t\to a \quad \text{as }t\to \infty.
	\end{equation}
\end{theo}

In order to prove Theorem~\ref{theo:convergence to the minimal travelling wave for finite initial mass}, we will show that $(L_t-L^H_t)_{t\ge 0}$ is non-decreasing and bounded, and that $L_t$ satisfies the asymptotics~\eqref{eq:O(1) asymptotics of free-bdy}.
Then we will show that for any $s\ge 0$ fixed, $L_{t+s}-L_t\to \sqrt 2 s$ as $t\to \infty$, which we will combine with the Brunet-Derrida formula (Theorem~\ref{theo:magic formula}) to prove convergence to $\Pi_{\min}$.
We begin by proving the following simple consequence of Lemma~\ref{lem:less stretching means slower boundary}. 
	\begin{lem} \label{lem:LtLhtincr}
	Suppose $U_0$ satisfies Assumption~\ref{assum:standing assumption ic},
let $L_0=\inf\{x\in \R:U_0(x)<1\}$,	
	 and let $(U(t,x),L_t)$ denote the solution of~\eqref{eq:FBP_CDF}.
	Then $t\mapsto L_t-L^H_t$ is a non-decreasing function on $[0,\infty)$. 
	\end{lem}
	\begin{proof}
	Take $0<s\le t <\infty$.
	Then by Proposition~\ref{prop:fbpsoln} we have $L_s,L^H_s>-\infty$, and by Lemma~\ref{lem:extended maximum principle} we have $L_s\ge_s L^H_s$.
	By Lemma~\ref{lem:less stretching means slower boundary}, it follows that $L_t-L_s\ge L^H_t-L^H_s$.
	Hence $t\mapsto L_t-L^H_t$ is a non-decreasing function on $(0,\infty)$, and the result follows since $L_t-L^H_t\to L_0-L^H_0$ as $t\to 0$ by Proposition~\ref{prop:fbpsoln}.
	\end{proof}
	For $t\ge 0$, let
	\begin{equation} \label{eq:mtfinitemassdefn}
		m(t):=\sqrt{2}t-\frac{3}{2\sqrt{2}}\log (t+1).
	\end{equation}
Recall from Lemma~\ref{lem:LHtlower} that there exists $C_0<\infty$ such that 
\begin{equation} \label{eq:LHtlowerinconv}
L^H_t\ge m(t) -C_0 \quad \forall t\ge 0.
\end{equation}
We now use Theorem~\ref{theo:extension of BBHR} and the Feynman-Kac representation in Lemma~\ref{lem:FKforwardstime} to show that
under a suitable condition on the initial condition $U_0$, we have
$L_t\le m(t)+\mathcal O(1)$ as $t\to \infty$.	
	\begin{lem} \label{lem:L-mbound}
	Suppose $u_0$ is a Borel probability measure supported on $[0,\infty)$, and suppose $u_0$ satisfies
	\[
	\int_{[0,\infty)}xe^{\sqrt{2}x}u_0(dx)<\infty.
	\]
	Let $(U(t,x),L_t)$ denote the solution of~\eqref{eq:FBP_CDF} with initial condition $U_0$ given by $U_0(x)=u_0((x,\infty))$ $\forall x\in \R$.
	Then
		\[
		\sup_{t> 0} (L_t-m(t))< \infty.
		\]
	\end{lem}
	\begin{proof}
We assume, aiming for a contradiction, that there exists $(t_k)_{k=1}^\infty\subset (0,\infty)$ with $t_k\to \infty$ as $k\to \infty$ such that
		\begin{equation}\label{eq:assumption for contradiction Lt larger than tilde L t}
			L_{t_k}-m(t_k)\ra \infty \quad  \text{as }k\to \infty.
		\end{equation}
Let $L_0=\inf\{x\in \R:u_0((x,\infty))<1\}$; then $L_0\ge 0$ because $u_0$ is supported on $[0,\infty)$.
Therefore, by Lemma~\ref{lem:LtLhtincr} and since $L_0-L^H_0\ge 0$,
and then by~\eqref{eq:LHtlowerinconv},		
		\begin{equation}\label{eq:comparison 1 of all boundaries in proof}
			L_t\geq L^H_t\geq m(t)-C_0 \quad \forall t\ge 0.
		\end{equation}
		
		Now define the stopping times 
		\begin{equation} \label{eq:tauLtaumdefn}
		\tau_L:=\inf\{t>0:B_t\leq L_t\} \quad \text{and} \quad \tau_m:=\inf\{t>0:B_t\leq m(t)-C_0\}.
		\end{equation}		
Clearly, for $t>0$, if $\tau_L>t$ then $B_t>L_t$. 
Moreover, by~\eqref{eq:comparison 1 of all boundaries in proof},
we have $\tau_L\leq \tau_m$, and so $\tau_L>t$ implies that $\tau_m>t$.
Therefore, recalling the definition of $\mathbb P_{u_0}$ from before Lemma~\ref{lem:FKforwardstime}, for $t>0$,
\begin{equation} \label{eq:tauLgiventaum}
\psub{u_0}{\tau_L>t}=\psub{u_0}{\tau_m>t,\tau_L>t}
\leq \psub{u_0}{\tau_m>t,B_t>L_t}.
\end{equation}
By Theorem~\ref{theo:extension of BBHR}, there exists $K\in (0,\infty)$ such that as $t\to \infty$, uniformly in $y\ge 0$,
\[
e^t \psub{u_0}{B_t>m(t)-C_0+y,\, \tau_m>t}\to K \Pi_{\min}(y).
\]
Fix $y_0\in (0,\infty)$, and using~\eqref{eq:assumption for contradiction Lt larger than tilde L t}, take $k_0\in \Nm$ sufficiently large that
\[
L_{t_k}-m(t_k)\ge y_0-C_0 \; \forall k\ge k_0.
\]
Then for $k\ge k_0$, if $B_{t_k}>L_{t_k}$ then $B_{t_k}>m(t_k)+y_0-C_0$.
It follows that
\begin{equation*}
\limsup_{k\to \infty} e^{t_k}\psub{u_0}{B_{t_k}>L_{t_k},\, \tau_m>t_k}\le K\Pi_{\min}(y_0).
%\text{and} \quad \lim_{k\to \infty} e^{t_k}\psub{u_0}{\tau_m>t_k}&= K.
\end{equation*}
Therefore, by~\eqref{eq:tauLgiventaum} and since $y_0$ can be taken arbitrarily large,
		\[
		e^{t_k}\psub{u_0}{\tau_L>t_k}\ra 0 \quad \text{as }k\to \infty.
		\]
 But by the Feynman-Kac representation in Lemma~\ref{lem:FKforwardstime}, for any $t>0$,
		\[
		e^t \psub{u_0}{\tau_L>t}=1.
		\]
This gives us a contradiction, and completes the proof.	
	\end{proof}
	
	We are now ready to prove Theorem~\ref{theo:convergence to the minimal travelling wave for finite initial mass}.
\begin{proof}[Proof of Theorem~\ref{theo:convergence to the minimal travelling wave for finite initial mass}]
Define $U^+_0:\R\to [0,1]$ by letting
\[
U^+_0(x)=
\begin{cases}
1 \quad &\text{for }x<1,\\
U_0(x) \quad &\text{for }x\ge 1.
\end{cases}
\]
Then let $u_0^+$ denote the Borel probability measure on $\R$ given by $u^+_0((x,\infty))=U_0^+(x)$ $\forall x\in \R$.
Let $(U^+(t,x),L^+_t)$ denote the solution of~\eqref{eq:FBP_CDF} with initial condition $U^+_0$.
Since $U_0\le U^+_0$, by Proposition~\ref{prop:fbpcomparison} we have $L_t\le L^+_t$ $\forall t>0$.
Moreover, since $U_0$ satisfies~\eqref{eq:finiteinitialmassinsection}, we have
\[
\int_0^{\infty} y e^{\sqrt 2 y} U^+_0(y)  d y<\infty,
\]
and so using integration by parts,
\[
\int_{[0,\infty)}xe^{\sqrt{2}x}u^+_0(dx)=\int_0^\infty (1+\sqrt{2}x)e^{\sqrt 2 x}U^+_0(x)dx <\infty.
\]
Therefore, since $U^+_0$ is supported on $[0,\infty)$,
by Lemma~\ref{lem:L-mbound} there exists $C'<\infty$ such that
$L^+_t \le m(t)+C'$ $\forall t>0$.
By Lemma~\ref{lem:LtLhtincr} and then by~\eqref{eq:LHtlowerinconv}, we also have that for $t\ge 1$,
\[
L_t\ge L_t^H+L_1-L^H_1\ge m(t)-C_0+L_1-L^H_1.
\]
We now have that for $t\ge 1$,
\begin{equation} \label{eq:LtLHtsandwich}
m(t)-C_0+L_1-L^H_1\le L^H_t+L_1-L^H_1\le L_t\le L^+_t \le m(t)+C',
\end{equation}
which, in particular, completes the proof of~\eqref{eq:O(1) asymptotics of free-bdy}.

By~\eqref{eq:LtLHtsandwich}, we also have that there exists $A<\infty$ such that 
\begin{equation} \label{eq:LtLHtseqbound}
|L_t-L^H_t|\le A \quad \forall t\ge 1.
\end{equation}
Moreover, $t\mapsto L_t-L^H_t$ is non-decreasing by Lemma~\ref{lem:LtLhtincr}, and so there exists $a\in \R$ such that 
$\lim_{t\to \infty}(L_t-L^H_t)=a$, which establishes~\eqref{eq:theoLHtLtconv}.

It remains to prove~\eqref{eq:theoconvtoPimin}; we will use the Brunet-Derrida relation (Theorem~\ref{theo:magic formula}), which tells us that for $t>0$ and $r\in (-\infty,\sqrt{2})\setminus\{0\}$,
\begin{equation} \label{eq:magicforconvtoPi}
\int_0^{\infty} U(t,x+L_t)e^{rx}dx=-\frac 1 r + \frac 1r \int_0^\infty e^{r(L_{t+s}-L_t)-(1+\frac 12 r^2)s}ds.
\end{equation}
For $\varepsilon>0$, recalling the definition of $t_{\sqrt{2}-\varepsilon}$ in Lemma~\ref{lem:boundary locally Lipschitz from the left}, we have by Lemma~\ref{lem:boundary locally Lipschitz from the left} and Lemma~\ref{lem:boundary Lipschitz from the right when less stretched than a travelling wave} that for $t\ge t_{\sqrt{2}-\varepsilon}$ and $s\ge 0$,
\begin{equation} \label{eq:LHtspeed}
(\sqrt{2}-\varepsilon)s \le L^H_{t+s}-L^H_t \le \sqrt{2} s.
\end{equation}
Hence for $s\ge 0$ fixed,
\begin{equation} \label{eq:LtsLtconv}
L_{t+s}-L_t=L^H_{t+s}-L^H_t +(L_{t+s}-L^H_{t+s})-(L_t-L^H_t)\to \sqrt 2 s
\end{equation}
as $t\to \infty$, by~\eqref{eq:LHtspeed} and since $L_u-L^H_u$ converges as $u\to \infty$. Moreover, for $t\ge t_{\sqrt{2}-\varepsilon}\vee 1$ and $s\ge 0$,
by~\eqref{eq:LtLHtseqbound} and then by~\eqref{eq:LHtspeed},
\[
|L_{t+s}-L_t -\sqrt 2 s|\le 2A +|L^H_{t+s}-L^H_t -\sqrt 2 s|
\le 2A +\varepsilon s.
\]
For $r\in (-\infty,\sqrt{2})\setminus\{0\}$, take $\varepsilon>0$ sufficiently small that $|r|\varepsilon -(\frac{r}{\sqrt 2}-1)^2<0.$
Then for $t\ge t_{\sqrt{2}-\varepsilon}\vee 1$ and $s\ge 0$,
\begin{equation} \label{eq:fordomconvmagicPi}
r(L_{t+s}-L_t)-(1+\tfrac 12 r^2)s
=-(\tfrac{r}{\sqrt 2}-1)^2s+r(L_{t+s}-L_t-\sqrt 2 s)
\le (|r|\varepsilon -(\tfrac{r}{\sqrt 2}-1)^2)s+2A.
\end{equation}
Therefore, by~\eqref{eq:magicforconvtoPi},~\eqref{eq:LtsLtconv},~\eqref{eq:fordomconvmagicPi} and dominated convergence,
for any $r\in (-\infty,\sqrt{2})\setminus\{0\}$, as $t\to \infty$,
\begin{equation*} 
\int_0^{\infty} U(t,x+L_t)e^{rx}dx\to -\frac 1 r + \frac 1r \int_0^\infty e^{(\sqrt 2 r-1-\frac 12 r^2)s}ds=\frac{2\sqrt 2 -r}{(r-\sqrt 2)^2}
=\int_0^\infty \Pi_{\min}(x)e^{rx}dx,
\end{equation*}
where the last equality follows from~\eqref{eq:Pimindefn} and~\eqref{eq:minimal travelling wave}.
By integration by parts, it follows that
\begin{equation*} 
\int_0^{\infty} (-\partial_x U(t,x+L_t))e^{rx}dx\to 
\int_0^\infty \pi_{\min}(x)e^{rx}dx \quad \text{as }t\to \infty,
\end{equation*}
and so $-\partial_x U(t,x+L_t)dx$ converges weakly to $\pi_{\min}(x)dx$ as $t\to \infty$. 
Hence $U(t,\cdot+L_t)\to \Pi_{\min}$ pointwise as $t\to \infty$.

For $n\in \Nm$, letting $x_k:=\Pi_{\min}^{-1}(k/n)$ for $k\in \{1,\ldots ,n-1\}$, take $t$ sufficiently large that 
$|U(t,x_k+L_t)-\Pi_{\min}(x_k)|\le n^{-1}$ $\forall k\in \{1,\ldots ,n-1\}$.
Then since $U(t,\cdot)$ and $\Pi_{\min}$ are non-increasing and non-negative, with $U(t,x+L_t)=1=\Pi_{\min}(x)$ $\forall x\le 0$,
we have $|U(t,x+L_t)-\Pi_{\min}(x)|\le 2n^{-1}$ $\forall x\in \R$.
Therefore $U(t,\cdot+L_t)\to \Pi_{\min}$ uniformly as $t\to \infty$, which completes the proof.
\end{proof}

\subsection{Bootstrapping the boundary asymptotics} \label{subsec:bootstrap}

In this subsection, we will improve the asymptotics for $L_t$ given by~\eqref{eq:O(1) asymptotics of free-bdy} in Theorem~\ref{theo:convergence to the minimal travelling wave for finite initial mass} by proving the following finer asymptotics for $L^H_t$, which we can then combine with~\eqref{eq:theoLHtLtconv}.
\begin{theo}\label{theo:bootstrap boundary asymptotics}
There exists $x^H\in \Rm$ such that
	\begin{equation}
		L^H_t=\sqrt{2}t-\frac{3}{2\sqrt{2}}\log t+x^H+o(1) \quad \text{as }t\to \infty.
	\end{equation}
\end{theo}	
	As in~\eqref{eq:mtfinitemassdefn}, for $t\ge 0$, let
		\[
	m(t):=\sqrt{2}t-\frac{3}{2\sqrt{2}}\log (t+1).
	\]
We have from~\eqref{eq:O(1) asymptotics of free-bdy} in Theorem~\ref{theo:convergence to the minimal travelling wave for finite initial mass} and Proposition~\ref{prop:fbpsoln} that there exists $C\in (0,\infty)$ such that
	\begin{equation} \label{eq:sandwich Lt}
		m(t)-C\leq L^H_t\leq m(t)+C   \quad \forall t\geq 0.
	\end{equation}
	We define the stopping time
	\begin{equation} \label{eq:tauforLH}
	\tau:=\inf\{t>0:B_t\leq L^H_t\}.
	\end{equation}
	We will use the following technical lemma in the proof of Theorem~\ref{theo:bootstrap boundary asymptotics}.
	\begin{lem}\label{lem:hitting boundary plus a conditional on having high final value is small}
		For any $\epsilon>0$, there exist $T_0\in (10,\infty)$ and $R\in [10C,\infty)$ such that
		\begin{equation}\label{eq:hitting boundary plus a conditional on having high final value is small}
			\limsup_{t\to \infty}\psub{0}{\left. \exists s\in [T_0,t] \text{ such that } B_s\leq L^H_s+2C \, \right| B_t> L^H_t+R,\tau\ge t}\leq \epsilon.
		\end{equation}
	\end{lem}
	We defer the proof of Lemma \ref{lem:hitting boundary plus a conditional on having high final value is small} until after the proof of Theorem~\ref{theo:bootstrap boundary asymptotics}.
	We will also use the following elementary lemma in the proof of Theorem~\ref{theo:bootstrap boundary asymptotics}.
	\begin{lem} \label{lem:Pimin}
	For $d\ge 0$, let
	\begin{equation} \label{eq:fmindefn}
	f_{\min}(d):=\sup_{x\ge 0}\frac{\Pi_{\min}(x+d)}{\Pi_{\min}(x)}.
	\end{equation}
	Then $f_{\min}(d)<1$ $\forall d>0$.
	\end{lem}
	\begin{proof}
	 We fix $d>0$. We observe that since $x\mapsto \frac{\Pi_{\min}(x+d)}{\Pi_{\min}(x)}$ is continuous on $[0,\infty)$ and since $\Pi_{\min}(x+d)<\Pi_{\min}(x)$ for $x\ge 0$, it suffices to prove that $\limsup_{x\ra\infty}\frac{\Pi_{\min}(x+d)}{\Pi_{\min}(x)}<1$. But by~\eqref{eq:minimal travelling wave} and~\eqref{eq:Pimindefn}, we have that for $x\ge 0$,
		\[
		\Pi_{\min}(x+d)=\int_{x+d}^{\infty}2ye^{-\sqrt{2}y}dy=\int_{x}^{\infty}2(y+d)e^{-\sqrt{2}(y+d)}dy=e^{-\sqrt{2}d}\Big(\Pi_{\min}(x)+2d\int_x^{\infty}e^{-\sqrt{2}y}dy\Big).
		\]
		It follows that for any $x\ge 0$ and $d>0$,
		\begin{equation}\label{eq:inequality for ratio of Pi mins}
			\frac{\Pi_{\min}(x+d)}{\Pi_{\min}(x)}\leq e^{-\sqrt{2}d}+\frac{de^{-\sqrt{2}d}}{x}.
		\end{equation}
		Taking $\limsup_{x\to \infty}$ on both sides, the result follows.
		\end{proof}
	We now use Lemma~\ref{lem:hitting boundary plus a conditional on having high final value is small}, Theorem~\ref{theo:convergence to the minimal travelling wave for finite initial mass}, Lemma~\ref{lem:FKforwardstime}, Theorem~\ref{theo:extension of BBHR} and Lemma~\ref{lem:Pimin} to prove Theorem~\ref{theo:bootstrap boundary asymptotics}.
	\begin{proof}[Proof of Theorem~\ref{theo:bootstrap boundary asymptotics}]
	Recall the definition of $C\in (0,\infty)$ in~\eqref{eq:sandwich Lt}, and
	let
	\begin{equation} \label{eq:IandS}
	I=\liminf_{t\ra \infty}(m(t)+C-L^H_t)\quad \text{and} \quad S=\limsup_{t\ra\infty}(m(t)+C-L^H_t).
	\end{equation}
	We now fix an arbitrary $\epsilon>0$.
	Then by Lemma~\ref{lem:hitting boundary plus a conditional on having high final value is small}, we can fix $T_0\in (10,\infty)$ and $R\in [10C,\infty)$ satisfying~\eqref{eq:hitting boundary plus a conditional on having high final value is small} for this $\epsilon>0$. 
	We then define the modified boundary
	\begin{equation} \label{eq:Lt'defn}
	L'_t:=\begin{cases}
		L^H_t\quad &\text{for }0\leq t<T_0\\
		m(t)+C\quad &\text{for }t\geq T_0,
	\end{cases}
	\end{equation}
	and define the stopping time
	\begin{equation} \label{eq:tau'defn}
	\tau '=\inf\{t> 0:B_t\le L'_t\}.
	\end{equation}
	We note that by~\eqref{eq:sandwich Lt},
	\begin{equation} \label{eq:L'tsandwich}
	L^H_t\leq L'_t\leq L^H_t+2C\quad \forall t\geq 0.
	\end{equation}
	Recall the definition of $\tau$ in~\eqref{eq:tauforLH}.
By the Feynman-Kac representation in Lemma~\ref{lem:FKforwardstime}, we have
\begin{equation} \label{eq:taulimit}
	e^t\psub{0}{B_t> L^H_t+R,\tau\ge t} = U^H(t,L^H_t+R) \ra \Pi_{\min}(R)
\end{equation}
	as $t\ra \infty$, where the convergence follows from~\eqref{eq:theoconvtoPimin} in Theorem~\ref{theo:convergence to the minimal travelling wave for finite initial mass}.
	
	We now take an arbitrary increasing sequence $(t_n)_{n=1}^\infty\subset (0,\infty)$ such that $t_n\to \infty$ as $n\to \infty$ and $L'_{t_n}-L^H_{t_n}$
	converges as $n\ra\infty$. We let $c:=\lim_{n\to \infty}(L'_{t_n}-L^H_{t_n}) \in [0,2C]$.
For $n\in \Nm$ sufficiently large that $t_n\ge T_0$, let
\begin{equation} \label{eq:a_ndefn}
a_n:=e^{t_n-T_0}\psub{0}{\left. B_{t_n}> L^H_{t_n}+R,\tau'> t_n \right| \tau'>T_0}.
\end{equation}	
Then we can write
\begin{equation} \label{eq:proba_n}
e^{t_n}\psub{0}{B_{t_n}> L^H_{t_n}+R,\tau'> t_n}=a_n e^{T_0}\psub{0}{\tau'>T_0}.
\end{equation}
We will now use Theorem~\ref{theo:extension of BBHR} to determine the limit of the left-hand side of~\eqref{eq:proba_n} as $n\to \infty$.
By the definition of $\tau'$ in~\eqref{eq:tau'defn} and then by the definition of $\tau$ in~\eqref{eq:tauforLH},
\[
\{\tau'>T_0\}=\{B_{T_0}>m(T_0)+C,\, B_s>L^H_s\; \forall s\in (0,T_0)\}
=\{B_{T_0}>m(T_0)+C,\, \tau>T_0\}.
\]
Therefore, by Lemma~\ref{lem:FKforwardstime} and then by Proposition~\ref{prop:fbpsoln},
\begin{equation} \label{eq:probtau'}
\psub{0}{\tau' >T_0}=\psub{0}{\tau>T_0,\, B_{T_0}>m(T_0)+C}=e^{-T_0}U^H(T_0,m(T_0)+C)>0.
\end{equation}
Let $\mu_{T_0}$ denote the law of $B_{T_0}$ under $\psub{0}{\cdot | \tau'>T_0}$.
Then
\[
\int_{[0,\infty)}xe^{\sqrt 2 x}\mu_{T_0}(dx)
\le \frac{\Esub{0}{B_{T_0}e^{\sqrt 2 B_{T_0}}}}{\psub{0}{\tau' >T_0}}<\infty.
\]
Hence by Theorem~\ref{theo:extension of BBHR}, there exists $C_{T_0}\in (0,\infty)$ depending on $T_0$ such that
\[
e^t \psub{\mu_{T_0}}{B_t>m(t+T_0)+C+y, B_s>m(s+T_0)+C \; \forall s\in (0,t)}\to C_{T_0}\Pi_{\min}(y)
\]
as $t\to \infty$, uniformly in $y\ge 0$.
By the definition of $\tau'$ in~\eqref{eq:tau'defn} and the definition of $L'_t$ in~\eqref{eq:Lt'defn}, it follows that
\[
e^{t-T_0} \psub{0}{\left. B_t>L'_t+y, \tau'\ge t \right| \tau'>T_0}\to C_{T_0}\Pi_{\min}(y)
\]
as $t\to \infty$, uniformly in $y\ge 0$.
Then for $\delta>0$, since $c=\lim_{n\to \infty}(L'_{t_n}-L^H_{t_n})$ we have that for $n$ sufficiently large, $|L'_{t_n}-L^H_{t_n}-c|<\delta$, and so, recalling the definition of $a_n$ in~\eqref{eq:a_ndefn},
\[
\limsup_{n\to \infty}a_n\le C_{T_0}\Pi_{\min}(R-c-\delta)
\quad \text{and }\quad \liminf_{n\to \infty}a_n\ge C_{T_0}\Pi_{\min}(R-c+\delta).
\]
Since $\delta>0$ was arbitrary and $\Pi_{\min}$ is continuous, it follows that $\lim_{n\to \infty}a_n= C_{T_0}\Pi_{\min}(R-c)$. Therefore by~\eqref{eq:proba_n} and~\eqref{eq:probtau'},
letting $C'_{T_0}=C_{T_0}U^H(T_0,m(T_0)+C)\in (0,\infty)$,
\begin{equation} \label{eq:tau'limit}
\lim_{n\to \infty} e^{t_n}\psub{0}{B_{t_n}> L^H_{t_n}+R,\tau'\ge  t_n}=  C'_{T_0}\Pi_{\min}(R-c).
\end{equation}

By~\eqref{eq:L'tsandwich}, we have $\tau'\ge \tau$, and so for $t\ge T_0$ sufficiently large,  
\begin{align*}
\psub{0}{B_{t}> L^H_{t}+R,\tau\geq t}
&\geq \psub{0}{B_{t}> L^H_{t}+R,\tau' \geq t}\\
&=\psub{0}{B_{t}> L^H_{t}+R,\tau\geq t}\psub{0}{\left. \tau'\ge t \right| B_{t}> L^H_{t}+R,\tau\geq t}\\
&\geq (1-2\epsilon)\psub{0}{B_{t}> L^H_{t}+R,\tau\geq t},
\end{align*}
where in the last line we used~\eqref{eq:L'tsandwich} and the fact
that we chose $T_0$ and $R$ 
 satisfying~\eqref{eq:hitting boundary plus a conditional on having high final value is small}
in Lemma~\ref{lem:hitting boundary plus a conditional on having high final value is small}.
It then follows from~\eqref{eq:taulimit} and~\eqref{eq:tau'limit} that for any $c\in \R$ such that there exists an increasing sequence $t_n\to \infty$ with $c=\lim_{n\to \infty}(L'_{t_n}-L^H_{t_n})$, we have
	\begin{equation}\label{eq:relationship between m(R-delta) and m(R)}
		\Pi_{\min}(R)\geq C'_{T_0}\Pi_{\min}(R-c)\geq (1-2\epsilon)\Pi_{\min}(R).
	\end{equation}
	
	We now note that by the definition of $I$ and $S$ in~\eqref{eq:IandS} and the definition of $L'_t$ in~\eqref{eq:Lt'defn}, we have
	\[
	I=\liminf_{t\ra \infty}(L'_t-L^H_t)\quad \text{and} \quad S=\limsup_{t\ra\infty}(L'_t-L^H_t).
	\]
	We can therefore take increasing sequences $t^-_n,t^+_n\ra \infty$ such that
	\[
	L'_{t^+_n}-L^H_{t^+_n}\ra S\quad\text{and}\quad L'_{t^-_n}-L^H_{t^-_n}\ra I \quad \text{as }n\to \infty.
	\]
Since $\Pi_{\min}$ is non-increasing, it follows from \eqref{eq:relationship between m(R-delta) and m(R)} that
	\[
	\Pi_{\min}(R)\geq C'_{T_0}\Pi_{\min}(R-S)\geq C'_{T_0}\Pi_{\min}(R-I)\geq (1-2\epsilon)\Pi_{\min}(R).
	\]
	Therefore
	\begin{equation}\label{eq:inequality of mass to the right of of R-I to R-S}
		1\geq \frac{\Pi_{\min}(R-I)}{\Pi_{\min}(R-S)}\geq 1-2\epsilon .
	\end{equation}
	Note that $R$ depends on our choice of $\epsilon$, but $I$ and $S$ do not.
	Recall that we chose $R\ge 10C$ at the beginning of the proof, and that $I,S \in [0,2C]$ by~\eqref{eq:sandwich Lt} and~\eqref{eq:IandS}.
	Therefore, by the definition of $f_{\min}$ in~\eqref{eq:fmindefn}, we can 
observe that~\eqref{eq:inequality of mass to the right of of R-I to R-S} implies that
	\[
	f_{\min}(S-I)\geq 1-2\epsilon.
	\]
	Since we chose $\epsilon>0$ to be fixed and arbitrary, it follows that $f_{\min}(S-I)=1$, and hence $S=I$ by Lemma~\ref{lem:Pimin}. By~\eqref{eq:IandS}, this completes the proof.
	\end{proof}		
	We finally conclude this section by proving Lemma \ref{lem:hitting boundary plus a conditional on having high final value is small}, using the asymptotics for $L^H_t$ in Theorem~\ref{theo:convergence to the minimal travelling wave for finite initial mass}, the Feynman-Kac representation in Lemma~\ref{lem:FKforwardstime}, and consequences of the stretching lemma.	
	\begin{proof}[Proof of Lemma \ref{lem:hitting boundary plus a conditional on having high final value is small}]
		We fix $T_0\in (10,\infty)$ and $R\in [10C,\infty)$ to be chosen later. 
		Define the stopping time
		\begin{equation} \label{eq:hattaudefn}
			\hat{\tau}:=\inf\{t\geq T_0:B_t\leq L^H_t+2C\}\wedge \tau,
		\end{equation}
		where $\tau=\inf\{t>0:B_t\le L^H_t\}$ as in~\eqref{eq:tauforLH}.
		
		For $s\ge 0$ and $x\in \R$, we write $\Pm_{x,s}$ for the probability measure under which $(B_t)_{t\ge s}$ is a Brownian motion started from time $s$ at position $x$, so that $\Pm_{x,s}(B_s=x)=1$. 
		For $s\ge 0$, we let
			\begin{equation} \label{eq:tausdefn}
		\tau_s:=\inf\{t\geq s:B_t\leq L^H_t \}.
		\end{equation}
We now fix an arbitrary $s\geq T_0$ and $x\in [L^H_s,L^H_s+2C]$;		
		we will obtain an upper bound on
		\[
		\Pm_{x,s}(B_t> L^H_t+R,\tau_s \ge t).
		\]
	
	We define another boundary
	\begin{equation} \label{eq:barLdefn}
	L^-_t:=L^H_s+L^H_{t-s+1}-L^H_1 \quad \text{for }t\ge s.
\end{equation}
By Lemma~\ref{lem:extended maximum principle}, and since $s>1$, we have $U^H(s,\cdot)\ge_s U^H(1,\cdot)$, and so by Lemma~\ref{lem:less stretching means slower boundary}, for $t\ge s$,
\[
L^H_t-L^H_s \ge L^H_{t-s+1}-L^H_1.
\]
Therefore
\[
L^-_t\le L^H_t \quad \forall t\ge s.
\]
We now compare the event of hitting the boundary $u\mapsto L^H_{s+u}$ with the event of hitting $u\mapsto \bar L_{u+s}$.
 We define
\begin{equation} \label{eq:tau-defn}
\tau^-_s:=\inf\{t\geq s:B_t\leq L^-_t \}.
\end{equation} 
We see that since $L^-_t\le L^H_t$ $\forall t\ge s$ and so $\tau^-_s\ge \tau_s$,
and then by the definitions of $L^-_u$ in~\eqref{eq:barLdefn} and $\tau_u$ in~\eqref{eq:tausdefn}, for $t\ge s$,
		\begin{align} \label{eq:probhitLH}
			\Pm_{x,s}(B_t> L^H_t+R,\, \tau_s\ge t)
			&\leq \Pm_{x,s}(B_t> L^-_{t}+(L^H_t-L^-_t)+R,\, \tau^-_s\ge t) \notag \\
			&= \Pm_{x-L^H_s+L^H_1,1}(B_{t-s+1}> L^H_{t-s+1}+(L^H_t-L^-_t)+R,\, \tau_1\ge t-s+1) \notag \\
			&\leq \Pm_{L^H_1+2C,1}(B_{t-s+1}> L^H_{t-s+1}+(L^H_t-L^-_t)+R,\, \tau_1\ge t-s+1) \notag \\
			&\leq \frac{\Pm_{0,0}(B_{t-s+1}> L^H_{t-s+1}+(L^H_t-L^-_t)+R,\, \tau\ge t-s+1)}{\Pm_{0,0}(\tau\ge 1,\, B_1> L^H_1+2C)},
		\end{align}
		where the penultimate inequality follows since $x-L^H_s+L^H_1\le L^H_1+2C$ by our assumption on $x$, and the last inequality follows from the Markov property at time $1$.
		
		By Lemma~\ref{lem:FKforwardstime} and then by Proposition~\ref{prop:fbpsoln}, for $r\ge 0$ and $y>0$ we have
		\[
		e^r \Pm_{0,0}(\tau\ge r,B_r> L^H_r+y)
		=U^H(r,L^H_r+y) \in (0,1).
		\]
		Therefore we can define
		\begin{equation} \label{eq:c1defn}
		c_1:=\frac{1}{\Pm_{0,0}(\tau\ge 1,B_1> L^H_1+2C)}\in (0,\infty).
		\end{equation}
Moreover, since $U^H(r,\cdot)\le_s \Pi_{\min}$ for $r\ge 0$ by Lemma~\ref{lem:extended maximum principle}, we have 
\begin{equation}\label{eq:prob tau at least r, Br at least LrH+y} e^r \Pm_{0,0}(\tau\ge r,B_r> L^H_r+y)=U^H(r,L^H_r+y)\le \Pi_{\min}(y) \quad \forall r\ge 0,\, y>0 
\end{equation}
by Lemma~\ref{lem:stretchdecr}. By substituting~\eqref{eq:c1defn} and~\eqref{eq:prob tau at least r, Br at least LrH+y} into~\eqref{eq:probhitLH},
it follows that	for $t\ge s$,
\begin{equation} \label{eq:probhitLH2}
		\Pm_{x,s}(B_t> L^H_t+R,\tau_s\ge t)\leq c_1 e^{-(t-s+1)}\Pi_{\min}(L^H_t-L^-_t+R) \le c_1 e^{-(1+t-s)}.
\end{equation}
		
		We now observe that by~\eqref{eq:barLdefn}, for $t\ge s$, 
		\[
		L^H_t-L^-_t=L^H_t-L^H_s-L^H_{t-s+1}+L^H_1.
		\]	
		Therefore, by~\eqref{eq:O(1) asymptotics of free-bdy} in Theorem~\ref{theo:convergence to the minimal travelling wave for finite initial mass} and since $s\ge 1$, there exists a constant $K_1<\infty$ (not depending on $s$) such that for any $t\ge s$,
		\begin{equation} \label{eq:LHL-bound}
			L^H_t-L^-_t\ge -\frac{3}{2\sqrt{2}}\Big(\log t-\log s-\log(t-s+1)\Big)-K_1.
		\end{equation}
		To shorten notation we now define, for $0< s \le t<\infty$,
		\begin{equation} \label{eq:FstRdefn}
			F(s ,t,R):=c_1 e^{-1}\Pi_{\min}\Big(-\frac{3}{2\sqrt{2}}\Big(\log t-\log s -\log(t-s +1)\Big)-K_1+R\Big).
		\end{equation}
		Thus by~\eqref{eq:probhitLH2} and~\eqref{eq:LHL-bound}, for any $t\ge s\ge T_0$ and $x\in [L^H_s,L^H_s+2C]$,
		\begin{equation} \label{eq:upper bound Pimin}
	\Pm_{x,s}(B_t> L^H_t+R,\tau_s\ge t)\leq e^{-(t-s)}F(s,t,R).
		\end{equation} 
		
		Recall the definition of $\hat{\tau}$ in~\eqref{eq:hattaudefn} and the definition of $\tau_s$ in~\eqref{eq:tausdefn}.
Then for $t\ge 2T_0$, using the strong Markov property at time $\hat\tau$ in the first line, and using~\eqref{eq:upper bound Pimin} and the definition of $\hat\tau$ in the second line,		
		\begin{align} \label{eq:probtauhatearly}
			\psub{0}{\hat{\tau}\leq t-T_0,\tau \ge t,B_t> L^H_t+R}
			&=\Esub{0}{\1_{\{\hat{\tau}\leq t-T_0,\, \tau\geq \hat \tau\}}\psub{B_{\hat{\tau}},\hat{\tau}}{\tau_{\hat \tau}\ge t,B_{t}> L^H_t+R}} \notag \\
			&\leq \Esub{0}{\1_{\{T_0\le \hat{\tau}\leq t-T_0,\, \tau\geq \hat \tau\}}e^{-(t-\hat \tau)}F(\hat \tau, t,R)} \notag \\
			&= \sum_{n=\lfloor T_0\rfloor}^{\lfloor t-T_0\rfloor}
			\Esub{0}{\1_{\{\hat{\tau}\in [n,n+1)\}}\1_{\{T_0\le \hat{\tau}\leq t-T_0,\, \tau\geq \hat \tau\}}e^{-(t-\hat \tau)}F(\hat \tau, t,R)} \notag \\
			&\le \sum_{n=\lfloor T_0\rfloor}^{\lfloor t-T_0\rfloor}
			e^{n-t+1}\sup_{s\in [n,n+1)}F(s,t,R)
			\psub{0}{\tau\ge n} \notag \\
			&=e^{-t+1}\sum_{n=\lfloor T_0\rfloor}^{\lfloor t-T_0\rfloor}\sup_{s\in [n,n+1)}F(s,t,R),
		\end{align}
		where the last line follows since $\psub{0}{\tau\ge n}=e^{-n}$ for each $n$ by Lemma~\ref{lem:FKforwardstime}.
		
		We now claim that for any $t>1$ fixed,
		\begin{equation} \label{eq:Fstclaim}
			s\mapsto F(s,t,R)\quad \text{is}\quad \begin{cases}
				\text{non-increasing on }\; (1,(1+t)/2],\\
				\text{non-decreasing on }\; [(1+t)/2,t].
			\end{cases}
		\end{equation}
		Indeed, to prove the claim, for $s\in (1,t]$,
			we calculate 
			\[
			\frac{\partial}{\partial s}(\log t-\log s-\log (t-s+1))=-\frac{1}{s}+\frac{1}{t-s+1}\begin{cases}
				\le 0\quad \text{for }s\in (1,(1+t)/2],\\
				\ge 0\quad \text{for }s\in [(1+t)/2,t].
			\end{cases}
			\]
			The claim~\eqref{eq:Fstclaim} then follows from the definition of $F$ in~\eqref{eq:FstRdefn} and the fact that $x\mapsto \Pi_{\min}(x)$ is non-increasing.

		It follows from~\eqref{eq:Fstclaim} that for $t\ge 2T_0$ and $n\in \mathbb Z \cap [\lfloor T_0 \rfloor, \lfloor t-T_0 \rfloor ]$,
		\[
		\sup_{s\in [n,n+1)} F(s,t,R)\leq F(n,t,R)+F(n+1,t,R).
		\]
		Therefore, by~\eqref{eq:probtauhatearly}, for $t\ge 2T_0$, 
		\[
		\psub{0}{\hat{\tau}\leq t-T_0,\tau \ge t,B_t> L^H_t+R}\leq 2e^{-t+1}\sum_{n=\lfloor T_0\rfloor}^{\lfloor t-T_0\rfloor+1}F(n,t,R).
		\]
		On the other hand, by Lemma~\ref{lem:FKforwardstime} and then by~\eqref{eq:theoconvtoPimin} in Theorem~\ref{theo:convergence to the minimal travelling wave for finite initial mass}, we have that for any $t>0$,
		\begin{equation}\label{eq:probability of surviving and being above Lt+R}
			e^t\psub{0}{\tau\ge t,B_t> L^H_t+R}=U^H(t,L^H_t+R)
			\ra \Pi_{\min}(R)
		\end{equation}
		as $t\ra\infty$. It follows that
		\begin{equation} \label{eq:sum of F}
			\limsup_{t\ra\infty}\psub{0}{\hat{\tau}\leq t-T_0\Big\lvert \tau\ge t,B_t> L^H_t+R}\leq 2e \limsup_{t\ra\infty}\sum_{n=\lfloor T_0\rfloor}^{\lfloor t-T_0\rfloor+1}\frac{F(n,t,R)}{\Pi_{\min}(R)}.
		\end{equation}
		
		We now write for $0<s\le t<\infty$, recalling that $K_1<\infty$ is the fixed constant in the definition of $F$ in~\eqref{eq:FstRdefn},
\begin{equation} \label{eq:Deltastdefn}
\Delta_{s,t}:=-\frac{3}{2\sqrt{2}}\Big(\log t-\log s-\log(t-s+1)\Big)-K_1.
\end{equation}
		Note that since $s(1+t-s)\ge \frac 12 t (s\wedge (t-s))$ for $0<s\le t$,
		we have that if $s\wedge (t-s)$ is sufficiently large then
		$\Delta_{s,t}\ge 1$. 
		We then use the definition of $F$ in~\eqref{eq:FstRdefn},
		and~\eqref{eq:inequality for ratio of Pi mins} in the proof of Lemma~\ref{lem:Pimin}, to see that
		for $T_0$ sufficiently large (not depending on anything else),
		for $t\ge 2T_0$ and $n\in  [\lfloor T_0 \rfloor, \lfloor t-T_0 \rfloor +1]$,
		\[
		\frac{F(n,t,R)}{\Pi_{\min}(R)}=c_1e^{-1}\frac{\Pi_{\min}(\Delta_{n,t}+R)}{\Pi_{\min}(R)}\leq c_1 e^{-1} e^{-\sqrt{2}\Delta_{n,t}}\left(1+\frac{\Delta_{n,t}}{R}\right).
		\]
		
Now fix $\delta>0$ sufficiently small that
\begin{equation} \label{eq:choosedelta}
\tfrac{3(\sqrt{2}-\delta)}{2\sqrt{2}}\in (1,2).
\end{equation}
Recall that we are assuming that $R\ge 10C$, and recall the definition of $c_1$ in~\eqref{eq:c1defn}.
Then there exists a constant $K_2<\infty$ depending only upon our choice of $\delta$ and on $C$ (which is fixed), such that
for any $T_0$ sufficiently large (not depending on anything else),
		for $t\ge 2T_0$ and $n\in  [\lfloor T_0 \rfloor, \lfloor t-T_0 \rfloor +1]$,
		\[
		\frac{F(n,t,R)}{\Pi_{\min}(R)}\leq K_2 e^{-(\sqrt{2}-\delta)\Delta_{n,t}}.
		\]
By~\eqref{eq:sum of F} and then by the definition of $\Delta_{n,t}$ in~\eqref{eq:Deltastdefn}, we therefore have that for $T_0$ sufficiently large,
		\begin{align*}
&\limsup_{t\ra\infty}\psub{0}{\hat{\tau}\leq t-T_0\Big\lvert \tau\ge t,B_t> L^H_t+R}\\&\leq 2eK_2 \limsup_{t\ra\infty}\sum_{n=\lfloor T_0\rfloor}^{\lfloor t-T_0\rfloor+1}e^{-(\sqrt{2}-\delta)\Delta_{n,t}}\\
			&\leq 2eK_2 e^{(\sqrt 2 -\delta)K_1} \limsup_{t\ra\infty}
			\sum_{n=\lfloor T_0\rfloor}^{\lfloor t-T_0\rfloor+1}\Big(\frac{t}{n(t-n+1)}\Big)^{\frac{3(\sqrt{2}-\delta)}{2\sqrt{2}}}.
		\end{align*}
		By our choice of $\delta$ in~\eqref{eq:choosedelta}, we have $r:=\frac{3(\sqrt{2}-\delta)}{2\sqrt{2}}-1\in (0,1)$.
		Therefore, for $T_0$ sufficiently large,
	\begin{align} \label{eq:limsupearlytauhatbd}
		\limsup_{t\ra\infty}\psub{0}{\hat{\tau}\leq t-T_0\Big\lvert \tau\ge t,B_t> L^H_t+R}
			&\leq 2eK_2 e^{\sqrt 2 K_1} \limsup_{t\ra\infty} \int_{T_0-2}^{t+2-T_0}\Big(\frac{1}{s}+\frac{1}{t-s+1}\Big)^{1+r}ds \notag \\
			&\leq 2eK_2 e^{\sqrt 2 K_1}\cdot 2^{2+r}\int_{T_0-2}^{\infty}s^{-(1+r)}ds \notag \\
			&= 2^{3+r}r^{-1} eK_2 e^{\sqrt 2 K_1}(T_0-2)^{-r}.
	\end{align}
		
		We now fix $\epsilon>0$. It follows from~\eqref{eq:limsupearlytauhatbd} that we can choose $T_0\in [10,\infty)$ sufficiently large that for any $R\in [10C,\infty)$ we have
		\begin{equation}\label{eq:bound on cond prob 1}
			\limsup_{t\ra\infty}\psub{0}{\hat{\tau}\leq t-T_0\Big\lvert \tau\ge t,B_t> L^H_t+R}<\epsilon.
		\end{equation}
		We henceforth fix such a $T_0\in [10,\infty)$.
		
		On the other hand, for $t\ge 2T_0$ we have the crude bound
		\begin{align} \label{eq:probtauhatcrude}
		\psub{0}{t-T_0\leq \hat{\tau}\leq t,\, \tau>t,\, B_t> L^H_t+R}
		&\le \Esub{0}{ \1_{\{t-T_0\leq \hat{\tau}\leq t\}} \1_{\{\tau\ge \hat \tau\}} \psub{B_{\hat\tau},\hat \tau}{B_t>L^H_t+R}} \notag \\
		&\le \psub{0}{\tau\geq t-T_0}\sup_{t-T_0\leq s\leq t}\psub{L^H_s+2C,s}{B_t> L^H_t+R} \notag \\
		&=e^{-( t-T_0)}\sup_{t-T_0\leq s\leq t}\psub{L^H_s+2C,s}{B_t> L^H_t+R},
		\end{align}
		where the first inequality follows from the strong Markov property at time $\hat \tau$, the second inequality follows from the definition of $\hat\tau$ in~\eqref{eq:hattaudefn}, and the last inequality follows from Lemma~\ref{lem:FKforwardstime}.
		
		By Lemma~\ref{lem:less stretching means slower boundary}, for $0\le s \le t$
		we have 
		\[
		L^H_t-L^H_s\ge L^H_{t-s}-L^H_0\ge \inf_{u\ge 0}L^H_u>-\infty,
		\]
		where the last inequality follows from Lemma~\ref{lem:LHtlower}.
		Letting $K_3:=-\inf_{u\ge 0}L^H_u<\infty$, 
		we can write
		\begin{align} \label{eq:crudeGaussian}
		\sup_{t-T_0\leq s\leq t}\psub{L^H_s+2C,s}{B_t> L^H_t+R}
		&\le \sup_{t-T_0\leq s\leq t}\psub{0}{B_{t-s}> R-2C-K_3} \notag \\
		&\le \1_{\{R\le 2C+K_3\}}+\psub{0}{B_{T_0}> R-2C-K_3}\1_{\{R> 2C+K_3\}}.
		\end{align}
		
		We now combine~\eqref{eq:probtauhatcrude} and~\eqref{eq:crudeGaussian} with~\eqref{eq:probability of surviving and being above Lt+R} to see that
		\begin{align*}
		&\limsup_{t\ra\infty}\psub{0}{t-T_0\leq \hat{\tau}\leq t\Big| \tau\ge t,B_t> L^H_t+R}\\
		&\quad \leq e^{T_0}\frac{\1_{\{R\le 2C+K_3\}}+\psub{0}{B_{T_0}> R-2C-K_3}\1_{\{R> 2C+K_3\}}}{\Pi_{\min}(R)}.
		\end{align*}
		Since $\psub{0}{B_{T_0}> \cdot}$ has Gaussian tails and $\Pi_{\min}(x)$ decays exponentially as $x\to \infty$, we can choose $R\in [10C,\infty)$ sufficiently large that 
		\begin{equation}\label{eq:bound on cond prob 2}
			\limsup_{t\ra\infty}\psub{0}{t-T_0\leq \hat{\tau}\leq t\Big| \tau\ge t,B_t> L^H_t+R}\leq \epsilon.
		\end{equation}
		Combining \eqref{eq:bound on cond prob 1} and \eqref{eq:bound on cond prob 2}, we see that for any $\epsilon>0$ we can choose $ T_0\in [10,\infty)$ and $R\in [10C,\infty)$ such that
		\[
		\limsup_{t\ra\infty}\psub{0}{\hat{\tau}\leq t\Big| \tau\ge t,B_t> L^H_t+R}\leq 2\epsilon.
		\]
		By the definition of $\hat \tau$ in~\eqref{eq:hattaudefn}, this concludes the proof of Lemma \ref{lem:hitting boundary plus a conditional on having high final value is small}.	
	\end{proof}

\section{Infinite initial mass case} \label{sec:infiniteinitialmass}

In this section, we will often make the following assumptions on the initial condition $U_0$:
\begin{align}
	&\limsup_{x\to \infty} \tfrac 1x \log U_0(x)\le -\sqrt{2} \label{eq:assump_infmass1}\\
	&\text{and }\qquad \int_0^{\infty} ye^{\sqrt{2}y}U_0(y)dy=\infty.\label{eq:assump_infmass2}
\end{align}
We will prove the following results in this section.
\begin{theo} \label{theo:infinitemassconv}
	Suppose $U_0$ satisfies Assumption~\ref{assum:standing assumption ic},~\eqref{eq:assump_infmass1} and~\eqref{eq:assump_infmass2}.
	Let $(U(t,x),L_t)$ denote the solution of the free boundary problem~\eqref{eq:FBP_CDF} with initial condition $U_0$.
	Then
	\[
	\sup_{x\in \Rm}|U(t,x+L_t)-\Pi_{\min}(x)|\to 0 \quad \text{as }t\to \infty.
	\]
\end{theo}

\begin{theo} \label{theo:infmassfront}
	Suppose $U_0$ satisfies Assumption~\ref{assum:standing assumption ic},~\eqref{eq:assump_infmass1} and~\eqref{eq:assump_infmass2}, and moreover that for some $\gamma<1/2$,
\begin{equation}\label{eq:stretched exponential U0 proof section}
U_0(x)\leq e^{x^{\gamma}-\sqrt{2}x}
\end{equation}
for all $x$ sufficiently large. 
	 %for $t$ sufficiently large,
	%\begin{equation} \label{eq:bcond}
	%b(t):=2^{-1 / 2} \log \left(\int_0^{\infty} y e^{\sqrt 2 y} U_0(y) e^{-y^2 / 2 t} d y+1\right) \leq t^\delta.
	%\end{equation}
	Let $(U(t,x),L_t)$ denote the solution of the free boundary problem~\eqref{eq:FBP_CDF} with initial condition $U_0$. For $t>0$, let
	\begin{equation} \label{eq:bdefninfmass}
	b(t):=2^{-1 / 2} \log \left(\int_0^{\infty} y e^{\sqrt 2 y} U_0(y) e^{-y^2 / (2 t)} d y+1\right).
	\end{equation}
	Then
	\[
	L_t-m(t)\to 0 \quad \text{as }t\to \infty,
	\]
	where
	\begin{equation} \label{eq:mtheodef}
	m(t)=\sqrt 2 t-\frac 3{2\sqrt 2} \log t+b(t)-\frac 1 {\sqrt{2}} \log {\sqrt{\pi}}.
	\end{equation}
\end{theo}

We now give a heuristic overview of the proof of Theorems~\ref{theo:infinitemassconv} and~\ref{theo:infmassfront}.
The proof follows a simplified version of the strategy used by Bramson in~\cite{Bramson1983} to prove the corresponding results for solutions of the FKPP equation with initial condition $U_0$ satisfying~\eqref{eq:assump_infmass1} and~\eqref{eq:assump_infmass2}.

By the Feynman-Kac formula in Lemma~\ref{lem:FKforinfinitemass}, for $t>0$ and $x\in \R$,
\begin{equation} \label{eq:infheur_lower}
U(t,x)\ge e^t \Esub{x}{U_0(B_t)\1_{\{B_s>L_{t-s}\; \forall s\in [0,t]\}}}.
\end{equation}
For $2\le r\le t$, we will define a curve $\underline{\mathcal M}_{r,t}$ where $\underline{\mathcal M}_{r,t}(t-s)\approx L_{t-s}-\Theta(1)(s\wedge (t-s))^{1/3}$ for $s\in [3r,t-3r]$.
Then we can show (see Proposition~\ref{prop:upperbdinfmass}) that for large $r$, there exists $N_r(t)\to \infty$ as $t\to\infty$ such that for large $t$ and $x\ge L_t$,
\begin{equation} \label{eq:infheur_upper}
U(t,x)\le (1+o(1))e^t \Esub{x}{U_0(B_t)\1_{\{B_s>\underline{\mathcal M}_{r,t}(t-s)\; \forall s\in [3r,t-3r]\}}\1_{\{B_t>N_r(t)\}}}.
\end{equation}
Indeed, to establish~\eqref{eq:infheur_upper}, we show that
\begin{itemize}
\item trajectories $(B_s)_{s\in [0,t]}$ that go below $\underline{\mathcal M}_{r,t}(t-s^*)$ at some time $s^*\in [3r,t-3r]$ are likely to spend a long time below $L_{t-s}$, and so $\Leb(\{s\in [0,t]:B_s\ge L_{t-s}\})$ is likely to be significantly less than $t$ for these trajectories, which means that their contribution to the formula for $U(t,x)$ in Lemma~\ref{lem:FKforinfinitemass} is small (see Proposition~\ref{prop:Gxy}).
\item using the infinite initial mass condition~\eqref{eq:assump_infmass2}, trajectories $(B_s)_{s\in [0,t]}$ that stay above $\underline{\mathcal M}_{r,t}(t-s)$ with $B_t\le \mathcal O(1)$ make a small contribution to the formula for $U(t,x)$ in Lemma~\ref{lem:FKforinfinitemass} (see Lemma~\ref{lem:Bt_ends_big}).
\end{itemize}
We then use~\eqref{eq:infheur_upper} and Brownian bridge estimates (using that $\underline{\mathcal M}_{r,t}$ is close enough to a straight line from $o(t)$ at time $0$ to $L_t=\sqrt{2}t+o(t)$ at time $t$)
to show that (see Corollary~\ref{cor:x2x1}) for large $r$ and $t$, for $x_1,x_2\in [L_t,L_t+o(t)]$ with $x_1\le x_2$, with $x_1-L_t$ large and $x_2-x_1$ not too large,
\begin{align*}
U(t,x_2)
&\le o(1)\\
&\;\; +(1+o(1))e^t e^{-(\sqrt 2-o(1))(x_2-x_1)}\frac{x_2-L_t}{x_1-L_t}\Esub{x_1}{U_0(B_t)\1_{\{B_s>\underline{\mathcal M}_{r,t}(t-s)\; \forall s\in [3r,t-3r]\}}\1_{\{B_t>N_r(t)\}}}\\
&\le o(1)+(1+o(1)) e^{-(\sqrt 2-o(1))(x_2-x_1)}\frac{x_2-L_t}{x_1-L_t}U(t,x_1),
\end{align*}
where the second inequality follows by~\eqref{eq:infheur_lower}, and since $N_r(t)\to \infty$ as $t\to \infty$, and since when $x-L_t$ and $y-L_0$ are large,
\begin{equation} \label{eq:infheur_sameprob}
\p{\xi^t_{x,y}(s)>\underline{\mathcal M}_{r,t}(t-s)\; \forall s\in [3r,t-3r]}
=(1+o(1))\p{\xi^t_{x,y}(s)>L_{t-s}\; \forall s\in [0,t]}.
\end{equation}
Indeed,~\eqref{eq:infheur_sameprob} follows from an entropic repulsion result for Brownian bridge (see Lemma~\ref{lem:probaboveinterval}).

We now have an upper bound on the tail behaviour of $U(t,\cdot)$ for large $t$.
By the stretching lemma, along with the convergence of $U^H(t,L^H_t+\cdot)$ to $\Pi_{\min}$ from Theorem~\ref{theo:convergence to the minimal travelling wave for finite initial mass}, and the tail asymptotics of $\Pi_{\min}$, we can also show that for large $t$ and $L_t\le x_1\le x_2$ with $x_1-L_t$ large (but $x_2-L_t$ not too large),
\[
U(t,x_2)
\ge (1-o(1)) e^{-\sqrt 2(x_2-x_1)}\frac{x_2-L_t}{x_1-L_t}U(t,x_1).
\]
By combining the upper and lower bounds on the tail behaviour of $U(t,\cdot)$ for large $t$, we can prove Theorem~\ref{theo:infinitemassconv}.
(See Proposition~\ref{prop:sandwichmeansconv} for a tail condition on $U(t,m(t)+\cdot)$ for suitable $m(t)$ that implies convergence to $\Pi_{\min}$.)

For the proof of Theorem~\ref{theo:infmassfront}, we will show (see Proposition~\ref{prop:s0}) using (roughly)~\eqref{eq:infheur_lower},~\eqref{eq:infheur_upper} and~\eqref{eq:infheur_sameprob} that for any $s_0\ge 0$, for $t$ large and $x-L_t$ large,
\[
U(t,x)=(1+o(1)) e^t \Esub{x}{U_0(B_t)\1_{\{B_s>L_{t-s}\; \forall s\in [0,t-s_0]\}}}.
\]
Armed with this expression for $U(t,x)$ and the convergence result in Theorem~\ref{theo:infinitemassconv}, the same estimates  on quantities of the form $\Esub{x}{U_0(B_t)\1_{\{B_s>m(t-s) \; \forall s\in [0,t-s_0]\}}}$ as proved and used in~\cite{Bramson1983} (see Proposition~\ref{prop:frontinfiniteinitial})
will imply Theorem~\ref{theo:infmassfront}.

\subsection{Convergence to $\Pi_{\min}$} \label{subsec:inf_conv}

In this subsection, we use the same strategy as in~\cite{Bramson1983} to show that if $U(t,x+m(t))$ can be sandwiched between two suitable approximations of $\Pi_{\min}(x)$ for large $t$ and large $x$, then $U(t,x+m(t))$ must converge to $\Pi_{\min}(x)$ uniformly in $x$ as $t\to \infty$.
We begin by proving the following lemma, which corresponds to~\cite[Lemma 3.4]{Bramson1983};
we prove the result for general travelling waves $\Pi_c$ instead of just the minimal travelling wave $\Pi_{\min}$ because we will use the general result in the proof of Theorem~\ref{theo:slowerdecay} in Section~\ref{sec:mainthmpfs}.
Recall the definition of $\Pi_c$ in~\eqref{eq:Picdefn} and~\eqref{eq:speed c travelling wave}.
\begin{lem} \label{lem:Bramsonconv}
	Take $c\ge \sqrt{2}$, and suppose $U_0$ satisfies Assumption~\ref{assum:standing assumption ic} with $U_0(x)=\gamma(x)\Pi_{c}(x)$ $\forall x\in \R$, where $\gamma(x)\to 1$ as $x\to \infty$. 
	Let $(U(t,x),L_t)$ solve the free boundary problem~\eqref{eq:FBP_CDF}.
	Then
	\[
	\sup_{x\ge -\frac 18 \log t}|U(t,x+ct)-\Pi_{c}(x)|\to 0 \quad \text{as }t\to \infty.
	\]
\end{lem}
The proof of Lemma~\ref{lem:Bramsonconv} will use the following elementary lemma.
\begin{lem} \label{lem:equalonpositive}
	There exists $C_1<\infty$ such that the following holds.
	Suppose $U^1_0$ and $U^2_0$ satisfy Assumption~\ref{assum:standing assumption ic}, and let
	$(U^1(t,x),L^1_t)$ and $(U^2(t,x),L^2_t)$ solve~\eqref{eq:FBP_CDF} with initial conditions $U_0^1$ and $U_0^2$ respectively. 
	If $U^1_0(x)=U^2_0(x)$ $\forall x>0$, then for $t>0$,
	\[
	|U^2(t,x)-U^1(t,x)|\le C_1 t^{-1/4} \quad \forall x\ge \sqrt{2}t-\tfrac{1}{4\sqrt{2}}\log t.
	\]
\end{lem}
\begin{proof}
	This follows from~\cite[Lemma 3.3]{Bramson1983}, which is a simple consequence of the comparison principle for FKPP equations, applied to solutions $U^{n,1}$ and $U^{n,2}$ of~\eqref{eq:CDF_n} with initial conditions $U_0^1$ and $U_0^2$ respectively, and then letting $n\to \infty$ and using~\eqref{eq:UntoU} in Proposition~\ref{prop:Un}.
	(Note that the constant $C$ in~\cite[Lemma 3.3]{Bramson1983} does not depend on $n$.)
\end{proof}
\begin{proof}[Proof of Lemma~\ref{lem:Bramsonconv}]
The proof follows the same argument as for~\cite[Lemma 3.4]{Bramson1983}.
	For $N\in \Nm$, set 
	\[
	U^N_0(x)=\begin{cases}
		U_0(x) \quad &\text{for }x>N,\\
		\Pi_{c}(x+\delta_N) \quad &\text{for }x\le N,
	\end{cases}
	\]
	where $\delta_N\in \Rm$ is chosen so that $\Pi_{c}(N+\delta_N)=U_0(N)$.
	Then by our assumption on $U_0$, and since $\Pi_{\min}(x)=(\sqrt 2 x+1)e^{-\sqrt 2 x}$ for $x>0$
	and, for $c>\sqrt{2}$, 
	\[
	\Pi_c(x)\sim (c^{2}-2)^{-1/2}(c-\sqrt{c^{2}-2})^{-1}e^{(-c+\sqrt{c^{2}-2})x} \quad \text{as }x\to \infty,
	\]  
	we have  $\delta_N\to 0$ as $N\to \infty$. Hence
	for each $\delta>0$, there exists $N(\delta)\in \Nm$ such that for $N\ge N(\delta)$, we have
	\[
	\Pi_{c}(x+\delta)\le U^N_0(x)\le \Pi_{c}(x-\delta) \quad \forall x\in \Rm .
	\]
	Therefore, letting $(U^N(t,x),L^N_t)$ denote the solution of the free boundary problem~\eqref{eq:FBP_CDF} with initial condition $U^N_0$,
	by the comparison principle (Proposition~\ref{prop:fbpcomparison}) we have that for $N\ge N(\delta)$,
	\begin{equation} \label{eq:sandwichUN}
		\Pi_{c}(x+\delta)\le U^N(t,x+ct)\le \Pi_{c}(x-\delta) \quad  \forall t>0,\, x\in \Rm.
	\end{equation}
	By Lemma~\ref{lem:equalonpositive}, for $t>0$,
	\begin{equation} \label{eq:UNU}
		|U^N(t,x)-U(t,x)|\le C_1 t^{-1/4} \quad \forall x\ge \sqrt{2}t-\tfrac{1}{4\sqrt{2}}\log t+N.
	\end{equation}
	Take $\delta(t)\to 0$ as $t\to \infty$ sufficiently slowly that $N(\delta(t))-\frac 1{4\sqrt{2}}\log t\le -\frac 18 \log t$ for $t$ sufficiently large.
	Combining~\eqref{eq:sandwichUN} and~\eqref{eq:UNU}, for $t$ sufficiently large, for
	$x\ge -\frac 1 8\log t \ge (\sqrt{2}-c)t+N(\delta(t))-\frac 1 {4\sqrt{2}}\log t$,
	\[
	\Pi_{c}(x+\delta(t))-C_1 t^{-1/4}\le U(t,x+ct)\le \Pi_{c}(x-\delta(t))+C_1 t^{-1/4}.
	\]
	Since $\Pi_{c}$ is uniformly continuous, the result follows.
\end{proof}
We can now prove convergence of $U(t,\cdot +m(t))$ to $\Pi_{c}$ as $t\to \infty$ if condition~\eqref{eq:propsandwich} below holds; this result corresponds to~\cite[Proposition 3.3]{Bramson1983}.
\begin{prop}\label{prop:sandwichmeansconv}
	Suppose $U_0$ satisfies Assumption~\ref{assum:standing assumption ic}, and $(U(t,x),L_t)$ solves~\eqref{eq:FBP_CDF} with initial condition $U_0$.
	Suppose for some $c\ge \sqrt{2}$, $(m(t),t\ge 0)$, $\gamma_1\in C(\R)$, $\gamma_2:[0,\infty)\to [0,\infty)$ and $N<\infty$,
	\begin{equation} \label{eq:propsandwich}
		\gamma_1^{-1}(x)\Pi_{c}(x)-\gamma_2(t)\le U(t,x+m(t))\le \gamma_1(x)\Pi_{c}(x)+\gamma_2(t)
		\quad \forall t>0,\, x\ge N,
	\end{equation}
	where $\gamma_1(x)\to 1$ as $x\to \infty$ and $\gamma_2(t)\to 0$ as $t\to \infty$,
	and $\sup_{0\le s \le t<\infty}(m(s)-m(t))<\infty$.
	Then
	\[
	\sup_{x\in \R}|U(t,x+m(t))-\Pi_{c}(x)|\to 0 \quad \text{as }t\to \infty.
	\]
\end{prop}
\begin{proof}
The proof follows the same strategy as for~\cite[Proposition 3.3]{Bramson1983}.
	By increasing $N$ if necessary, we can assume that $N$ is sufficiently large that $\gamma_1(x)\Pi_{c}(x)\in (0,1)$ and $\gamma_1^{-1}(x)\Pi_{c}(x)\in (0,1)$ $\forall x\ge N$.
	Let $V^{\pm}(t,x)$ denote solutions of~\eqref{eq:FBP_CDF} with initial conditions
	\[
	V^+_0(x)=\begin{cases}
		\sup_{y\ge x}(\gamma_1(y)\Pi_{c}(y)) \quad &\text{for }x\ge N,\\
		1 \quad &\text{for }x<N,
	\end{cases}
	\]
	and
	\[
	V^-_0(x)=\begin{cases}
		\sup_{y\ge x}(\gamma_1^{-1}(y)\Pi_{c}(y)) \quad &\text{for }x\ge N,\\
		\sup_{y\ge N}(\gamma_1^{-1}(y)\Pi_{c}(y)) \quad &\text{for }x<N
	\end{cases}
	\]
	respectively.
	By Lemma~\ref{lem:Bramsonconv}, for any $\epsilon>0$, there exists $s_0=s_0(\epsilon,N)$ such that for $s\ge s_0$,
	\begin{equation} \label{eq:VpmPimin}
		|V^{\pm}(s,x+cs)-\Pi_{c}(x)|<\epsilon \quad \forall x\ge -\tfrac 18 \log t.
	\end{equation}
	But for $0<s<t$,
	by~\eqref{eq:propsandwich} and since $U(t-s,\cdot)$ is bounded by $1$ and non-increasing, we have
	\[
	U(t-s,x+m(t-s))-V^+_0(x)\le \gamma_2(t-s) \quad 
	\text{and} \quad U(t-s,x+m(t-s))-V^-_0(x)\ge -\gamma_2(t-s) \quad \forall x\in \Rm.
	\] 
	Therefore, by Lemma~\ref{lem:Bramsongronwall}, for $s>0$, for $t$ sufficiently large that $\gamma_2(t-s)<\epsilon e^{-s}$,
	\[
	U(t,x+m(t-s))-V^+(s,x)<\epsilon \quad \text{and}\quad
	U(t,x+m(t-s))-V^-(s,x)>-\epsilon \quad \forall x\in \Rm .
	\]
	Hence using~\eqref{eq:VpmPimin}, for $s\ge s_0(\epsilon,N)$, for $t$ sufficiently large (depending on $s$),
	\begin{equation} \label{eq:mt-snearPi}
		|U(t,x+m(t-s)+cs)-\Pi_{c}(x)|<2\epsilon \quad \forall x\ge -\tfrac 18 \log t.
	\end{equation}
	By~\eqref{eq:propsandwich}, for $t>0$ we have 
	\[
	|U(t,N+m(t))-\Pi_{c}(N)|\in [(\gamma_1^{-1}(N)-1)\Pi_{c}(N)-\gamma_2(t),(\gamma_1(N)-1)\Pi_{c}(N)+\gamma_2(t)].
	\]
	Therefore, for $s\ge s_0(\epsilon,N)$ fixed, for $t$ sufficiently large, applying~\eqref{eq:mt-snearPi} with $x=m(t)-m(t-s)-cs+N\ge -\frac 1 8\log t$ for $t$ sufficiently large,
	\[
	|\Pi_{c}(N)-\Pi_{c}(m(t)-m(t-s)-cs+N)|<((\gamma_1(N)-1)\vee (1-\gamma_1^{-1}(N))\Pi_{c}(N)+\gamma_2(t)+2\epsilon.
	\]
	Hence for $\delta>0$, for $N$ sufficiently large, for $\epsilon$ sufficiently small and $s\ge s_0(\epsilon,N)$, for $t$ sufficiently large we have
	\[
	|m(t)-m(t-s)-cs|<\delta.
	\]
	Therefore, by~\eqref{eq:mt-snearPi}, for $t>0$ sufficiently large, there exists $z\in [-\delta,\delta]$ such that
	\[
	|U(t,x+m(t))-\Pi_{c}(x-z)|<2\epsilon \quad \forall x\ge -\tfrac 18 \log t+\delta.
	\]
	Since $\Pi_{c}$ is uniformly continuous with $\Pi_{c}(x)\to 1$ as $x\to -\infty$, and since $U(t,\cdot)$ is non-increasing for every $t>0$, this completes the proof.
\end{proof}
In order to prove Theorem~\ref{theo:infinitemassconv}, we will prove that~\eqref{eq:propsandwich} holds with $c=\sqrt{2}$ for some $m,\gamma_1, \gamma_2$ and $N$, and then apply Proposition~\ref{prop:sandwichmeansconv} along with the following result.
\begin{lem} \label{lem:mtoL}
Suppose $U_0$ satisfies Assumption~\ref{assum:standing assumption ic}, and let $(U(t,x),L_t)$ solve~\eqref{eq:FBP_CDF}.
Suppose that for some $c\ge \sqrt{2}$ and $(m(t),t\ge 0)$, 
\[
\sup_{x\in \R}|U(t,x+m(t))-\Pi_{c}(x)|\to 0 \quad \text{as }t\to \infty.
\]
Then
\[
\sup_{x\in \R}|U(t,x+L_t)-\Pi_{c}(x)|\to 0 \quad \text{as }t\to \infty,
\]
and
 $L_t-m(t)\to 0$ as $t\to \infty$.
\end{lem}
\begin{proof}
By~\eqref{eq:speed c travelling wave}, there exists $a>0$ such that $\pi_c(y)\sim ay$ as $y\downarrow 0$, and so $1-\Pi_c(y)\sim \frac 12 ay^2$ as $y\downarrow 0$.
Therefore we can fix $\varepsilon>0$ sufficiently small that
\begin{equation} \label{eq:epschoice}
e^{2\varepsilon}(1-\varepsilon)(1-4e^{-\varepsilon^{-1/3}/4})>1
\text{ and }
e^{2\varepsilon}(\Pi_c(\varepsilon^{1/3})+\varepsilon+e^{-\varepsilon^{-1/3}/4})<1.
\end{equation}
Recall the definition of $t_0$ in Lemma~\ref{lem:boundary locally Lipschitz from the left}.
 Then take $T_\varepsilon>t_0$ sufficiently large that
\begin{equation}
\sup_{x\in \R}|U(t,x+m(t))-\Pi_{c}(x)|\le \varepsilon \quad \forall t\ge T_\varepsilon \label{eq:UPiTeps}
\end{equation}
Take $t\ge T_\varepsilon$; we will now show that $L_{t+2\varepsilon}$ and $m(t+2\varepsilon)$ are both close to $m(t)$.

First suppose, aiming for a contradiction, that $L_{t+s}< m(t)-2\varepsilon^{1/3}$ $\forall s\in [0,2\varepsilon]$.
Then by Lemma~\ref{lem:FKforinfinitemass} (the Feynman-Kac formula),~\eqref{eq:UPiTeps} and since $\Pi_{c}(x)=1$ $\forall x\le 0$, and then by the reflection principle and a Gaussian tail bound, we have
\begin{align*}
U(t+2\varepsilon,m(t)-\varepsilon^{1/3})
&\ge e^{2\varepsilon}(1-\varepsilon)\psub{m(t)-\varepsilon^{1/3}}{B_s\in [m(t)-2\varepsilon^{1/3},m(t)] \; \forall s\in [0,2\varepsilon]}\\
&\ge e^{2\varepsilon}(1-\varepsilon)(1-4e^{-\varepsilon^{2/3}/(4\varepsilon)})\\
&>1,
\end{align*}
where the last line follows from our choice of $\varepsilon$ in~\eqref{eq:epschoice}.
This gives us a contradiction, and so we must have $\sup_{s\in [0,2\varepsilon]}L_{t+s}\ge m(t)-2\varepsilon^{1/3}$.
Since we chose $T_\varepsilon>t_0$, it follows from Lemma~\ref{lem:boundary locally Lipschitz from the left} that
$L_{t+2\varepsilon}\ge m(t)-2\varepsilon^{1/3}$.

Applying Lemma~\ref{lem:FKforinfinitemass} and~\eqref{eq:UPiTeps} again, and then since $\Pi_{c}$ is non-increasing, we can write
\begin{align*}
U(t+2\varepsilon,m(t)+2\varepsilon^{1/3})
&\le e^{2\varepsilon}\Esub{2\varepsilon^{1/3}}{\left( \Pi_{c}(B_{2\varepsilon})+\varepsilon\right)\wedge 1}\\
&\le e^{2\varepsilon}\left(\Pi_c(\varepsilon^{1/3})+\varepsilon+\psub{2\varepsilon^{1/3}}{B_{2\varepsilon}<\varepsilon^{1/3}}\right)\\
&\le  e^{2\varepsilon}(\Pi_c(\varepsilon^{1/3})+\varepsilon+e^{-\varepsilon^{2/3}/(4\varepsilon)})\\
&<1,
\end{align*}
where the third inequality follows from a Gaussian tail bound, and the last inequality
by our choice of $\varepsilon$ in~\eqref{eq:epschoice}. Therefore $L_{t+2\varepsilon}\le m(t)+2\varepsilon^{1/3}$.
We now have $|L_{t+2\varepsilon}- m(t)|\le 2\varepsilon^{1/3}$.

Now fix $x_{1/2}>0$ such that $\Pi_{c}(x_{1/2})=1/2$.
Then
\[
\left|\Esub{x_{1/2}}{\Pi_{c}(B_{2\varepsilon})}-\tfrac 12 \right|\le \Esub{0}{\|\pi_{c}\|_{\infty}|B_{2\varepsilon}|}
=2\pi^{-1/2}\varepsilon^{1/2}\|\pi_{c}\|_{\infty}.
\]
Hence by Lemma~\ref{lem:FKforinfinitemass} and~\eqref{eq:UPiTeps},
\begin{align*}
U(t+2\varepsilon,m(t)+x_{1/2})
\in \left[\tfrac 12 -2\pi^{-1/2}\varepsilon^{1/2}\|\pi_{c}\|_{\infty}-\varepsilon,e^{2\varepsilon}(\tfrac 12 +2\pi^{-1/2}\varepsilon^{1/2}\|\pi_{c}\|_{\infty}+\varepsilon)\right],
\end{align*}
and so in particular there exists a constant $K_1<\infty$ such that
\[
|U(t+2\varepsilon,m(t)+x_{1/2})-\Pi_{c}(x_{1/2})|\le K_1 \varepsilon^{1/2}.
\]
Since, by~\eqref{eq:UPiTeps}, $|U(t+2\varepsilon,m(t+2\varepsilon)+x)-\Pi_{c}(x)|\le \varepsilon$ $\forall x\in \R$,
by the triangle inequality we have
\begin{equation} \label{eq:PiPidiff}
|\Pi_{c}(m(t)-m(t+2\varepsilon)+x_{1/2})-\Pi_{c}(x_{1/2})|\le K_1\varepsilon^{1/2}+\varepsilon.
\end{equation}
There exists a constant $K_2<\infty$ such that for $\varepsilon>0$ sufficiently small,~\eqref{eq:PiPidiff} implies that
\[
|m(t)-m(t+2\varepsilon)|\le K_2 \varepsilon^{1/2}.
\]

We now have that $|L_{t+2\varepsilon}-m(t+2\varepsilon)|\le 2\varepsilon^{1/3}+K_2 \varepsilon^{1/2}$ for any $t\ge T_\varepsilon$.
Hence for $t\ge T_\varepsilon$ and $x\in \R$, using the triangle inequality and then~\eqref{eq:UPiTeps},
\begin{align*}
&|U(t+2\varepsilon,x+L_{t+2\varepsilon})-\Pi_{c}(x)|\\
&\le |U(t+2\varepsilon,x+L_{t+2\varepsilon})-\Pi_{c}(x+L_{t+2\varepsilon}-m(t+2\varepsilon))|
+|\Pi_{c}(x+L_{t+2\varepsilon}-m(t+2\varepsilon))-\Pi_{c}(x)|\\
&\le \varepsilon + (2\varepsilon^{1/3}+K_2 \varepsilon^{1/2})\|\pi_{c}\|_{\infty}.
\end{align*}
Since $\varepsilon>0$ can be taken arbitrarily small, the result follows.
\end{proof}

\subsection{Trajectories ending below $\beta_r(t)$ make a small contribution} \label{subsec:inf_belowbeta}

In order to establish that~\eqref{eq:propsandwich} holds for some $m,\gamma_1,\gamma_2$ and $N$, we now study the expression for $U(t,x)$ given by the Feynman-Kac formula in Lemma~\ref{lem:FKforinfinitemass}.
Suppose $U_0$ satisfies Assumption~\ref{assum:standing assumption ic}, and let $(U(t,x),L_t)$ denote the solution of~\eqref{eq:FBP_CDF}. 
For $t>0$ and $r\in [0,t]$, let
\begin{align}
	\beta_r(t)&=\inf\{x\in \Rm: x+\sqrt 2 s \ge L_s \; \forall s\in [r,t]\} \label{eq:betadef}\\
	\text{and }\quad b_r(t)&=\inf\{s\in [r,t]:\beta_r(t)+\sqrt{2}s=L_s\}. \label{eq:brtdef}
\end{align}
These definitions are the equivalent of the definitions of $\beta_r(t)$ and $b_r(t)$ in~\cite{Bramson1983}. The quantity $\beta_r(t)$ will be
used in Section~\ref{subsec:inf_belowM} to define the curve $\underline{\mathcal M}_{r,t}$ mentioned in the proof overview at the start of Section~\ref{sec:infiniteinitialmass}, and the definition of $\beta_r(t)$ will be used to prove Lemma~\ref{lem:lrtLs}, which then implies a useful property of $\underline{\mathcal M}_{r,t}$ in Lemma~\ref{lem:Mrtestimates}.
 
We have the following bound on $\beta_r(t)$ as a consequence of Lemma~\ref{lem:Lttsqrt2}.
\begin{lem} \label{lem:betartbound}
Suppose $U_0$ satisfies Assumption~\ref{assum:standing assumption ic} and~\eqref{eq:assump_infmass1}.
Let $(U(t,x),L_t)$ denote the solution of~\eqref{eq:FBP_CDF},
and define $\beta_r(t)$ as in~\eqref{eq:betadef}.
Then there exists $(\beta(t),t\ge 0)$ such that $|\beta_r(t)|\le \beta(t)$ $\forall t\ge 0$ and $r\in [0,t]$, and
\[
\lim_{t\to \infty}\frac{\beta(t)}t=0.
\]
\end{lem}
\begin{proof}
For $\varepsilon>0$, suppose for some $s_\varepsilon\ge 0$ that $L_s\in [(\sqrt 2-\varepsilon)s,(\sqrt 2+\varepsilon)s]$ $\forall s\ge s_\varepsilon$.
Then for $0\le r \le t$ with $t\ge s_\varepsilon$ we have
$-\varepsilon t\le \beta_r(t)\le \varepsilon t+\sup_{s\in [0,s_\varepsilon]}L_s.$
Hence $|\beta_r(t)|\le 2\varepsilon t$ $\forall t\ge s_\varepsilon \vee (\varepsilon^{-1} \sup_{s\in [0,s_\varepsilon]}L_s)$ and $r\in [0,t]$.
The result then follows directly from Lemma~\ref{lem:Lttsqrt2}.
\end{proof}
In the following result, which corresponds to~\cite[Proposition~7.6]{Bramson1983}, we show that for large $r$, for $t\ge r$ and $x\ge L_t$, trajectories $(B_s)_{s\in [0,t]}$ with $B_t\le \beta_r(t)$ only make a small contribution to the expectation in Lemma~\ref{lem:FKforinfinitemass}.
\begin{prop} \label{prop:belowbeta}
Suppose $U_0(x)$ satisfies Assumption~\ref{assum:standing assumption ic}, and let $(U(t,x),L_t)$ denote the solution of~\eqref{eq:FBP_CDF}.
Define $\beta_r(t)$ as in~\eqref{eq:betadef}.
	As $r\to \infty$,
	\[
	\frac{\Esub{x}{e^{\Leb(\{s\le t:B_s\ge L_{t-s}\})}U_0(B_t)\1_{\{B_t\le \beta_r(t)\}}}}{U(t,x)}\to 0
	\]
	uniformly in $t\ge r$ and $x\ge L_t$.
\end{prop}
\begin{proof}
The proof follows a simplified version of the strategy in the proof of~\cite[Proposition~7.6]{Bramson1983}.
	Take $0\le r \le t$.
	Let $(V(s,x),L^V_s)$ denote the solution of~\eqref{eq:FBP_CDF} with initial condition
	\[
	V_0(x)=\begin{cases}
		U_0(x) \quad &\text{for }x< \beta_r(t),\\
		0 \quad &\text{for }x\ge \beta_r(t).
	\end{cases}
	\] 
	Recall from Section~\ref{subsec:notation} that we write 
$(U^H(s,x),L^H_s)$ for the solution of~\eqref{eq:FBP_CDF} with Heaviside initial condition.
	By the comparison principle we have $V(s,y)\le U(s,y)$ $\forall s>0,$ $y\in \Rm$, and so
	$L^V_s\le L_s$ $\forall s>0$.
	Hence for any $x\in \Rm$,
	\begin{align*}
		\Esub{x}{e^{\Leb(\{s\le t:B_s\ge L_{t-s}\})}U_0(B_t)\1_{\{B_t\le \beta_r(t)\}}}
		&\le \Esub{x}{e^{\Leb(\{s\le t:B_s\ge L^V_{t-s}\})}U_0(B_t)\1_{\{B_t\le \beta_r(t)\}}}\\
		&= \Esub{x}{e^{\Leb(\{s\le t:B_s\ge L^V_{t-s}\})}V_0(B_t)}\\
		&= V(t,x)\\
		&\le U^H(t,x-\beta_r(t)),
	\end{align*}
	where the third line follows from Lemma~\ref{lem:FKforinfinitemass}, and the last inequality follows from the comparison principle.
	
	Therefore it suffices to prove that
	\begin{equation} \label{eq:UHUlimitclaim}
		\frac{U^H(t,x-\beta_r(t))}{U(t,x)}\to 0 \quad \text{as }r\to \infty
	\end{equation}
	uniformly in $t\ge r$ and $x\ge L_t$.
	
	By Theorem~\ref{theo:convergence to the minimal travelling wave for finite initial mass}, there exists $K_1<\infty$ such that for $s\ge 1$,
	\begin{equation} \label{eq:Lheateqnbd}
		L^H_s\le \sqrt{2} s-\frac 3 {2\sqrt 2}\log s+K_1.
	\end{equation}
	Now take $1\le r \le t$.
	By the definition of $b_r(t)$ in~\eqref{eq:brtdef}, we have
	\[
	L_{b_r(t)}=\sqrt{2}b_r(t)+\beta_r(t).
	\]
	Hence by~\eqref{eq:Lheateqnbd} with $s=b_r(t)\ge r\ge 1$, 
	\begin{equation} \label{eq:Lbrt}
		L_{b_r(t)}-(L^H_{b_r(t)}+\beta_r(t))\ge \frac 3 {2\sqrt 2}\log b_r(t)-K_1\ge \frac 3 {2\sqrt 2}\log r-K_1,
	\end{equation}
	i.e. if $r$ is large then at time $b_r(t)$, the free boundary $L_\cdot$ is far ahead of $L^H_{\cdot}+\beta_r(t)$.
	
	We now show that an analogue of inequality~\eqref{eq:Lbrt} holds at time $t$.
	Indeed, by the stretching lemma (Lemma~\ref{lem:extended maximum principle}),
	we have $U(b_r(t),\cdot)\ge_s U^H(b_r(t),\cdot)$, and so by Lemma~\ref{lem:less stretching means slower boundary},
	\[
	L_t-L_{b_r(t)}\ge L^H_t-L^H_{b_r(t)}.
	\]
	It follows that
	\begin{equation} \label{eq:LtLHt}
		L_t-(L^H_t+\beta_r(t))\ge L_{b_r(t)}-(L^H_{b_r(t)}+\beta_r(t))\ge \frac 3 {2\sqrt 2}\log r-K_1,
	\end{equation}
	where the second inequality follows from~\eqref{eq:Lbrt}.
	
	We can now establish~\eqref{eq:UHUlimitclaim} and complete the proof.
	Suppose $r\ge 1$ is sufficiently large that $\frac 3 {2\sqrt 2}\log r-K_1>0$.
	Take $t\ge r$ and $x\ge L_t$, and let $y=x-L_t\ge 0$.
	Then we can write
	\begin{equation} \label{eq:twofrac}
		\frac{U^H(t,x-\beta_r(t))}{U(t,x)}
		=\frac{U^H(t,y+L_t-\beta_r(t))}{U^H(t,y+L^H_t)}\cdot \frac{U^H(t,y+L^H_t)}{U(t,y+L_t)}.
	\end{equation}
	Now note that by Lemma~\ref{lem:extended maximum principle} we have $U^H(t,\cdot)\le_s U(t,\cdot)$, and so by Lemma~\ref{lem:stretchdecr} and since $y\ge 0$ we have
	\begin{equation} \label{eq:UHUy}
	U^H(t,y+L^H_t)\le U(t,y+L_t).
	\end{equation}
It remains to bound the first fraction on the right-hand side of~\eqref{eq:twofrac}.
By Lemma~\ref{lem:extended maximum principle} we have $U^H(t,\cdot)\le_s \Pi_{\min}$.
Therefore, setting $c_y:= U^H(t,y+L^H_t)\le 1$, by Lemma~\ref{lem:stretchdecr} we have	
\[
U^H(t,y+L^H_t+z)\le \Pi_{\min}(z+\Pi^{-1}_{\min}(c_y)) \; \forall z\ge 0.
\]
Since $L_t-(L^H_t+\beta_r(t))>0$ by~\eqref{eq:LtLHt} and our choice of $r$, it follows that
\begin{equation} \label{eq:Ufirstfrac}
\frac{U^H(t,y+L_t-\beta_r(t))}{U^H(t,y+L^H_t)}\le \frac{\Pi_{\min}(L_t-(L^H_t+\beta_r(t))+\Pi_{\min}^{-1}(c_y))}{\Pi_{\min}(\Pi_{\min}^{-1}(c_y))}.
\end{equation}
For $z_1,z_2\ge 0$, we have
\begin{equation} \label{eq:Pitail}
\frac{\Pi_{\min}(z_1+z_2)}{\Pi_{\min}(z_1)}=\frac{\sqrt{2}(z_1+z_2)+1}{\sqrt{2}z_1+1}e^{-\sqrt 2 z_2}\le (1+\sqrt 2 z_2)e^{-\sqrt 2 z_2}\le K_2 e^{-z_2}
\end{equation}
for some constant $K_2<\infty$. Combining~\eqref{eq:twofrac},~\eqref{eq:UHUy},~\eqref{eq:Ufirstfrac} and~\eqref{eq:Pitail}
with $z_1=\Pi_{\min}^{-1}(c_y)$ and $z_2=L_t-(L^H_t+\beta_r(t))$, and then using~\eqref{eq:LtLHt}, it follows that
\[
\frac{U^H(t,x-\beta_r(t))}{U(t,x)}\le K_2 e^{-(L_t-(L^H_t+\beta_r(t)))}
\le K_2 e^{-\frac 3 {2\sqrt 2}\log r+K_1}.
\]
The claim~\eqref{eq:UHUlimitclaim} follows, which completes the proof.
\end{proof}

\subsection{Trajectories going below $\underline{\mathcal M}_{r,t}(t-\cdot)$ make a small contribution} \label{subsec:inf_belowM}

Suppose $U_0$ satisfies Assumption~\ref{assum:standing assumption ic}, and let $(U(t,x),L_t)$ denote the solution of~\eqref{eq:FBP_CDF}. 
Recall the definition of $\beta_r(t)$ in~\eqref{eq:betadef}.
The following definitions are analogues of definitions in~\cite{Bramson1983}.
For $t>0$ and $r,s\in [0,t]$, let
\begin{equation} \label{eq:lrt_defn}
	\ell_{r,t}(s)=\frac s t L_t +\frac{t-s}t \beta_r(t).
\end{equation}
For $t> 0$ and $r,s\in [0,t]$, let
\begin{equation} \label{eq:Krtdefn}
G_{r,t}(s)=L_s-\frac st L_t-\frac{t-s}t \beta_r(t)=L_s-\ell_{r,t}(s).
\end{equation}
For $t\ge 4$ and $r\in [2,t/2]$,
define $\theta_{r,t}$ as follows: for a function $G:[0,t]\to \Rm$,
let 
\[
\theta_{r,t}\circ G(s)=
\begin{cases}
	G(s+s^{1/3})+4 s^{1/3} \quad &\text{for }s\in [r,t/2],\\
	G(s+(t-s)^{1/3})+4 (t-s)^{1/3} \quad &\text{for } s\in [t/2, t-2r],\\
	G(s)\quad &\text{for }s\in [0,r)\cup (t-2r,t].
\end{cases}
\]
Then for $G:[0,t]\to \Rm$, let 
\begin{equation} \label{eq:theta-defn}
\theta^{-1}_{r,t}\circ G=\inf\{\ell:[0,t]\to \Rm \; \text{ s.t. }\theta_{r,t}\circ \ell \ge G\}.
\end{equation}
The right-hand side of~\eqref{eq:theta-defn} is the pointwise infimum of the set of functions $\ell$ that satisfy $\theta_{r,t}\circ \ell \ge G$ pointwise.
The following more explicit expression for $\theta_{r,t}^{-1}\circ G$ will be useful later.
Define $h_{t}:[2,t]\to \Rm$ such that
\begin{equation}\label{eq:htdefn}
h_{t}(s+s^{1/3})=s \text{ for }s\in [2,t/2] \quad  \text{ and } \quad h_{t}(s+(t-s)^{1/3})=s \text{ for }s\in [t/2,t-2].
\end{equation}
Then for $G:[0,t]\to \Rm$,
\begin{equation} \label{eq:theta-expr}
\theta_{r,t}^{-1}\circ G(s)=
\begin{cases}
	-\infty \quad &\text{for }s\in [r,r+r^{1/3}),\\
	G(h_{t}(s))-4(h_{t}(s))^{1/3} \quad &\text{for } s\in [r+r^{1/3},t/2+(t/2)^{1/3}],\\
	G(h_{t}(s))-4(t-h_{t}(s))^{1/3} \quad &\text{for } s\in [t/2+(t/2)^{1/3},t-2r],\\
	G(s)\vee (G(h_{t}(s))-4(t-h_{t}(s))^{1/3}) \quad &\text{for }s\in (t-2r,t-2r+(2r)^{1/3}],\\
	G(s) \quad &\text{for }s\in [0,r)\cup (t-2r+(2r)^{1/3} ,t].
\end{cases}
\end{equation}
We can now define the curve $\underline{\mathcal M}_{r,t}$ mentioned in the proof overview at the start of Section~\ref{sec:infiniteinitialmass}.
For $t\ge 4$, $r\in [2,t/2]$ and $s\in [0,t]$, let 
\begin{equation} \label{eq:Munderdefn}
	\underline{\mathcal M}_{r,t}(s)=\theta^{-1}_{r,t}\circ G_{r,t}(s)+\frac st L_t +\frac{t-s}t \beta_r(t)=\theta^{-1}_{r,t}\circ G_{r,t}(s)+\ell_{r,t}(s).
\end{equation}
We will see that the curve $\underline{\mathcal M}_{r,t}$ is chosen in such a way that if a Brownian bridge $\xi(\cdot)$ hits $\underline{\mathcal M}_{r,t}(t-\cdot)$ at some time $s$, then (roughly speaking) it is likely that $\xi(\cdot)$ stays below $L_{t-\cdot}$ for a time interval of length $\frac 12 (s\wedge (t-s))^{1/3}$. 

Recall from Lemma~\ref{lem:FKforinfinitemass} that for $t>0$ and $x\in \R$, we can write
\begin{equation} \label{eq:UtxBBformula}
U(t,x)=\int_{-\infty}^\infty U_0(y) \frac{e^{-\frac 1 {2t}(x-y)^2}}{\sqrt{2\pi t}}\Es{e^{\Leb(\{s\in [0,t]:\xi^t_{x,y}(s)\ge L_{t-s}\})}}dy,
\end{equation}
where, as in Section~\ref{subsec:notation}, we let $(\xi^t_{x,y}(s),0\le s \le t)$ denote a Brownian bridge from $x$ at time 0 to $y$ at time $t$.
In this subsection, we will show (in Proposition~\ref{prop:Gxy} below) that for large $r$,
for $t$ sufficiently large and $x\ge L_t$, $y\ge \beta_r(t)$ not too large,
the contribution to~\eqref{eq:UtxBBformula} from trajectories $(\xi^t_{x,y}(s),0\le s \le t)$ with $\xi^t_{x,y}(s)\le \underline{\mathcal M}_{r,t}(t-s)$ for some $s\in [3r,t-3r]$ is small.

Moreover, we will see in Lemma~\ref{lem:probaboveinterval} below in Section~\ref{subsec:inf_asympt} that by an entropic repulsion result of Bramson~\cite{Bramson1983}, for large $r$, the probability that $(\xi^t_{x,y}(s),0\le s \le t)$ stays above $\underline{\mathcal M}_{r,t}(t-s)$ for all $s\in [3r,t-3r]$ is close to the probability of staying above $L_{t-s}$ on the whole time interval $[0,t]$.

In summary, the heuristic to have in mind is that $\underline{\mathcal M}_{r,t}$ is chosen to be far enough below $L$ that trajectories going below $\underline{\mathcal M}_{r,t}(t-\cdot)$ make a small contribution to $U(t,x)$, but close enough to $L$ that the probability of staying above $\underline{\mathcal M}_{r,t}(t-\cdot)$ is close to the probability of staying above $L_{t-\cdot}$.

We begin by proving two simple deterministic lemmas concerning $\ell_{r,t}$ and $\underline{\mathcal M}_{r,t}$.
First, we show that the line $\ell_{r,t}$ does not get too far below the free boundary $L$; this result corresponds to~\cite[Lemma 9.1]{Bramson1983}.
\begin{lem} \label{lem:lrtLs}
	There exists $C_2<\infty$ such that the following holds.
	Suppose $U_0$ satisfies Assumption~\ref{assum:standing assumption ic}, and let $(U(t,x),L_t)$ denote the solution of~\eqref{eq:FBP_CDF}.
	For $t>0$ and $r\in [0,t]$, defining $\ell_{r,t}$ as in~\eqref{eq:lrt_defn},
	\[
	\ell_{r,t}(s)-L_s \ge -\frac 3{2\sqrt 2} \log ((s\wedge (t-s))+1)-C_2 \quad \forall s\in  [r,t].
	\]
\end{lem}
\begin{proof}
	The proof follows the same argument as the proof of~\cite[Lemma 9.1]{Bramson1983}.
	By the definitions of $\beta_r(t)$ and $b_r(t)$ in~\eqref{eq:betadef} and~\eqref{eq:brtdef}, we have
	$L_s\le \beta_r(t)+\sqrt{2}s$ $\forall s\in [r,t]$ and
	$\beta_r(t)=L_{b_r(t)}-\sqrt{2}b_r(t)$. Therefore, by~\eqref{eq:lrt_defn}, for $s\in [r,t]$,
	\begin{equation} \label{eq:lrts-Ls}
		\ell_{r,t}(s)-L_s\ge \frac st L_t-\frac st \beta_r(t)-\sqrt{2}s=
		\frac st \left((L_t-L_{b_r(t)})-\sqrt{2}(t-b_r(t))\right).
	\end{equation}
	Recall from Section~\ref{subsec:notation} that we let $(U^H(s,x),L_s^H)$ denote the solution of~\eqref{eq:FBP_CDF} with Heaviside initial condition.
	By Lemma~\ref{lem:LHtlower},
	we have
	\begin{equation} \label{eq:LHlower}
		L^H_s\ge \sqrt{2}s-\frac{3}{2\sqrt{2}}\log (s+1)-C_0 \quad \forall s\ge 0.
	\end{equation}
	For $s\in [0,t]$, by Lemma~\ref{lem:less stretching means slower boundary}
	 and then by~\eqref{eq:LHlower}, 
	\begin{equation} \label{eq:LtLslower}
		L_t-L_s\ge L^H_{t-s}\ge \sqrt{2}(t-s)-\frac{3}{2\sqrt{2}}\log (t-s+1)-C_0.
	\end{equation}
	Substituting this into~\eqref{eq:lrts-Ls}, for $s\in [r,t]$ we have
	\begin{equation} \label{eq:ellLlower1}
		\ell_{r,t}(s)-L_s\ge -\frac{3}{2\sqrt{2}}\frac st \log (t-b_r(t)+1)-C_0\ge -\frac{3}{2\sqrt{2}}\frac st \log (t+1)-C_0\ge -\frac{3}{2\sqrt{2}}\log (s+1)-C_0,
	\end{equation}
	where in the last inequality we used that $\log$ is concave. 
	
	Now note that  for $s\in [r,t]$,
	by~\eqref{eq:lrt_defn} and then by~\eqref{eq:LtLslower} and since $\beta_r(t)-L_t\ge -\sqrt{2}t$ by~\eqref{eq:betadef},
	\[
		\ell_{r,t}(s)-L_s=\frac{t-s}t(\beta_r(t)-L_t)+L_t-L_s
		\ge -\frac{3}{2\sqrt{2}}\log (t-s+1)-C_0.
	\]
	Combining this with~\eqref{eq:ellLlower1} completes the proof.
\end{proof}
The following bounds on $\underline{\mathcal M}_{r,t}$ will be used in the proofs later in this subsection.
For large $r$, for $s$ not too close to $0$ or $t$, we show that $\underline{\mathcal M}_{r,t}(s)$ is below $\ell_{r,t}(s)$ and $L_s$, and (roughly speaking) we show that $\underline{\mathcal M}_{r,t}(s)$ is at least distance $(s\wedge (t-s))^{1/3}+\mathcal O(1)$ below $L_{s'}$ if $|s-s'|\le \frac 1{50} (s\wedge (t-s))^{1/3}+\mathcal O(1)$. We also give a lower bound on $\underline{\mathcal M}_{r,t}(s)$ that will be used in the proof of Lemma~\ref{lem:Bt_ends_big} in Section~\ref{subsec:inf_asympt}.
\begin{lem} \label{lem:Mrtestimates}
Suppose $U_0$ satisfies Assumption~\ref{assum:standing assumption ic} and~\eqref{eq:assump_infmass1}, and let $(U(t,x),L_t)$ denote the solution of~\eqref{eq:FBP_CDF}.
For $r\ge 2$ sufficiently large and $t\ge 6r$,	defining $\underline{\mathcal M}_{r,t}$ as in~\eqref{eq:Munderdefn} and $\ell_{r,t}$ as in~\eqref{eq:lrt_defn},
\begin{align} 
\underline{\mathcal M}_{r,t}(t-s) &\le \ell_{r,t}(t-s) \quad \forall s\in [2r,t-r] \label{eq:Mrtest1a} \\
\text{and }\quad \underline{\mathcal M}_{r,t}(t-s) &< L_{t-s} \qquad \quad \forall s\in [3r,t-3r]. \label{eq:Mrtest1b} 
\end{align}
Moreover, for $j\in [\lfloor 3r \rfloor ,t/2]$ and $s^*\in [j,j+1]$,
\begin{equation} \label{eq:Mrtest2}
\underline{\mathcal M}_{r,t}(t-s^*)+j^{1/3}<L_{t-s} \quad \forall s\in [j-\tfrac 1{50} j^{1/3},j],
\end{equation}
and for $s^*\in [t-j-1,t-j]$,
\begin{equation} \label{eq:Mrtest3}
\underline{\mathcal M}_{r,t}(t-s^*)+j^{1/3}<L_{t-s} \quad \forall s\in [t-j,t-j+\tfrac 1{50} j^{1/3}].
\end{equation}
For $r\ge 2$ and $\varepsilon>0$, there exists $R>0$ such that for $t$ sufficiently large, defining $\underline{\mathcal M}_{r,t}$ as in~\eqref{eq:Munderdefn},
\begin{equation} \label{eq:Mrtest4}
\underline{\mathcal M}_{r,t}(s)\ge L_{s-s^{1/3}-1} -\tfrac 12 s^{1/3+\varepsilon} \quad \forall s\in [R,t-R].
\end{equation}
\end{lem}
\begin{proof}
For $s\in [0,t]$, write $s_{\wedge}:=s\wedge (t-s)$.
We begin by proving~\eqref{eq:Mrtest1a}.
By the definition of $\underline{\mathcal M}_{r,t}$ in~\eqref{eq:Munderdefn}, and using~\eqref{eq:theta-expr} and~\eqref{eq:htdefn}, and then by~\eqref{eq:Krtdefn}, for $s\in [r,t-2r]$ with $s+s_\wedge^{1/3}\le t-2r$,
	\[
	\underline{\mathcal M}_{r,t}(s+s_\wedge^{1/3} )-\ell_{r,t}(s+s_\wedge^{1/3})=G_{r,t}(s)-4 s_\wedge^{1/3}=L_s-\ell_{r,t}(s)-4 s_\wedge^{1/3},
	\]
	and for $s\in [r,r+r^{1/3})$, $\underline{\mathcal M}_{r,t}(s)=-\infty$.
	For $s'\in [2r,t-r]$, we have either $t-s'\in [r,r+r^{1/3})$ or there exists $s\in [r,t-2r]$ such that $s+s^{1/3}_\wedge=t-s'$.
	It follows from Lemma~\ref{lem:lrtLs} that if $r$ is sufficiently large then~\eqref{eq:Mrtest1a} holds.
	
We now prove the remaining bounds~\eqref{eq:Mrtest1b}-\eqref{eq:Mrtest4}.
By~\eqref{eq:Munderdefn} and~\eqref{eq:theta-expr},
	for $s\in [r+r^{1/3},t-2r]$,
	\begin{equation} \label{eq:Mhrtexpr}
	\underline{\mathcal M}_{r,t}(s)
	=L_{h_{t}(s)}-\ell_{r,t}(h_{t}(s))+\ell_{r,t}(s)-4 h_{t}(s)_\wedge^{1/3}.
	\end{equation}
	Recall the definition of $h_{t}$ before~\eqref{eq:theta-expr}.
We now claim that
\begin{equation} \label{eq:hrtclaim}
|h_{t}(u)-(u-u_\wedge ^{1/3})|=|h_{t}(u)_{\wedge}^{1/3}-u_{\wedge}^{1/3}|\le r^{-1/3} \quad \forall u\in [r+r^{1/3},t-r].
\end{equation}
Indeed, 
for $u\in [r+r^{1/3},t-r]$, 
we have $u=s+s_\wedge^{1/3}$ for some $s\in [r,t-r]$, and so
by the definition of $h_{t}$ in~\eqref{eq:htdefn} we have $h_t(u)=s$ and therefore $h_{t}(u)+h_{t}(u)_{\wedge}^{1/3}=u$.
Hence, using in the first inequality below that $|a^{1/3}-b^{1/3}|\le b^{-2/3}|a-b|$ $\forall a,b\ge 0$, and in the second inequality that $|a_\wedge-b_\wedge|\le |a-b|$ $\forall a,b\in [0,t]$,
\begin{align*} 
|h_{t}(u)-(u-u_{\wedge}^{1/3})|
=|u-h_{t}(u)_{\wedge}^{1/3}-u+u_{\wedge}^{1/3}|
&=|h_{t}(u)_{\wedge}^{1/3}-u_{\wedge}^{1/3}| \notag \\
&\le h_{t}(u)_{\wedge}^{-2/3}|h_{t}(u)_{\wedge}-u_{\wedge}| \notag \\
&\le h_{t}(u)_{\wedge}^{-2/3}|h_{t}(u)-u| \notag \\
&=h_{t}(u)_{\wedge}^{-2/3}\cdot h_{t}(u)_{\wedge}^{1/3},
\end{align*}
and the claim~\eqref{eq:hrtclaim} follows since for any $u\in [r+r^{1/3},t-r]$ we have $h_{t}(u)\in [r,t-r]$.
	
	Now suppose, using Lemma~\ref{lem:Lttsqrt2} and Lemma~\ref{lem:betartbound}, that $t$ is sufficiently large that for any $r\in [2,t]$ we have $|L_t|+|\beta_r(t)|\le 2t$.
	To bound the right-hand side of~\eqref{eq:Mhrtexpr}, for $s\in [r+r^{1/3},t-2r]$, using~\eqref{eq:lrt_defn} we write
	\begin{align*} 
	-\ell_{r,t}(h_{t}(s))+\ell_{r,t}(s)-4 h_{t}(s)_\wedge^{1/3} 
	&=(s-h_{t}(s))t^{-1}(L_t-\beta_r(t))-4 h_{t}(s)_\wedge^{1/3} \notag \\
	&\le 2|s-h_{t}(s)|-4 h_{t}(s)_\wedge^{1/3} \notag \\
	&\le 6r^{-1/3}-2 s_\wedge^{1/3},
	\end{align*}
	where the last inequality follows from~\eqref{eq:hrtclaim} applied to each term.
	Similarly, we can write
	\[
	-\ell_{r,t}(h_{t}(s))+\ell_{r,t}(s)-4 h_{t}(s)_\wedge^{1/3} 
	\ge - 2|s-h_{t}(s)|-4 h_{t}(s)_\wedge^{1/3} 
	\ge -6r^{-1/3}-6 s_\wedge^{1/3},
	\]
	again by applying~\eqref{eq:hrtclaim} to each term for the second inequality.
	Hence by~\eqref{eq:Mhrtexpr} and~\eqref{eq:hrtclaim}, we have that for $t$ sufficiently large,
	for $s\in [r+r^{1/3},t-2r]$,
	\begin{equation} \label{eq:Mboundfromclaim}
	\inf_{u\in [s-s_\wedge^{1/3}-1,s-s_\wedge^{1/3}+1]}L_u-6r^{-1/3}-6 s_\wedge^{1/3} \le \underline{\mathcal M}_{r,t}(s)
	\le \sup_{u\in [s-s_\wedge^{1/3}-1,s-s_\wedge^{1/3}+1]}L_u+6r^{-1/3}-2 s_\wedge^{1/3}.
	\end{equation}
	By Lemma~\ref{lem:boundary locally Lipschitz from the left}, there exists $t_0>0$ such that $u\mapsto L_u$ is non-decreasing on $[t_0,\infty)$.
	Therefore,~\eqref{eq:Mrtest1b},~\eqref{eq:Mrtest2} and~\eqref{eq:Mrtest3} 
	follow directly from the second inequality in~\eqref{eq:Mboundfromclaim} by taking $r$ sufficiently large.
	Moreover,~\eqref{eq:Mrtest4} follows directly from the first inequality in~\eqref{eq:Mboundfromclaim} by taking $R$ sufficiently large.
\end{proof}
Recall from Section~\ref{subsec:notation} that for $x\in \R$, we write $\mathbb P_x$ for the probability measure under which $(B_t)_{t\ge 0}$ is a Brownian motion started at $x$, and for $t\ge 0$ and  $x,y\in \Rm$, we write $(\xi^t_{x,y}(s),0\le s \le t)$ for a Brownian bridge from $x$ at time 0 to $y$ at time $t$.
We now state some standard properties of the Brownian bridge $\xi^t_{x,y}$, which we will use in the proofs later in this section.
\begin{lem} \label{lem:BBfacts}
For $t>0$,
\begin{enumerate}[(a)]
\item $(\xi^t_{0,0}(s), \, 0\le s \le t)$ under $\mathbb P$ has the same law as $(\frac{t-s}tB_{st/(t-s)} , 0\le s \le t)$ under $\mathbb P_0$.
\item for $x,y>0$, 
\[
\p{\xi^t_{x,y}(s)>0 \; \forall s\in [0,t]}=1 -e^{-2xy/t}.
\]
\item for $x>0$ and $s_0\in (0,t)$,
\[
\p{\xi^t_{0,0}(s)>-x \; \forall s\in [0,s_0]}>1-2s_0^{1/2}x^{-1}e^{-x^2/(2s_0)}.
\]
\item for $\ell_1,\ell_2, \ell:[0,t]\to \R\cup \{-\infty\}$ bounded above with $\p{\xi^t_{x,y}(s)>\ell_2(s) \; \forall s\in [0,t]}>0$ and $\ell_1(s)\le \ell_2(s)$ $\forall s\in [0,t]$,
\begin{align*}
&\p{\left. \xi^t_{x,y}(s)>\ell(s) \; \forall s\in [0,t] \right| \xi^t_{x,y}(s)>\ell_1(s) \; \forall s\in [0,t]}\\
&\quad \le 
\p{\left. \xi^t_{x,y}(s)>\ell(s) \; \forall s\in [0,t] \right| \xi^t_{x,y}(s)>\ell_2(s) \; \forall s\in [0,t]}.
\end{align*}
\item for $\ell_1,\ell_2:[0,t]\to \Rm$ upper semi-continuous except at finitely many points with $\ell_1(s)\le \ell_2(s)$ $\forall s\in [0,t]$, if the denominator is non-zero then
	\[
	\frac{\p{\xi^t_{x,y}(s)>\ell_2(s)\; \forall s\in [0,t]}}{\p{\xi^t_{x,y}(s)>\ell_1(s)\; \forall s\in [0,t]}}
	\]
	is increasing in $x$ and $y$.
\end{enumerate}
\end{lem}
\begin{proof}
Property (a) is stated in e.g.~\cite[Lemma 2.1(b)]{Bramson1983}, and can be easily checked by calculating the covariance.
Properties (b) and (c) are proved in e.g.~\cite[Lemma 2.2(a) and (b)]{Bramson1983}.
Property (d) is proved in e.g.~\cite[Lemma 2.6]{Bramson1983} (note that the condition on $\Lambda$ is not needed in the proof).
Property (e) is proved in e.g.~\cite[Proposition 6.3]{Bramson1983}.
\end{proof}
We now begin to show that for large $r$,
for large $t$ and suitable $x,y$,
the contribution to~\eqref{eq:UtxBBformula} from trajectories $(\xi^t_{x,y}(s),0\le s \le t)$ with $\xi^t_{x,y}(s)\le \underline{\mathcal M}_{r,t}(t-s)$ for some $s\in [3r,t-3r]$ is small.
We will need to introduce some notation;
the following notation is the equivalent of notation in~\cite{Bramson1983}.
For $0\le r \le t$ and $x,y\in \R$, let
\begin{align}
	S^{(1)}_{r,t,x,y}&:=\sup(\{0\}\cup \{s\in [2r,t/2]: \xi^t_{x,y}(s)\le \underline{\mathcal M}_{r,t}(t-s)\}) \label{eq:S1defn}\\
	\text{and }\quad S^{(2)}_{r,t,x,y}&:=\inf(\{t\}\cup \{s\in [t/2,t-r]: \xi^t_{x,y}(s)\le \underline{\mathcal M}_{r,t}(t-s)\}). \label{eq:S2defn}
\end{align}
Then let
\[S_{r,t,x,y}:=
\begin{cases}
	S^{(1)}_{r,t,x,y} \quad &\text{if }S^{(1)}_{r,t,x,y}>t-S^{(2)}_{r,t,x,y}\\
	S^{(2)}_{r,t,x,y} \quad &\text{otherwise,}
\end{cases}
\]
so that $S_{r,t,x,y}$ is the time nearest $t/2$ when $\xi^t_{x,y}(\cdot)$ is below $\underline {\mathcal M}_{r,t}(t-\cdot)$.

For $u\in [r,t/2]$, define the event
\begin{align} \label{eq:Gxtdefn}
	E_{r,t,x,y}(u)
	:=\{S_{r,t,x,y}\in [u,t-u]\}
	=\{\exists s\in [u\vee (2r),t-u] \text{ s.t. }\xi^t_{x,y}(s)\le \underline{\mathcal M}_{r,t}(t-s)\}.
\end{align}
Let $E_{r,t,x,y}:=E_{r,t,x,y}(r)$.
Define the time intervals
\[
I_j:= [j,j+1)\cup (t-j-1,t-j] \; \forall j\in \{0,1,\ldots , \lfloor t/2\rfloor -1\} \quad \text{ and }\quad I_{\lfloor t/2\rfloor}:=[\lfloor t/2\rfloor,t-\lfloor t/2\rfloor]
\]
For $j\in \{0,1,\ldots , \lfloor t/2\rfloor\}$, define the event $A_{r,t,x,y}(j):=\{S_{r,t,x,y}\in I_j\},$ and let
\begin{align} \label{eq:Aj1Aj2defn}
	A^{+}_{r,t,x,y}(j)&=A_{r,t,x,y}(j)\cap \{\xi^t_{x,y}(S_{r,t,x,y})>-(S_{r,t,x,y}\wedge (t-S_{r,t,x,y}))+\tfrac {S_{r,t,x,y}}t y +\tfrac{t-S_{r,t,x,y}}t x \} \notag \\
	\text{and }\quad A^{-}_{r,t,x,y}(j)&:=A_{r,t,x,y}(j)\setminus A^{+}_{r,t,x,y}(j).
\end{align}
Later in this subsection, in Lemma~\ref{lem:e-LebA1j} below, we will establish an upper bound on the expectation in the integrand of~\eqref{eq:UtxBBformula} on the event $A^+_{r,t,x,y}(j)$ in terms of the probability of the event $A^+_{r,t,x,y}(j)$. We now give an upper bound on the probability of the event $(E_{r,t,x,y}(u))^c$; since $A^+_{r,t,x,y}(j)\subseteq (E_{r,t,x,y}(j+1))^c$ for $j\le \lfloor t/2 \rfloor -1$, we will be able to combine this lemma with Lemma~\ref{lem:e-LebA1j} to prove the main result of this subsection (Proposition~\ref{prop:Gxy}).
This result corresponds to~\cite[Lemma~7.1]{Bramson1983}.
\begin{lem} \label{lem:Gcxyr}
	There exists $C_3\in (0,\infty)$ such that the following holds.
	Suppose $U_0$ satisfies Assumption~\ref{assum:standing assumption ic} and~\eqref{eq:assump_infmass1}, and let $(U(t,x),L_t)$ denote the solution of~\eqref{eq:FBP_CDF}.
	Define $\beta_r(t)$ as in~\eqref{eq:betadef}.
For $r\ge 2$ sufficiently large and $t\ge 6r$,
	for $u\in [2r,t/2]$, $y\ge \beta_r(t)$ and $x\ge L_t$,
	\begin{align}
		\p{(E_{r,t,x,y}(u))^c }&\le C_3 \frac{u}r \p{(E_{r,t,x,y})^c} \label{eq:Ertu1}\\ 
		\text{and }\quad \p{(E_{r,t,x,y})^c}&\ge 2C_3^{-1}\frac r t. \label{eq:Ert2}
	\end{align}
\end{lem}
\begin{proof}
The proof is similar to the proof of~\cite[Lemma~7.1]{Bramson1983}.
Take $r$ sufficiently large that Lemma~\ref{lem:Mrtestimates} holds.
	By~\eqref{eq:lrt_defn}, and then by~\eqref{eq:Mrtest1a} in Lemma~\ref{lem:Mrtestimates}, for $y\ge \beta_r(t)$ and $x\ge L_t$,
	\begin{equation} \label{eq:Munderlinebelow}
		\underline{\mathcal M}_{r,t}(t-s)-\frac{t-s}t x-\frac st y\le \underline{\mathcal M}_{r,t}(t-s)-\ell_{r,t}(t-s)\le 0 \;\; \forall s\in [2r,t-r].
	\end{equation}
For $u\in [2r,t/2]$, by the definition of the event $E_{r,t,x,y}(u)$ in~\eqref{eq:Gxtdefn} and by Lemma~\ref{lem:BBfacts}(d) with $\ell_1=-\infty$,
and then by~\eqref{eq:Munderlinebelow},
	\begin{align} \label{eq:Gcxylower}
		\p{\left. (E_{r,t,x,y})^c \, \right| (E_{r,t,x,y}(u))^c}
		&\ge \p{\xi^t_{x,y}(s)>\underline{\mathcal M}_{r,t}(t-s) \; \forall s\in [2r,u)\cup (t-u,t-r]} \notag \\
		&\ge \p{\xi^t_{0,0}(s)>0 \; \forall s\in [2r,u)\cup (t-u,t-r]} \notag \\
		&\ge \p{\xi^t_{0,0}(s)>0 \; \forall s\in [r,u]}^2,
	\end{align}
	where the last inequality follows by Lemma~\ref{lem:BBfacts}(d) again.
	Then by Lemma~\ref{lem:BBfacts}(a), we can write
	\begin{align} \label{eq:BBBM}
		\p{\xi^t_{0,0}(s)>0 \; \forall s\in [r,u]}
		&\ge \psub{0}{\left\{ B_{rt/(t-r)}>r^{1/2}\right\}
		\cap \left\{\tfrac{t-s}t B_{st/(t-s)}>0\; \forall s\in [r,u] \right\}} \notag \\
		&\ge \psub{0}{B_{rt/(t-r)}>r^{1/2}}
		\psub{r^{1/2}}{B_{s'}>0\; \forall s'\in [0,2u] },
	\end{align}
	where the second inequality follows since $u\le t/2$ and so $ut/(t-u)-rt/(t-r)\le 2u$.
	By the reflection principle,
	\[
	\psub{r^{1/2}}{B_{s}>0\; \forall s\in [0,2u] }
	=1-2\psub{0}{B_{2u}\ge r^{1/2}}=\psub{0}{|B_{2u}|\le r^{1/2}}
	\ge \frac{2r^{1/2}}{\sqrt{2\pi\cdot 2u}}e^{-r/4u}.
	\]
	Since $r/(4u)\le 1/8$ and $rt/(t-r)\ge r$, it follows from~\eqref{eq:BBBM} that there exists a constant $K_1\in (0,\infty)$ such that
	\begin{equation} \label{eq:BBabovelowerbd}
	\p{\xi^t_{0,0}(s)>0 \; \forall s\in [r,u]}\ge \frac 1 {K_1} \sqrt{\frac{r}{u}}.
	\end{equation}
	The statement~\eqref{eq:Ertu1} now follows from~\eqref{eq:Gcxylower} and~\eqref{eq:BBabovelowerbd}, since $(E_{r,t,x,y})^c\subseteq (E_{r,t,x,y}(u))^c$.
	The second statement~\eqref{eq:Ert2} also follows from~\eqref{eq:BBabovelowerbd}, since
	\[
	\p{ (E_{r,t,x,y})^c }
	\ge \p{\xi^t_{0,0}(s)>0 \; \forall s\in [r,t/2]}^2,
	\]
	by the same argument as for~\eqref{eq:Gcxylower}.
\end{proof}
We now give an upper bound on the expectation in the integrand of the expression for $U(t,x)$ in~\eqref{eq:UtxBBformula} 
on the event $A^{+}_{r,t,x,y}(j)$, for large $r$, $t$ and $j$.
This result corresponds to~\cite[Lemma~7.2]{Bramson1983}.
\begin{lem} \label{lem:e-LebA1j}
Suppose $U_0$ satisfies Assumption~\ref{assum:standing assumption ic} and~\eqref{eq:assump_infmass1}, and let $(U(t,x),L_t)$ denote the solution of~\eqref{eq:FBP_CDF}.
	Define $\beta_r(t)$ as in~\eqref{eq:betadef}.
For $r\ge 2$ sufficiently large and $t\ge 6r$,
for $j\in \Nm\cap [\lfloor 3r \rfloor,t/2]$,
	$y\in [\beta_r(t),\beta_r(t)+20t]$ and $x\in [L_t,L_t+20t]$,
	\[
	\Es{e^{-\Leb(\{s\in [0,t]:\xi^t_{x,y}(s)<L_{t-s}\})}\1_{A^{+}_{r,t,x,y}(j)}}\le 2e^{-j^{1/3} /50}\p{A^{+}_{r,t,x,y}(j)}.
	\]
\end{lem}
\begin{proof}
The proof is similar to the proof of~\cite[Lemma~7.2]{Bramson1983}.
Take $r$ sufficiently large that Lemma~\ref{lem:Mrtestimates} holds.

	Recall the definition of $S^{(1)}_{r,t,x,y}$ in~\eqref{eq:S1defn}. Take $j\in \Nm\cap [\lfloor 3r \rfloor,t/2]$, and define the event
	\[
	D^{(1)}_j:= \{\xi^t_{x,y}(s)-\xi^t_{x,y}(S^{(1)}_{r,t,x,y})\le j^{1/3} \; \forall s\in [j-\tfrac 1 {50}j^{1/3} ,S^{(1)}_{r,t,x,y}]\}.
	\]
	By~\eqref{eq:Aj1Aj2defn}, on the event $A^{+}_{r,t,x,y}(j)\cap \{S_{r,t,x,y}<t/2\}$ we have
	$S^{(1)}_{r,t,x,y}\in [j,j+1)$.
	Therefore, on the event $A^{+}_{r,t,x,y}(j)\cap \{S_{r,t,x,y}<t/2\}\cap D^{(1)}_j$, for $s\in [j-\frac 1 {50}j^{1/3} ,j]$ we have
	\[
	\xi^t_{x,y}(s)\le j^{1/3} +\underline{\mathcal M}_{r,t}(t-S^{(1)}_{r,t,x,y})
	 <L_{t-s},
	\]
	where the last inequality follows from~\eqref{eq:Mrtest2} in Lemma~\ref{lem:Mrtestimates}.
	Hence on the event $A^{+}_{r,t,x,y}(j)\cap \{S_{r,t,x,y}<t/2\}\cap D^{(1)}_j$,
	\[
	\Leb(\{s\in [0,t]:\xi^t_{x,y}(s)<L_{t-s}\})\ge \tfrac 1 {50} j^{1/3}.
	\]
	Therefore 
	\begin{align} \label{eq:ELebA1}
		&\Es{e^{-\Leb(\{s\in [0,t]:\xi^t_{x,y}(s)<L_{t-s}\})}\1_{A^{+}_{r,t,x,y}(j)}\1_{\{S_{r,t,x,y}<t/2\}}} \notag \\
		&\le \p{A^{+}_{r,t,x,y}(j)\cap \{S_{r,t,x,y}<t/2\}}
		\left(e^{-\frac 1 {50}j^{1/3}}+\p{ (D^{(1)}_j)^c \left|A^{+}_{r,t,x,y}(j)\cap \{S_{r,t,x,y}<t/2\}\right.}\right).
	\end{align}
	We now bound the last term on the right-hand side of~\eqref{eq:ELebA1}.
	On the event $A^{+}_{r,t,x,y}(j)\cap \{S_{r,t,x,y}<t/2\}$, letting $S^{(1)}=S^{(1)}_{r,t,x,y}$, by~\eqref{eq:Aj1Aj2defn} and then by our choice of $x$ and $y$ we have
	\begin{equation} \label{eq:xiS1bound}
		\xi^t_{x,y}(S^{(1)})-x
		> -S^{(1)}+\frac{S^{(1)}}t (y-x)
		\ge -S^{(1)} +\frac{S^{(1)}}t (\beta_r(t)-L_t-20t)
		\ge -23 S^{(1)},
	\end{equation}
	where the last inequality follows
	for $t$ sufficiently large by Lemma~\ref{lem:Lttsqrt2} and Lemma~\ref{lem:betartbound} (using that $U_0$ satisfies~\eqref{eq:assump_infmass1}).
	Moreover, $\frac 12 t-S^{(1)}_{r,t,x,y}$ is a stopping time for the process $(\xi^t_{x,y}(t-s),s\in [0,t])$, and so, writing $S^{(1)}=S^{(1)}_{r,t,x,y}$ again, conditional on
	$(S^{(1)},\xi^t_{x,y}(S^{(1)}))$, the process $(\xi^t_{x,y}(S^{(1)}-s),s\in [0,S^{(1)}])$ is a Brownian bridge from $\xi^t_{x,y}(S^{(1)})$ to $x$ in time $S^{(1)}$.
	Therefore, using~\eqref{eq:xiS1bound}, 
	\begin{align} \label{eq:Djcbound}
		\p{ (D^{(1)}_j)^c \left|A^{+}_{r,t,x,y}(j)\cap \{S_{r,t,x,y}<t/2\}\right.}
		&\le \sup_{t'\ge j, \, y'>x-23t'}\p{\exists s\in [0,\tfrac 1{50}j^{1/3} +1] : \xi^{t'}_{y',x}(s)>y'+j^{1/3}} \notag \\
		&\le \sup_{t'\ge j}\p{\exists s\in [0,\tfrac 1{50}j^{1/3} +1] : \xi^{t'}_{0,0}(s)>-23s+ j^{1/3} } \notag \\
		&\le \sup_{t'\ge j}\p{\exists s\in [0,\tfrac 1{50}j^{1/3} +1] : \xi^{t'}_{0,0}(s)>\tfrac 12 j^{1/3} },
		\end{align}
		where the last inequality holds for $r$ (and hence $j$) sufficiently large.
	Then by Lemma~\ref{lem:BBfacts}(c), for $t'\ge j$ we have
		\begin{equation} \label{eq:Djcbound2}
		\p{\exists s\in [0,\tfrac 1{50}j^{1/3} +1] : \xi^{t'}_{0,0}(s)>\tfrac 12 j^{1/3} } 
		\le \frac{2(\frac 1 {50}j^{1/3} +1)^{1/2}}{\frac 12 j^{1/3}}e^{-\frac{\frac 14 j^{2/3}}{2(\frac 1 {50}j^{1/3} +1)}} 
		\le e^{-j^{1/3}},
	\end{equation}
	where the second inequality holds for $r$ (and hence $j$) sufficiently large.
	By combining~\eqref{eq:ELebA1},~\eqref{eq:Djcbound} and~\eqref{eq:Djcbound2},
	we now have that 
	\begin{equation} \label{eq:LebAj1S1}
		\Es{e^{-\Leb(\{s\in [0,t]:\xi^t_{x,y}(s)<L_{t-s}\})}\1_{A^{+}_{r,t,x,y}(j)}\1_{\{S_{r,t,x,y}<t/2\}}}
		\le 
		2e^{-\frac 1 {50}j^{1/3}}\p{A^{+}_{r,t,x,y}(j)\cap \{S_{r,t,x,y}<t/2\}}.
	\end{equation}
	
	We now prove the analogue of~\eqref{eq:LebAj1S1} with the event $\{S_{r,t,x,y}\ge t/2\}$ in place of $\{S_{r,t,x,y}<t/2\}$.
	Recall the definition of $S^{(2)}_{r,t,x,y}$ in~\eqref{eq:S2defn}.
	Take $j\in \Nm\cap [\lfloor 3r \rfloor,t/2]$, and define the event
	\[
	D^{(2)}_j:= \{\xi^t_{x,y}(s)-\xi^t_{x,y}(S^{(2)}_{r,t,x,y})\le j^{1/3} \; \forall s\in [S^{(2)}_{r,t,x,y},t-j+\tfrac 1 {50}j^{1/3} ]\}.
	\]
	By~\eqref{eq:Aj1Aj2defn}, on the event $A^{+}_{r,t,x,y}(j)\cap \{S_{r,t,x,y}\ge t/2\}$, we have
	$S^{(2)}_{r,t,x,y}\in (t-j-1,t-j]$.
	Hence on the event $A^{+}_{r,t,x,y}(j)\cap \{S_{r,t,x,y}\ge t/2\}\cap D^{(2)}_j$, for $s\in [t-j,t-j+\tfrac 1 {50}j^{1/3} ]$,
	\[
	\xi^t_{x,y}(s)\le j^{1/3} +\underline{\mathcal M}_{r,t}(t-S^{(2)}_{r,t,x,y}) 
	<L_{t-s},
	\]
	where the last inequality follows from~\eqref{eq:Mrtest3} in Lemma~\ref{lem:Mrtestimates}.
	Therefore, by the same argument as for~\eqref{eq:ELebA1},
	\begin{align} \label{eq:ELebA12}
		&\Es{e^{-\Leb(\{s\in [0,t]:\xi^t_{x,y}(s)<L_{t-s}\})}\1_{A^{+}_{r,t,x,y}(j)}\1_{\{S_{r,t,x,y}\ge t/2\}}} \notag \\
		&\le \p{A^{+}_{r,t,x,y}(j)\cap \{S_{r,t,x,y}\ge t/2\}}
		\left(e^{-\frac 1 {50}j^{1/3}}+\p{ (D^{(2)}_j)^c \left|A^{+}_{r,t,x,y}(j)\cap \{S_{r,t,x,y}\ge t/2\}\right.}\right).
	\end{align}
	On the event $A^{+}_{r,t,x,y}(j)\cap \{S_{r,t,x,y}\ge t/2\}$, letting $S^{(2)}=S^{(2)}_{r,t,x,y}$, by~\eqref{eq:Aj1Aj2defn} and then by our choice of $x$ and $y$,
	\begin{align*}
		\xi^t_{x,y}(S^{(2)})-y> -(t-S^{(2)})+\frac{t-S^{(2)}}t (x-y)
		&\ge -(t-S^{(2)}) +\frac{t-S^{(2)}}t (L_t-\beta_r(t)-20t)\\
		&\ge -23 (t-S^{(2)}),
	\end{align*}
	where the last inequality follows for $t$ sufficiently large by Lemma~\ref{lem:Lttsqrt2} and Lemma~\ref{lem:betartbound}.
Moreover, $S^{(2)}_{r,t,x,y}-\frac 12 t$ is a stopping time for the process $(\xi^t_{x,y}(s),s\in [0,t])$.	
	Therefore, using the same argument as in~\eqref{eq:Djcbound} and~\eqref{eq:Djcbound2}, for $r$ (and hence $j$) sufficiently large,
	\begin{align*}
		\p{ (D^{(2)}_j)^c \left|A^{+}_{r,t,x,y}(j)\cap \{S_{r,t,x,y}\ge t/2\}\right.}
		&\le  e^{-j^{1/3}}.
	\end{align*}
	By~\eqref{eq:ELebA12}, it follows that
	\begin{align*}
		\Es{e^{-\Leb(\{s\in [0,t]:\xi^t_{x,y}(s)<L_{t-s}\})}\1_{A^{+}_{r,t,x,y}(j)}\1_{\{S_{r,t,x,y}\ge t/2\}}}
		&\le 
		2e^{-\frac 1 {50}j^{1/3}}\p{A^{+}_{r,t,x,y}(j)\cap \{S_{r,t,x,y}\ge t/2\}},
	\end{align*}
	and by~\eqref{eq:LebAj1S1} this completes the proof.
\end{proof}
We can now use Lemmas~\ref{lem:e-LebA1j} and~\ref{lem:Gcxyr} to give an upper bound on the contribution to the expression for $U(t,x)$ in~\eqref{eq:UtxBBformula}, for large $r$ and large $t$, from trajectories $(\xi^t_{x,y}(s),0\le s \le t)$ such that $\xi^t_{x,y}(s)\le \underline{\mathcal M}_{r,t}(t-s)$ for some $s\in [3r,t-3r]$ (with $x\ge L_t$, $y\ge \beta_r(t)$ not too large).
This result corresponds to~\cite[Proposition~7.3]{Bramson1983}.
\begin{prop} \label{prop:Gxy}
Suppose $U_0$ satisfies Assumption~\ref{assum:standing assumption ic} and~\eqref{eq:assump_infmass1}, and let $(U(t,x),L_t)$ denote the solution of~\eqref{eq:FBP_CDF}.
	Define $\beta_r(t)$ as in~\eqref{eq:betadef}.
For $r\ge 2$ sufficiently large and $t\ge 6r$,
for $y\in [\beta_r(t),\beta_r(t)+20t]$ and $x\in [L_t,L_t+20t]$,
	\[
	\Es{e^{-\Leb(\{s\in [0,t]:\xi^t_{x,y}(s)<L_{t-s}\})}\1_{E_{r,t,x,y}(3r)}}\le r^{-2}\p{(E_{r,t,x,y})^c}.
	\]
\end{prop}
\begin{proof}
Take $r$ sufficiently large that Lemmas~\ref{lem:Gcxyr} and~\ref{lem:e-LebA1j} hold.
	By the definitions of the events $E_{r,t,x,y}(3r)$, $A^+_{r,t,x,y}(j)$ and $A^-_{r,t,x,y}(j)$ in~\eqref{eq:Gxtdefn} and~\eqref{eq:Aj1Aj2defn}, we can write
	\begin{align} \label{eq:sumoverj}
		&\Es{e^{-\Leb(\{s\in [0,t]:\xi^t_{x,y}(s)<L_{t-s}\})}\1_{E_{r,t,x,y}(3r)}} \notag \\
		&\le \sum_{j=\lfloor 3r \rfloor }^{\lfloor t/2\rfloor }\Es{e^{-\Leb(\{s\in [0,t]:\xi^t_{x,y}(s)<L_{t-s}\})}\1_{A^+_{r,t,x,y}(j)}}
		+\sum_{j=\lfloor 3r \rfloor }^{\lfloor t/2 \rfloor }\p{A^-_{r,t,x,y}(j)}.
	\end{align}
	For the first sum on the right-hand side of~\eqref{eq:sumoverj}, by Lemma~\ref{lem:e-LebA1j}, 
	\begin{align} \label{eq:sumA+}
		\sum_{j=\lfloor 3r \rfloor }^{\lfloor t/2\rfloor }\Es{e^{-\Leb(\{s\in [0,t]:\xi^t_{x,y}(s)<L_{t-s}\})}\1_{A^+_{r,t,x,y}(j)}}
		&\le \sum_{j=\lfloor 3r \rfloor }^{\lfloor t/2\rfloor}2e^{-\frac 1 {50}j^{1/3}}\p{A^+_{r,t,x,y}(j)} \notag \\
		&\le \sum_{j=\lfloor 3r \rfloor }^{\lfloor t/2\rfloor -1}2e^{-\frac 1 {50}j^{1/3}}\p{(E_{r,t,x,y}(j+1))^c}
		+2e^{-\frac 1 {50}\lfloor t/2\rfloor^{1/3}},
	\end{align}
	where the second inequality follows from the fact that, by~\eqref{eq:Aj1Aj2defn} and~\eqref{eq:Gxtdefn}, for $j\le \lfloor t/2 \rfloor-1$ we have $A^+_{r,t,x,y}(j)\subseteq A_{r,t,x,y}(j)\subseteq (E_{r,t,x,y}(j+1))^c$.
	For the second sum on the right hand side of~\eqref{eq:sumoverj}, by~\eqref{eq:Aj1Aj2defn} and~\eqref{eq:Gxtdefn} again, for $\lfloor 3r \rfloor \le j \le \lfloor t/2\rfloor-1$ we have
	\begin{align} \label{eq:A-bound}
		\p{A^-_{r,t,x,y}(j)}
		&\le \p{\{\exists s\in I_j : \xi^t_{x,y}(s)\le -(s\wedge (t-s))+\tfrac s t y +\tfrac{t-s}t x \}\cap  (E_{r,t,x,y}(j+1))^c} \notag \\
		&\le \p{\exists s\in I_j :\xi^t_{0,0}(s)\le -(s\wedge (t-s))}
		\p{ (E_{r,t,x,y}(j+1))^c},
	\end{align}
	where the last inequality follows from Lemma~\ref{lem:BBfacts}(d).
	By~\eqref{eq:Aj1Aj2defn} we also have
	\begin{equation} \label{eq:A-t2bound}
	\p{A^-_{r,t,x,y}(\lfloor t/2 \rfloor)}
	\le \p{\exists s\in I_{\lfloor t/2 \rfloor} : \xi^t_{0,0}(s)\le -(s\wedge (t-s)) }.
	\end{equation}
	By Lemma~\ref{lem:BBfacts}(a),
	for $\lfloor 3r \rfloor \le j\le \lfloor t/2\rfloor $ we can write
	\begin{align*}
		\p{\exists s\in [j,(j+1)\wedge \tfrac t2] :\xi^t_{0,0}(s)\le -(s\wedge (t-s))}
		&\le \psub{0}{\exists s\in [j,(j+1)\wedge \tfrac t2] : B_{st/(t-s)}\le -j }\\
		&\le \psub{0}{\exists s\in [j,2(j+1)]: B_s\le -j  }\\
		&\le \psub{0}{B_j\le -j/2}+\psub{0}{\sup_{s\le j+2}B_s\ge j/2}\\
		&\le e^{-j/8}+2e^{-j^2/8(j+2)},
	\end{align*}
	where the last inequality follows by the reflection principle and a Gaussian tail bound.
	Hence using the same argument for $[(t-j-1)\vee \tfrac t2,t-j]$ in place of $[j,(j+1)\wedge \tfrac t2]$, for $r$ sufficiently large, for 
	$\lfloor 3r \rfloor \le j\le \lfloor t/2\rfloor $ we have
	\begin{equation} \label{eq:probA-Ij}
	\p{\exists s\in I_j : \xi^t_{0,0}(s)\le -(s\wedge (t-s)) }\le 6e^{-j/10}.
	\end{equation}
	Therefore, by~\eqref{eq:sumoverj}-\eqref{eq:probA-Ij},
	and then using Lemma~\ref{lem:Gcxyr} in the second inequality,
	\begin{align*}
		&\Es{e^{-\Leb(\{s\in [0,t]:\xi^t_{x,y}(s)<L_{t-s}\})}\1_{E_{r,t,x,y}(3r)}}\\
		&\le \sum_{j=\lfloor 3r \rfloor }^{\lfloor t/2 \rfloor-1}(2e^{-\frac 1 {50}j^{1/3}}+6e^{-j/10})\p{(E_{r,t,x,y}(j+1))^c}
		+2e^{-\frac 1 {50}\lfloor t/2 \rfloor^{1/3} }+6e^{-\lfloor t/2 \rfloor/10}\\
		&\le \sum_{j=\lfloor 3r \rfloor }^{\lfloor t/2 \rfloor-1}(2e^{-\frac 1 {50}j^{1/3}}+6e^{-j/10})C_3 \frac{j+1}r \p{(E_{r,t,x,y})^c}
		+(2e^{-\frac 1 {50}\lfloor t/2 \rfloor^{1/3} }+6e^{-\lfloor t/2 \rfloor/10})\frac{\p{(E_{r,t,x,y})^c}}{2C_3^{-1}rt^{-1}}
		\\
		&\le r^{-2}\p{(E_{r,t,x,y})^c}
	\end{align*}
	for $r$ sufficiently large (and $t\ge 6r$), which completes the proof.
\end{proof}

\subsection{Asymptotic formula for $U(t,x)$} \label{subsec:inf_asympt}

In this subsection, we will use results from Sections~\ref{subsec:inf_belowbeta} and~\ref{subsec:inf_belowM}
to prove an upper bound on $U(t,x)$ (see Proposition~\ref{prop:upperbdinfmass} below). This upper bound will then be used to prove Corollary~\ref{cor:x2x1}, which will in turn allow us to control the tail of $U(t,\cdot)$ in the proof of Theorem~\ref{theo:infinitemassconv} in Section~\ref{subsec:inf_thm1pf}.
We will then establish an asymptotic formula for $U(t,x)$, for $t$ and $x-L_t$ large (in Proposition~\ref{prop:s0}), which will be used in the proof of Theorem~\ref{theo:infmassfront} in Section~\ref{subsec:inf_thm2pf}.

The following lemma will allow us to show that we can neglect the contribution to the expression for $U(t,x)$ in Lemma~\ref{lem:FKforinfinitemass} from trajectories $(B_s)_{s\in [0,t]}$ with $B_t=\mathcal O(1)$. The infinite initial mass condition~\eqref{eq:assump_infmass2} is crucial here; see~\cite[Proof of Lemma 9.2]{Bramson1983}.
\begin{lem} \label{lem:Bt_ends_big}
Suppose $U_0$ satisfies Assumption~\ref{assum:standing assumption ic},~\eqref{eq:assump_infmass1} and~\eqref{eq:assump_infmass2}. 
	For some $r_0\in (0,\infty)$ and $\delta \in (0,1/2)$, suppose for each $t\ge 2r_0$ that $\ell_t:[0,t]\to \R$
	 with $\ell_t(s)\ge \sqrt{2}s-s^{\delta }$ $\forall s\in [r_0,t-r_0]$.
	Then for $r_1,r_2\ge 0$, there exists $M(t)\to \infty$ as $t\to \infty$ such that
	\begin{equation} \label{eq:Bt_ends_big}
		\frac{\Esub{x}{U_0(B_t)\1_{\{B_t\le M(t), B_s>\ell_t(t-s)\; \forall s\in [r_1,t-r_2]\}}}}{\Esub{x}{U_0(B_t)\1_{\{B_s>\ell_t(t-s)\; \forall s\in [r_1,t-r_2]\}}}}\to 0 \quad \text{as }t\to \infty
	\end{equation}
	uniformly in $x\ge \sqrt{2}t-t^{\delta} +1,$ if the denominator in~\eqref{eq:Bt_ends_big} is positive.
	
	Let $(U(t,x),L_t)$ denote the solution of~\eqref{eq:FBP_CDF}, and define $\underline{\mathcal M}_{r,t}$ as in~\eqref{eq:Munderdefn}.
	In particular, for any $r\ge 2$, the limit in~\eqref{eq:Bt_ends_big} with $r_1,r_2>0$ and $\ell_t(s)=\underline{\mathcal M}_{r,t}(s)$ holds uniformly in $x\ge L_t$, and the limit in~\eqref{eq:Bt_ends_big} with $\ell_t(s)=L_s$ holds uniformly in $x\ge L_t+1$.
\end{lem}
\begin{proof}
	The main statement is identical to~\cite[Lemma 9.2]{Bramson1983}, where the lower bound on $\ell_t$ is used to show that it suffices to prove similar estimates for a Brownian motion staying above a straight line, and then the
	 `infinite initial mass' condition~\eqref{eq:assump_infmass2} implies~\eqref{eq:Bt_ends_big}. Therefore, to complete the proof, it suffices to show that for $r\ge 2$, there exist $r_0\in (0,\infty)$ and $\delta \in (0,1/2)$ such that for $t$ sufficiently large,
	$\underline{\mathcal M}_{r,t}(s)\ge \sqrt{2}s-s^{\delta }$ $\forall s\in [r_0,t-r_0]$ and $L_s\ge \sqrt{2}s-s^{\delta}+1$ $\forall s\ge r_0$.
	
	Fix $\delta \in (1/3,1/2)$. 	
	Recall from Section~\ref{subsec:notation} that we let
	$(U^H(t,x),L^H_t)$ denote the solution of~\eqref{eq:FBP_CDF} with Heaviside initial condition. By Lemma~\ref{lem:less stretching means slower boundary} and then by Lemma~\ref{lem:LHtlower}, for $s\ge 1$ we have
	\begin{equation} \label{eq:Lslowerbddelta}
	L_s\ge L^H_{s-1}+L_1\ge \sqrt{2}s-\tfrac 13 s^{\delta}+1
	\end{equation}
	for $s$ sufficiently large.
	By~\eqref{eq:Mrtest4} in Lemma~\ref{lem:Mrtestimates}, there exists $R\in (0,\infty)$ such that for $t$ sufficiently large, for $s\in [R,t-R]$,
	\[
	\underline{\mathcal M}_{r,t}(s)\ge L_{s-s^{1/3}-1}-\tfrac 12 s^\delta \ge  \sqrt{2} s -s^{\delta} ,
	\]
	where the second inequality follows from~\eqref{eq:Lslowerbddelta} by taking $R$ sufficiently large.
\end{proof}
We can now combine Proposition~\ref{prop:belowbeta}, Proposition~\ref{prop:Gxy} and Lemma~\ref{lem:Bt_ends_big} to prove the following upper bound on the expression for $U(t,x)$ in~\eqref{eq:UtxBBformula}.
 This result corresponds to~\cite[Proposition~9.1]{Bramson1983}.
\begin{prop} \label{prop:upperbdinfmass}
Suppose $U_0$ satisfies Assumption~\ref{assum:standing assumption ic},~\eqref{eq:assump_infmass1} and~\eqref{eq:assump_infmass2}, and let $(U(t,x),L_t)$ denote the solution of~\eqref{eq:FBP_CDF}.
Define $\beta_r(t)$ as in~\eqref{eq:betadef} and define $\underline{\mathcal M}_{r,t}$ as in~\eqref{eq:Munderdefn}.
	There exists $A^{(1)}_r\to 1$ as $r\to \infty$, and there exists $N_r(t)\ge \beta_r(t)$ for $r\ge 2$ and $t\ge r$, with $N_r(t)\to \infty$ as $t\to \infty$ for each $r\ge 2$, such that the following holds.
	For $r$ sufficiently large, for $t\ge 6r$ and $x\in [L_t,L_t+7t]$,
	\begin{align*}
		U(t,x)
		&\le A^{(1)}_r e^t \int_{N_r(t)}^\infty U_0(y) \frac {e^{-\frac 1 {2t}(x-y)^2}} {\sqrt{2\pi t}}  \p{\xi^t_{x,y}(s)>\underline {\mathcal M}_{r,t}(t-s)\; \forall s\in [3r,t-3r]} dy.
	\end{align*}
\end{prop}
\begin{proof}
	Using Lemma~\ref{lem:Lttsqrt2} and Lemma~\ref{lem:betartbound}, take $r$ sufficiently large that for $t\ge 6r$ we have $0\le L_t\le 2t$ and $|\beta_r(t)|\le t$.
	Since $U_0$ is non-increasing with $U_0(x)\to 1$ as $x\to -\infty$, there exists $m_0\in \R$ such that $U_0(y)\ge 1/2$ $\forall y\le m_0$.
	Therefore, by the Feynman-Kac formula in Lemma~\ref{lem:FKforinfinitemass}, for $t\ge 6r$ and
	$0\le x\le L_t+7t\le 9t$,
	\begin{align} \label{eq:Ucrudelower}
		U(t,x)
		\ge \frac 12 \int_{-\infty}^{m_0}  \frac{e^{-\frac 1 {2t}(x-y)^2}}{\sqrt{2\pi t}}dy 
		\ge \frac 12 \frac{e^{-\frac 1 {2t}(x+|m_{0}|+1)^2}}{\sqrt{2\pi t}} 
		\ge e^{-41t}.
	\end{align}
	where the second inequality holds since $|x-y|\le x+|m_{0}|+1$ $\forall y\in [m_{0}-1,m_{0}]$, and the last inequality holds for $t$ sufficiently large, since $0\le x \le 9t$.
	Moreover,
	for $0\le x\le L_t+7t\le 9t$, we have 
	\begin{equation} \label{eq:integralcrudeupper}
		\int_{\beta_r(t)+20t}^\infty U_0(y)e^t \frac{e^{-\frac 1 {2t}(x-y)^2}}{\sqrt{2\pi t}}dy
		\le e^t \int_{19t}^\infty  \frac{e^{-\frac 1 {2t}(x-y)^2}}{\sqrt{2\pi t}}dy 
		\le e^t e^{-\frac 1 {2t}(19t-x)^2} 
		\le e^{-49t} 
		\le e^{-8t}U(t,x),
	\end{equation}
	where the second inequality follows by a Gaussian tail bound, and the last inequality follows from~\eqref{eq:Ucrudelower} for $t$ sufficiently large.
	Therefore,
	by Lemma~\ref{lem:FKforinfinitemass},~\eqref{eq:integralcrudeupper} and Proposition~\ref{prop:belowbeta}, there exists $K^{(1)}_r\to 1$ as $r\to \infty$ such that
	for $r$ sufficiently large, for $t\ge 6r$ and $x\in [L_t,L_t+7t]$, 
	\begin{align*}
		U(t,x)&\le K^{(1)}_r\int_{\beta_r(t)}^{\beta_r(t)+20t} U_0(y) \frac{e^{-\frac 1 {2t}(x-y)^2}}{\sqrt{2\pi t}}\Es{e^{t-\Leb(\{s\in [0,t]:\xi^t_{x,y}(s)<L_{t-s}\})}}dy.
	\end{align*}
	By Proposition~\ref{prop:Gxy} and since $(E_{r,t,x,y})^c\subseteq (E_{r,t,x,y}(3r))^c$ by~\eqref{eq:Gxtdefn}, for $r$ sufficiently large, for $t\ge 6r$, $y\in [\beta_r(t),\beta_r(t)+20t]$ and $x\in [L_t,L_t+20t]$ we have
	\[
	\Es{e^{-\Leb(\{s\in [0,t]:\xi^t_{x,y}(s)<L_{t-s}\})}\1_{E_{r,t,x,y}(3r)}}\le r^{-2}\p{(E_{r,t,x,y}(3r))^c}.
	\]
	Therefore, for $r$ sufficiently large, for $t\ge 6r$ and $x\in [L_t,L_t+7t]$ we have
	\begin{align} \label{eq:Uupperbetter}
		U(t,x)&\le K^{(1)}_r \int_{\beta_r(t)}^{\beta_r(t)+20t} U_0(y) \frac{e^{-\frac 1 {2t}(x-y)^2}}{\sqrt{2\pi t}}e^t 
		\left(r^{-2}\p{(E_{r,t,x,y}(3r))^c}+\p{(E_{r,t,x,y}(3r))^c}\right)dy \notag \\
		&= K^{(1)}_r(1+r^{-2})
		\int_{\beta_r(t)}^{\beta_r(t)+20t} U_0(y) \frac{e^{-\frac 1 {2t}(x-y)^2}}{\sqrt{2\pi t}}e^t 
		\p{\xi^t_{x,y}(s)>\underline{\mathcal M}_{r,t}(t-s) \; \forall s\in [3r,t-3r]}dy \notag \\
		&\le K^{(1)}_r(1+r^{-2})
		\Esub{x}{U_0(B_t)e^t\1_{\{B_s>\underline{\mathcal M}_{r,t}(t-s) \; \forall s\in [3r,t-3r]\}}\1_{\{B_t>\beta_r(t)\}}},
	\end{align}
	where the second line follows from the definition of the event $E_{r,t,x,y}(3r)$ in~\eqref{eq:Gxtdefn}.
	
	By Lemma~\ref{lem:Bt_ends_big}, for fixed $r\ge 2$, there exist $K^{(2)}_r(t)\to 1$ as $t\to \infty$ and $M_r(t)\to \infty$ as $t\to \infty$ such that 
	for $t\ge 6r$ and $x\ge L_t$,
	\begin{equation*}
		\Esub{x}{U_0(B_t)\1_{\{B_s>\underline{\mathcal M}_{r,t}(t-s)\; \forall s\in [3r,t-3r]\}}}
		\le K^{(2)}_r(t)\Esub{x}{U_0(B_t)\1_{\{B_t>M_r(t)\}}\1_{\{B_s>\underline{\mathcal M}_{r,t}(t-s)\; \forall s\in [3r,t-3r]\}}}.
	\end{equation*}
	Therefore, for $r$ sufficiently large, for $t\ge 6r$
	and $x\in [L_t,L_t+7t]$,
	\begin{equation} \label{eq:UupperwithM}
		U(t,x)\le K^{(1)}_r(1+r^{-2})K^{(2)}_r(t)\Esub{x}{U_0(B_t)e^t\1_{\{B_s>\underline{\mathcal M}_{r,t}(t-s) \; \forall s\in [3r,t-3r]\}}\1_{\{B_t>M_r(t)\}}}.
	\end{equation}
	For $2\le r\le t$, we now let
	\[
	N_r(t)=\begin{cases} M_r(t) \quad &\text{if }K^{(2)}_r(t)\le 1+r^{-1}\text{ and }M_r(t)\ge \beta_r(t),\\
		\beta_r(t) \quad &\text{otherwise}.
	\end{cases}
	\]
	Note that $N_r(t)\ge \beta_r(t)$ for every $2\le r\le t$, and for fixed $r\ge 2$, since $K^{(2)}_r(t)\le 1+r^{-1}$ for $t$ sufficiently large, we have that for $t$ sufficiently large, $N_r(t)\ge M_r(t)$ and so $N_r(t)\to \infty$ as $t\to \infty$.
	Let $K^{(3)}_r=K^{(1)}_r(1+r^{-2})(1+r^{-1})$.
	Then using~\eqref{eq:UupperwithM} in the case $N_r(t)=M_r(t)$, in which case $K^{(2)}_r(t)\le 1+r^{-1}$, and~\eqref{eq:Uupperbetter} in the case $N_r(t)=\beta_r(t)$,
	for $r$ sufficiently large, for $t\ge 6r$ and $x\in [L_t,L_t+7t]$,
	\[
	U(t,x)
	\le K^{(3)}_r
	\int_{N_r(t)}^{\infty} U_0(y) \frac{e^{-\frac 1 {2t}(x-y)^2}}{\sqrt{2\pi t}}e^t 
	\p{\xi^t_{x,y}(s)>\underline{\mathcal M}_{r,t}(t-s) \; \forall s\in [3r,t-3r]}dy.
	\]
	Then $K^{(3)}_r\to 1$ as $r\to \infty$ and $N_r(t)\to \infty$ as $t\to \infty$ for each $r\ge 2$, with $N_r(t)\ge \beta_r(t)$ for every $2\le r \le t$, as required.
\end{proof}

Suppose $U_0$ satisfies Assumption~\ref{assum:standing assumption ic} and~\eqref{eq:assump_infmass1}.
Then for
$b\in (0,\sqrt{2})$, there exists $T_b(U_0)\in (0,\infty)$ such that
\begin{equation} \label{eq:Tbdefn}
U_0(y) \le e^{-by} \;\;\forall y\ge T_b(U_0).
\end{equation}
The next result, which will follow easily from Proposition~\ref{prop:upperbdinfmass} and standard Brownian bridge estimates, will be one of the main ingredients in the proof of Theorem~\ref{theo:infinitemassconv} in the next subsection.
 This result corresponds to~\cite[Corollary 1 of Proposition 9.1]{Bramson1983}.
\begin{cor} \label{cor:x2x1}
Suppose $U_0$ satisfies Assumption~\ref{assum:standing assumption ic},~\eqref{eq:assump_infmass1} and~\eqref{eq:assump_infmass2}, and let $(U(t,x),L_t)$ denote the solution of~\eqref{eq:FBP_CDF}.
Define $A^{(1)}_r$ as in Proposition~\ref{prop:upperbdinfmass}.
For $r$ sufficiently large, for $t\ge 6r$, $b\in (0,\sqrt 2)$ and $\delta>0$ with $\delta t\ge T_b(U_0)$, for $x_1,x_2\in (L_t,L_t+7t]\cap[\delta t,(b+\delta)t]$ with $x_1\le x_2$, 
	\begin{align*}
		U(t,x_2)&\le A^{(1)}_r e^{-b(x_2-\sqrt{2}t)+\frac 12 (\sqrt{2}-b)^2t} e^{-\frac 1 {2t}((b+\delta) t -x_2)^2}\\
		&\quad +A^{(1)}_r e^t\int_{N_r(t)}^{\infty} U_0(y)\frac{e^{-\frac 1 {2t}(x_1-y)^2}}{\sqrt{2\pi t}} e^{-(\frac{L_t} t-\delta)(x_2-x_1)}\frac{x_2-L_t}{x_1-L_t}\\
		&\hspace{4cm}\cdot  \p{\xi^t_{x_1,y}(s)>\underline{\mathcal M}_{r,t}(t-s) \; \forall s\in [3r,t-3r]}dy.
	\end{align*}
\end{cor}
\begin{proof}
	Suppose $b\in (0,\sqrt 2)$ and $\delta t \ge T_b(U_0)$; then by~\eqref{eq:Tbdefn}, for $x\in [ \delta t,(b+\delta)t]$,
	\begin{align*}
		e^t \int_{\delta t}^\infty U_0(y) \frac{e^{-\frac 1 {2t}(x-y)^2}}{\sqrt{2\pi t}}dy 
		\le e^t \int_{\delta t}^\infty e^{-by} \frac{e^{-\frac 1 {2t}(x-y)^2}}{\sqrt{2\pi t}}dy
		&=e^t e^{-bx}e^{\frac 12 b^2 t}\int_{\delta t}^\infty \frac{e^{-\frac 1 {2t}(x-bt-y)^2}}{\sqrt{2\pi t}}dy\\
		&\le e^t e^{-bx}e^{\frac 12 b^2 t}e^{-\frac 1 {2t}((b+\delta)t-x)^2}\\
		&= e^{-b(x-\sqrt 2 t)}e^{\frac 12 (\sqrt 2 -b)^2 t}e^{-\frac 1 {2t}((b+\delta)t-x)^2},
	\end{align*}
	where the second inequality follows from a Gaussian tail bound.
	
	Take $r$ sufficiently large that Proposition~\ref{prop:upperbdinfmass} and Lemma~\ref{lem:Mrtestimates} hold.
	Then by Proposition~\ref{prop:upperbdinfmass}, for $t\ge 6r$ and $b\in (0,\sqrt 2)$ with $\delta t\ge T_b(U_0)$ and
	$x_2\in [L_t,L_t+7t]\cap [\delta t, (b+\delta) t]$,
	\begin{equation}
	\begin{split} 
		U(t,x_2)&\le 
		A^{(1)}_r e^{-b(x_2-\sqrt{2}t)+\frac 12 (\sqrt{2}-b)^2t} e^{-\frac 1 {2t}((b+\delta) t -x_2)^2}\label{eq:Utx2upper}
		\\
		&\hspace{0.5cm} +A^{(1)}_r e^t\int_{N_r(t)}^{(\delta t) \vee N_r(t)} U_0(y) \frac{e^{-\frac 1 {2t}(x_2-y)^2}}{\sqrt{2\pi t}}\p{\xi^t_{x_2,y}(s)>\underline{\mathcal M}_{r,t}(t-s) \; \forall s\in [3r,t-3r]}dy.
	\end{split}
	\end{equation}
	By~\eqref{eq:Mrtest1a} in Lemma~\ref{lem:Mrtestimates}, for $t\ge 6r$ we have
	$
	\underline{\mathcal M}_{r,t}(t-s)\le \ell_{r,t}(t-s)$ $\forall s\in [3r,t-3r].
	$
	Therefore, by Lemma~\ref{lem:BBfacts}(e),
	for $y\in \Rm$,
	\[
	\frac{\p{\xi^t_{x,y}(s)>\ell_{r,t}(t-s) \; \forall s\in [0,t]}}{\p{\xi^t_{x,y}(s)>\underline{\mathcal M}_{r,t}(t-s)\; \forall s\in [3r,t-3r]}}
	\]
	is increasing in $x$.
	Recall from Proposition~\ref{prop:upperbdinfmass} that $N_r(t)\ge \beta_r(t)$, and recall the definition of $\ell_{r,t}$ in~\eqref{eq:lrt_defn}.
	For $x,y\in \R$, we can write
	$\xi^t_{x,y}(s)=\xi^t_{x-L_t,y-\beta_r(t)}(s)+\ell_{r,t}(t-s)$ $\forall s\in [0,t]$.
	Therefore, for $x_1,x_2> L_t$ with $x_1\le x_2$ and $y> N_r(t)\ge \beta_r(t)$,
	\begin{align} \label{eq:probaboveratio}
		\frac{\p{\xi^t_{x_2,y}(s)>\underline{\mathcal M}_{r,t}(t-s)\; \forall s\in [3r,t-3r]}}{\p{\xi^t_{x_1,y}(s)>\underline{\mathcal M}_{r,t}(t-s)\; \forall s\in [3r,t-3r]}}
		&\le \frac{\p{\xi^t_{x_2,y}(s)>\ell_{r,t}(t-s) \; \forall s\in [0,t]}}{\p{\xi^t_{x_1,y}(s)>\ell_{r,t}(t-s) \; \forall s\in [0,t]}}\\
		&= \frac{\p{\xi^t_{x_2-L_t,y-\beta_r(t)}(s)>0\; \forall s\in [0,t]}}{\p{\xi^t_{x_1-L_t,y-\beta_r(t)}(s)>0\; \forall s\in [0,t]}} \notag \\
		&=\frac{1-e^{-\frac 2 t (x_2-L_t)(y-\beta_r(t))}}{1-e^{-\frac 2 t (x_1-L_t)(y-\beta_r(t))}} \notag \\
		&\le \frac{x_2-L_t}{x_1-L_t},
	\end{align}
	where the equality follows from Lemma~\ref{lem:BBfacts}(b), and the
	last inequality follows since $\frac{1-e^{-b}}{1-e^{-a}}\le \frac b a$ for $0< a \le b$.
	For $y\le \delta t$ and $x_2\ge x_1\ge L_t$, we also have
	\begin{equation} \label{eq:gaussiantailratio}
		\frac{e^{\frac 1 {2t}(x_1-y)^2}}{e^{\frac 1 {2t}(x_2-y)^2}}
		=e^{\frac 1 {2t}(x_1-y)^2-\frac 1 {2t}((x_1-y)^2+2(x_2-x_1)(x_1-y)+(x_2-x_1)^2)}
		\le e^{-\frac 1t (x_2-x_1)(x_1-y)}\le e^{-(\frac{L_t}t -\delta)(x_2-x_1)},
	\end{equation}
	where the final inequality follows since $x_2-x_1\ge 0$ and $x_1-y\ge L_t-\delta t$.
	By substituting~\eqref{eq:probaboveratio} and~\eqref{eq:gaussiantailratio} into~\eqref{eq:Utx2upper},
	the result follows.
\end{proof}
In the remainder of this subsection, we use the upper bound on $U(t,x)$ in Proposition~\ref{prop:upperbdinfmass} to establish the asymptotic formula for $U(t,x)$ in Proposition~\ref{prop:s0} below.
We will use the following entropic repulsion result for the Brownian bridge from~\cite{Bramson1983}.
Recall the definition of $\theta_{r,t}^{-1}$ in~\eqref{eq:theta-defn}.
\begin{lem}[Proposition 6.2, \cite{Bramson1983}] \label{lem:Bramsonentropicrepulsion}
Suppose $r_0\in (0,\infty)$, and for each $t\ge 1$, $G_t:[0,t]\to \R$ with
\begin{equation} \label{eq:Ktupperbd}
G_t(s)\le  8 (s\wedge (t-s))^{1/3}\;
\forall s\in [r_0,t-r_0].
\end{equation}
Then
\[
	\frac{\mathbb{P}\left(\xi^t_{0,0}(s)>G_t(s) \; \forall s \in [3r,t-3r]\right)}{\mathbb{P}\left(\xi^t_{0,0}(s)>\theta_{r,t}^{-1}\circ G_t(s) \; \forall s \in [3r,t-3r]\right)} \rightarrow 1 \quad \text { as } r \rightarrow \infty 
\]
	uniformly in $t\ge 6r$ and in $G_t$ satisfying~\eqref{eq:Ktupperbd}.
\end{lem}
(Note that~\cite[Proposition 6.2]{Bramson1983} is stated with $[r,t-r]$ in place of $[3r,t-3r]$, but the result stated here follows from the same proof as in~\cite{Bramson1983}.) 
In~\cite{Bramson1983},~\cite[Proposition 6.2]{Bramson1983} is proved by first showing, using the Cameron-Martin formula, that for $\tilde G_t$ satisfying~\eqref{eq:Ktupperbd},
\[
	\frac{\mathbb{P}\left(\xi^t_{0,0}(s)>\tilde G_t(s)+8(s\wedge (t-s))^{1/3} \; \forall s \in [r,t-r]\right)}{\mathbb{P}\left(\xi^t_{0,0}(s)>\tilde G_t(s)-8(s\wedge (t-s))^{1/3} \; \forall s \in [r,t-r]\right)} \rightarrow 1 \quad \text { as } r \rightarrow \infty 
\]
uniformly in $t\ge 2r$, and then showing that the result follows from Gaussian tail estimates.

Recall from Proposition~\ref{prop:fbpsoln} that for $U_0$ satisfying Assumption~\ref{assum:standing assumption ic}, letting $(U(t,x),L_t)$ denote the solution of~\eqref{eq:FBP_CDF}, we have $\sup_{s\in [0,t]}L_s<\infty$ for any $t>0$.

We now use Lemma~\ref{lem:Bramsonentropicrepulsion} to show that for large $r$, for suitable $x$ and $y$, the probability that $\xi^t_{x,y}(\cdot)$ stays above $\underline{\mathcal M}_{r,t}(t-\cdot)$ on the interval $[3r,t-3r]$ is close to the probability that $\xi^t_{x,y}(\cdot)$ stays above $L_{t-\cdot}$ on the whole interval $[0,t]$.
This result corresponds to~\cite[Proposition 7.4]{Bramson1983}.
\begin{lem} \label{lem:probaboveinterval}
Suppose $U_0$ satisfies Assumption~\ref{assum:standing assumption ic} and~\eqref{eq:assump_infmass1}, and let $(U(t,x),L_t)$ denote the solution of~\eqref{eq:FBP_CDF}.
	Define $\beta_r(t)$ as in~\eqref{eq:betadef} and $\underline{\mathcal M}_{r,t}$ as in~\eqref{eq:Munderdefn}.
	For $t\ge 0$, let
\begin{equation} \label{eq:barLtdefn}
\bar{L}_t:=\sup_{s\in [0,t]}L_s \vee 0.
\end{equation}	
	There exists $A^{(2)}_r\to 1$ as $r\to \infty$ such that the following holds.
For $r$ sufficiently large, for $t \ge 6 r,$ $ x \ge L_{t}+14 r$ and
	$y\ge \beta_r(t)\vee (2(\bar{L}_{3r}+r))$,
	\begin{equation} \label{eq:probaboveinterval1}
	\p{\xi^t_{x,y}(s)>\underline{\mathcal M}_{r,t}(t-s)\; \forall s\in [3r,t-3r]}
	\le A^{(2)}_r\p{\xi^t_{x,y}(s)>L_{t-s}\; \forall s\in [0,t]}.
	\end{equation}
Moreover, for $s_0\in [0,\infty)$, for $r$ sufficiently large, for $t\ge 6r$, $x> L_t$ and $y\ge 2(\bar{L}_{3r}+r)$,
	\begin{equation} \label{eq:probaboveinterval2}
\p{\xi^t_{x,y}(s)>L_{t-s}\;\forall s\in [0,t-s_0]}
	\le A^{(2)}_r\p{\xi^t_{x,y}(s)>L_{t-s}\;\forall s\in [0,t]}.
	\end{equation}
\end{lem}
\begin{proof}
We claim that
	there exists $K^{(1)}_r\to 1$ as $r\to\infty$ such that
	for $r$ sufficiently large, for $t\ge 6r$, $x\ge L_t$ and $y\ge \beta_r(t)$,
	\begin{equation} \label{eq:underMLxiabove}
	\p{\xi^t_{x,y}(s)>\underline{\mathcal M}_{r,t}(t-s)\; \forall s\in [3r,t-3r]}
	\le K^{(1)}_r\p{\xi^t_{x,y}(s)>L_{t-s}\; \forall s\in [3r,t-3r]}.
	\end{equation}
To prove the claim,	note first that by Lemma~\ref{lem:lrtLs} and~\eqref{eq:lrt_defn}, for $r$ sufficiently large and $t\ge 2r$,
	\begin{equation} \label{eq:Lsupperdelta}
		L_s\le \frac s t L_t+\frac{t-s}t \beta_r(t)+8(s\wedge (t-s))^{1/3}  \; \; \forall s\in [r,t-r].
	\end{equation}
By~\eqref{eq:Mrtest1b} in Lemma~\ref{lem:Mrtestimates}, for $r$ sufficiently large and $t\ge 6r$,  	
\begin{equation} \label{eq:MbelowLbd}
\underline{\mathcal M}_{r,t}(s)< L_s \quad \forall s\in [3r,t-3r].
\end{equation}
For $r\ge 1$, $t\ge 6r$ and $x,y\in \R$, we can write
	\begin{align} \label{eq:ratioKK}
		& \frac{\mathbb{P}\left(\xi_{x, y}^{t}(s)>L_{t-s}\;  \forall s \in[3 r, t-3 r]\right)}{\mathbb{P}\left(\xi^{t}_{x, y}(s)>\underline{\mathcal M}_{r, t}(t-s) \; \forall s \in[3 r, t-3 r]\right)} \notag \\
		& =\frac{\mathbb{P}\left(\xi^t_{x-L_{t}, y-\beta_{r}(t)}(s)>L_{t-s}-\frac{t-s}{t} L_{t}-\frac{s}{t} \beta_{r}(t) \; \forall s\in [3r,t-3r]\right)}{\mathbb{P}\left(\xi^t_{x-L_{t}, y-\beta_{r}(t)}(s)>\underline{\mathcal  M}_{r,t}(t-s)-\frac{t-s}{t} L_{t}-\frac{s}{t} \beta_{r}(t) \;\forall s\in [3r,t-3r]\right)} \notag \\
		& =\frac{\mathbb{P}\left(\xi^t_{x-L_{t}, y-\beta_{r}(t)}(s)>G_{r,t}(t-s) \; \forall s\in [3r,t-3r]\right)}{\mathbb{P}\left(\xi^t_{x-L_{t}, y-\beta_{r}(t)}(s)>\theta^{-1}_{r,t}\circ G_{r,t}(t-s) \; \forall s\in [3r,t-3r]\right)},
	\end{align}
	where the last line follows
	by~\eqref{eq:Krtdefn} and~\eqref{eq:Munderdefn}.
	By Lemma~\ref{lem:Bramsonentropicrepulsion}, using~\eqref{eq:Lsupperdelta} and~\eqref{eq:Krtdefn}, we have
	$$
	\frac{\mathbb{P}\left(\xi^t_{0,0}(s)>G_{r,t}(t-s) \; \forall s \in [3r,t-3r]\right)}{\mathbb{P}\left(\xi^t_{0,0}(s)>\theta^{-1}_{r,t}\circ G_{r,t}(t-s) \; \forall s \in [3r,t-3r]\right)} \rightarrow 1 \quad \text { as } r \rightarrow \infty 
	$$
	uniformly in $t\ge 6r$.	
	Then by Lemma~\ref{lem:BBfacts}(e) and~\eqref{eq:MbelowLbd}, we have that the left-hand side of~\eqref{eq:ratioKK} is increasing in $x$ and $y$, and bounded above by 1, which completes the proof of the claim~\eqref{eq:underMLxiabove}.

	By Lemma~\ref{lem:BBfacts}(d) applied twice, for $r>0$, $t\ge 6r$, $x> L_t$ and $y\in \R$, we have
	\begin{equation} 
	\begin{split}\label{eq:xiLs0frac}
		&\frac{\p{\xi^t_{x,y}(s)>L_{t-s}\; \forall s\in [0,t]}}{\p{\xi^t_{x,y}(s)>L_{t-s}\; \forall s\in [3r,t-3r]}}
		\ge \p{\xi^t_{x,y}(s)>L_{t-s}\; \forall s\in [0,3r]\cup [t-3r,t]} \\
		& \ge \p{\xi^t_{x,y}(s)>L_{t-s}\; \forall s\in [0,3r]}\p{\xi^t_{x,y}(s)>L_{t-s}\; \forall s\in  [t-3r,t]}.
	\end{split}
	\end{equation}
	By Lemma~\ref{lem:boundary locally Lipschitz from the left}, there exists $t_0>0$ such that $L_{s_{2}} \ge L_{s_1}$ $ \forall t_{0} \leq s_{1} \leq s_{2}$. Therefore, for $r$ sufficiently large, for $t \ge 6r$ and $x,y\in \R$ we have
	\begin{equation} \label{eq:probLt-s03r}
	\mathbb{P}\left(\xi_{x, y}^{t}(s)>L_{t-s} \; \forall s \in[0,3 r]\right) \ge \mathbb{P}\left(\xi_{0,0}^{t}(s)>L_{t}-\tfrac{s}{t} y-\tfrac{t-s}{t} x \; \forall s \in[0,3 r]\right) .
	\end{equation}
	Using Lemmas~\ref{lem:Lttsqrt2} and~\ref{lem:betartbound}, we can take $r$ sufficiently large that for
	$t\ge 6r$ we have $L_{t}-\beta_{r}(t) \le 2 t$.
	Then for $t\ge 6r$, $x \ge L_{t}+14 r$, $y \ge \beta_{r}(t)$, and $s \in[0,3 r]\subseteq [0,t/2]$ we have
\[
L_{t}-\tfrac{s}{t} y-\tfrac{t-s}{t} x=\tfrac{s}{t} (L_t-y)-\tfrac{t-s}{t} (x-L_t)
\le 2s -7r\le -r.
\]	
Therefore, substituting into~\eqref{eq:probLt-s03r},
	\begin{align} \label{eq:aboveL3r}
		\mathbb{P}\left(\xi_{x, y}^{t}(s)>L_{t-s} \; \forall s \in[0,3 r]\right)  \ge \mathbb{P}\left(\xi_{0,0}^{t}(s)>-r \; \forall s \in[0,3 r]\right) 
		 &\ge \mathbb{P}_0\left(B_{s}>-r \; \forall s \in[0,6 r]\right) \notag \\
		&\ge 1-2 e^{-r / 12},
	\end{align}
	where the second inequality follows from Lemma~\ref{lem:BBfacts}(a),
	and the last inequality follows
	by the reflection principle and a Gaussian tail bound.
	
	To establish~\eqref{eq:probaboveinterval1}, it remains to bound the second term on the right-hand side of~\eqref{eq:xiLs0frac}.
	Recall the definition of $\bar{L}_{\cdot}$ in~\eqref{eq:barLtdefn}.
	Using Lemma~\ref{lem:Lttsqrt2}, take $r$ sufficiently large that for
	$t\ge 6r$ we have $L_{t} \ge 0$.
	Then for $t\ge 6r$, $x \ge L_{t}\ge 0$ and $y \ge 2(\bar{L}_{3r}+r)$,
	we have
	\begin{align} \label{eq:aboveLt-3r}
		\mathbb{P}\left(\xi^{t}_{x,y}(s)>L_{t-s} \; \forall s \in[t-3 r, t]\right) &\ge \mathbb{P}\left(\xi^t_{0,0}(s)>\bar{L}_{3r}-\tfrac{t-s}{t} x-\tfrac{s}{t} y
		\;\forall s \in[t-3 r, t]\right) \notag \\
		& \ge \mathbb{P}\left(\xi_{0,0}^{t}(s)>-r \; \forall s \in[t-3 r, t]\right) \notag \\
		& \geq 1-2 e^{-r / 12},
	\end{align}
	where the last inequality follows by the same argument as in~\eqref{eq:aboveL3r}.
By combining \eqref{eq:underMLxiabove}, \eqref{eq:xiLs0frac}, \eqref{eq:aboveL3r} and \eqref{eq:aboveLt-3r}, the first statement~\eqref{eq:probaboveinterval1} follows.

For the second statement~\eqref{eq:probaboveinterval2}, take $r>s_0/3$ and note that by Lemma~\ref{lem:BBfacts}(d),
for $t\ge 6r$, $x> L_t$ and $y\in \R$,
\[
\frac{\p{\xi^t_{x,y}(s)>L_{t-s}\; \forall s\in [0,t]}}{\p{\xi^t_{x,y}(s)>L_{t-s}\; \forall s\in [0,t-s_0]}}
		\ge \p{\xi^t_{x,y}(s)>L_{t-s}\; \forall s\in [t-3r,t]}. 
\]
The result then follows from~\eqref{eq:aboveLt-3r}.
\end{proof}
By combining Proposition~\ref{prop:upperbdinfmass}, Lemma~\ref{lem:probaboveinterval} and Lemma~\ref{lem:Bt_ends_big}, we can now prove the following result, which will be used in the proof of Theorem~\ref{theo:infmassfront}.
This result corresponds to~\cite[Proposition 6.3]{Bramson1983}.
\begin{prop} \label{prop:s0}
Suppose $U_0$ satisfies Assumption~\ref{assum:standing assumption ic},~\eqref{eq:assump_infmass1} and~\eqref{eq:assump_infmass2}, and let $(U(t,x),L_t)$ denote the solution of~\eqref{eq:FBP_CDF}.
For $s_0\ge 0$, there exists $(C_{t,z}(s_0))_{t>0,z>0}$, with $\lim_{t\wedge z\to \infty}C_{t,z}(s_0)=1$ for each $s_0\ge 0$, such that for $t>0$ and $x\in (L_t,L_t+7t]$,
\[
U(t,x)=C_{t,x-L_t}(s_0)e^t \Esub{x}{U_0(B_t)\1_{\{B_s>L_{t-s} \; \forall s\in [0,t-s_0]\}}}.
\]
\end{prop}

\begin{proof}
Recall the definition of $N_r(t)$ in Proposition~\ref{prop:upperbdinfmass}.
Since $N_r(t)\to \infty$ as $t\to \infty$ and $N_r(t)\ge \beta_r(t)$ $\forall t\ge r \ge 2$, we have by Lemma~\ref{lem:probaboveinterval} that for $r$ sufficiently large, for $t\ge 6r$ sufficiently large that $N_r(t)\ge 2(\bar{L}_{3r}+r)$ and $x\ge L_t+14r$,
	\[
	\p{\xi^t_{x,y}(s)>\underline{\mathcal M}_{r,t}(t-s) \; \forall s\in [3r,t-3r]}
	\le A^{(2)}_r \p{\xi^t_{x,y}(s)>L_{t-s} \; \forall s\in [0,t]} \quad \forall y\ge N_r(t).
	\]
	By Proposition~\ref{prop:upperbdinfmass}, it follows that
	for $r$ sufficiently large, for $t\ge 6r$ sufficiently large that $N_r(t)\ge 2(\bar{L}_{3r}+r)$ and $x\in [L_t+14r,L_t+7t]$,
\begin{align} \label{eq:s0lemupper}
U(t,x)&\le A^{(1)}_r A^{(2)}_r e^t \Esub{x}{U_0(B_t)\1_{\{B_s>L_{t-s}\; \forall s\in [0,t]\}}\1_{\{B_t>N_r(t)\}}} \notag \\
&\le A^{(1)}_r A^{(2)}_r e^t \Esub{x}{U_0(B_t)\1_{\{B_s>L_{t-s}\; \forall s\in [0,t-s_0]\}}}.
\end{align}
By Lemma~\ref{lem:Bt_ends_big} with $\ell_t(s)=L_s$, we can take $M(t)\to \infty$ as $t\to \infty$ and $\varepsilon_t\to 0$ as $t\to \infty$ such that for $t>0$ and $x\ge L_t+1$,
\begin{equation} \label{eq:useBtendsbig}
\Esub{x}{U_0(B_t)\1_{\{B_t\le M(t),B_s>L_{t-s}\;\forall s\in [0,t-s_0]\}}}\le \varepsilon_t \Esub{x}{U_0(B_t)\1_{\{B_s>L_{t-s}\;\forall s\in [0,t-s_0]\}}}.
\end{equation}
Take $r>0$ sufficiently large that~\eqref{eq:probaboveinterval2} in Lemma~\ref{lem:probaboveinterval} holds, and then take $t\ge 6r$ sufficiently large that $M(t)\ge 2(\bar{L}_{3r}+r)$. 
By the Feynman-Kac formula in Lemma~\ref{lem:FKforinfinitemass}, and then by conditioning on $B_t$ and using~\eqref{eq:probaboveinterval2} in Lemma~\ref{lem:probaboveinterval},
 for $x\ge L_t+1$,
 \begin{align} \label{eq:s0lemlower}
U(t,x)& \ge e^t \Esub{x}{U_0(B_t)\1_{\{B_s>L_{t-s}\; \forall s\in [0,t]\}}\1_{\{B_t>M(t)\}}} \notag \\
&\ge e^t (A^{(2)}_r)^{-1}\Esub{x}{U_0(B_t)\1_{\{B_s>L_{t-s}\; \forall s\in [0,t-s_0]\}}\1_{\{B_t>M(t)\}}} \notag \\
&\ge e^t (A^{(2)}_r)^{-1}(1-\varepsilon_t)\Esub{x}{U_0(B_t)\1_{\{B_s>L_{t-s}\; \forall s\in [0,t-s_0]\}}},
\end{align}
where the last inequality follows from~\eqref{eq:useBtendsbig}.
The result follows by combining~\eqref{eq:s0lemupper} and~\eqref{eq:s0lemlower}, since $A^{(1)}_r\to 1$ and $A^{(2)}_r\to 1$ as $r\to \infty$.
\end{proof}

\subsection{Proof of Theorem~\ref{theo:infinitemassconv}} \label{subsec:inf_thm1pf}

We now use Lemma~\ref{lem:probaboveinterval}, Corollary~\ref{cor:x2x1}, the stretching lemma, and the convergence of $U^H(t,L^H_t+\cdot)$ to $\Pi_{\min}$ from Theorem~\ref{theo:convergence to the minimal travelling wave for finite initial mass}, to show that the tail bounds on $U(t,\cdot)$ in~\eqref{eq:propsandwich} in Proposition~\ref{prop:sandwichmeansconv} hold.
\begin{prop} \label{prop:UhasPimintail}
Suppose $U_0$ satisfies Assumption~\ref{assum:standing assumption ic},~\eqref{eq:assump_infmass1} and~\eqref{eq:assump_infmass2}, and let $(U(t,x),L_t)$ denote the solution of~\eqref{eq:FBP_CDF}.
	There exist $N>0$, $T>0$,
	$\gamma_1(z)\to 1$ as $z\to \infty$, 
	$\gamma_2(t)\to 0$ as $t\to \infty$ 
	and $(m(t),t\ge T)$ with $\sup_{T\le s \le t<\infty}(m(s)-m(t))<\infty$
	such that
	\begin{equation} \label{eq:propUsandwich}
	\gamma_1^{-1}(z)\Pi_{\min}(z)-\gamma_2(t)\le U(t,z+m(t))\le \gamma_1(z)\Pi_{\min}(z)+\gamma_2(t)
	\quad \forall z\ge N, \, t\ge T.
	\end{equation}
\end{prop}
\begin{proof}
We begin by establishing an upper bound on $U(t,x_2)$ in terms of $U(t,x_1)$ for large $t$ and suitable $L_t\le x_1\le x_2$, using Corollary~\ref{cor:x2x1}.

Note first that by Lemma~\ref{lem:probaboveinterval}, for $r$ sufficiently large, for $t\ge 6r$ sufficiently large that $N_r(t)\ge 2(\bar{L}_{3r}+r)$, for $x_1\ge L_t+14r$,
	\begin{equation} \label{eq:useprobabovelem}
	\p{\xi^t_{x_1,y}(s)>\underline{\mathcal M}_{r,t}(t-s)\; \forall s\in [3r,t-3r]}
	\le A^{(2)}_r\p{\xi^t_{x_1,y}(s)>L_{t-s}\; \forall s\in [0,t]} \quad \forall y\ge N_r(t).
	\end{equation}
	Recall the definition of $T_b(U_0)$ in~\eqref{eq:Tbdefn}.
Using Lemma~\ref{lem:Lttsqrt2}, take $\delta_t\in (0,\sqrt 2)$ $\forall t>0$ with $\delta_t$ non-increasing in $t$ and $\delta_t\to 0$ as $t\to \infty$ sufficiently slowly that
\begin{equation} \label{eq:deltatconditions}
\frac{\delta_t^2 t}{\log t}\to \infty \quad
\text{and} \quad
\delta_t^{-2} \left|\frac{L_t}t -\sqrt 2 \right| \to 0 
\quad \text{ as }t\to \infty,
\end{equation}
and that for $t$ sufficiently large, 
\begin{equation} \label{eq:deltatconditions2}
\delta_t t\ge T_{\sqrt{2}-\frac 14 \delta_t}(U_0).
\end{equation}
Take $r$ sufficiently large that Corollary~\ref{cor:x2x1} holds, and then take $t\ge 6r$ sufficiently large that~\eqref{eq:deltatconditions2} holds, that $\sqrt 2-\frac 14 \delta_t>0$, and that (using~\eqref{eq:deltatconditions})
\begin{equation} \label{eq:tbigforthm}
L_t-\sqrt 2 t \ge -\frac{\delta_t^2 t}{32\sqrt 2} \quad \text{and}\quad
\left|\frac{L_t}t -\sqrt 2\right| \le \delta_t.
\end{equation}
Then by Corollary~\ref{cor:x2x1} (with $\delta=\delta_t$ and $b=\sqrt 2-\frac 14 \delta_t$) combined with~\eqref{eq:useprobabovelem},
for $x_1,x_2\in [L_t+14r,L_t+7t]\cap [\delta_t t,(\sqrt 2+\frac 14 \delta _t)t]$ with $0\le x_2-x_1\le \delta_t^{-1/2}$,
and then in the second inequality using that $\sqrt 2-\frac 14 \delta_t>0$ and $L_t\le x_2\le (\sqrt 2+\frac 14 \delta_t)t$ for the first term, and using the second part of~\eqref{eq:tbigforthm} for the second term,
\begin{align} \label{eq:usecorforthm}
&U(t,x_2) \notag\\
&\le A^{(1)}_r e^{-(\sqrt 2-\frac 14 \delta_t)(x_2-\sqrt 2 t)+\frac 1 {32}\delta_t^2 t}e^{-\frac 1 {2t}((\sqrt 2+\frac 34 \delta_t)t-x_2)^2} \notag \\
&\quad +A^{(1)}_r e^t e^{-(\frac{L_t}t-\delta_t)(x_2-x_1)}\frac{x_2-L_t}{x_1-L_t}
A^{(2)}_r \int_{N_r(t)}^\infty U_0(y)\frac{e^{-\frac{1}{2t}(x_1-y)^2}}{\sqrt{2\pi t}}\p{\xi^t_{x_1,y}(s)>L_{t-s}\; \forall s\in [0,t]}dy \notag \\
&\le A^{(1)}_r e^{-(\sqrt 2-\frac 14 \delta_t)(L_t-\sqrt 2 t)+\frac 1 {32}\delta_t^2 t}e^{-\frac 1 {2t}\frac 14 \delta_t^2 t^2} \notag \\
&\quad +A^{(1)}_r e^t e^{-(\sqrt 2 -2\delta_t)(x_2-x_1)}\frac{x_2-L_t}{x_1-L_t}
A^{(2)}_r \Esub{x_1}{U_0(B_t)\1_{\{B_s>L_{t-s}\; \forall s\in [0,t]\}}} \notag \\
&\le A^{(1)}_r e^{-\frac 1 {16}\delta_t^2 t}
+A^{(1)}_r A^{(2)}_r  e^{2\delta_t^{1/2}}e^{-\sqrt 2(x_2-x_1)}\frac{x_2-L_t}{x_1-L_t}U(t,x_1),
\end{align}
and the last inequality follows by the first part of~\eqref{eq:tbigforthm} for the first term, and by the Feynman-Kac representation in Lemma~\ref{lem:FKforinfinitemass} and since $x_2-x_1\le \delta_t^{-1/2}$ for the second term.

We now use the stretching lemma and Theorem~\ref{theo:convergence to the minimal travelling wave for finite initial mass} to establish a lower bound on $U(t,x_2)$ in terms of $U(t,x_1)$ for $L_t\le x_1\le x_2$.
By Theorem~\ref{theo:convergence to the minimal travelling wave for finite initial mass}, there exists $\varepsilon_t\to 0$ as $t\to \infty$ such that
\begin{equation} \label{eq:heavisideconvforinfthm}
\sup_{z\in \R}|U^H(t,L^H_t+z)-\Pi_{\min}(z)|\le \varepsilon_t \quad \forall t>0.
\end{equation}
We will assume that $(\varepsilon_t)_{t>0}$ is non-increasing.
Recall from~\eqref{eq:minimal travelling wave} and~\eqref{eq:Pimindefn}
that 
\begin{equation} \label{eq:Piminformula}
\Pi_{\min}(z)=(\sqrt 2 z+1)e^{-\sqrt 2 z} \; \forall z\ge 0.
\end{equation}
By Lemma~\ref{lem:extended maximum principle} and Lemma~\ref{lem:stretchdecr}, and then by~\eqref{eq:heavisideconvforinfthm} and~\eqref{eq:Piminformula}, for $t>0$ and $z\ge 0$ we have
\begin{equation} \label{eq:lowerbdfromstretch}
U(t,L_t+z)\ge U^H(t,L^H_t+z)\ge (\sqrt 2 z+1)e^{-\sqrt 2 z}-\varepsilon_t\ge e^{-\sqrt 2 z}-\varepsilon_t.
\end{equation}
Also, by Lemma~\ref{lem:extended maximum principle} and Lemma~\ref{lem:stretchdecr} and then~\eqref{eq:heavisideconvforinfthm} again,
for $t>0$ and $x_1>L_t$, taking $z_1>0$ such that $U^H(t,L^H_t+z_1)=U(t,x_1)$, for any $z\ge 0$ we have
\begin{align*}
U(t,x_1+z)&\ge U^H(t,L^H_t+z_1+z)\ge \Pi_{\min}(z_1+z)-\varepsilon_t\\
\text{and }\quad U(t,x_1)&=U^H(t,L^H_t+z_1)\le \Pi_{\min}(z_1)+\varepsilon_t.
\end{align*}
Since $U(t,L_t+z_1)\ge U^H(t,L^H_t+z_1)=U(t,x_1)$ by~\eqref{eq:lowerbdfromstretch}, we have $x_1\ge L_t+z_1$.
Hence for $z\ge 0$, using~\eqref{eq:Piminformula} in the second inequality,
and in the third inequality using that $[\Pi_{\min}(x_1-L_t)]^{-1}\le e^{\sqrt 2 (x_1-L_t)}$ and that letting $y=x_1-L_t>0$ we have $\sqrt 2(\sqrt 2 y+1)^{-1}\ge y^{-1}(1-y^{-1})$, it follows that
\begin{align} \label{eq:lowerbdforthm}
\frac{U(t,x_1+z)}{U(t,x_1)}
&\ge \frac{\Pi_{\min}(z_1+z)-\varepsilon_t}{\Pi_{\min}(z_1)+\varepsilon_t} \notag \\
&\ge \left(\frac{\sqrt 2(x_1-L_t+z)+1}{\sqrt 2(x_1-L_t)+1}e^{-\sqrt 2 z}-\varepsilon_t [\Pi_{\min}(x_1-L_t)]^{-1}
\right)
(1+\varepsilon_t [\Pi_{\min}(x_1-L_t)]^{-1})^{-1} \notag \\
&\ge e^{-\sqrt 2 z}\frac{x_1-L_t+z}{x_1-L_t}(1-(x_1-L_t)^{-1}-\varepsilon_t e^{\sqrt 2 (x_1-L_t+z)})
(1+\varepsilon_t e^{\sqrt 2 (x_1-L_t)})^{-1}.
\end{align}

Take $M_t\to \infty$ as $t\to \infty$ sufficiently slowly that
\begin{equation} \label{eq:Mtconds}
\varepsilon_t^{1/2}e^{\sqrt 2 M_t}\to 0\quad \text{as }t\to \infty,
\end{equation}
and that, using~\eqref{eq:deltatconditions},
for $t$ sufficiently large,
\begin{equation} \label{eq:Mtconds2}
L_t+M_t\le (\sqrt 2+\tfrac 14 \delta_t)t, \;
M_t\le 7t, \;
M_t\le \delta_t^{-1/2}, \;
e^{-\sqrt 2 M_t}\ge 2\varepsilon_t
\;\text{ and }\;
e^{2\sqrt 2 M_t}\le e^{\frac 1 {32}\delta_t^2 t}.
\end{equation}
Note that by Lemma~\ref{lem:Lttsqrt2} and since $\delta_t\to 0$ as $t\to \infty$, we also have $\delta_t t\le L_t$ for $t$ sufficiently large.
Take $r_0>0$ sufficiently large that $A^{(2)}_r>1/2$ for $r\ge r_0$ and that Corollary~\ref{cor:x2x1} holds for $r\ge r_0$.
Then by~\eqref{eq:usecorforthm},~\eqref{eq:lowerbdfromstretch} and~\eqref{eq:Mtconds2},
 there exists $(t_r)_{r\ge r_0}$ such that 
 for $r\ge r_0$, $t\ge t_r$ and $x_1,x_2\in [L_t+14r,L_t+M_t]$ with $x_1\le x_2$,
 \begin{align} \label{eq:Uratioupper}
 \frac{U(t,x_2)}{U(t,x_1)}
 &\le 
 A^{(1)}_r e^{-\frac 1 {16}\delta_t^2 t}(e^{-\sqrt 2 M_t}-\varepsilon_t)^{-1}+
A^{(1)}_r A^{(2)}_r  e^{2\delta_t^{1/2}}e^{-\sqrt 2(x_2-x_1)}\frac{x_2-L_t}{x_1-L_t} \notag \\
 &\le 
A^{(1)}_r A^{(2)}_r  e^{2\delta_t^{1/2}}e^{-\sqrt 2(x_2-x_1)}\frac{x_2-L_t}{x_1-L_t}(1+4e^{-\frac 1 {32}\delta_t^2 t}),
 \end{align}
 where the second inequality follows from~\eqref{eq:Mtconds2} and our assumption that $(A^{(2)}_r)^{-1}<2$.
 Moreover, by increasing $(t_r)_{r\ge r_0}$ if necessary, by
 ~\eqref{eq:lowerbdforthm} and~\eqref{eq:Mtconds},
  for $r\ge r_0$, $t\ge t_r$ and $x_1,x_2\in [L_t+14r,L_t+M_t]$ with $x_1\le x_2$,
 \begin{equation} \label{eq:Uratiolower}
\frac{U(t,x_2)}{U(t,x_1)}\ge 
e^{-\sqrt 2(x_2-x_1)}\frac{x_2-L_t}{x_1-L_t}(1-(x_1-L_t)^{-1}-\varepsilon_t^{1/2})(1+\varepsilon_t^{1/2})^{-1}.
 \end{equation}
We can assume that $(t_r)_{r\ge r_0}$ is non-decreasing and that $t_r\to \infty$ as $r\to \infty$.
By reducing $(M_t)_{t>0}$ if necessary, we assume from now on that $M_s\to \infty$ sufficiently slowly as $s\to \infty$ that for $s$ sufficiently large,
\begin{equation} \label{eq:tauMtcond}
t_{M_s}\le s.
\end{equation}

Take $r_1\ge r_0\vee 1$ sufficiently large that $A^{(1)}_{r_1} A^{(2)}_{r_1} \le 2$.
Then by~\eqref{eq:lowerbdfromstretch}, for $t$ sufficiently large, $U(t,L_t+14r_1)\in [\frac 12 e^{-\sqrt 2 \cdot 14r_1},1]$, and by~\eqref{eq:Uratioupper} and~\eqref{eq:Uratiolower}, for $t\ge t_{r_1}$ sufficiently large we have
\[
\frac 12 \frac{M_t}{14r_1}e^{-\sqrt 2(M_t-14r_1)}\le \frac{U(t,L_t+M_t)}{U(t,L_t+14r_1)}\le 3\frac{M_t}{14r_1}e^{-\sqrt 2(M_t-14r_1)}.
\]
Therefore, there exists $N_0>0$ such that for $t$ sufficiently large, there exists $\zeta(t)\in [-N_0,N_0]$ such that
\begin{equation} \label{eq:zetatdefn}
U(t,L_t+M_t)=\sqrt 2 M_t e^{-\sqrt 2(M_t-\zeta(t))}.
\end{equation}
Then by~\eqref{eq:Uratioupper} with $x_1=x$ and $x_2=L_t+M_t$, for $r\ge r_0$ and $t\ge t_r$, for $x\in [L_t+14r,L_t+M_t]$,
\begin{align*}
U(t,x)
&\ge (A^{(1)}_r A^{(2)}_r)^{-1} e^{-2\delta_t^{1/2}}e^{\sqrt 2(L_t+M_t-x)}\frac{x-L_t}{M_t}(1+4e^{-\frac 1 {32}\delta_t^2 t})^{-1}\cdot \sqrt 2 M_t e^{-\sqrt 2(M_t-\zeta(t))}\\
&=(A^{(1)}_r A^{(2)}_r)^{-1} e^{-2\delta_t^{1/2}}(1+4e^{-\frac 1 {32}\delta_t^2 t})^{-1}\sqrt 2(x-L_t)e^{-\sqrt 2(x-L_t-\zeta(t))}.
\end{align*}
Also, by~\eqref{eq:Uratiolower} with $x_1=x$ and $x_2=L_t+M_t$, 
\begin{align*}
U(t,x)
&\le e^{\sqrt 2(L_t+M_t-x)}\frac{x-L_t}{M_t}(1-(x-L_t)^{-1}-\varepsilon_t^{1/2})^{-1}(1+\varepsilon_t^{1/2})\cdot \sqrt 2 M_t e^{-\sqrt 2(M_t-\zeta(t))}\\
&=(1-(x-L_t)^{-1}-\varepsilon_t^{1/2})^{-1}(1+\varepsilon_t^{1/2})\sqrt 2(x-L_t)e^{-\sqrt 2(x-L_t-\zeta(t))}.
\end{align*}
Take $T_0>0$ sufficiently large that~\eqref{eq:tauMtcond} and~\eqref{eq:zetatdefn} hold for $t\ge T_0$, and $M_t\ge 14r_0$ for $t\ge T_0$, and also $\delta_t^2 t\ge \log t$ for $t\ge T_0$ (using~\eqref{eq:deltatconditions}).
Then for $s\ge T_0$ and $x\in [L_s+14r_0,L_s+M_s]$, letting $r=(x-L_s)/14$, we have $r\ge r_0$, $s\ge t_{M_s}\ge t_r$ and $x\in [L_s+14r,L_s+M_s]$, and so, using that $(\varepsilon_u)_{u>0}$ and $(\delta_u)_{u>0}$ are non-increasing,
\[
(1-r^{-1}-\varepsilon_{t_r}^{1/2})^{-1}(1+\varepsilon_{t_r}^{1/2})\ge \frac{U(s,x)}{\sqrt 2(x-L_s)e^{-\sqrt 2(x-L_s-\zeta(s))}}
\ge (A^{(1)}_r A^{(2)}_r)^{-1} e^{-2\delta_{t_r}^{1/2}}(1+4t_r^{-1/32})^{-1}.
\]
For $t\ge T_0$, let $m(t)=L_t+\zeta(t)$.
Note that by Lemma~\ref{lem:boundary locally Lipschitz from the left}
and since $|\zeta(t)|\le N_0$ $\forall t\ge T_0$, we have
$\sup_{T_0\le s \le t<\infty}(m(s)-m(t))<\infty$.
For $z\ge N_0+14r_0$, let
\begin{align*}
\gamma_1(z)&=\tilde{\gamma}_1((z-N_0)/14) \left(\sup_{|y|\le N_0}\frac{\sqrt 2(z+y)e^{-\sqrt 2 z}}{\Pi_{\min}(z)}\vee 1\right)\left(\inf_{|y|\le N_0}\frac{\sqrt 2(z+y)e^{-\sqrt 2 z}}{\Pi_{\min}(z)}\wedge 1\right)^{-1},\\
\text{where }\quad \tilde{\gamma}(r)&=\max((1-r^{-1}-\varepsilon_{t_r}^{1/2})^{-1}(1+\varepsilon_{t_r}^{1/2}), A^{(1)}_r A^{(2)}_r e^{2\delta_{t_r}^{1/2}}(1+4t_r^{-1/32}))\quad \text{for }r\ge r_0.
\end{align*}
Then since $\tilde \gamma(r)\to 1$ as $r\to \infty$, we have $\gamma_1(z)\to 1$ as $z\to \infty$, and for $t\ge T_0$ and $z\in [14r_0+N_0,M_t-N_0]$,
\begin{equation} \label{eq:Umtsandwich}
\gamma_1(z)^{-1}\Pi_{\min}(z)\le U(t,m(t)+z)\le \gamma_1(z)\Pi_{\min}(z).
\end{equation}
For $t\ge T_0$, let
\[
\gamma_2(t)=\sup_{z\ge M_t-N_0}(\gamma_1(z)^{-1}\Pi_{\min}(z))+U(t,m(t)+M_t-N_0).
\]
By~\eqref{eq:Umtsandwich}, we have $U(t,m(t)+M_t-N_0)\le \gamma_1(M_t-N_0)\Pi_{\min}(M_t-N_0)$ for $t\ge T_0$, and so $\gamma_2(t)\to 0$ as $t\to \infty$.
Since $U(t,\cdot)\ge 0$ is non-increasing, we now have that~\eqref{eq:propUsandwich} holds with $N=14r_0+N_0$ and $T=T_0$, which completes the proof.
\end{proof}

We can now complete the proof of Theorem~\ref{theo:infinitemassconv}.

\begin{proof}[Proof of Theorem~\ref{theo:infinitemassconv}]
	The result follows directly from Proposition~\ref{prop:UhasPimintail}, Proposition~\ref{prop:sandwichmeansconv}, and
	 Lemma~\ref{lem:mtoL}.
\end{proof}

\subsection{Proof of Theorem~\ref{theo:infmassfront}} \label{subsec:inf_thm2pf}

Using Proposition~\ref{prop:s0} and Theorem~\ref{theo:infinitemassconv}, Theorem~\ref{theo:infmassfront} will now follow from exactly the same arguments as for~\cite[Theorem 5]{Bramson1983} and~\cite[Application of Theorems 4 and 5]{Bramson1983}.
We will state how the proofs in~\cite{Bramson1983} can be applied here, and give a brief heuristic overview of the proofs in~\cite{Bramson1983}; the reader is referred to~\cite{Bramson1983} for the full details.
The following result corresponds to~\cite[Theorem 5]{Bramson1983}.

\begin{prop} \label{prop:frontinfiniteinitial}
	Suppose $U_0$ satisfies Assumption~\ref{assum:standing assumption ic},~\eqref{eq:assump_infmass1} and~\eqref{eq:assump_infmass2}.
	Let $(U(t,x),L_t)$ denote the solution of the free boundary problem~\eqref{eq:FBP_CDF}.
	Suppose for some $\delta\in (0,1/2)$, $s_0, s_1>0$ and $C_0>0$ that $(m(t),t\ge 0)$ satisfies
	\begin{equation} \label{eq:mconditions}
	\begin{aligned}
	&\lim_{t\to \infty}\frac{m(t)}t =\sqrt 2,\\
		&\sup\{|m(s)|:s_0\le s \le t\}<\infty \quad \forall t>0,\\
		&m(t)\ge \sqrt{2}t-t^\delta \quad \forall t\ge s_1,\\
		\text{and }\quad &m(t)-m(s)\ge \sqrt 2 (t-s)-(t-s)^\delta -C_0 \quad \forall s_0\le s \le t.
	\end{aligned}
	\end{equation}
	Suppose also that for some $D:\R_+\times \R\to [0,\infty)$ and $v:\R\to [0,\infty)$,
	\begin{multline}
	\label{eq:Dvasymp}
	e^t \Esub{x}{U_0(B_t)\1_{\{B_s>m(t-s) \; \forall s\in [0,t-s_0]\}}}=D(t,x-m(t))v(x-m(t)) \quad \forall t>0, x\in \R,\\
	\text{where }\lim_{z\to \infty} \limsup_{t\to \infty}D(t,z)=1 \text{ and }\lim_{z\to \infty} \liminf_{t\to \infty}D(t,z)=1.
	\end{multline}
	Then there exists $x_0\in \R$ such that
	\[
	m(t)-L_t\to x_0 \quad \text{as } t\to \infty.
	\]
	Moreover,~\eqref{eq:Dvasymp} holds for some $D$ with $v(z)=\Pi_{\min}(z+x_0)$ $\forall z\in \R$.
\end{prop}

\begin{proof}
Note that by Lemma~\ref{lem:Lttsqrt2}, Lemma~\ref{lem:less stretching means slower boundary} and Theorem~\ref{theo:convergence to the minimal travelling wave for finite initial mass}, $(L_t)_{t\ge 0}$ satisfies~\eqref{eq:mconditions}.
The proof is then identical to the proof of Theorem~5 in~\cite[pp. 173-181]{Bramson1983}
using Proposition~\ref{prop:s0} in place of~\cite[(9.39)]{Bramson1983}
and Theorem~\ref{theo:infinitemassconv} in place of~\cite[Theorem 4]{Bramson1983}.
Indeed, in~\cite[pp. 173-181]{Bramson1983}, Bramson first proves a bound on
\begin{equation} \label{eq:Eaboveratio}
\frac{\Esub{x}{U_0(B_t)\1_{\{B_s>m_1(t-s) \; \forall s\in [0,t-s_0]\}}}}{\Esub{x}{U_0(B_t)\1_{\{B_s>m_2(t-s) \; \forall s\in [0,t-s_0]\}}}}
\end{equation}
in terms of $\sup_{s_0\le s \le t}|m_1(s)-m_2(s)|$, for $m_1$ and $m_2$ satisfying~\eqref{eq:mconditions} and $U_0$ satisfying~\eqref{eq:assump_infmass2}. 
He also proves a bound on the tail of
\begin{equation} \label{eq:Eabovetail}
\Esub{x}{U_0(B_t)\1_{\{B_s>m_0(t-s) \; \forall s\in [0,t-s_0]\}}}
\end{equation}
for $m_0$ satisfying~\eqref{eq:mconditions} and $U_0$ satisfying~\eqref{eq:assump_infmass2}. 
By applying the bound on~\eqref{eq:Eaboveratio} with $m_1(t)=m(t)$ and $m_2(t)=L_t$, and using the assumption~\eqref{eq:Dvasymp}, and Proposition~\ref{prop:s0} together with Theorem~\ref{theo:infinitemassconv}, along with the estimates on the tail of~\eqref{eq:Eabovetail}, it follows from exactly the same arguments as in~\cite{Bramson1983} that $|m(t)-L_t|$ is bounded as $t\to \infty$, and then by the same arguments as in~\cite{Bramson1983}, the result follows.
\end{proof}

In~\cite{Bramson1983}, Bramson proves the analogue of Theorem~\ref{theo:infmassfront} for the classical FKPP equation under the following condition: for some $\delta<1/3$, for $t$ sufficiently large,
\begin{equation} \label{eq:bcond}
	b(t):=2^{-1 / 2} \log \left(\int_0^{\infty} y e^{\sqrt 2 y} U_0(y) e^{-y^2 / 2 t} d y+1\right) \leq t^\delta.
	\end{equation}
The following lemma will allow us to instead impose the simpler condition, \eqref{eq:stretched exponential U0 proof section}, on $U_0$.
\begin{lem}\label{lem:stretched exponential implies b condition}
Suppose $U_0$ satisfies Assumption~\ref{assum:standing assumption ic}, and that for some $\gamma<1/2$,
\[
U_0(x)\leq e^{x^{\gamma}-\sqrt{2}x}
\]
for all $x$ sufficiently large.
Define $b(t)$ as in~\eqref{eq:bdefninfmass}; then there exists $\delta<1/3$ such that $b(t)\le t^\delta$ for $t$ sufficiently large.
\end{lem}
In fact, it is straightforward to show that these two conditions are equivalent, and so stating Theorem~\ref{theo:infmassfront} under the condition~\eqref{eq:stretched exponential U0 proof section} rather than~\eqref{eq:bcond} incurs no loss of generality, but we will not need that to prove Theorem~\ref{theo:infmassfront}.
\begin{proof}[Proof of Lemma~\ref{lem:stretched exponential implies b condition}]
Take $\gamma '\in (\gamma,1/2)$; then there exists $A \in (0,\infty)$ such that 
\[
U_0(x)\leq Ax^{-1}e^{x^{\gamma'}-\sqrt{2}x} \quad \forall x\in (0,\infty).
\]
Therefore, for $t>0$,
\begin{equation} \label{eq:daggerU0blem}
\int_0^{\infty} y e^{\sqrt 2 y} U_0(y) e^{-y^2 / (2 t)} d y\leq A\int_0^{\infty}f_t(y)dy,
\end{equation}
where, for $t>0$ and $y\in \R$, we let
\[
f_t(y):=e^{y^{\gamma '}-y^2/(2t)}.
\]
Fixing $t>0$ and differentiating with respect to $y$, we see that for $y\in \R$,
\begin{equation} \label{eq:*U0blem}
\partial_y f_t(y)=(\gamma ' y^{\gamma '-1}-yt^{-1})f_t(y).
\end{equation}
Therefore, we observe that
$\partial_y f_t(y)>0$ for $y\in (0,(\gamma ' t)^{1/(2-\gamma ')})$ and $\partial_y f_t(y)<0$ for $y>(\gamma ' t)^{1/(2-\gamma ')}$.
We deduce that the maximum of $f_t(y)$ is attained at $y=(\gamma ' t)^{1/(2-\gamma ')}$, and so
\[
\sup_{y>0}f_t(y)\leq \exp[(\gamma ' t)^{\gamma '/(2-\gamma ')}].
\]
By~\eqref{eq:*U0blem} and since $\gamma '<1/2$, we also have that $\partial_y f_t(y)\leq -f_t(y)$ for $t\ge 1$ and $y\geq 2 t\ge 2$. 
Hence by Gr\"onwall's inequality, we see that for $t\ge 1$,
\[
f_t(y)\leq e^{-(y-2 t)}f_t(2t)\leq e^{-(y-2 t)}\sup_{z>0}f_t(z)\quad \text{for $y>2 t$.}
\]
We conclude that for $t\ge 1$,
\[
\int_0^{\infty}f_t(y)dy \leq 2 t \sup_{y>0}f_t(y)+\int_{2 t}^{\infty}e^{-(y-2 t)}dy \sup_{z>0}f_t(z)\leq (2 t+1)\exp[(\gamma ' t)^{\gamma'/(2-\gamma ')}].
\]
It therefore follows from~\eqref{eq:daggerU0blem} and~\eqref{eq:bdefninfmass} that for $t$ sufficiently large,
\[
\sqrt 2 b(t)\leq 1+\log A+\log (2t+1)+(\gamma ' t)^{\gamma'/(2-\gamma')}.
\]
Since $\gamma'<1/2$, we see that $\frac{\gamma'}{2-\gamma'}< \frac{1/2}{2-1/2}=1/3$, and the result follows.
\end{proof}

We can now use Theorem~\ref{theo:infinitemassconv} and Proposition~\ref{prop:frontinfiniteinitial} to prove Theorem~\ref{theo:infmassfront}.

\begin{proof}[Proof of Theorem~\ref{theo:infmassfront}]
By Lemma~\ref{lem:stretched exponential implies b condition} and since we assume~\eqref{eq:stretched exponential U0 proof section}, we have that $b(t)$ satisfies~\eqref{eq:bcond} for some $\delta<1/3$.
Recall the definition of $m(t)$ in~\eqref{eq:mtheodef} in the statement of the result.
By an identical proof to the proof of~\cite[Application of Theorems 4 and 5, p.~181]{Bramson1983},
using Theorem~\ref{theo:infinitemassconv} in place of~\cite[Theorem 4]{Bramson1983}
and Proposition~\ref{prop:frontinfiniteinitial} in place of~\cite[Theorem 5]{Bramson1983}, we obtain that
\[
m(t)-L_t\to x_0 \quad \text{as } t\to \infty,
\]
where $x_0\in \R$ is such that $\Pi_{\min}(x+x_0)\sim \sqrt{\frac{2}{\pi}}xe^{-\sqrt{2}x}$ as $x\to \infty$.
Indeed, the result is proved in~\cite{Bramson1983} by determining the asymptotic behaviour of the left-hand side of~\eqref{eq:Dvasymp} for $m$ as defined in~\eqref{eq:mtheodef}, for large $x-m(t)$. The condition~\eqref{eq:bcond} is used to show that $m$ is sufficiently close to linear that the left-hand side of~\eqref{eq:Dvasymp} can be estimated by replacing $m(s)$ with $m(t)s/t$, and that $|m(t)-\sqrt 2 t|$ is sufficiently small that the asymptotics of this simplified expression can then be calculated, so that Proposition~\ref{prop:frontinfiniteinitial} can then be applied. The expression for $x_0$ comes from the result of this calculation and the fact that, as stated in Proposition~\ref{prop:frontinfiniteinitial},~\eqref{eq:Dvasymp} must hold with $v(z)=\Pi_{\min}(z+x_0)$.

Since $\Pi_{\min}(x)=(\sqrt 2 x+1)e^{-\sqrt 2 x}$ for $x>0$ by~\eqref{eq:minimal travelling wave}, it follows that $x_0=\log(\sqrt{\pi})/\sqrt{2}$, which completes the proof.
\end{proof}

\section{Proof of Theorems~\ref{theo:convtoPimin},~\ref{theo:Ltposition} and~\ref{theo:slowerdecay}} \label{sec:mainthmpfs}

We can now use results proved in Sections~\ref{section:properties of free-boundary}-\ref{sec:infiniteinitialmass} to prove Theorems~\ref{theo:convtoPimin},~\ref{theo:Ltposition} and~\ref{theo:slowerdecay}.

\begin{proof}[Proof of Theorem~\ref{theo:convtoPimin}]
\textbf{1.} $\Leftrightarrow $ \textbf{2.}
Suppose $\limsup_{x\ra\infty}\frac{1}{x}\log U_0(x)\leq -\sqrt{2}$;
then for any $\epsilon>0$, we have $U_0(x)\le e^{-(\sqrt 2-\epsilon)x}$ for $x$ sufficiently large, and so
 by the definition of $r_0(U_0)$ in~\eqref{eq:r0U0defn} we have $r_0(U_0)=\sqrt 2$.
By Theorem~\ref{theo:initial condition limsup bdy relation}, it follows that $\limsup_{t\to \infty}\frac{L_t}{t}\le \sqrt 2$.

Conversely, if $\limsup_{t\to \infty}\frac{L_t}{t}\le \sqrt 2$, then by Theorem~\ref{theo:initial condition limsup bdy relation} and since $r\mapsto \frac 1 r + \frac r 2$  is decreasing on $(0,\sqrt 2]$, we must have $r_0(U_0)=\sqrt 2$.
Now suppose (aiming for a contradiction) that $\limsup_{x\ra\infty}\frac{1}{x}\log U_0(x)> -\sqrt{2}$, and so there exist $\epsilon>0$ and $(x_n)_{n=1}^\infty\subset \R$ with $x_n\to \infty$ as $n\to \infty$ such that $U_0(x_n)\ge e^{-(\sqrt 2-\epsilon)x_n}$ $\forall n\in \Nm$.
Then since $U_0$ is non-increasing, for any $r\ge \sqrt 2-\frac 12 \epsilon$ and $n\in \Nm$ sufficiently large we have 
\[\int_{x_n-1}^{x_n}U_0(x)e^{rx}dx
\ge e^{-(\sqrt 2-\epsilon)x_n}e^{r(x_n-1)}
\ge e^{\frac 12 \epsilon x_n}e^{-(\sqrt 2 -\frac 12 \epsilon)}\to \infty
\]
as $n\to \infty$, and so $r_0(U_0)\le \sqrt{2}-\frac 12 \epsilon$, giving us a contradiction.
Therefore $\limsup_{x\ra\infty}\frac{1}{x}\log U_0(x)\leq -\sqrt{2}$.

\noindent \textbf{2.} $\Leftrightarrow $ \textbf{3.}
Recall the definition of $(U^H(t,x),L^H_t)$ in Section~\ref{subsec:notation}.
By Lemma~\ref{lem:less stretching means slower boundary} and then by Lemma~\ref{lem:LHtlower}, for $t\ge 1$ we have
\[
L_t-L_1\ge L^H_{t-1}\ge \sqrt 2(t-1)-\frac 3{2\sqrt 2}\log t -C_0.
\]
It follows that $\liminf_{t\to\infty}\frac{L_t}{t}\ge \sqrt 2$, and so $\limsup_{t\to\infty}\frac{L_t}{t}\le \sqrt 2$ if and only if $\lim_{t\to\infty}\frac{L_t}{t}=\sqrt 2$.

\noindent \textbf{1.} $\Rightarrow $ \textbf{4.}
Suppose $\limsup_{x\ra\infty}\frac{1}{x}\log U_0(x)\leq -\sqrt{2}$.
Then by Theorem~\ref{theo:convergence to the minimal travelling wave for finite initial mass} if~\eqref{eq:finiteinitialmassinsection} holds, and by Theorem~\ref{theo:infinitemassconv} if instead~\eqref{eq:assump_infmass2} holds,
we have $U(t,x+L_t)\to \Pi_{\min}(x)$ uniformly in $x$ as $t\to \infty$.

\noindent \textbf{4.} $\Rightarrow $ \textbf{2.}
Suppose $U(t,x+L_t)\to \Pi_{\min}(x)$ uniformly in $x$ as $t\to \infty$.
Then for $\delta>0$, there exists $T<\infty$ such that $|U(t,x+L_t)-\Pi_{\min}(x)|\le \delta$ $\forall t\ge T, $ $x\in \R$.
Hence by Lemma~\ref{lem:Bramsongronwall} and since $\Pi_{\min}$ is a travelling wave solution of~\eqref{eq:FBP_CDF} with speed $\sqrt 2$, we have 
\[
|U(t+1,x+L_t)-\Pi_{\min}(x-\sqrt 2)|\le e\delta \quad \forall t\ge T,\, x\in \R.
\]
Therefore $U(t+1,y)<1$ for $y>L_t+\sqrt 2+\Pi_{\min}^{-1}(1-e\delta)$, and so $L_{t+1}-L_t\le \sqrt 2+\Pi_{\min}^{-1}(1-e\delta)$.
Since $\delta>0$ was arbitrary, and $\Pi_{\min}^{-1}(1-e\delta)\to 0$ as $\delta \to 0$, it follows that $\limsup_{t\to\infty}\frac{L_t}{t}\le \sqrt 2$.
This completes the proof.
\end{proof} 

\begin{proof}[Proof of Theorem~\ref{theo:Ltposition}]
We begin by assuming that~\eqref{eq:stretched exponential U0} is satisfied for some $\gamma<1/2$.
In the case $\int_0^{\infty} y e^{\sqrt 2 y} U_0(y)  d y<\infty$,
then there exists $b\in \R$ such that $b(t)\to b$ as $t\to \infty$, and so the result follows by combining~\eqref{eq:theoLHtLtconv} in Theorem~\ref{theo:convergence to the minimal travelling wave for finite initial mass}
with Theorem~\ref{theo:bootstrap boundary asymptotics}.
In the case $\int_0^{\infty} y e^{\sqrt 2 y} U_0(y)  d y=\infty$, 
since~\eqref{eq:assump_infmass1} follows from~\eqref{eq:stretched exponential U0},
the result follows from Theorem~\ref{theo:infmassfront}.

All that remains is to check that, in the case $\int_0^{\infty} y e^{\sqrt 2 y} U_0(y)  d y=\infty$, the conclusion $L_t-[\sqrt{2}t-\frac{3}{2\sqrt{2}}\log t]\ra \infty$ as $t\ra\infty$ remains valid without assuming~\eqref{eq:stretched exponential U0}. If $U_0(x)\leq e^{-\sqrt 2 x}$ for all $x$ large enough, then~\eqref{eq:stretched exponential U0} is satisfied. Therefore we can assume, without loss of generality, that there exist arbitrarily large $x$ such that $U_0(x)>e^{-\sqrt 2 x}$.

We now define
\[
\tilde{U}_0(x):=U_0(x)\wedge e^{-\sqrt 2 x} \quad \text{for }x\in \R.
\]
Then there exists an increasing sequence $x_n\ra \infty$ as $n\ra\infty$ with $x_1\ge 2$ such that $\tilde{U}_0(x_n)=e^{-\sqrt 2 x_n}$ and $x_{n+1}-x_n>1$ for all $n\in \mathbb N$. Therefore $\int_0^{\infty} y e^{\sqrt 2 y} \tilde{U}_0(y)  d y\geq \sum_{n\in \mathbb N}e^{-\sqrt 2 x_n}e^{\sqrt{2}(x_n-1)}=\infty$. 

Moreover, $\tilde{U}_0$ satisfies Assumption~\ref{assum:standing assumption ic} and \eqref{eq:stretched exponential U0} (since $\tilde{U}_0(x)\leq e^{-\sqrt 2 x}$). We take $(\tilde{U},L^{\tilde{U}})$ to be the solution to \eqref{eq:FBP_CDF} with initial condition $\tilde{U}_0$. We have $L^{\tilde{U}}_t\leq L_t$ for all $t>0$ by the comparison principle (Proposition \ref{prop:fbpcomparison}), and $L^{\tilde{U}}_t-[\sqrt{2}t-\frac{3}{2\sqrt{2}}\log t]\ra \infty$ as $t\ra \infty$ by the previous parts of Theorem~\ref{theo:Ltposition}.
\end{proof}

\begin{proof}[Proof of Theorem~\ref{theo:slowerdecay}]
First note that since $U_0$ is non-decreasing, for any $b>0$ and $h>0$ we have
\begin{equation} \label{eq:slowdecayequiv}
\lim_{x\to \infty}\frac 1 x \log U_0(x)=-b \quad \Leftrightarrow \quad \lim_{t\to \infty}\frac 1t \log \left(\int_t^{(1+h)t}U_0(y)dy\right)=-b.
\end{equation}
Indeed, take $\varepsilon \in (0,b)$ and take $K<\infty$ sufficiently large that $h(K-1)>1$.
Then if $U_0(x)\in [e^{-(b+\varepsilon)x},e^{-(b-\varepsilon)x}]$ $\forall x\ge K$, it follows that for $t\ge K$,
\[
(b+\varepsilon)^{-1}e^{-(b+\varepsilon)t}(1-e^{-(b+\varepsilon)ht})
\le \int_t^{(1+h)t}U_0(y)dy
\le (b-\varepsilon)^{-1}e^{-(b-\varepsilon)t}.
\]
Conversely, if there exists $x\ge K$ such that $U_0(x)<e^{-(b+\varepsilon)x}$ then
\[
\int_x^{(1+h)x}U_0(y)dy
\le hx e^{-(b+\varepsilon)x},
\]
and if there exists $x\ge K$ such that $U_0(x)>e^{-(b-\varepsilon)x}$ then
\[
\int_{x-1}^{(1+h)(x-1)}U_0(y)dy
\ge \int_{x-1}^{x}U_0(y)dy \ge  e^{-(b-\varepsilon)x}.
\]
The claim~\eqref{eq:slowdecayequiv} follows.

For $t\ge 0$ and $x\in \R$, let $\varphi(t,x)=e^t\Esub{x}{U_0(B_t)}$.
Then by the Feynman-Kac formula in Lemma~\ref{lem:FKforinfinitemass}, we have $U(t,x)\le \varphi(t,x)$ $\forall t\ge 0$, $x\in \R$.
Therefore, recalling the definition of $m(t)$ in~\eqref{eq:mtslowdecay}, we have $L_s\le m(s)$ $\forall s\ge 0$, and so by Lemma~\ref{lem:FKforinfinitemass} again, for $t\ge 0$ and $x\in \R$,
\begin{equation} \label{eq:slowdecayFKsandwich}
\Esub{x}{U_0(B_t)e^{\Leb(\{s\in [0,t]:B_s\ge m(t-s)\})}} \le U(t,x)\le e^t \Esub{x}{U_0(B_t)}.
\end{equation}

Now fix $c>\sqrt{2}$, and suppose $\lim_{x\ra\infty}\frac{1}{x}\log U_0(x)=-c+\sqrt{c^2-2}$.
Let
\[
x_c:=\frac{1}{c-\sqrt{c^{2}-2}} \left(\log \sqrt{c^{2}-2}+\log\big(c-\sqrt{c^{2}-2}\big)\right).
\]
By the definition of $\Pi_c$ in~\eqref{eq:Picdefn} and~\eqref{eq:speed c travelling wave}, we have $e^{(c-\sqrt{c_2-2})x}\Pi_c(x-x_c)\to 1$ as $x\to \infty$.
Then by exactly the same argument as in~\cite[Proof of Theorem~1, p.69-74]{Bramson1983}, using~\eqref{eq:slowdecayequiv}, and using~\eqref{eq:slowdecayFKsandwich} in place of~\cite[(5.4)]{Bramson1983} and Proposition~\ref{prop:sandwichmeansconv} in place of~\cite[Proposition~3.3]{Bramson1983}, we have that $U(t,x+m(t))\ra \Pi_c(x-x_c)$ uniformly in $x$ as $t\ra\infty$.
By Lemma~\ref{lem:mtoL} it follows that $U(t,x+L_t)\ra \Pi_c(x)$ uniformly in $x$ as $t\ra\infty$, and $L_t-m(t)-x_c\to 0$ as $t\to\infty$.

Suppose instead that $U(t,x+L_t)\ra \Pi_c(x)$ uniformly in $x$ as $t\ra\infty$.
By Lemma~\ref{lem:Bramsongronwall}, 
\[
\sup_{x\in \R}|U(t+1,x+c+L_t)-\Pi_c(x)|\le e \sup_{x\in \R}|U(t,x+L_t)-\Pi_c(x)|\to 0
\]
 as $t\to \infty$.
 Therefore, by Lemma~\ref{lem:mtoL} we have $L_{t+1}-L_t-c\to 0$ as $t\to \infty$.
 Since $L_{\lfloor t \rfloor}\le L_t \le L_{\lceil t \rceil}$ for $t\ge t_0+1$ by Lemma~\ref{lem:boundary locally Lipschitz from the left}, it follows that
 \[
 \lim_{t\to\infty}\frac{L_t}t=c.
 \]
 Hence by exactly the same argument as in~\cite[Proof of Theorem~2, p.79-83]{Bramson1983}, using~\eqref{eq:slowdecayFKsandwich}, we have that 
  $\lim_{t\to \infty}\frac 1t \log \left(\int_t^{(1+h)t}U_0(y)dy\right)=-c+\sqrt{c^2-2}$, and by~\eqref{eq:slowdecayequiv} it follows that 
  \[
  \lim_{x\ra\infty}\tfrac{1}{x}\log U_0(x)=-c+\sqrt{c^2-2}. \qedhere
  \]
\end{proof}

\section{Results for the generalised free boundary problem \eqref{eq:generalised FBP_CDF}}\label{section:proof on generalised FBP}

In this section, we state in full our results about the long-term behaviour of solutions of the generalised free boundary problem~\eqref{eq:generalised FBP_CDF}, and prove these results, as well as proving Theorems~\ref{theo:existence uniqueness classical solution V} and~\ref{theo:mapping between FBP and general problem}.

We begin by defining travelling wave solutions.
Recalling~\eqref{eq:Picdefn}, for $c\ge \sqrt 2 $ and $\beta>0$, let
\begin{equation}\label{eq:travelling waves general beta}
	\Pi^{(\beta)}_c(x)=\begin{cases}
		\Pi_c(x)+\frac{\beta}{2}\partial_x\Pi_c(x),\quad & x>0,\\
		1,\quad &x\leq 0.
	\end{cases}
\end{equation}
These are travelling wave solutions of~\eqref{eq:generalised FBP_CDF}, in the sense that for any $c\ge \sqrt 2 $ and $\beta>0$,
\begin{equation}\label{eq:general beta V travelling}
V(t,x):=\Pi^{(\beta)}_c(x-ct)
\end{equation}
is a solution of~\eqref{eq:generalised FBP_CDF}; this can be verified directly from~\eqref{eq:Picdefn},~\eqref{eq:speed c travelling wave} and~\eqref{eq:minimal travelling wave}. 
We have the following result concerning the existence of {\it non-negative} travelling wave solutions of~\eqref{eq:generalised FBP_CDF}; this result is also stated in~\cite[Section 2]{Berestycki2018a}.
\begin{prop}\label{prop:travelling waves general beta}
For $c\ge \sqrt 2 $ and $\beta>0$,
the travelling wave $\Pi^{(\beta)}_c$ defined in~\eqref{eq:travelling waves general beta} is non-negative if and only if $c\geq c^{(\beta)}_{\min}$, where 
\begin{equation} \label{eq:cbetamindef}
c^{(\beta)}_{\min}:=
\begin{cases}
\sqrt{2},\quad &0<\beta \leq \sqrt{2},\\
\frac{\beta}{2}+\frac{1}{\beta},\quad &\beta >\sqrt{2}.
\end{cases}
\end{equation}
Furthermore, for $c\geq c^{(\beta)}_{\min}$, the travelling wave $ \Pi^{(\beta)}_c$ is non-increasing and satisfies $ \Pi^{(\beta)}_c(x)\to 0$ as $x\to \infty$.
For $\beta>0$, let 
\begin{equation} \label{eq:generalbetaminTW}
\Pi^{(\beta)}_{\min}:=\Pi^{(\beta)}_{c^{(\beta)}_{\min}}
\end{equation}
denote the minimal non-negative travelling wave. 
Then there exists $A_\beta\in (0,\infty)$ such that as $x\to \infty$,
\begin{equation} \label{eq:Pibetaminasymptotics}
\Pi^{(\beta)}_{\min}(x)\sim \begin{cases}
A_{\beta}xe^{-\sqrt{2}x}\quad &\text{if }0<\beta<\sqrt{2},\\
A_{\beta}e^{-\beta x}\quad &\text{if }\beta\geq \sqrt{2}.
\end{cases}
\end{equation}
\end{prop}
Let $\tilde \Pi^{(\beta)}_{\min}(x)=e^{c^{(\beta)}_{\min}x}\Pi^{(\beta)}_{\min}(x)$ $\forall x\in \R$.
Then~\eqref{eq:Pibetaminasymptotics} is consistent with {\it pushed} behaviour when $\beta>\sqrt{2}$, since $\tilde \Pi^{(\beta)}_{\min}\in L^2(\Rm)$ (see \cite[(1.16)]{An2023b} and \cite[Remark 1]{Garnier2012}); {\it pushmi-pullyu} behaviour when $\beta=\sqrt{2}$, since $\tilde \Pi^{(\beta)}_{\min}\in L^{\infty}(\Rm)$ but $\tilde \Pi^{(\beta)}_{\min}\notin L^2(\Rm)$ (see \cite[(1.18)]{An2023b}), and {\it pulled} behaviour when $0<\beta<\sqrt{2}$, since $\tilde \Pi^{(\beta)}_{\min}\notin L^{\infty}(\Rm)$ and $\tilde \Pi^{(\beta)}_{\min}\notin L^2(\Rm)$ (see \cite[p.2077]{An2023b}). Note that the above references give conditions on $e^{\frac 12 c^{(\beta)}_{\min}x}\Pi^{(\beta)}_{\min}(x)$; the distinction comes from the $\frac{1}{2}$ in front of the Laplacian in~\eqref{eq:generalised FBP_CDF}.

By combining Theorems~\ref{theo:existence uniqueness classical solution V} and~\ref{theo:mapping between FBP and general problem} with Theorems~\ref{theo:convtoPimin},~\ref{theo:slowerdecay} and~\ref{theo:initial condition limsup bdy relation}, we can characterise the domain of attraction of this minimal travelling wave solution.
\begin{theo} \label{theo:convtoPimin general beta}
Take $\beta>0$, 
suppose that $V_0$ satisfies Assumption~\ref{assum:standing assumption initial condition general beta}, and let $(V(t,x),L_t)$ denote the solution of the free boundary problem~\eqref{eq:generalised FBP_CDF}. 
Define $c^{(\beta)}_{\min}$ as in~\eqref{eq:cbetamindef}.
Then the following are equivalent:
	\begin{enumerate}
		\item $\limsup_{x\ra\infty}\frac{1}{x}\log V_0(x)\leq -\min(\sqrt{2},\frac{2}{\beta})$; \label{enum:bound on tails of U0 general beta} 
		%=\begin{cases}-\sqrt{2},\quad & 0<\beta\leq \sqrt 2 ,\\
		%-\frac{2}{\beta},\quad &\beta>\sqrt{2};\end{cases}$
		\item $\limsup_{t\ra\infty}\frac{L_t}{t}\leq c^{(\beta)}_{\min}$;
		\item $\lim_{t\ra\infty}\frac{L_t}{t}= c^{(\beta)}_{\min}$;\label{enum:convergence of velocity general beta}
		\item $V(t,x+L_t)\ra \Pi^{(\beta)}_{\min}(x)$ uniformly in $x$ as $t\ra\infty$. \label{enum:convergence of profile general beta}
	\end{enumerate}
\end{theo}
We now turn to the asymptotics of the front position for solutions of~\eqref{eq:generalised FBP_CDF}. As for solutions of~\eqref{eq:FBP_CDF}, the behaviour of the front position will depend on whether the initial condition $V_0$ has \textit{finite} or \textit{infinite initial mass}, but the meaning of this will vary between the cases $0<\beta< \sqrt{2}$ and $\beta\geq \sqrt{2}$.
For $\beta>0$ and $V_0$ satisfying Assumption~\ref{assum:standing assumption initial condition general beta}, as in~\eqref{eq:Ibetaintro} we let 
\begin{equation}\label{eq:I integral finite initial mass general beta}
I_\beta=I_\beta(V_0):=
\begin{cases}
\int_0^{\infty}xe^{\sqrt{2}x}V_0(x)dx \quad &\text{if }0<\beta< \sqrt{2}, \\
\int_{-\infty}^{\infty}e^{\frac 2 \beta x}V_0(x)dx \quad &\text{if }\beta\geq  \sqrt{2}.
\end{cases}
\end{equation}
We say that $V_0$ has \textit{finite initial mass} if $I_\beta<\infty$ and $V_0$ has \textit{infinite initial mass} if instead $I_\beta=\infty$.
Note that in the case $\beta< \sqrt 2$, this is the same definition of finite and infinite initial mass as in Definition~\ref{defin: fin init mass}.

By combining Theorems~\ref{theo:existence uniqueness classical solution V} and~\ref{theo:mapping between FBP and general problem} with Theorems~\ref{theo:Ltposition} and~\ref{theo:slowerdecay}, we obtain the following result about the long-term asymptotics of the front position for solutions of~\eqref{eq:generalised FBP_CDF}.
As discussed in Section~\ref{subsec:pushedFBP}, the result is divided into the following three regimes:
\begin{enumerate}
\item $0<\beta <\sqrt{2}$ is the pulled regime;
\item $\beta=\sqrt{2}$ is the pushmi-pullyu regime; and
\item $\beta>\sqrt{2}$ is the pushed regime.
\end{enumerate}
The behaviour of the front position is very different in each of these three cases.
Note that, as mentioned in Section~\ref{subsec:pushedFBP}, the constant term in the asymptotics is given explicitly in the pushmi-pullyu and pushed cases, but in the pulled case the constant term is not explicit in the finite initial mass case.
\begin{theo} \label{theo:Ltposition general beta}
Take $\beta>0$, and suppose that $V_0$ satisfies Assumption~\ref{assum:standing assumption initial condition general beta} and 
\[
\limsup_{x\ra\infty}\tfrac{1}{x}\log V_0(x)\leq -\min(\sqrt{2},\tfrac{2}{\beta}).
\]
Define $I_\beta=I_\beta(V_0)$ as in~\eqref{eq:I integral finite initial mass general beta} and $c^{(\beta)}_{\min}$ as in~\eqref{eq:cbetamindef}.
Let $(V(t,x),L_t)$ denote the solution of \eqref{eq:generalised FBP_CDF} with initial condition $V_0$.
\begin{enumerate}
\item \label{enum:pulled Lt generalised}
Suppose that $0<\beta < \sqrt{2}$. Suppose that for some $\gamma<1/2$,
\begin{equation}\label{eq:stretched exponential U0 general beta}
V_0(x)\leq e^{x^{\gamma}-\sqrt{2}x}
\end{equation}
for all $x$ sufficiently large. Let
\begin{equation}\label{eq:definition of b general beta}
	 b(t)=2^{-1 / 2} \log \left(\frac{2}{\beta}\int_{0}^{\infty} y e^{(\sqrt{2}-\frac{2}{\beta}) y}
\int_{-\infty}^y e^{\frac{2}{\beta}z}V_0(z)dz \, 
 e^{-y^2 / (2 t)} d y+1\right) \quad \text{for }t>0.
	\end{equation}
\commentout{Suppose that for some $\delta<1 / 3$, for $t$ sufficiently large,
	\begin{equation}\label{eq:assumption on b general beta}
	b(t)\le t^\delta.
	%\int_0^{\infty} y e^{\sqrt 2 y} U_0(y) e^{-y^2 / 2 t} d y \leq e^{t^\delta}.
	\end{equation}}
Let
\begin{equation}\label{eq:definition of m general beta}
m(t)=\sqrt{2}t-\frac{3}{2\sqrt{2}}\log t+b(t) \quad \text{for }t>0.
\end{equation}
Then there exists $a=a(V_0)\in \mathbb{R}$ such that
\[
L_t-m(t)\ra a\quad\text{as } t\ra\infty.
\]
In particular, if $I_\beta<\infty$, then there exists $c=c(V_0)\in \mathbb{R}$ such that
\begin{equation}\label{eq:Lt asymp finite init mass beta < sqrt 2}
L_t=\sqrt{2}t-\frac{3}{2\sqrt{2}}\log t+c+o(1)\quad\text{as $t\ra\infty$.}
\end{equation}
If instead $I_\beta=\infty$, then
\[
a(V_0)=-\frac{1}{\sqrt{2}}\log \sqrt{\pi}.
\]
\item \label{enum:pushmi-pullyu Lt generalised}
Suppose that $\beta =\sqrt{2}$. Define $b$ as in \eqref{eq:definition of b general beta}, suppose that~\eqref{eq:stretched exponential U0 general beta} is satisfied for some $\gamma<1/2$, for $x$ sufficiently large, and define $m$ as in \eqref{eq:definition of m general beta}. Then
\begin{equation}\label{eq:asymptotics 1 of Lt beta sqrt 2}
L_t-m(t)\ra -\frac{1}{\sqrt{2}}\log\sqrt{\pi}\quad\text{as } t\ra\infty.
\end{equation}
In particular, if $I_{\sqrt 2}<\infty$, then 
\begin{equation}\label{eq:Lt asymp finite init mass beta sqrt 2}
L_t=\sqrt{2}t-\frac{1}{2\sqrt{2}}\log t+\frac{1}{\sqrt{2}}\Big(\log(\sqrt 2 I_{\sqrt 2})-\log\sqrt{\pi}\Big)+o(1) \quad\text{as } t\ra\infty.
\end{equation}
\item \label{enum:pushed Lt generalised}
Suppose that $\beta>\sqrt{2}$. Define $U_0$ as in~\eqref{eq:U0 formula from V0}, and for $t> 0$, let
%m(t):=\sup\left\{x\in \R:e^t \int_{-\infty}^\infty e^{-\frac{2}{\beta}y}\left[e^{\frac{2}{\beta}L_0}+\frac{2}{\beta}\int_{L_0}^ye^{\frac{2}%%{\beta}z}V_0(z)dz\right]\frac{e^{-(x-y)^2/(2t)}}{\sqrt{2\pi t}}dy\ge 1\right\}.
\begin{equation} \label{eq:mtslowdecay generalised}
m(t):=\sup\left\{x\in \R:e^t\int_{-\infty}^{\infty}U_0(y)\frac{e^{-(x-y)^2/(2t)}}{\sqrt{2\pi t}}dy\ge 1\right\}.
\end{equation}
Then 
\begin{equation}\label{eq:Lt in terms of mt}
L_t-m(t)-\frac{\beta}{2} \left(\log(\beta^2-2)-2\log \beta\right) \rightarrow 0\quad\text{as }t\rightarrow \infty.
\end{equation}
In particular, if $I_{\beta}<\infty$, then 
\begin{equation}\label{eq:Lt asymp finite init mass beta grt sqrt 2}
L_t=c^{(\beta)}_{\min}t+\frac{\beta}{2} \left(\log (\tfrac{2}{\beta}I_{\beta})+\log(\beta^2-2)-2\log \beta\right)+o(1)\quad\text{as } t\ra\infty.
\end{equation}
\end{enumerate}
Now suppose that $\beta>0$ and $I_{\beta}=\infty$. Then (without assuming~\eqref{eq:stretched exponential U0 general beta}):
\begin{enumerate}[a)]
\item if $0<\beta<\sqrt{2}$, then $L_t-[\sqrt{2}t-\frac{3}{2\sqrt{2}}\log t]\ra \infty$ as $t\ra\infty$;\label{enum:infinitely far from pulled beta<sqrt 2}
\item if $\beta=\sqrt{2}$, then $L_t-[\sqrt{2}t-\frac{1}{2\sqrt{2}}\log t]\ra \infty$ as $t\ra\infty$;\label{enum:infinitely far from pulled beta=sqrt 2}
\item if $\beta>\sqrt{2}$, then $L_t-c^{(\beta)}_{\min} t\ra \infty$ as $t\ra\infty$.\label{enum:infinitely far from pulled beta>sqrt 2}
\end{enumerate}
\end{theo}
\begin{rmk}\label{rmk:necessary and sufficient for asymptotics generalised}
Theorem~\ref{theo:Ltposition general beta} tells us that (for some $c=c(V_0)\in \Rm$) we have the asymptotics 
\begin{enumerate}
\item $L_t=\sqrt{2}t-\frac{3}{2\sqrt{2}}\log t+c+o(1)$ as $t\ra\infty$, when $0<\beta<\sqrt{2}$;
\item $L_t=\sqrt{2}t-\frac{1}{2\sqrt{2}}\log t+c+o(1)$ as $t\ra\infty$, when $\beta=\sqrt{2}$; or
\item $L_t=c^{(\beta)}_{\min} t+c+o(1)$ as $t\ra\infty$, when $\beta>\sqrt{2}$,
\end{enumerate}
if and only if the finite initial mass condition $I_\beta<\infty$ holds. Such a necessary and sufficient condition has been previously established in the pulled setting for FKPP equations by Bramson \cite{Bramson1983}, but we are not aware of such a necessary and sufficient characterisation having previously been established in the pushmi-pullyu or pushed settings
(see~\cite{Rothe1981} for sufficient conditions in the pushed case, for reaction-diffusion equations with no free boundary).
\end{rmk}

We further remark that if $\beta=\sqrt{2}+\epsilon$ for small $\epsilon>0$, and $I_{\beta}<\infty$, then Theorem \ref{theo:Ltposition general beta} tells us that
\[
L_t=c^{(\beta)}_{\min}t+\frac{\beta}{2} \left(\log (\tfrac{2}{\beta}I_{\beta})-2\log \beta\right)+\frac{\sqrt{2}+\epsilon}{2}\log (2\sqrt 2 \epsilon+\epsilon^2)+o(1)\quad\text{as } t\ra\infty.
\] 
We see that the final $\frac{\sqrt{2}+\epsilon}{2}\log (2\sqrt 2\epsilon+\epsilon^2)$ term converges to $-\infty$ as $\epsilon \downarrow 0$, reflecting (at least formally) that the pushmi-pullyu ($\beta=\sqrt{2}$) correction from the linear speed (the $-\frac{1}{2\sqrt{2}}\log t$ term) goes to $-\infty$ as $t\ra\infty$.

\begin{rmk}\label{rmk:FKPP asymptotics heavy-tailed pushmi-pullyu}
As an illustrative example of the results in Theorem~\ref{theo:Ltposition general beta}, we now suppose that $V_0(x) \sim Ax^\nu e^{-\min(\sqrt 2,\frac 2{\beta}) x}$ as $x\to \infty$ for some fixed $\nu\in \R$ and $A\in (0,\infty)$. We assume that $\nu\geq -2$ if $0<\beta<\sqrt{2}$, and $\nu\geq -1$ if $\beta\geq \sqrt{2}$; otherwise the asymptotics of $L_t$ are stated above in~\eqref{eq:Lt asymp finite init mass beta < sqrt 2},~\eqref{eq:Lt asymp finite init mass beta sqrt 2} and~\eqref{eq:Lt asymp finite init mass beta grt sqrt 2}, for the $0<\beta<\sqrt{2}$, $\beta=\sqrt{2}$ and $\beta>\sqrt{2}$ cases respectively.

\begin{enumerate}
\item If $0<\beta<\sqrt{2}$, we have:
\begin{enumerate}
\item if $\nu=-2$, then $L_t=\sqrt{2}t-\frac{3}{2\sqrt{2}}\log t+\frac{1}{\sqrt{2}}\log \log t+\frac{1}{\sqrt{2}}\log \left(\frac{A}{\sqrt{\pi}(2-\sqrt{2}\beta)}\right)+o(1)$ as $t\to \infty$.
	\item if $\nu>-2$, then $L_t=\sqrt{2}t+\frac{\nu-1}{2\sqrt{2}}\log t+\frac{1}{\sqrt{2}}\log \left(\frac{2A}{\sqrt{\pi}(2-\sqrt{2}\beta)}\int_{0}^{\infty} y^{1+\nu} e^{-y^2 / 2 } d y\right)+o(1)$ as $t\to \infty$.
\end{enumerate}
\item If $\beta=\sqrt{2}$, we have
\begin{enumerate}
	\item if $\nu=-1$, then $L_t =\sqrt 2 t -\frac{1}{2\sqrt 2} \log t + \frac 1 {\sqrt 2}\log \log t+  \frac 1 {\sqrt 2}\log \left(\frac{A}{2\sqrt{\pi}}\right)+o(1)$ as $t\to \infty$.
		\item if $\nu>-1$, then $L_t =\sqrt 2 t +\frac{\nu}{2\sqrt 2} \log t + \frac{1}{\sqrt{2}}\log\left(\frac{\sqrt 2 A}{\sqrt{\pi}}\int_0^{\infty}y^{\nu}e^{-y^2/2}dy\right) +o(1)$ as $t\to \infty$.
\end{enumerate}
\item If $\beta>\sqrt{2}$, we have
\begin{enumerate}
\item if $\nu=-1$, then $L_t=c^{(\beta)}_{\min}t+\frac{\beta}{2}\log\log t+\frac{\beta}{2}\left[\log\left(2A(\beta^2-2)\right)-3\log \beta\right]+o(1)$ as $t\to \infty$.
\item if $\nu>-1$, then $L_t=c^{(\beta)}_{\min}t+\frac{\beta}{2}(1+\nu)\log t+\frac{\beta}{2}\left[(2+\nu)\log(\frac{\beta}{2}-\frac{1}{\beta})+\log (\frac{4A}{(1+\nu)\beta^2})\right]+o(1)$ as $t\to \infty$.
\end{enumerate}
\end{enumerate}
\end{rmk}

In the remainder of this section, we will prove Theorems~\ref{theo:existence uniqueness classical solution V} and~\ref{theo:mapping between FBP and general problem}, Proposition~\ref{prop:travelling waves general beta}, and Theorems~\ref{theo:convtoPimin general beta} and~\ref{theo:Ltposition general beta}.
We begin by proving the mapping from solutions of~\eqref{eq:FBP_CDF} to solutions of~\eqref{eq:generalised FBP_CDF}, Theorem~\ref{theo:mapping between FBP and general problem}. We then establish Theorem~\ref{theo:existence uniqueness classical solution V} by showing how one can invert this mapping. We will next prove Proposition~\ref{prop:travelling waves general beta}, characterising all the non-negative travelling wave solutions of~\eqref{eq:generalised FBP_CDF}. We will then prove a lemma relating exponential moments of $V_0$ and $U_0$ under the mapping~\eqref{eq:U0 formula from V0}. Following this, we will prove Theorem~\ref{theo:convtoPimin general beta}, and then Theorem~\ref{theo:Ltposition general beta}.

\subsection{Proof of Theorem \ref{theo:mapping between FBP and general problem}}

The following Feynman-Kac formula for solutions of~\eqref{eq:FBP} with bounded initial conditions will be used in the proof of Theorem~\ref{theo:mapping between FBP and general problem}.
Recall from Section~\ref{subsec:notation} that under $\mathbb P_x$, $(B_t)_{t\ge 0}$ is a Brownian motion started at $x$.
\begin{lem} \label{lem:FKformulaforu0}
Suppose $u_0\in L^\infty(\R)\cap L^1(\R)$ with $\int_{-\infty}^\infty u_0(x)dx =1$, and let $(u(t,x),L_t)$ denote the solution of~\eqref{eq:FBP} with initial condition $u_0(x)dx$.
Let $L_0:=\inf\{x\in \R:U_0(x)<1\}\in \{-\infty\}\cup \R$.
For $t>0$, define the stopping time
\begin{equation}\label{eq:stopping time F-K for u}
\tau^t:=\inf\{s>0:B_s\leq L_{(t-s)\vee 0}\}.
\end{equation}
Then for $t>0$ and $x>L_t$,
\[
u(t,x)=e^t\expE_x[\Ind_{\{\tau^t>t\}} u_0(B_t)]. 
\]
\end{lem}
\begin{proof}
For $x\in \R$, let $U_0(x)=\int_x^\infty u_0(y)dy$; then let $(U(t,x),L_t)$ denote the solution of~\eqref{eq:FBP_CDF} with initial condition $U_0$ (using Proposition~\ref{prop:fbpsoln}(v) to see that the free boundary $L_t$ is the same as the free boundary for the solution of~\eqref{eq:FBP}). By Proposition~\ref{prop:fbpsoln}(v) we have that $u=-\partial_x U$.

Fix $t>0$ and $x>L_t$, and take $\delta \in (0,t)$.
Since $u \in C^{1,2}(\{(t,x):t>0,\, x>L_t\})\cap C( (0,\infty) \times \R)$ and $u$ is bounded on $\{(s,y):s\ge \delta, \, y\in \R\}$ (by Proposition~\ref{prop:fbpsoln}(v)),  we can apply It\^o's lemma to see that
\[
( e^su(t-s,B_{s}):0\leq s\leq  \tau^t \wedge (t-\delta)  ) 
\]
is a $\Pm_x$-martingale. Therefore we have 
\[
u(t,x)=\expE_x[e^{\tau^t\wedge (t-\delta)}u(t-(\tau^t\wedge (t-\delta)),B_{\tau^t\wedge(t-\delta)})].
\]
Since $u$ vanishes along $\{(t,L_t):t>0\}$, we have
\begin{equation} \label{eq:udeltaFK}
u(t,x)=e^{t-\delta}\expE_x[\Ind_{\{\tau^t>t-\delta\}}u(\delta,B_{t-\delta})].
\end{equation}

Recall from Proposition~\ref{prop:fbpsoln}(ii) that $L_s\to L_0$ as $s\to 0$.
Therefore $s\mapsto L_s$ is continuous on $[0,\infty)$, and so the set $\{(s,y)\in (0,\infty)\times \Rm:y>L_{(t-s)\vee 0}\}$ is open.
We have that for fixed $(t,x)$ with $x>L_t$, $v(s,dy):=\Pm_x(B_{s}\in dy,\tau^t>s)$ satisfies the heat equation in a weak sense on $\{(s,y)\in (0,\infty)\times \Rm:y>L_{(t-s)\vee 0}\}$.
Recall that weak solutions of the heat equation are $C^{\infty}$ on any open set (e.g.~by H\"ormander's theorem~\cite{Hormander1967}).
Hence $v(s,dy)$
has a $C^{\infty}$ density on $\{(s,y)\in (0,\infty)\times \Rm:y>L_{(t-s)\vee 0}\}$
 with respect to Lebesgue measure, which we denote by $w(s,y)$; 
 %note that we do not require or claim anything about the boundary conditions for $w$. 
 for convenience we set $w(s,y)=0$ for $(s,y)\in (0,\infty)\times \Rm$ with $y\le L_{(t-s)\vee 0}$.

We now establish that
\begin{equation}\label{eq:convergence of u(t,x) to F-K of u0}
u(t,x)=e^{t-\delta}\int_{L_{\delta}}^{\infty}w(t-\delta,y)u(\delta,y)dy\ra e^t\int_{L_{0}}^{\infty}w(t,y)u_0(y)dy,
\end{equation}
as $\delta \ra 0$. We already have the above equality for all $\delta\in (0,t)$ by~\eqref{eq:udeltaFK}.

We now an fix arbitrary $\epsilon>0$. Since $u_0=-U_0'$ is bounded and $U(t,x)\ra U_0(x)$ for Lebesgue-almost every $x>L_0$ (by Proposition~\ref{prop:fbpsoln}(vi)), we can find $L_0<x_{\epsilon}^-<x_{\epsilon}^+<\infty$ such that $U_0(x_{\epsilon}^-)>1-\epsilon$, $U_0(x_{\epsilon}^+)<\epsilon$, and $U(t,x_{\epsilon}^{\pm})\ra U_0(x_{\epsilon}^{\pm})$ as $t\ra 0$. We then take $\phi\in C_c^{\infty}((L_0,\infty);[0,1])$ with $\phi(x)=1$ for all $x\in (x_{\epsilon}^-,x_{\epsilon}^+)$. 
We observe that there exists $\bar \delta \in (0,t)$ such that $w(s,\cdot)\phi(\cdot)\in C_b(\Rm)$ $\forall s\in [t-\bar \delta ,t]$, and moreover,
\begin{align}\label{eq:map time to profile of w phi}
W:[t-\bar \delta ,t]&\to C_b(\Rm) \notag \\
s &\mapsto  w(s,\cdot)\phi(\cdot)
\end{align} 
is continuous with respect to the uniform norm on $C_b(\Rm)$, using the fact that $L_s\ra L_0$ as $s\ra 0$ (by Proposition \ref{prop:fbpsoln}\eqref{enum:Lt_time0_limit}), $w(s,y)$ is jointly continuous in $(s,y)$ on $\{(s,y)\in (0,\infty)\times \Rm:y>L_{(t-s)\vee 0}\}$, and the support of $\phi$ is a compact subset of $(L_0,\infty)$. 

By the triangle inequality, for $\delta \in (0,\bar \delta)$,
\begin{align} \label{eq:convwithphi}
&\left \lvert e^{t-\delta}\int_{L_{\delta}}^{\infty}\phi(y)w(t-\delta,y)u(\delta,y)dy- e^t\int_{L_0}^{\infty}\phi(y)w(t,y)u_0(y)dy\right\rvert \notag \\
&\leq  e^{t-\delta}\int_{L_{\delta}}^{\infty}\lvert \phi(y)w(t-\delta,y)-\phi(y)w(t,y)\rvert u(\delta,y)dy \notag \\
&\;\;+\left \lvert e^{t-\delta}\int_{L_{\delta}}^{\infty}\phi(y)w(t,y)u(\delta,y)dy- e^t\int_{L_0}^{\infty}\phi(y)w(t,y)u_0(y)dy\right\rvert .
\end{align}
As $\delta \ra 0$, the first term on the right-hand side of~\eqref{eq:convwithphi} converges to $0$ since $W$ defined in~\eqref{eq:map time to profile of w phi} is continuous with respect to the uniform norm,
and since $\int_{L_{\delta}}^{\infty}u(\delta,y)dy=1$ $\forall \delta>0$.
The second term on the right-hand side of~\eqref{eq:convwithphi} converges to $0$ as $\delta \ra 0$ because $w(t,\cdot)\phi(\cdot)\in C_b(\Rm)$ and $u(\delta,y)dy$ converges weakly to $u_0(y)dy$.

Since $w$ is bounded by the density of Brownian motion without absorption, there exists $C<\infty$ such that $w(s,y)\leq C$ for all $(s,y)\in [t-\bar \delta ,t]\times \Rm$.
Therefore, since $\phi=1$ on $(x_{\epsilon}^-,x_{\epsilon}^+)$, and then since $u=-\partial_x U$,
for $\delta \in (0,\bar \delta)$ we have
\begin{align} \label{eq:difftonophidelta}
\left\lvert \int_{L_{\delta}}^{\infty}w(t-\delta,y)u(\delta,y)dy-\int_{L_{\delta}}^{\infty}\phi(y)w(t-\delta,y)u(\delta,y)dy\right\rvert
&\leq C\int_{L_{\delta}}^{x^-_{\epsilon}}u(\delta,y)dy+C\int_{x_{\epsilon}^+}^{\infty}u(\delta,y)dy \notag \\
&=C[1-U(\delta,x^-_{\epsilon})+U(\delta,x_{\epsilon}^+)],
\end{align}
and similarly, since $u_0=-U'_0$,
\begin{align} \label{eq:difftonophi0}
\left\lvert \int_{L_{0}}^{\infty}w(t,y)u_0(y)dy-\int_{L_{0}}^{\infty}\phi(y)w(t,y)u_0(y)dy\right\rvert
&\leq C\int_{L_{0}}^{x^-_{\epsilon}}u_0(y)dy+C\int_{x_{\epsilon}^+}^{\infty}u_0(y)dy \notag \\
&=C[1-U_0(x^-_{\epsilon})+U_0(x_{\epsilon}^+)].
\end{align}

Putting~\eqref{eq:convwithphi},~\eqref{eq:difftonophidelta} and~\eqref{eq:difftonophi0} together with the triangle inequality and using the convergences $U(\delta,x_{\epsilon}^{\pm})\ra U_0(x_{\epsilon}^{\pm})$ as $\delta\ra 0$, and then using the inequalities $U_0(x^-_{\epsilon})>1-\epsilon$, $U_0(x^+_{\epsilon})<\epsilon$, we see that
\[
\limsup_{\delta\ra 0}\left\lvert e^{t-\delta}\int_{L_{\delta}}^{\infty}w(t-\delta,y)u(\delta,y)dy- e^t\int_{L_0}^{\infty}w(t,y)u_0(y)dy\right\rvert \leq 2Ce^t[1-U_0(x_{\epsilon}^-)+U_0(x_{\epsilon}^+)]\leq 4Ce^t\epsilon.
\]
Now since $\epsilon>0$ was arbitrary, we have established \eqref{eq:convergence of u(t,x) to F-K of u0}.
The result now follows from the definition of $w$.
\end{proof}

The following elementary lemma will be used in the proofs of Theorems~\ref{theo:existence uniqueness classical solution V} and~\ref{theo:mapping between FBP and general problem}.

\begin{lem} \label{lem:assumpV0impliesU0}
Suppose $V_0:\R\to [0,1]$ is a c\`adl\`ag function, and define $U_0$ as in~\eqref{eq:U0 formula from V0}. Then 
\begin{enumerate}[(i)]
\item $U_0$ is differentiable, and $V_0(x)=U_0(x)+\frac{\beta}{2}U'_0(x)$ $\forall x\in \R$;
\item $U_0'$ is bounded; 
\item For any $x\in \R$, if $V_0$ is continuous at $x$ then $U_0'$ is continuous at $x$;
\item If $V_0$ satisfies Assumption~\ref{assum:standing assumption initial condition general beta}, 
then $U_0$ satisfies Assumption~\ref{assum:standing assumption ic}.
\end{enumerate}

\end{lem}
\begin{proof}
We begin by proving~(i).
For succinctness we define $k:=\frac{2}{\beta}>0$. Then by~\eqref{eq:U0 formula from V0}, for $x\in \R$,
\begin{equation} \label{eq:U0kaboveL0}
U_0(x)=ke^{-kx}\int_{-\infty}^xe^{kz}V_0(z)dz.
\end{equation}
Differentiating, we obtain
\begin{equation}\label{eq:derivative of U0 proof of transform}
U_0'(x)=-k^2e^{-kx}\int_{-\infty}^xe^{kz}V_0(z)dz+kV_0(x).
\end{equation}
Therefore, for $x\in \R$ we calculate 
\[
U_0(x)+\frac{\beta}{2}U'_0(x)
=U_0(x)-ke^{-kx}\int_{-\infty}^xe^{kz}V_0(z)dz+V_0(x)
=V_0(x),
\]
where the second equality follows from~\eqref{eq:U0 formula from V0}; this establishes~(i). 
We can see~(ii) and~(iii) directly from~\eqref{eq:derivative of U0 proof of transform}, using that
$0\le V_0 \le 1$. 

From now on, suppose $V_0$ satisfies Assumption~\ref{assum:standing assumption initial condition general beta}. Then for any $\epsilon>0$ there exists $a\in (0,\infty)$ such that $V_0(y)<\epsilon$ for all $y>a$ and $V_0(y)>1-\epsilon$ for all $y<-a$. Hence for any $x\in \R$, 
\[
\int_{-\infty}^xe^{\frac{2}{\beta}z}V_0(z)dz\leq  \int_{-\infty}^a e^{\frac{2}{\beta}z}dz + \epsilon \int_{-\infty}^xe^{\frac{2}{\beta}z}dz = \frac{\beta}{2} e^{\frac{2}{\beta}a}+  \frac{\beta}{2} \epsilon e^{\frac{2}{\beta}x}.
\]
Moreover, for any $x<-a$,
\[
\int_{-\infty}^xe^{\frac{2}{\beta}z}V_0(z)dz\geq  (1-\epsilon)\int_{-\infty}^x e^{\frac{2}{\beta}z}dz= (1-\epsilon)\frac{\beta}2 e^{\frac{2}{\beta}x}.
\]
Therefore, by~\eqref{eq:U0 formula from V0} and since $\epsilon>0$ was arbitrary, we have 
\[
\limsup_{x\ra\infty}U_0(x)\leq 0 \quad \text{and}\quad \liminf_{x\ra -\infty}U_0(x)\geq 1.
\]
On the other hand, we can see directly from~\eqref{eq:U0 formula from V0} that since $0\le V_0\le 1$ we have $0\le U_0\le 1$.

Since $U_0$ is continuous,
to establish~(iv) it remains to show that $U_0$ is non-increasing. 
Since $V_0$ is non-increasing, we have that for $x\in \R$,
\[
\int_{-\infty}^xe^{kz}V_0(z)dz\geq V_0(x)\int_{-\infty}^xe^{kz}dz=V_0(x)k^{-1}e^{kx}.
\]
Therefore, substituting into~\eqref{eq:derivative of U0 proof of transform}, 
we see that $U_0'(x)\le 0$, which completes the proof.
\end{proof}

We can now use Lemmas~\ref{lem:FKformulaforu0} and~\ref{lem:assumpV0impliesU0} to prove Theorem~\ref{theo:mapping between FBP and general problem}.
\begin{proof}[Proof of Theorem~\ref{theo:mapping between FBP and general problem}]
Define $U_0$ as~\eqref{eq:U0 formula from V0}; by Lemma~\ref{lem:assumpV0impliesU0}(iv) we have that $U_0$ satisfies Assumption~\ref{assum:standing assumption ic}. 
Now let $(U(t,x),L_t)$ denote the unique classical solution of~\eqref{eq:FBP_CDF} with initial condition $U_0$, and define $V$ in terms of $U$ as in~\eqref{eq:V defined in terms of U}. 
Let $L_0:=\inf\{x\in \R:V_0(x)<1\}\in \{-\infty\}\cup \R$.

Recall the definition of a classical solution of~\eqref{eq:generalised FBP_CDF} from the start of Section~\ref{subsec:pushedFBP}.
By the definition of a classical solution of~\eqref{eq:FBP_CDF} from the start of Section~\ref{subsec:mainresults}, condition~(i) is satisfied.
By Proposition~\ref{prop:fbpsoln}(v), we have $V\in C^{1,2} (\{ (t,x): t>0, \, x>L_t \}) \cap C( (0,\infty)\times \R)$. We now check that $(V,L)$ satisfies the remaining conditions~(ii)-(iv).

From the fact that solutions of the heat equation are $C^{\infty}$ on any open set (e.g.~by classical parabolic regularity~\cite{Evans1998} or by H\"ormander's theorem~\cite{Hormander1967}), we know that $U\in C^{\infty}(\{(t,x):t>0,\,x>L_t\})$. Then on $\{t>0,\,x>L_t\}$ we have that 
\[
\partial_tV=\partial_tU+\tfrac{\beta}{2}\partial_x\partial_tU=\tfrac{1}{2}\Delta U+U+\tfrac{\beta}{2}(\tfrac{1}{2}\partial_x\Delta U+\partial_x U)=\tfrac{1}{2}\Delta V+V.
\]

We now check that $V$ satisfies the initial condition in~\eqref{eq:generalised FBP_CDF}. Recall from Proposition~\ref{prop:fbpsoln}(v) that letting  $u(t,x):=-\partial_xU(t,x)$ for $t>0$, $x\in \R$, we have that $u$ satisfies~\eqref{eq:FBP} with initial condition $u_0:=-U_0'$.
%Therefore $u(t,x)dx\ra u_0(x)dx$ as $t\to 0$ in the sense of weak convergence of measures. 
We also note from Lemma~\ref{lem:assumpV0impliesU0}(ii) that $u_0$ is bounded. 
Therefore, by Lemma~\ref{lem:FKformulaforu0}, we have that for $t>0$ and $x>L_t$,
\begin{equation} \label{eq:FKforu}
u(t,x)=e^t\expE_x[\Ind_{\{\tau^t>t\}} u_0(B_t)],
\end{equation}
where $\tau^t$ is defined in~\eqref{eq:stopping time F-K for u}.
By taking the $t\ra 0$ limit (using that $s\mapsto L_s$ is continuous on $[0,\infty)$ by Proposition~\ref{prop:fbpsoln}(ii), and using again that $u_0$ is bounded), we obtain that
for any $x>L_0$ such that $u_0$ is continuous at $x$,
\begin{equation}\label{eq:generalised proof convergence u init cond}
	u(t,x)\ra u_0(x) \quad \text{as }t\to 0.
\end{equation}

By Lemma~\ref{lem:assumpV0impliesU0}(iii), and since $V_0$ is non-increasing and therefore continuous a.e., we have that $u_0$ is continuous Lebesgue-almost everywhere.
Furthermore, by Proposition~\ref{prop:fbpsoln}(vi),
we know that $U(t,x) \to U_0(x)$ as $t \to 0$ at all continuity points $x$ of $U_0$. Thus, by~\eqref{eq:V defined in terms of U}, for Lebesgue-almost every $x>L_0$,
\[
V(t,x)=U(t,x)+\tfrac{\beta}{2}\partial_xU(t,x)\ra U_0(x)+\tfrac{\beta}{2}U'_0(x) \quad \text{as }t\to 0.
\]
By Lemma~\ref{lem:assumpV0impliesU0}(i), it follows that $V(t,x)\to V_0(x)$ as $t\to 0$ for Lebesgue-almost every $x>L_0$.

Since $L_t\ra L_0$ as $t\ra 0$ (by Proposition \ref{prop:fbpsoln}\eqref{enum:Lt_time0_limit}), 
we have trivially by the definition of $V$ in~\eqref{eq:V defined in terms of U} that
$V(t,x)\to V_0(x)$ as $t\to 0$ for all $x<L_0$.
By~\eqref{eq:FKforu} and since $u_0$ is bounded (by Lemma~\ref{lem:assumpV0impliesU0}(ii)), we have that $u$ is bounded on $(0,T)\times \R$ for any $T\in (0,\infty)$.
Therefore, using~\eqref{eq:V defined in terms of U} and since $0\le U\le 1$, we have that
\begin{equation} \label{eq:Visbounded}
\sup\{|V(t,x)|:t\in (0,T), \, x\in \R\}<\infty 
\end{equation}
for any $T\in (0,\infty)$, and so $V(t,\cdot)\to V_0(\cdot)$ in $L^1_{\mathrm{loc}}$ as $t\to 0$.
We have therefore established the initial condition in~\eqref{eq:generalised FBP_CDF}.

We now take $t^*>0$ and $(t_n,x_n)\to (t^*,L_{t^*})$ as $n\to \infty$ with $x_n>L_{t_n}$ $\forall n\in \mathbb N$.
We will show that
$\partial_x V(t_n,x_n)\to -\beta$ as $n\to \infty$; this will establish condition~(iii) in the definition of a classical solution of~\eqref{eq:generalised FBP_CDF}.
On $\{t>0,\,x>L_t\}$, using that $\partial_tU=\frac{1}{2}\Delta U+U$, we calculate
\[
\partial_xV=\partial_xU+\tfrac{\beta}{2}\Delta U=\partial_xU-\beta(U-\partial_tU).
\]
By Proposition~\ref{prop:fbpsoln}(v),
we know that $\partial_xU(t_n,x_n) \ra 0$ and $U(t_n,x_n)\ra 1$ as $n\to \infty$. Therefore, in order to see that $\partial_xV(t_n,x_n)\ra -\beta$ as $n\to \infty$, it suffices to show that
\begin{equation}\label{eq:time derivative of U at t0}
\partial_tU(t_n,x_n)\ra 0 \quad  \text{as }n\to \infty.
\end{equation}

We now define, for $\epsilon\in [0,1)$ and $t>0$, 
\[
\gamma(\epsilon,t):=\inf\{x>L_t:U(t,x)<1-\epsilon\}.
\]
Then by~\cite[Theorem 1]{Chen2022} and the penultimate sentence in~\cite[Proof of Proposition 3, p.4708]{Chen2022}, we have that for some $\delta>0$,
\begin{equation} \label{eq:gammaisC1}
\gamma\in C^{1}([0,\delta)\times (t^*-\delta,t^*+\delta)).
\end{equation}
Therefore, for $\epsilon>0$ sufficiently small and $t\in (t^*-\delta,t^*+\delta)$, we can calculate
\[
0=\frac{d}{dt}U(t,\gamma(\epsilon,t))=\partial_tU(t,\gamma(\epsilon,t))+\partial_xU(t,\gamma(\epsilon,t))\partial_t\gamma (\epsilon,t).
\]
For $n\in \mathbb N$, let $\epsilon_n=1-U(t_n,x_n)$; note that $\epsilon_n\to 0$ as $n\to \infty$.
Then by Proposition~\ref{prop:fbpsoln}(iv), we have $\gamma(\epsilon_n,t_n)=x_n$.
Therefore, for $n$ sufficiently large,
\[
\partial_t U(t_n,x_n)=-\partial_x U(t_n,x_n)\partial_t \gamma(\epsilon_n,t_n).
\]
By~\eqref{eq:gammaisC1}, we have that
\[
\sup_{(\epsilon,t)\in [0,\delta/2]\times [t^*-\delta/2,t^*+\delta/2]}|\partial_t \gamma(\epsilon,t)|<\infty,
\]
and so since $\partial_xU(t_n,x_n) \ra 0$ as $n\to \infty$,~\eqref{eq:time derivative of U at t0} holds, which completes the proof that $\partial_x V(t_n,x_n)\to -\beta$ as $n\to \infty$.

To show that $(V,L)$ is a classical solution of~\eqref{eq:generalised FBP_CDF}, it remains to show that $0\le V\le 1$.
We now complete the proof by showing
that $V(t,\cdot)$ is non-increasing and $0\le V(t,\cdot)\le 1$ for $t>0$. We have the following facts about $V$:
\begin{enumerate}
\item $V\in C^{1,2}(\{(t,x):t>0,\, x>L_t\})\cap C( (0,\infty) \times \R)$, by~\eqref{eq:V defined in terms of U} and Proposition~\ref{prop:fbpsoln}(v).
\item $V(t,x)\le 1$ $\forall t>0$, $x\in \R$, by~\eqref{eq:V defined in terms of U} and since $U\le 1$ and $\partial_x U\le 0$ by Proposition~\ref{prop:fbpsoln}(i) and~(iii).
Moreover, for any $T\in (0,\infty)$, $V$ is bounded on $(0,T)\times \R$ by~\eqref{eq:Visbounded}.
\item We have already established that $V(t,\cdot)\ra V_0(\cdot)$ in $L^1_{\mathrm{loc}}$ as $t\ra 0$. 
\end{enumerate}
For $t>0$, define $\tau^t$ as in~\eqref{eq:stopping time F-K for u},
and for $y\ge 0$, define
\[
\tau^t_y:=\inf\{s>0:B_s\leq L_{(t-s)\vee 0}+y\},
\]
so that $\tau^t_0=\tau^t$.
Using facts 1-3 above, we can apply the Feynman-Kac formula in~\cite[Proposition~3.1, condition~2]{Berestycki2018} to $V$, to obtain that for any $y\ge 0$ and $x>L_t$,
\begin{equation} \label{eq:FKforVwithy}
V(t,x)=\expE_x[\1_{\{\tau^t_y\ge t\}}e^t V_0(B_t)+\Ind_{\{\tau^t_y< t\}}e^{\tau^t_y}V(t-\tau^t_y,B_{\tau^t_y})].
\end{equation}
Setting $y=0$ in~\eqref{eq:FKforVwithy}, and using that $V(t-s,L_{t-s})=1$ for $s\in [0,t)$, we have
\begin{equation} \label{eq:FKforV}
V(t,x)=\expE_x[\1_{\{\tau^t\ge t\}}e^tV_0(B_t)+\Ind_{\{\tau^t< t\}}e^{\tau^t}],
\end{equation}
and the non-negativity of $V$ immediately follows.

Now take $t>0$, $x>L_t$ and $\delta>0$. By~\eqref{eq:FKforVwithy}, and then
by coupling a Brownian motion started from $x$ with a Brownian motion started from $x+\delta$ by simple translation, we have
\[
\begin{split}
V(t,x+\delta)&=\expE_{x+\delta}[\1_{\{\tau^t_{\delta}\ge t\}}e^t V_0(B_t)+\Ind_{\{\tau^t_{\delta}< t\}}e^{\tau^t_{\delta}}V(t-\tau^t_{\delta},B_{\tau^t_{\delta}})]\\
&=\expE_{x}[\1_{\{\tau^t\ge t\}}e^t V_0(B_t+\delta)+\Ind_{\{\tau^t< t\}}e^{\tau^t}V(t-\tau^t,B_{\tau^t}+\delta)]\\
&\le \expE_{x}[\1_{\{\tau^t\ge t\}}e^t V_0(B_t)+\Ind_{\{\tau^t< t\}}e^{\tau^t}]\\
&=V(t,x),
\end{split}
\]
where the above inequality follows because $V_0$ is non-increasing and $V(s,y)\le 1$ $\forall s>0$, $y\in \R$, and the last line follows from~\eqref{eq:FKforV}. Since $V(t,x)=1$ for $x\le L_t$, it follows that $V(t,\cdot)$ is non-increasing, which completes the proof.
\end{proof}

\subsection{Proof of Theorem \ref{theo:existence uniqueness classical solution V}}

Our strategy will be to invert the mapping given by Theorem \ref{theo:mapping between FBP and general problem}. We accomplish this in the following proposition.
Recall the definition of a classical solution of~\eqref{eq:generalised FBP_CDF} from the start of Section~\ref{subsec:pushedFBP}, and the definition of a classical solution of~\eqref{eq:FBP_CDF} from the start of Section~\ref{subsec:mainresults}.
\begin{prop}\label{prop:mapping inversion}
Let $\beta>0$, suppose $V_0$ satisfies Assumption~\ref{assum:standing assumption initial condition general beta}, and suppose that $(V(t,x),L_t)$ is a classical solution of~\eqref{eq:generalised FBP_CDF} with initial condition $V_0$. For $t>0$, let
\begin{equation} \label{eq:UtfromVtclassical}
U(t,x):=
\frac{2}{\beta}e^{-\frac{2}{\beta}x}\int_{-\infty}^xe^{\frac{2}{\beta}z}V(t,z)dz \quad \forall x\in \R.
\end{equation}
Then $(U(t,x),L_t)$ is a classical solution of \eqref{eq:FBP_CDF} with initial condition, $U_0$, given by~\eqref{eq:U0 formula from V0}.
\end{prop}

Before proving Proposition~\ref{prop:mapping inversion}, we show that it implies Theorem~\ref{theo:existence uniqueness classical solution V}.

\begin{proof}[Proof of Theorem~\ref{theo:existence uniqueness classical solution V}]
Take $V_0$ satisfying Assumption~\ref{assum:standing assumption initial condition general beta}.
By Theorem~\ref{theo:mapping between FBP and general problem}, a classical solution of~\eqref{eq:generalised FBP_CDF} exists; it remains to show uniqueness.

Suppose $(V,L)$ and $(\tilde{V},\tilde{L})$ are two classical solutions of~\eqref{eq:generalised FBP_CDF}, both with the same initial condition $V_0$. 
Then by Proposition~\ref{prop:mapping inversion}, defining $U$ as in~\eqref{eq:UtfromVtclassical}, and defining $\tilde U$ as in~\eqref{eq:UtfromVtclassical} with $V$ replaced by $\tilde V$, we have that
$(U,L)$ and $(\tilde{U},\tilde{L})$ are classical solutions of~\eqref{eq:FBP_CDF}, with the same initial condition, $U_0$, 
given by~\eqref{eq:U0 formula from V0}.
By Theorem~\ref{theo:mapping between FBP and general problem}, we have that $U_0$
satisfies Assumption \ref{assum:standing assumption ic}. 
Therefore, by Proposition~\ref{prop:fbpsoln}, $L=\tilde{L}$ and $U=\tilde{U}$. 
By Lemma~\ref{lem:assumpV0impliesU0}(i),
it then immediately follows that $V=\tilde{V}$, which completes the proof.
\end{proof}

We now prove Proposition~\ref{prop:mapping inversion}.

\begin{proof}[Proof of Proposition \ref{prop:mapping inversion}]
It is immediately clear from~\eqref{eq:UtfromVtclassical} and the definition of a classical solution of~\eqref{eq:generalised FBP_CDF} (at the start of Section~\ref{subsec:pushedFBP}) that:
\begin{enumerate}
\item  $L_t\in \R$ $\forall t>0$ and $t\mapsto L_t$ is continuous on $(0,\infty)$;
\item  $U:(0,\infty) \times \R  \to [0, 1]$ with $U \in C( (0,\infty) \times \R)$;
\item $U(t,x)=1$ for $t>0$, $x\leq L_t$;
\item $\partial_xU(t,L_t)=0$ for $t>0$;
\item $U(t,\cdot ) \to U_0(\cdot)$ in $L^1_{\mathrm{loc}}$ as $t\to 0$. 
\end{enumerate}
All that remains is to check that $U \in C^{1,2}(\{(t,x):t>0,\, x>L_t\})$ with $\partial_tU=\frac{1}{2}\Delta U +U$ on $\{(t,x):t>0,x>L_t\}$.

Fix $t_0>0$. By~(i) and~(iii) in the definition of a classical solution of~\eqref{eq:generalised FBP_CDF}, there exists $\delta\in (0,t_0)$ such that
\begin{equation} \label{eq:dxVnear-beta}
\partial_x V(t,x)\in [-2\beta,-\tfrac{1}2 \beta] \quad \forall t\in [t_0-\delta,t_0+\delta], \, x\in (L_t,L_t+\delta].
\end{equation}
For $n\in \mathbb N$ and $t>0$, let
\begin{equation} \label{eq:ellndefn}
\ell_n(t):=\inf\{x>L_t:V(t,x)<1-n^{-1}\}.
\end{equation}
Then by~\eqref{eq:dxVnear-beta} and since $V(t,L_t)=1$ $\forall t>0$, there exists $n_0\in \mathbb N$ such that for $n\ge n_0$ we have $\ell_n(t)\in (L_t,L_t+\delta]$ $\forall t\in [t_0-\delta,t_0+\delta]$. Moreover, for each $t\in [t_0-\delta,t_0+\delta]$ we have $\ell_n(t)\to L_t$ as $n\to \infty$.

Since solutions of the heat equation are $C^{\infty}$ on any open set (e.g.~by classical parabolic regularity~\cite{Evans1998} or by H\"ormander's theorem~\cite{Hormander1967}), we have that $V\in C^{\infty}(\{(t,x):t>0,\,x>L_t\})$.
Therefore, for $n\ge n_0$, by the implicit function theorem and since $\partial_x V<0$ on $\{(t,x):t\in [t_0-\delta,t_0+\delta],(L_t,L_t+\delta] \}$ by~\eqref{eq:dxVnear-beta},
\begin{equation}\label{eq:smoothness of level set}
 \ell_n\in C^{\infty}((t_0-\delta,t_0+\delta)).
\end{equation}
For $n\ge n_0$, we write $\mathcal{O}_n:=\{(t,x):t\in (t_0-\delta, t_0+\delta), x>\ell_{n}(t)\}$, and define
$U_n:  [t_0-\delta, t_0+\delta]\times \R\to [0,1]$ by letting
\begin{equation} \label{eq:Undefn}
U_n(t,x):=\frac{2}{\beta}e^{-\frac{2}{\beta}x}\int_{-\infty}^x (\1_{\{z<\ell_n(t)\}} (1-n^{-1})+\1_{\{z\ge \ell_n(t)\}})e^{\frac{2}{\beta}z}V(t,z)dz.
\end{equation}
Then for $(t,x)\in \mathcal{O}_n$ we can write
\begin{equation}\label{eq:Un formula}
U_n(t,x)=e^{-\frac{2}{\beta}x}\left((1-n^{-1})e^{\frac{2}{\beta}\ell_{n}(t)}+\frac{2}{\beta}\int_{\ell_{n}(t)}^xe^{\frac{2}{\beta}z}V(t,z)dz\right),
\end{equation}
and so by~\eqref{eq:smoothness of level set} we have that
 $U_n\in C^{\infty}(\mathcal{O}_n)$.

It will be convenient to define, for $(t,x)\in \mathcal{O}_n$,
\begin{align}
A_n(t,x) &:=(1-n^{-1})e^{-\frac{2}{\beta}(x-\ell_{n}(t))} \label{eq:AndefnforUn}\\
\text{and } \quad B_n(t,x) &:=e^{-\frac{2}{\beta}(x-\ell_{n}(t))}\partial_xV(t,\ell_{n}(t)). \label{eq:BndefnforUn}
\end{align}
Write $\dot{\ell}_n(t)$ for the time derivative of $\ell_n(t)$.
For $(t,x)\in \mathcal{O}_n$, we can calculate directly from~\eqref{eq:Un formula} that
\begin{align} \label{eq:dtUncalc}
\partial_tU_n(t,x)&=e^{-\frac{2}{\beta}x}\left( (1-n^{-1}) \frac{2}{\beta}\dot{\ell}_{n}(t)e^{\frac{2}{\beta}\ell_{n}(t)}-\frac{2}{\beta}\dot{\ell}_{n}(t)e^{\frac{2}{\beta}\ell_{n}(t)}V(t,\ell_{n}(t))\right) \notag \\
&\qquad+\frac{2}{\beta}e^{-\frac{2}{\beta}x}\int_{\ell_{n}(t)}^xe^{\frac{2}{\beta}z}[\tfrac{1}{2}\Delta V+V](t,z)dz \notag \\
&= \frac{2}{\beta}e^{-\frac{2}{\beta}x}\int_{\ell_{n}(t)}^xe^{\frac{2}{\beta}z}[\tfrac{1}{2}\Delta V+V](t,z)dz,
\end{align}
where in the second equality we used the fact that, by~\eqref{eq:ellndefn}, $V(t,\ell_n(t))=1-n^{-1}$
to see that the first term is $0$. We then use integration by parts twice to calculate
\[
\begin{split}
\int_{\ell_n(t)}^xe^{\frac{2}{\beta}z}\Delta V(t,z)dz&=\left[e^{\frac{2}{\beta}z}\partial_z V(t,z)\right]_{z=\ell_{n}(t)}^{z=x}-\frac{2}{\beta}\int_{\ell_n(t)}^xe^{\frac{2}{\beta}z}\partial_z V(t,z)dz\\
&=\left[e^{\frac{2}{\beta}z}\partial_z V(t,z)\right]_{z=\ell_{n}(t)}^{z=x}-\frac{2}{\beta}\left[e^{\frac{2}{\beta}z}V(t,z)\right]_{z=\ell_{n}(t)}^{z=x}+\frac{4}{\beta^2}\int_{\ell_n(t)}^xe^{\frac{2}{\beta}z} V(t,z)dz.
\end{split}
\]
Therefore, substituting into~\eqref{eq:dtUncalc} and using~\eqref{eq:BndefnforUn}, and then using~\eqref{eq:Un formula},~\eqref{eq:AndefnforUn} and the fact that $V(t,\ell_n(t))=1-n^{-1}$, for $(t,x)\in \mathcal O_n$ we have
\begin{align} \label{eq:dtUncalc2}
\partial_tU_n(t,x)&=
\left(1+\frac{2}{\beta^2}\right)\frac{2}{\beta} e^{-\frac{2}{\beta}x}\int_{\ell_{n}(t)}^xe^{\frac{2}{\beta}z}V(t,z)dz
+\frac{1}{\beta}[\partial_xV(t,x)-B_n(t,x)] \notag \\
&\qquad -\frac{2}{\beta^2}[V(t,x)-e^{-\frac 2 \beta x}e^{\frac 2 \beta \ell_n(t)} V(t,\ell_n(t))] \notag \\
&=(U_n(t,x)-A_n(t,x))\left(1+\frac{2}{\beta^2}\right)+\frac{1}{\beta}[\partial_xV(t,x)-B_n(t,x)]-\frac{2}{\beta^2}[V(t,x)-A_n(t,x)].
\end{align}
We now calculate directly from \eqref{eq:Un formula} that for $(t,x)\in \mathcal O_n$,
\begin{align} \label{eq:DeltaUncalc}
\Delta U_n(t,x)&=\frac{4}{\beta^2}U_n(t,x)-\frac{8}{\beta^2}V(t,x)+\frac{2}{\beta}\left[\frac{2}{\beta}V(t,x)+\partial_xV(t,x)\right] \notag \\
&=\frac{4}{\beta^2}U_n(t,x)-\frac{4}{\beta^2}V(t,x)+\frac{2}{\beta}\partial_xV(t,x).
\end{align}
Therefore, combining~\eqref{eq:dtUncalc2} and~\eqref{eq:DeltaUncalc}, for $(t,x)\in \mathcal O_n$,
\begin{equation} \label{eq:Unheateqapprox}
[\partial_tU_n-\tfrac{1}{2}\Delta U_n-U_n](t,x)
=-A_n(t,x)-\frac{1}{\beta}B_n(t,x)=\left(\frac{-1}{(1-n^{-1})\beta}\partial_xV(t,\ell_{n}(t))-1\right)A_n(t,x),
\end{equation}
where the second equality follows from~\eqref{eq:AndefnforUn} and~\eqref{eq:BndefnforUn}.

It will be convenient to write $\mathcal{O}:=\{(t,x):t\in (t_0-\delta, t_0+\delta), x>L_t\}$. We now fix $\phi\in C_c^{\infty}(\mathcal O)$ and write $\mathrm{supp}(\phi)$ for the support of $\phi$. Then there exist $t_0-\delta<t_1<t_2<t_0+\delta $ such that $\mathrm{supp}(\phi)\subset [t_1,t_2]\times \R$. We also observe that $\{(t,x):t\in [t_0-\delta, t_0+\delta],x\leq \ell_n(t)\}\cap \mathrm{supp}(\phi)$ is a non-increasing sequence of compact sets with empty intersection, hence we can find some $n_1\ge n_0$ such that $\phi(t,x)= 0$ for all $n\ge n_1$, $x\leq \ell_{n}(t)$ and $t\in [t_0-\delta,t_0+\delta]$. Then by integration by parts, for $n\ge n_1$, we can write
\begin{equation}\label{eq:weak solution for U_n from integration by parts}
\begin{split}
\int_{\mathcal{O}}(-\partial_t\phi -\tfrac{1}{2}\Delta \phi-\phi)(t,x) U_n(t,x) dtdx&=\int_{\mathcal{O}}\phi(t,x)(\partial_tU_n-\tfrac{1}{2}\Delta U_n -U_n)(t,x)dtdx\\
&=\int_{\mathcal{O}}\phi(t,x)\left(\frac{-1}{(1-n^{-1})\beta}\partial_xV(t,\ell_{n}(t))-1\right)A_n(t,x)dtdx,
\end{split}
\end{equation}
where the second line follows from~\eqref{eq:Unheateqapprox}.

From~\eqref{eq:Undefn} and~\eqref{eq:UtfromVtclassical}, we can see 
that $U_n(t,x)$ converges to $U(t,x)$ pointwise on $\mathcal O$ as $n\to \infty$, and $0\le U_n\le 1$ on $\mathcal O$. Hence by the dominated convergence theorem the left-hand side of~\eqref{eq:weak solution for U_n from integration by parts} converges to
\begin{equation}\label{eq:weak formulation U}
\int_{\mathcal{O}}(-\partial_t\phi -\phi-\tfrac{1}{2}\Delta \phi)(t,x)U(t,x) dtdx
\end{equation}
as $n\ra \infty$. 

On the other hand, we see that for $t\in [t_0-\delta,t_0+\delta]$,
\begin{equation}
\left \lvert \frac{-1}{(1-n^{-1})\beta}\partial_xV(t,\ell_{n}(t))-1\right\rvert \ra 0 \quad \text{as }n\to \infty, \label{eq:dxVlimit}
\end{equation}
since $\ell_n(t)\to L_t$ and so $\partial_xV(t,\ell_{n}(t))\ra -\beta$ as $n\ra \infty$.
Moreover, the left-hand side of~\eqref{eq:dxVlimit} is bounded on $t\in [t_0-\delta,t_0+\delta]$ for $n\ge n_0$,
using~\eqref{eq:dxVnear-beta} and our choice of $n_0$ after~\eqref{eq:ellndefn}.
We also note from~\eqref{eq:AndefnforUn} that $0\le A_n\le 1$ on $\mathcal O_n$. It therefore follows from the dominated convergence theorem that the right-hand side of~\eqref{eq:weak solution for U_n from integration by parts} converges to $0$ as $n\to \infty$.

 Therefore~\eqref{eq:weak formulation U} is zero for all $\phi \in C_c^{\infty}(\mathcal{O})$,
 i.e.~$U$ is a weak solution of $\partial_tU=\frac{1}{2}\Delta U +U$ on $\mathcal O$, and hence it is a $C^{\infty}(\mathcal{O})$ classical solution by the aforementioned parabolic regularity for the heat equation (e.g. by H\"ormander's theorem~\cite{Hormander1967}). 

Since $t_0>0$ was arbitrary, this completes the proof of Proposition \ref{prop:mapping inversion}.
\end{proof}

\subsection{Proof of Proposition \ref{prop:travelling waves general beta}}
Recall that $\Pi^{(\beta)}_{c}$ was defined for $\beta>0$ and $c\geq \sqrt{2}$ in \eqref{eq:travelling waves general beta}. We now determine which travelling waves $\Pi^{(\beta)}_{c}$ are non-negative.

We define $a_c:=c-\sqrt{c^2-2}$ and $b_c=c+\sqrt{c^2-2}$ for $c>\sqrt{2}$. We have from \eqref{eq:minimal travelling wave}, \eqref{eq:speed c travelling wave}, \eqref{eq:Pimindefn} and \eqref{eq:Picdefn} that for $c\ge \sqrt 2$, there exists $Z_c\in (0,\infty)$ such that for $x>0$,
\begin{equation}\label{eq:travelling waves CDF equation proof generalised prop}
\Pi_{c}(x)=\begin{cases}
e^{-\sqrt{2}x}(1+\sqrt{2}x)\quad &\text{if }c=\sqrt{2},\\
Z_c[b_ce^{-a_cx}-a_ce^{-b_cx}]\quad &\text{if }c>\sqrt{2}.
\end{cases}
\end{equation}
Therefore, for $x>0$,
\begin{equation} \label{eq:Pibetacexpression}
\Pi^{(\beta)}_{c}(x)=\begin{cases}
e^{-\sqrt{2}x}[1+(\sqrt{2}-\beta)x] \quad &\text{if }c=\sqrt{2},\\
\frac 12 Z_c[b_c(2-a_c\beta)e^{-a_cx}-a_c(2-b_c\beta)e^{-b_cx}] \quad &\text{if }c>\sqrt{2}.
\end{cases}
\end{equation}
In particular, we see that $ \Pi^{(\beta)}_c(x)\to 0$ as $x\to \infty$.
The non-negativity of $\Pi^{(\beta)}_{c}(x)$ is equivalent to the non-negativity of $e^{\sqrt{2}x}\Pi^{(\beta)}_{\sqrt{2}}(x)$ when $c=\sqrt{2}$, or of $e^{a_cx}\Pi^{(\beta)}_{c}(x)$ when $c>\sqrt{2}$. We therefore see that $\Pi^{(\beta)}_{c}$ is non-negative if and only if
\[
\begin{cases}
1+(\sqrt{2}-\beta)x\geq 0 \text{ for all $x>0$}\quad &\text{if }c=\sqrt{2},\\
b_c(2-a_c\beta)-a_c(2-b_c\beta)e^{-(b_c-a_c)x}\geq 0 \text{ for all $x>0$},\quad &\text{if } c>\sqrt{2}.
\end{cases}
\]
Therefore $\Pi^{(\beta)}_{\sqrt{2}}$ is non-negative if and only if $\beta\leq \sqrt{2}$. 
For $c>\sqrt{2}$, we observe that
\[
x\mapsto b_c(2-a_c\beta)-a_c(2-b_c\beta)e^{-(b_c-a_c)x}
\]
is monotone on $(0,\infty)$, and so
$\Pi^{(\beta)}_{c}$ is non-negative if and only if $\min(b_c(2-a_c\beta),2(b_c-a_c))\geq 0$. We always have $2(b_c-a_c)\geq 0$ since $b_c\geq a_c$. Hence $\Pi^{(\beta)}_{c}$ is non-negative if and only if
\begin{equation}\label{eq:beta leq 2/ac}
a_c\beta\leq 2.
\end{equation}

We now consider the cases $\beta\leq \sqrt{2}$ and $\beta>\sqrt{2}$ separately.

\noindent \textbf{Case 1: $\beta\leq \sqrt{2}$.}
 Since $c\mapsto a_c$ is strictly decreasing on $[\sqrt{2},\infty)$, we have $a_c\leq a_{\sqrt{2}}=\sqrt{2}$ for all $c\geq \sqrt{2}$. From this, we see that $a_c\beta \leq 2$ for all $\beta\leq \sqrt{2}$ and $c\geq \sqrt{2}$. Therefore $\Pi^{(\beta)}_{c}$ is non-negative for all $\beta\leq \sqrt{2}$ and $c>\sqrt{2}$.

\noindent \textbf{Case 2: $\beta> \sqrt{2}$.} We can check that $c=\frac{\beta}{2}+\frac{1}{\beta}>\sqrt 2$ is a solution to $a_c \beta=2$. Thus, since $c\mapsto a_c$ is strictly decreasing on $[\sqrt{2},\infty)$, we see that \eqref{eq:beta leq 2/ac} is equivalent to $c\geq \frac{\beta}{2}+\frac{1}{\beta}$, for $\beta>\sqrt{2}$.

\noindent Overall, by the definition of $c^{(\beta)}_{\min}$ in~\eqref{eq:cbetamindef}, we see that $\Pi^{(\beta)}_{c}$ is non-negative if and only if $c\geq c^{(\beta)}_{\min}$. 

We now check that $\Pi^{(\beta)}_{c}$ is non-increasing for $c\geq c^{(\beta)}_{\min}$. Differentiating~\eqref{eq:Pibetacexpression}, we see that for $x>0$,
\[
\frac{d}{dx}\Pi^{(\beta)}_{c}(x)=\begin{cases}
e^{-\sqrt{2}x}[(\beta-\sqrt{2})\sqrt 2 x-\beta],\quad &c=\sqrt{2},\\
\tfrac 12 Z_c a_c b_c e^{-a_cx}[(a_c\beta-2)-(b_c\beta-2)e^{-(b_c-a_c)x}],\quad &c>\sqrt{2}.
\end{cases}
\]
Using that we have \eqref{eq:beta leq 2/ac} whenever $c\geq c^{(\beta)}_{\min}$, 
and using that $(a_c\beta-2)-(b_c\beta-2)=(a_c-b_c)\beta \le 0$,
we immediately see that $\Pi^{(\beta)}_c$ is non-increasing for all $c\geq c^{(\beta)}_{\min}$.

Finally, we establish~\eqref{eq:Pibetaminasymptotics}.
For $\beta\le \sqrt 2$, the asymptotics in~\eqref{eq:Pibetaminasymptotics} follow directly from~\eqref{eq:Pibetacexpression}.
For $\beta>\sqrt{2}$, letting $c=c^{(\beta)}_{\min}=\frac{\beta}{2}+\frac{1}{\beta}$ we have $b_c \beta >a_c \beta=2$ and $b_c=\beta$, and so again~\eqref{eq:Pibetaminasymptotics} follows directly from~\eqref{eq:Pibetacexpression}.
\qed

\subsection{Proof of Theorem \ref{theo:convtoPimin general beta}}

The following lemma, relating exponential moments of $V_0$ and $U_0$ under the mapping~\eqref{eq:U0 formula from V0}, will be used in the proofs of Theorems~\ref{theo:convtoPimin general beta} and~\ref{theo:Ltposition general beta}.
\begin{lem}\label{lem:U0 V0 exponential moment relationship}
Take $\beta>0$, and
suppose that $V_0$ satisfies Assumption~\ref{assum:standing assumption initial condition general beta}. 
Define $U_0$ as in~\eqref{eq:U0 formula from V0}, i.e.~let
\begin{equation} \label{eq:U0defninlemma}
U_0(x):=
\frac{2}{\beta}e^{-\frac{2}{\beta}x}\int_{-\infty}^xe^{\frac{2}{\beta}z}V_0(z)dz
\quad \forall x \in \R.
\end{equation}
Then for $r<\frac{2}{\beta}$, 
\begin{equation}
\int_{0}^{\infty}e^{rx}U_0(x)dx=\frac{2}{2-r\beta}\Big(\int_{0}^{\infty}e^{rx}V_0(x)dx+\int_{-\infty}^{0} e^{\frac{2}{\beta}x}V_0(x)dx\Big)
\label{eq:U0 V0 exponential moment relationship}\end{equation}
and
\begin{equation}
\begin{split}
\int_{0}^{\infty}xe^{rx}U_0(x)dx
&=\frac{2\beta}{(2-r\beta)^2}\Big(\int_{0}^{\infty}e^{rx}V_0(x)dx+\int_{-\infty}^{0} e^{\frac{2}{\beta}x}V_0(x)dx\Big)
 +\frac{2}{2-r\beta}\int_{0}^{\infty}xe^{rx}V_0(x)dx .
\end{split}
\label{eq:U0 V0 x exponential moment relationship 2}
\end{equation}
\end{lem}
\begin{proof}
For $r<\frac 2 {\beta}$,
by~\eqref{eq:U0defninlemma} and Tonelli's theorem,
\begin{align*}
\int_{0}^{\infty}e^{rx}U_0(x)dx
&= \frac{2}{\beta}\int_{-\infty}^{\infty} e^{\frac{2}{\beta}z}V_0(z) \int_{z\vee 0}^\infty e^{(r-\frac{2}{\beta})x}dx dz\\
&=\frac{2}{\beta}\int_{-\infty}^{\infty} e^{\frac{2}{\beta}z}V_0(z)\frac 1 {\frac 2 \beta -r} e^{(r-\frac{2}{\beta})(z\vee 0)} dz\\
&=\frac{2}{2-r\beta}\Big(\int_{0}^{\infty}e^{rx}V_0(x)dx+\int_{-\infty}^{0} e^{\frac{2}{\beta}x}V_0(x)dx\Big),
\end{align*}
as claimed in~\eqref{eq:U0 V0 exponential moment relationship}.

Now let
\[
r_0:=\sup \Big\{r\in \R:\int_{0}^{\infty}e^{rx}V_0(x)dx<\infty\Big\}.
\]
For $r<\min(r_0,\frac{2}{\beta})$, we can differentiate both sides of~\eqref{eq:U0 V0 exponential moment relationship} with respect to $r$ to obtain~\eqref{eq:U0 V0 x exponential moment relationship 2}. Then if $r_0<\frac{2}{\beta}$, we obtain~\eqref{eq:U0 V0 x exponential moment relationship 2} for $r=r_0$ by taking the limit of both sides of~\eqref{eq:U0 V0 x exponential moment relationship 2} as $r\uparrow r_0$. Finally, if $r_0<\frac{2}{\beta}$ and $r\in (r_0,\frac{2}{\beta})$, then~\eqref{eq:U0 V0 exponential moment relationship} implies that both sides of~\eqref{eq:U0 V0 x exponential moment relationship 2} are infinite.
\end{proof}

\begin{proof}[Proof of Theorem~\ref{theo:convtoPimin general beta}]
Define $U_0$ as in~\eqref{eq:U0 formula from V0}.
By Theorem~\ref{theo:mapping between FBP and general problem}, $U_0$ satisfies Assumption~\ref{assum:standing assumption ic};
let $(U(t,x),L_t)$ denote the solution of~\eqref{eq:FBP_CDF} with initial condition $U_0$.
By Theorems~\ref{theo:existence uniqueness classical solution V} and~\ref{theo:mapping between FBP and general problem}, the solution $(V(t,x),L_t)$ of~\eqref{eq:generalised FBP_CDF} is given by~\eqref{eq:V defined in terms of U}.
Since $V_0$ is non-increasing, and then in the second line using Lemma~\ref{lem:U0 V0 exponential moment relationship}, and then finally since $U_0$ is non-increasing,
\begin{equation}\label{eq:equivalence limsup V and U}
\begin{split}
\limsup_{x\ra\infty}\tfrac{1}{x}\log V_0(x)\leq -\min(\sqrt{2},\tfrac{2}{\beta})
&\Leftrightarrow \int_0^{\infty}e^{rx}V_0(x)dx<\infty\quad\text{for all $r<\min(\sqrt{2},\tfrac{2}{\beta})$}\\
&\Leftrightarrow \int_0^{\infty}e^{rx}U_0(x)dx<\infty\quad\text{for all $r<\min(\sqrt{2},\tfrac{2}{\beta})$}
\\&\Leftrightarrow\limsup_{x\ra\infty}\tfrac{1}{x}\log U_0(x)\leq -\min(\sqrt{2},\tfrac{2}{\beta}).
\end{split}
\end{equation}
We observe from \eqref{eq:U0 formula from V0} that necessarily
\[
\liminf_{x\rightarrow\infty}\tfrac{1}{x}\log U_0(x)\geq -\tfrac{2}{\beta}.
\]
In the case $\beta>\sqrt 2$, by~\eqref{eq:equivalence limsup V and U} it follows that 
\begin{equation}\label{eq:equivalence of limsup V0 and lim U0}
\limsup_{x\ra\infty}\tfrac{1}{x}\log V_0(x)\leq -\tfrac{2}{\beta}
\quad \Leftrightarrow \quad \lim_{x\ra\infty}\tfrac{1}{x}\log U_0(x)= -\tfrac{2}{\beta}.
\end{equation}

\noindent \textbf{Equivalence of \ref{enum:bound on tails of U0 general beta}-\ref{enum:convergence of velocity general beta}:}

If $\beta\leq \sqrt{2}$ then since $c^{(\beta)}_{\min}=\sqrt{2}$ by~\eqref{eq:cbetamindef}, we can immediately conclude the equivalence of \ref{enum:bound on tails of U0 general beta}-\ref{enum:convergence of velocity general beta} in Theorem \ref{theo:convtoPimin general beta} by applying Theorem \ref{theo:convtoPimin} and~\eqref{eq:equivalence limsup V and U}. 

Suppose instead that $\beta>\sqrt{2}$. 
Using that $c^{(\beta)}_{\min}=\frac{\beta}{2}+\frac{1}{\beta}$ by~\eqref{eq:cbetamindef}, we can check that 
\begin{equation} \label{eq:cbetaminidentity}
c^{(\beta)}_{\min}>\sqrt 2 \quad \text{and} \quad -c^{(\beta)}_{\min}+\sqrt{(c^{(\beta)}_{\min})^2-2}=-\tfrac{2}{\beta}.
\end{equation} 
Therefore, by~\eqref{eq:equivalence of limsup V0 and lim U0} and Theorem~\ref{theo:slowerdecay}, we have
\begin{equation} \label{eq:limsupV0toLtlim}
\limsup_{x\ra\infty}\tfrac{1}{x}\log V_0(x)\leq -\tfrac{2}{\beta}
\quad \Rightarrow \quad \lim_{t\ra\infty}\tfrac{L_t}{t}= c^{(\beta)}_{\min}.
\end{equation}
By Theorem~\ref{theo:initial condition limsup bdy relation}, if $\limsup_{t\ra\infty}\tfrac{L_t}{t}\le c^{(\beta)}_{\min}$ then
$\frac{1}{r_0}+\frac{r_0}{2}\le c^{(\beta)}_{\min}$, where
\[
r_0:=\sup\left( \{0\}\cup \Big\{r\in (0,\sqrt{2}):\int_{0}^{\infty}e^{rx}U_0(x)dx<\infty\Big\}\right).
\]
Since $r\mapsto \frac 1r +\frac r2$ is decreasing on $(0,\sqrt 2)$, we must have $r_0\ge \frac 2 \beta$; it follows from~\eqref{eq:equivalence limsup V and U} that
\begin{equation} \label{eq:limsupLttoV0}
\limsup_{t\ra\infty}\tfrac{L_t}{t}\le c^{(\beta)}_{\min}
\quad \Rightarrow \quad \limsup_{x\ra\infty}\tfrac{1}{x}\log V_0(x)\leq -\tfrac{2}{\beta}.
\end{equation}
Therefore, by~\eqref{eq:limsupV0toLtlim} and~\eqref{eq:limsupLttoV0}, and since trivially $\lim_{t\ra\infty}\frac{L_t}{t}= c^{(\beta)}_{\min}$ implies
$\limsup_{t\ra\infty}\frac{L_t}{t}\leq c^{(\beta)}_{\min}$, we have
the equivalence of \ref{enum:bound on tails of U0 general beta}-\ref{enum:convergence of velocity general beta}.

\medskip

\noindent \textbf{\ref{enum:bound on tails of U0 general beta}} $\Rightarrow$ \textbf{\ref{enum:convergence of profile general beta}}:

Suppose $\limsup_{x\ra\infty}\tfrac{1}{x}\log V_0(x)\leq -\min(\sqrt{2},\tfrac{2}{\beta})$.
In the case $\beta\le \sqrt 2$, we use~\eqref{eq:equivalence limsup V and U} and Theorem~\ref{theo:convtoPimin}, and in the case $\beta> \sqrt 2$, we use~\eqref{eq:equivalence of limsup V0 and lim U0},~\eqref{eq:cbetaminidentity} and Theorem~\ref{theo:slowerdecay} to see that
\[
U(t, L_t+x)\ra \Pi_{c^{(\beta)}_{\min}}(x) \quad \text{uniformly in }x \text{ as }t\to\infty.
\]
Hence by dominated convergence, for any $r<0$,
\[
\int_0^{\infty}e^{rx}U(t, L_t+x)dx\ra \int_0^{\infty}e^{rx}\Pi_{c^{(\beta)}_{\min}}(x)dx \quad \text{as }t\ra\infty.
\]
It then follows from~\eqref{eq:V defined in terms of U},~\eqref{eq:travelling waves general beta} and integration by parts that
\[
\int_0^{\infty}e^{rx}V(t,L_t+x)dx\ra \int_0^{\infty}e^{rx}\Pi^{(\beta)}_{\min}(x)dx \quad \text{as }t\ra\infty.
\]

By Theorem \ref{theo:mapping between FBP and general problem}, for any $t>0$ we have that $V(t,L_t+\cdot)$ is non-increasing, $[0,1]$-valued and continuous.
Using Helly's selection theorem, we see that for any sequence of times, we can find a subsequence of times $(t_n)_{n\in \mathbb N_0}$ along which $V(t_n,L_{t_n}+\cdot)$ converges pointwise. Let $\tilde V$ denote a 
subsequential limit; then $\tilde V(x)=1$  $\forall x\in (-\infty,0)$, and for any $r<0$, by dominated convergence,
\begin{equation}\label{eq:equality of subsequential limit and Pimin exponential moments 1}
\int_0^{\infty}e^{rx}\tilde{V}(x)dx= \int_0^{\infty}e^{rx}\Pi^{(\beta)}_{\min}(x)dx.
\end{equation}

By conformal extension, we have that~\eqref{eq:equality of subsequential limit and Pimin exponential moments 1} holds for any $r\in \mathbb{C}$ with $\text{Re}(r)<0$, and by the monotone convergence theorem,~\eqref{eq:equality of subsequential limit and Pimin exponential moments 1} holds for $r=0$. Then by~\eqref{eq:Pibetaminasymptotics} in Proposition~\ref{prop:travelling waves general beta}, and using that~\eqref{eq:equality of subsequential limit and Pimin exponential moments 1} holds for $r=0$, $\tilde V$ is integrable on $(0,\infty)$, and so both sides of~\eqref{eq:equality of subsequential limit and Pimin exponential moments 1} are continuous on $\{r\in \mathbb{C}:\text{Re}(r)\leq 0\}$.

Therefore~\eqref{eq:equality of subsequential limit and Pimin exponential moments 1} holds for all imaginary $r$, and so $\tilde{V}$ and $\Pi^{(\beta)}_{\min}$ (restricted to $(0,\infty)$) have the same Fourier transform, and hence must be equal. 

Since (1) $V(t_n,L_{t_n}+\cdot)$ are $[0,1]$-valued and non-increasing for all $n$, (2) $\Pi^{(\beta)}_{\min}(x)\ra 0$ as $x\ra \infty$ and $\Pi^{(\beta)}_{\min}(x)\ra1$ as $x\ra -\infty$, and (3) $\Pi^{(\beta)}_{\min}$ is continuous, the pointwise convergence of $V(t_n,L_{t_n}+\cdot)$ to $\Pi^{(\beta)}_{\min}$ is in fact uniform. This can be seen by the same argument as at the end of the proof of Theorem~\ref{theo:convergence to the minimal travelling wave for finite initial mass}, i.e.~for arbitrary $n$, taking $x_1<\ldots<x_{n-1}$ such that $\Pi^{(\beta)}_{\min}(x_k)=k/n$ for all $k\in \{1,\ldots , n-1\}$, then combining convergence at each $x_k$ with monotonicity.

Therefore $V(t,x+L_t)\ra \Pi^{(\beta)}_{\min}(x)$ uniformly in $x$ as $t\ra\infty$.

\medskip

\noindent \textbf{\ref{enum:convergence of profile general beta}} $\Rightarrow$ \textbf{\ref{enum:bound on tails of U0 general beta}}:

Suppose $V(t,x+L_t)\ra \Pi^{(\beta)}_{\min}(x)$ uniformly in $x$ as $t\ra\infty$.
By~\eqref{eq:V defined in terms of U} and integration by parts, for $t>0$ and $x>0$ we have
\[
e^{-\frac{2}{\beta}x}\left(1+\tfrac{2}{\beta}\int_{0}^xe^{\frac{2}{\beta}z}V(t,z+L_t)dz\right)=U(t,x+L_t).
\]
Therefore, for any $x>0$, as $t\to\infty$,
\[
U(t,x+L_t)\to e^{-\frac{2}{\beta}x}\left(1+\tfrac{2}{\beta}\int_{0}^xe^{\frac{2}{\beta}z}\Pi^{(\beta)}_{\min}(z) dz\right)
=\Pi_{c^{(\beta)}_{\min}}(x),
\]
where the equality follows from~\eqref{eq:travelling waves general beta},~\eqref{eq:generalbetaminTW}
and integration by parts.
By the same argument as at the end of the proof of Theorem~\ref{theo:convergence to the minimal travelling wave for finite initial mass}, it follows that $U(t,x+L_t)\ra \Pi_{c^{(\beta)}_{\min}}(x)$ uniformly in $x$ as $t\ra\infty$.
Therefore, by Theorem~\ref{theo:convtoPimin} in the case $\beta\le \sqrt 2$, and by Theorem~\ref{theo:slowerdecay} and~\eqref{eq:cbetaminidentity} in the case $\beta> \sqrt 2$,
we have $\limsup_{x\ra\infty}\tfrac{1}{x}\log U_0(x)\leq -\min(\sqrt{2},\tfrac{2}{\beta})$.
By~\eqref{eq:equivalence limsup V and U} it follows that
$\limsup_{x\ra\infty}\tfrac{1}{x}\log V_0(x)\leq -\min(\sqrt{2},\tfrac{2}{\beta})$. 
\end{proof}

\subsection{Proof of Theorem \ref{theo:Ltposition general beta}}

Define $U_0$ as in~\eqref{eq:U0 formula from V0}.
By Theorem~\ref{theo:mapping between FBP and general problem}, $U_0$ satisfies Assumption~\ref{assum:standing assumption ic}, and 
letting $(U(t,x),L_t)$ denote the solution of~\eqref{eq:FBP_CDF} with initial condition $U_0$,
by Theorems~\ref{theo:existence uniqueness classical solution V} and~\ref{theo:mapping between FBP and general problem} we have that
$L_t$ is the free boundary in the solution $(V(t,x),L_t)$ of~\eqref{eq:generalised FBP_CDF}.

We separate the proof of~\ref{enum:pulled Lt generalised}-\ref{enum:pushed Lt generalised} and~\ref{enum:infinitely far from pulled beta<sqrt 2})-\ref{enum:infinitely far from pulled beta>sqrt 2}) into two cases: $\beta\leq \sqrt{2}$ and $\beta>\sqrt{2}$.

\noindent \underline{$\beta\leq \sqrt{2}$}

Since $\limsup_{x\ra\infty}\frac{1}{x}\log V_0(x)\leq -\sqrt{2}$, \eqref{eq:equivalence limsup V and U} implies that $\limsup_{x\ra\infty}\frac{1}{x}\log U_0(x)\leq -\sqrt{2}$. 
Recall the definition of $I_\beta$ in~\eqref{eq:I integral finite initial mass general beta} and the definition of finite initial mass for $U_0$ in~\eqref{eq:finite initial mass U0}.
We observe that for $\beta<\sqrt 2$, Lemma~\ref{lem:U0 V0 exponential moment relationship} implies that $I_\beta<\infty$ if and only if~\eqref{eq:finite initial mass U0} holds. 
Moreover, in the case $\beta=\sqrt{2}$,~\eqref{eq:U0 V0 x exponential moment relationship 2} in Lemma~\ref{lem:U0 V0 exponential moment relationship} holds for $r<\sqrt 2$, and taking the limit on both sides as $r\uparrow \sqrt 2$, 
we see that $\int_0^{\infty}xe^{\sqrt 2 x}U_0(x)dx=\infty$.

Suppose $V_0$ satisfies~\eqref{eq:stretched exponential U0 general beta} for some $\gamma<1/2$;
we now check that $U_0$ satisfies~\eqref{eq:stretched exponential U0}. In the following, $C,R<\infty$ are positive constants and $C<\infty$ may increase from line to line. For $x\in \R$, we have
\[
\int_{-\infty}^xe^{\frac{2}{\beta}z}V_0(z)dz \leq C+\1_{\{x\ge R\}}\int_{R}^xe^{(\frac{2}{\beta}-\sqrt{2})z+z^{\gamma}}dz\leq C(1+xe^{(\frac{2}{\beta}-\sqrt{2})x+x^{\gamma}}),
\]
where in the second inequality we are using that $\frac{2}{\beta}\geq \sqrt{2}$. Therefore, by~\eqref{eq:U0 formula from V0}, for $x\ge 0$,
\[
U_0(x)\leq \frac{2}{\beta}e^{-\frac{2}{\beta}x}C(1+xe^{(\frac{2}{\beta}-\sqrt{2})x+x^{\gamma}})\leq C(1+xe^{-\sqrt{2}x+x^{\gamma}}).
\]
Increasing $\gamma<1/2$ slightly, we see that $U_0$ satisfies~\eqref{eq:stretched exponential U0}. 

We can therefore conclude both the whole of~\ref{enum:pulled Lt generalised} and~\ref{enum:infinitely far from pulled beta<sqrt 2})  (the $\beta<\sqrt{2}$ case) and the convergence~\eqref{eq:asymptotics 1 of Lt beta sqrt 2} in~\ref{enum:pushmi-pullyu Lt generalised} (the $\beta=\sqrt{2}$ case) by applying Theorem~\ref{theo:Ltposition} to $(U(t,x),L_t)$.

We now establish the convergence~\eqref{eq:Lt asymp finite init mass beta sqrt 2}; we set $\beta=\sqrt{2}$, assume $I_{\sqrt 2}<\infty$ and compute $b(t)$, defined in~\eqref{eq:definition of b general beta}.  We define
\begin{equation} \label{eq:Jydefn}
J_y:=\int_{-\infty}^{y}e^{\sqrt{2}x}V_0(x)dx\quad\text{for $y\in \R$,}
\end{equation}
and observe that by~\eqref{eq:I integral finite initial mass general beta} we have $J_y\ra I_{\sqrt 2}$ as $y\ra\infty$. We then calculate
\begin{equation} \label{eq:btcalc}
\begin{split}
b(t)&=2^{-1/2}\log\Big(\sqrt 2\int_{0}^{\infty}yJ_y e^{-y^2/(2t)}dy+1\Big)\\
&=2^{-1/2}\log\Big(\sqrt 2\int_{0}^{\infty}yI_{\sqrt 2} e^{-y^2/(2t)}dy
+\sqrt 2\int_{0}^{\infty}y(J_y-I_{\sqrt 2})e^{-y^2/(2t)}dy+1\Big).
\end{split}
\end{equation}
For any $\epsilon>0$, we can take $R_\epsilon<\infty$ such that $\lvert I_{\sqrt 2}-J_y\rvert\leq \epsilon$ for all $y\geq R_\epsilon$. Then
\[
\Big\lvert \int_{0}^{\infty}{\sqrt 2}y(J_y-I_{\sqrt 2})e^{-y^2/(2t)}dy \Big\rvert \leq I_{\sqrt 2}\Big\lvert \int_0^{R_\epsilon}\sqrt 2 y dy\Big\rvert +\epsilon\Big\lvert \int_0^{\infty}\sqrt 2 ye^{-y^2/(2t)}dy\Big\rvert.
\]
Therefore, for any $\epsilon>0$, there exists $C_{\epsilon}<\infty$ such that for all $t<\infty$,
\[
\Big\lvert \int_{0}^{\infty}{\sqrt 2}y(J_y-I_{\sqrt 2})e^{-y^2/(2t)}dy \Big\rvert \leq C_{\epsilon}+\sqrt 2 \epsilon t.
\]
We now calculate
\begin{equation} \label{eq:btcalc1}
\sqrt 2\int_{0}^{\infty}yI_{\sqrt 2} e^{-y^2/(2t)}dy
=\sqrt 2I_{\sqrt 2} t.
\end{equation}
Moreover, we cannot have $I_{\sqrt 2}=0$, since this would imply $V_0\equiv 0$.
It follows that
\begin{equation} \label{eq:btcalc2}
\frac{\left\lvert \int_{0}^{\infty}y(J_y-I_{\sqrt 2})e^{-y^2/(2t)}dy \right\rvert}{\int_{0}^{\infty}yI_{\sqrt 2}e^{-y^2/(2t)}dy} \ra 0 \quad \text{as }t\ra\infty.
\end{equation}
Combining~\eqref{eq:btcalc1} and~\eqref{eq:btcalc2} with~\eqref{eq:btcalc}, we see that
\[
b(t)=2^{-1/2}\log\Big({\sqrt 2}I_{\sqrt 2} t\Big)+o(1)=2^{-1/2}\log t+2^{-1/2}\log({\sqrt 2}I_{\sqrt 2})+o(1)
\]
as $t\ra\infty$. Therefore $m(t)$, defined in \eqref{eq:definition of m general beta}, is given by
\[
m(t)=\sqrt{2}t-\frac{1}{2\sqrt{2}}\log t+\frac{1}{\sqrt{2}}\log({\sqrt 2}I_{\sqrt 2})+o(1)
\]
as $t\ra\infty$. Inputting this into \eqref{eq:asymptotics 1 of Lt beta sqrt 2}, we obtain \eqref{eq:Lt asymp finite init mass beta sqrt 2}.

Finally, we prove~\ref{enum:infinitely far from pulled beta=sqrt 2}); we set $\beta=\sqrt{2}$ and assume $I_{\sqrt 2}=\infty$.
Take $R>0$, and let
$V^R_0(x):=V_0(x)\Ind_{\{x< R\}}$ for $x\in \R$.
Then $V^R_0$ satisfies Assumption~\ref{assum:standing assumption initial condition general beta} and~\eqref{eq:stretched exponential U0 general beta}. We let $(V^R,L^R)$ denote the solution of~\eqref{eq:generalised FBP_CDF} with initial condition $V^R_0$. Then
by Theorems~\ref{theo:existence uniqueness classical solution V} and~\ref{theo:mapping between FBP and general problem}, and
 by the monotonicity of the mapping~\eqref{eq:U0 formula from V0} and the comparison principle (Proposition \ref{prop:fbpcomparison}), we see that $L^R_t\leq L_t$ for all $t>0$. 
Defining $J_y$ as in~\eqref{eq:Jydefn}, we have from~\eqref{eq:I integral finite initial mass general beta} that $I_{\sqrt 2}(V_0^R)=J_R<\infty$, 
and so by~\eqref{eq:Lt asymp finite init mass beta sqrt 2},
\[
L_t^R=\sqrt{2}t-\frac{1}{2\sqrt{2}}\log t+\frac{1}{\sqrt{2}}\Big(\log(\sqrt 2 J_R)-\log\sqrt{\pi}\Big)+o(1) \quad\text{as } t\ra\infty.
\]
Since $I_{\sqrt 2}=\infty$ we have $J_R\to \infty$ as $R\to \infty$, and $R>0$ was arbitrary,  and so the claim~\ref{enum:infinitely far from pulled beta=sqrt 2}) follows.

\medskip

\noindent \underline{$\beta>\sqrt{2}$}

Since $\limsup_{x\rightarrow\infty}\frac{1}{x}\log V_0(x)\leq -\frac{2}{\beta}$, \eqref{eq:equivalence of limsup V0 and lim U0} implies that $\lim_{x\rightarrow\infty}\frac{1}{x}\log U_0(x)= -\frac{2}{\beta}$. 
Since $c^{(\beta)}_{\min}=\frac{\beta}{2}+\frac{1}{\beta}$ by~\eqref{eq:cbetamindef}, letting $c=c^{(\beta)}_{\min}>\sqrt 2$ we have 
\[
-c+\sqrt{c^2-2}=-\frac{2}{\beta}
\qquad \text{and}\qquad
\sqrt{c^2-2}\left(c-\sqrt{c^2-2}\right)=\frac{\beta^2-2}{\beta^2}.
\]
We can therefore use Theorem~\ref{theo:slowerdecay} applied to $(U(t,x),L_t)$ to obtain \eqref{eq:Lt in terms of mt}. 

We now assume that $I_\beta<\infty$ and establish~\eqref{eq:Lt asymp finite init mass beta grt sqrt 2}, which we shall accomplish by computing $m(t)$, defined in~\eqref{eq:mtslowdecay generalised}. We write
\[
U_0(x)=e^{-kx}f(x) \quad \text{for }x\in \R,
\]
where we define $k:=\frac{2}{\beta}>0$, and
\[
f(x):=k\int_{-\infty}^{x}e^{ky}V_0(y)dy \quad \text{for }x\in \R
\]
Then, recalling the definition of $I_{\beta}$ from~\eqref{eq:I integral finite initial mass general beta}, we have that $f$ is non-decreasing and bounded, and $f(x)\ra kI_{\beta}\in (0,\infty)$ as $x\rightarrow \infty$.

For $t>0$, to compute $m(t)$, 
we substitute $z=(y-x)/\sqrt{t}$, and then $u=z+k\sqrt{t}$, obtaining
\[
\begin{split}
\int_{-\infty}^{\infty}U_0(y)\frac{e^{-(x-y)^2/(2t)}}{\sqrt{2\pi t}}dy
&=\int_{-\infty}^{\infty}U_0( x +\sqrt{t}z)\frac{e^{-z^2/2}}{\sqrt{2\pi }}dz\\
&=e^{-kx}\int_{-\infty}^{\infty}f( x +\sqrt{t}z)e^{-k\sqrt{t}z}\frac{e^{-z^2/2}}{\sqrt{2\pi }}dz\\
&=\frac{e^{\frac 12 k^2t-kx}}{\sqrt{2\pi }}\int_{-\infty}^{\infty}f( x +\sqrt{t}z)e^{-(z+k\sqrt{t})^2/2}dz\\
&=\frac{e^{\frac 12 k^2t-kx}}{\sqrt{2\pi }}\int_{-\infty}^{\infty}f( x +\sqrt{t}u-kt)e^{-u^2/2}du.
\end{split}
\]
We now recall from~\eqref{eq:cbetamindef} that $c^{(\beta)}_{\min}=\frac{1}{k}+\frac{k}{2}$, and so, in particular, $\frac 12 k^2-kc^{(\beta)}_{\min}+1=0$. Then letting $x'=x-c^{(\beta)}_{\min}t$, we can write that for $t>0$,
\begin{equation} \label{eq:intformupper}
e^t\int_{-\infty}^{\infty}U_0(y)\frac{e^{-(x-y)^2/(2t)}}{\sqrt{2\pi t}}dy
=e^{-kx'}\int_{-\infty}^{\infty}f( x' +\sqrt{t}u+(\tfrac{1}{k}-\tfrac{k}{2})t)\frac{e^{-u^2/2}}{\sqrt{2\pi }}du.
\end{equation}
Since $\frac{1}{k}>\frac{k}{2}$, by dominated convergence we see that for any fixed $x'\in \R$, the above integral converges to $kI_{\beta}$ as $t\ra\infty$. 
Therefore, 
for $x'\in \R$ fixed, letting $x=x'+c^{(\beta)}_{\min}t$ we have
\[
e^t\int_{-\infty}^{\infty}U_0(y)\frac{e^{-(x-y)^2/(2t)}}{\sqrt{2\pi t}}dy \to e^{-kx'}kI_{\beta} \quad \text{as }t\to \infty.
\]
By the definition of $m(t)$ in~\eqref{eq:mtslowdecay generalised}, it follows that $m(t)=c^{(\beta)}_{\min}t+k^{-1}\log (kI_{\beta})+o(1)$ as $t\to \infty$.
We now substitute this into \eqref{eq:Lt in terms of mt}, obtaining \eqref{eq:Lt asymp finite init mass beta grt sqrt 2}.

We finally turn to establishing~\ref{enum:infinitely far from pulled beta>sqrt 2}). Assume $I_\beta = \infty$.
As in the proof of~\ref{enum:infinitely far from pulled beta=sqrt 2}),
take $R>0$, and let
$V^R_0(x):=V_0(x)\Ind_{\{x< R\}}$ for $x\in \R$.
Then $V^R_0$ satisfies Assumption~\ref{assum:standing assumption initial condition general beta}.
Let $(V^R,L^R)$ denote the solution of~\eqref{eq:generalised FBP_CDF} with initial condition $V^R_0$; by Theorems~\ref{theo:existence uniqueness classical solution V} and~\ref{theo:mapping between FBP and general problem}, the monotonicity of the mapping~\eqref{eq:U0 formula from V0} and Proposition \ref{prop:fbpcomparison}, we have $L^R_t\leq L_t$ for all $t>0$.
Since $I_\beta(V_0^R)<\infty$, we have 
by~\eqref{eq:Lt asymp finite init mass beta grt sqrt 2} that
\[
L^R_t=c^{(\beta)}_{\min}t+\frac{\beta}{2} \left(\log (\tfrac{2}{\beta}I_{\beta}(V_0^R))+\log(\beta^2-2)-2\log \beta\right)+o(1)\quad\text{as } t\ra\infty.
\]
Then since $I_\beta = \infty$ we have $I_\beta(V_0^R)\to \infty$ as $R\to \infty$, and $R>0$ was arbitrary, so the claim~\ref{enum:infinitely far from pulled beta>sqrt 2}) follows.\qed

\bibliography{freeboundary.bib}
\bibliographystyle{plain}
\end{document}